\documentclass[a4paper,10pt]{article}
\usepackage{rotating}
\usepackage[english]{babel}
\usepackage{mathtools}
\usepackage{amsthm}
\usepackage{amssymb}
\usepackage{amsmath}
\usepackage{amsfonts}
\usepackage{url}
\usepackage{hyperref}
\usepackage{algorithm}
\usepackage{algorithmic}
\usepackage[margin=0.9in]{geometry}
\usepackage{cite}
\usepackage[capitalize]{cleveref}
\usepackage{times}
\usepackage{pst-node}
\usepackage{tikz-cd}
\usepackage{yfonts}
\usepackage{textcomp}
\usepackage{enumitem}
\usepackage{overpic}
\newcommand{\koop}{\mathcal{K}}
\newcommand{\dom}[1]{\mathcal{D}(#1)}
\newcommand{\injmod}{\sigma_{\mathrm{inf}}}
\newcommand{\m}[1]{\mathbf{#1}}

\DeclareUnicodeCharacter{02BC}{'}
\DeclareMathOperator{\dist}{dist}
\DeclareMathOperator{\spann}{span}
\DeclareMathOperator{\diag}{diag}

\DeclareMathOperator{\spec}{Sp}
\DeclareMathOperator{\res}{res}
\DeclareMathOperator{\closure}{Cl}
\DeclareMathOperator{\gridop}{Grid}

\newcommand*{\dd}{{\,\mathrm{d}}}
\newcommand*{\Kv}{{\mathbb{K}}}

\newtheorem{theorem}{Theorem}[section]
\newtheorem{proposition}[theorem]{Proposition}
\newtheorem{lemma}[theorem]{Lemma}
\newtheorem{example}[theorem]{Example}
\newtheorem{definition}[theorem]{Definition}

\newtheorem{corollary}[theorem]{Corollary}
\newtheorem{remark}[theorem]{Remark}
\usepackage{booktabs}

\DeclareUnicodeCharacter{211D}{\mathbb{R}}

\numberwithin{equation}{section}
\counterwithout{figure}{section}

\def\thick{0.8}

\date{}

\usepackage{tocloft}

\begin{document}

\title{\vspace{0mm}Convergent Methods for Koopman Operators\\on Reproducing Kernel Hilbert Spaces}
\author{Nicolas Boull\'e\thanks{Department of Mathematics, Imperial College London, London, SW7 2AZ UK (n.boulle@imperial.ac.uk).}\and Matthew Colbrook\thanks{DAMTP, University of Cambridge, Cambridge, CB3 0WA UK (m.colbrook@damtp.cam.ac.uk, gjc51@cam.ac.uk).}\and Gustav Conradie\footnotemark[2]}
\maketitle

\begin{abstract}
    Data-driven spectral analysis of Koopman operators is a powerful tool for understanding numerous real-world dynamical systems, from neuronal activity to variations in sea surface temperature. The Koopman operator acts on a function space and is most commonly studied on the space of square-integrable functions. However, defining it on a suitable reproducing kernel Hilbert space (RKHS) offers numerous practical advantages, including pointwise predictions with error bounds, improved spectral properties that facilitate computations, and more efficient algorithms, particularly in high dimensions. We introduce the first general, provably convergent, data-driven algorithms for computing spectral properties of Koopman and Perron–Frobenius operators on RKHSs. These methods efficiently compute spectra and pseudospectra with error control and spectral measures while exploiting the RKHS structure to avoid the large-data limits required in the $L^2$ settings. The function space is determined by a user-specified kernel, eliminating the need for quadrature-based sampling as in $L^2$ and enabling greater flexibility with finite, externally provided datasets. Using the Solvability Complexity Index hierarchy, we construct adversarial dynamical systems for these problems to show that no algorithm can succeed in fewer limits, thereby proving the optimality of our algorithms. Notably, this impossibility extends to randomized algorithms and datasets. We demonstrate the effectiveness of our algorithms on challenging, high-dimensional datasets arising from real-world measurements and high-fidelity numerical simulations, including turbulent channel flow, molecular dynamics of a binding protein, Antarctic sea ice concentration, and Northern Hemisphere sea surface height. The algorithms are publicly available in the software package \texttt{SpecRKHS}.

    \vspace{3mm}

    \vspace{1mm}
\end{abstract}

\tableofcontents

\section{Introduction}

Dynamical systems underpin numerous scientific and engineering disciplines, including classical mechanics, neuroscience, fluid dynamics, and finance~\cite{strogatz2024nonlinear}. Understanding how a system's state evolves over time is essential for tasks such as prediction, control, and pattern discovery~\cite{cross1993pattern,skogestad2005multivariable}. We consider a discrete-time, autonomous dynamical system $(F,\mathcal{X})$ governed by the evolution equation:
\begin{equation}
    \label{the_dynamical_system}
    x_{n+1} = F(x_n),\quad n\in\mathbb{N},\,x_0\in\mathcal{X},
\end{equation}
where $\mathcal{X}$ is the state space and $F:\mathcal{X}\rightarrow\mathcal{X}$. Such discrete-time formulations are natural when studying iterated maps or working with time-series data sampled from an underlying continuous process. Since the 1960s, the interplay between numerical analysis and dynamical systems has driven significant advances in both fields~\cite{kalman1963mathematical,stuart1998dynamical}. The rise of big data~\cite{hey2009fourth}, statistical learning~\cite{friedman2017elements}, and machine learning~\cite{mohri2018foundations} has led to the growing importance of data-driven methods for analyzing and understanding complex dynamics~\cite{schmidt2009distilling}. Analyzing dynamical systems presents several challenges. Local models offer insight near fixed points but fail to capture global nonlinear behavior, which remains a major open problem~\cite{brunton2019data}. In practice, limited system knowledge or sparse observations often preclude local modeling, and high dimensionality further complicates geometric analysis.

Koopman operator theory provides a powerful framework for analyzing nonlinear dynamics through a form of global linearization, originally introduced by Koopman and von Neumann in the early $20^{\rm th}$ century~\cite{koopman1931hamiltonian,koopman1932dynamical}. More recently, the development of interpretable modal decompositions and practical algorithms for computing Koopman spectral properties has fueled a surge of interest in data-driven approaches for analysis, model reduction, prediction, and control. This growing popularity has even earned the label `Koopmanism'~\cite{budivsic2012applied}; see reviews \cite{mezic2013analysis,budivsic2012applied,brunton2021modern,colbrook2023multiverse}. Notable successes include robot control \cite{haggerty2023control}, identifying patterns in climate variability \cite{froyland2021spectral}, state-of-the-art training of neural networks \cite{orvieto2023resurrecting}, studying disease spread \cite{proctor2015discovering}, analyzing neural brain recordings \cite{brunton2016extracting}, non-autonomous systems \cite{froyland2024revealing}, and interpretable neural networks \cite{lusch2018deep}.

This paper introduces the first general suite of provably convergent algorithms for the data-driven approximation of Koopman operators on reproducing kernel Hilbert spaces (RKHSs). This setting has grown rapidly in recent years due to its flexibility and suitability for modern data-driven methods, and it presents several advantages over classical square-integrable ($L^2$) function spaces. We prove that our algorithms are optimal by establishing impossibility results showing that (even randomized) algorithms cannot converge with a probability greater than 2/3. Our contributions are summarized in \cref{fig:introdiagram} and outlined in \cref{sec:contributions}.

\begin{figure}[th]
    \centering
    \begin{overpic}[width=\textwidth]{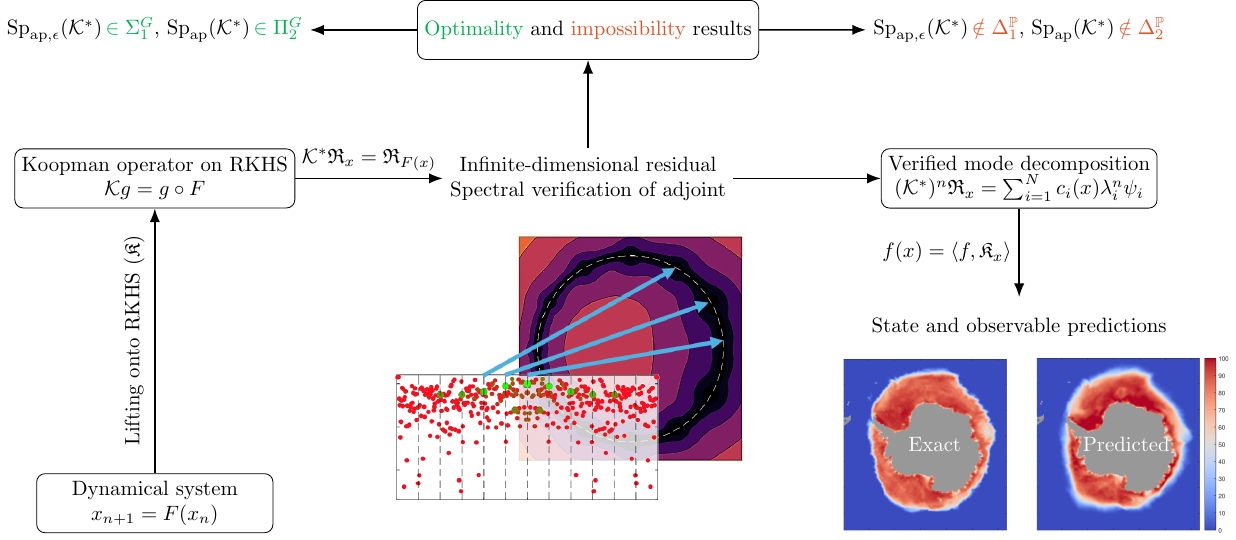}
    \end{overpic}
    \caption{Schematic of the main contributions summarized in \cref{algorithms-table,sci-classification-table}. We compute approximations of Koopman operators associated with dynamical systems on reproducing kernel Hilbert spaces (RKHS), and introduce a collection of provably convergent algorithms for computing spectral properties such as eigenvalues, approximate point pseudospectrum, spectral measures, and perform accurate state and observable predictions. The algorithms are supported by impossibility results, framed in the SCI hierarchy, to quantify the difficulty of each problem and prove the optimality of algorithms.}
    \label{fig:introdiagram}
\end{figure}

\subsection{Data-driven Koopman analysis}

Instead of linearizing the dynamics in the state space, which is generally impossible to do globally for nonlinear functions $F$, the Koopman approach considers an infinite-dimensional linear operator acting on observables—functions that map the state $x \in \mathcal{X}$ to some measurement of interest. Given an observable $g:\mathcal{X}\to\mathbb{C}$ in a suitable function space $\mathcal{H}$, the Koopman operator $\koop_F$ (denoted by $\koop$ if $F$ is clear from context) is defined by the composition rule:
$$
    [\koop_Fg](x)=[g\circ F](x)=g(F(x)).
$$
In other words, $\koop_F$ advances an observable $g$ one step along the dynamics. Because composition with $F$ is a linear operation, i.e., $(a g_1 + b g_2)\circ F = ag_1(F(x)) + bg_2(F(x))$, $\koop_F$ is a linear operator regardless of whether $F$ itself is linear or nonlinear. This remarkable fact means one can, in principle, analyze a nonlinear system by studying the linear action of $\koop_F$ on an appropriate function space of observables.

Koopman operators initially played a pivotal role in the ergodic theorems of von Neumann~\cite{neumann1932proof} and Birkhoff~\cite{birkhoff1931proof}, and became a key tool in ergodic theory~\cite{eisner2015operator}. Sparked by \cite{mezic2005spectral}, they have recently been extensively used in data-driven methods for studying dynamical systems to overcome the problem of incomplete knowledge of $F$. In this setting, one observes a collection of so-called `snapshots' of the system, i.e.,
\begin{equation}
    \label{snapshots_def}
    (x^{(m)},y^{(m)})\text{ such that } y^{(m)}=F(x^{(m)}),\quad\text{for }m\in\{1,\ldots,M\}.
\end{equation}
These snapshots could arise from a single trajectory or multiple trajectories. Finite-dimensional approximations of $\mathcal{K}$ can be constructed from the snapshot data in~\cref{snapshots_def} using tools from numerical linear algebra and statistics. For example, Dynamic Mode Decomposition (DMD) uses a linear model to fit trajectory data~\cite{schmid2010dynamic}, and Extended DMD (EDMD)~\cite{williams2015data} builds a Galerkin approximation of $\mathcal{K}$ using a dictionary of observables.

A major focus is on the Koopman operator's spectral properties (e.g., its eigenvalues, continuous spectrum, and eigenfunctions) since these encode essential characteristics of the dynamical system's behavior. For example, Koopman eigenfunctions can identify invariant modes and quasiperiodic structures in the dynamics, while continuous spectral components are associated with chaotic or mixing behavior \cite[Sec.~2.1.2]{colbrook2023multiverse}. Mixing and ergodicity can be characterized in terms of eigenvalues \cite{halmos_lectures_1956}. Koopman pseudoeigenfunctions exhibit approximate temporal coherence \cite{giannakis2019data,valva2023consistent} and can be used to provide reduced order models through a time-space separation of variables. Pseudoeigenfunctions encode information about the dynamics of the underlying system, such as the global stability of equilibria. Moreover, their level sets reveal structures such as ergodic partitions, invariant manifolds, and isostables \cite{budivsic2012applied,mauroy2013isostables}.

\subsection{The challenges of Koopman analysis and infinite-dimensional spectral computations}

Computing these spectral properties is an active area of research. However, since the function space $\mathcal{H}$ is typically infinite-dimensional, the spectrum of $\mathcal{K}$ is often far more intricate—and difficult to compute—than that of a finite matrix. Several interrelated challenges arise. One is spectral pollution, where discretizations introduce persistent spurious eigenvalues unrelated to the true spectrum~\cite{lewin2009spectral,colbrook2019compute}. Another is spectral invisibility, where parts of the spectrum are missed altogether. Methods such as EDMD suffer from both issues~\cite{colbrook2024limits}.

Even for a single observable $g\in\mathcal{H}$, the span of $\mathcal{K}^ng$ is generically infinite-dimensional.\footnote{A finite-dimensional invariant subspace of $\koop$ is a space of observables $\mathcal{G} = {\rm span}\{g_1,\ldots,g_k\}$ such that $\mathcal{K}g\in \mathcal{G}$ for all $g\in\mathcal{G}$.} A common assumption is that $\koop$ admits a nontrivial\footnote{If the constant function $1$ lies in $\mathcal{H}$, it generates an invariant subspace, but this is not informative for dynamics so `trivial'.} finite-dimensional invariant subspace, an assumption that often fails. Even when such subspaces exist, they may be hard to compute or insufficient to capture the dynamics of interest. Moreover, $\mathcal{K}$ can exhibit continuous spectra, which often has physical or dynamical significance~\cite[Sec.~2.1.2]{colbrook2023multiverse}. However, continuous spectra must be treated with considerable care as discretizing $\mathcal{K}$ destroys its presence.

In response to these challenges, a mature body of work has developed over the past decade on provably convergent algorithms for computing spectral properties of Koopman operators. These approaches almost exclusively operate in the Hilbert space $\mathcal{H}=L^2(\mathcal{X},\omega)$ of observables, where $\omega$ is a positive measure on $\mathcal{X}$. This setting is particularly suitable for measure-preserving dynamics on attractors, though it also accommodates more general scenarios. Residual Dynamic Mode Decomposition (ResDMD)\cite{colbrook2021rigorousKoop,colbrook2023residualJFM} has been developed to address the issues of spectral pollution and invisibility. Methods such as the Christoffel--Darboux kernel\cite{korda2020data}, Koopman resolvent analysis\cite{colbrook2025rigged}, and Measure-Preserving Extended Dynamic Mode Decomposition (mpEDMD)\cite{colbrook2023mpedmd} enable the computation of spectral measures and continuous spectra. In parallel, techniques for spectral analysis of flows directly from time-series data have also been advanced~\cite{giannakis2019data,das2021reproducing}. This list only scratches the surface—see~\cite{colbrook2023multiverse} for a recent review. Spectral computation for Perron--Frobenius operators is another active area~\cite{dellnitz1999approximation,froyland2007ulam}; on appropriate function spaces, these operators are adjoint to Koopman operators.

A potentially surprising consequence is that all these algorithms depend on several parameters that must be taken successively to infinity \cite[Table 1]{colbrook2024limits}. For example, in their study of chaotic systems, Dellnitz and Junge approximated measures that describe the statistical behavior of these systems using three successive limits~\cite[Thm.~4.4]{dellnitz1999approximation}: increasing the size of the data set, increasing the number of observables, and applying a smoothing step. EDMD, arguably the most popular algorithm for computing Koopman spectra, typically uses two limits: a large data limit $M\rightarrow\infty$ (increasing the number of snapshots in \cref{snapshots_def}) followed by a large subspace limit. This is not due to lack of analysis; it was shown in~\cite{colbrook2024limits} that these limits cannot, in general, be combined and that for many systems, no alternative algorithm converges in a single limit to the spectrum of Koopman operators in spaces $L^2$.

The need for multiple successive limits in infinite-dimensional spectral computations is a generic phenomenon captured by the Solvability Complexity Index (SCI) hierarchy~\cite{ben2015can,Hansen_JAMS,colbrook2020PhD}. The SCI framework delineates the boundaries of computational feasibility and has recently resolved the longstanding problem of computing spectra of general bounded operators on separable Hilbert spaces—a problem rooted in the work of Szeg{\H o} \cite{Szego} on finite section methods and orthogonal polynomials and Schwinger \cite{Schwinger} on finite-dimensional approximations in quantum mechanics. The algorithm $\Gamma_{n_3,n_2,n_1}$ from \cite{Hansen_JAMS} converges to the spectrum in the Hausdorff metric in three successive limits:
$$
    \lim_{n_3\rightarrow\infty}\lim_{n_2\rightarrow\infty}\lim_{n_1\rightarrow\infty}\Gamma_{n_3,n_2,n_1}(B)=\mathrm{Sp}(B)\quad\forall B\in\mathcal{B}(\ell^2(\mathbb{N})),
$$
where  $\mathcal{B}(\ell^2(\mathbb{N}))$ denotes the bounded operators on $\ell^2(\mathbb{N})$. It is impossible to reduce the number of successive limits via any algorithm \cite{ben2015can}. While traditional techniques rely on single-limit procedures—e.g., computing eigenvalues of increasingly large matrices via mesh refinement or dimensional truncation—this requirement for multiple limits highlights an intrinsic analytical complexity. For certain operator subclasses, single-limit algorithms with error control are possible~\cite{colbrook2019compute,colbrook3}. The SCI hierarchy thus not only guides algorithm design but also establishes optimality and the fundamental limitations of computational methods.

\subsection{Reproducing kernel Hilbert spaces of observables}

An RKHS $\mathcal{H}$ on $\mathcal{X}$ is a Hilbert space of functions in which point evaluation is continuous (precise definitions are delayed until \cref{rkhskedmd_sect}). Working in an RKHS corresponds to lifting the state $x$ through a (typically nonlinear) feature map into a high- or infinite-dimensional feature space, where linear operations encode nonlinear dynamics in the original state space. This approach offers several notable advantages over $L^2$ spaces\footnote{There are also important connections between RKHSs and $L^2$ spaces. For example, \cite{das2021reproducing} studies coherent observables of the continuous-time, unitary Koopman operator on $L^2$ by compactifying its generator. A family of compact, skew-adjoint operators is constructed that map between suitable RKHSs and spectral convergence is established in an appropriate sense. In a related development, \cite{valva2023consistent} applies similar compactification techniques to the resolvent of the Koopman generator. Additionally, \cite{bandtlow_edmd_2024} proves that for analytic expanding maps on the circle, an $L^2$ EDMD-type algorithm achieves exponentially fast spectral convergence to the spectrum on certain RKHSs.}:

\begin{itemize}
    \item First, this kernel-based lifting of dynamics enables the use of a much richer class of observables than traditional polynomial or $L^2$ bases, without explicitly constructing high-dimensional features (the so-called `kernel trick' \cite{scholkopf_kernel_2000}). Kernelized EDMD (kEDMD) \cite{williams2015kernel} is a powerful and widely used method for approximating the Koopman operator, particularly when the number of required observables exceeds the number of data points. Numerous efficient algorithms from kernel-based learning (e.g., kernel PCA, Gaussian process regression) can be adapted for Koopman spectral analysis. For instance, \cite{kawahara_dynamic_2016} projects the Koopman operator onto a Krylov subspace generated by the RKHS feature map, while \cite{klus_eigendecompositions_2020} develops RKHS-based formulations of stochastic Koopman and Perron--Frobenius operators. These techniques provide an effective way to address the challenges posed by high-dimensional state spaces.
    \item Second, since pointwise evaluation is bounded in an RKHS, error bounds on approximate eigenfunctions or Koopman mode decompositions (used for forecasting observables) can be translated into pointwise convergence guarantees (see \cref{sec:gauss_map_example} for an example showing the benefits of pointwise forecasts). For example, \cite{kohne_linfty-error_2024} shows that, for a broad class of dynamical systems, the Koopman operator is bounded on the native space of the Wendland kernel, which corresponds to a suitable Sobolev space. They derive error bounds for approximations of the Koopman operator that map functions from this Sobolev space to continuous bounded functions. These results are further applied to data-driven stability analysis and control in \cite{bold_kernel-based_2024}. Additionally, \cite{philipp_error_2023,philipp_error_2024} provide bounds for stochastic Koopman operators under ergodic and i.i.d. sampling.
    \item Third, the flexibility in choosing the kernel allows one to tailor the RKHS to capture specific dynamical features of interest. Many real-world systems are neither measure-preserving nor ergodic, and restricting to $L^2$ observables can be too coarse or impractical in data-driven settings—particularly when only finite samples are available. For example, \cite{mauroy_analytic_2024} studies RKHSs with analytic kernels and develops Analytic EDMD, which uses a data-driven Taylor expansion to capture the dissipative dynamics near fixed points. In \cite{mezic_spectrum_2020}, the modulated Fock space—an RKHS tailored to dissipative systems—is introduced, and the Koopman spectrum is computed on this space. Meanwhile, \cite{bevanda_koopman_2023} constructs a Koopman eigenfunction kernel from trajectory data, whose native space is invariant under the Koopman operator.
\end{itemize}

The RKHS approach thus combines the strengths of data-driven machine learning with the physics-based objective of obtaining linear representations of nonlinear dynamics. In the course of this paper, we shall uncover a fourth important advantage of RKHS spaces:
\begin{itemize}
    \item Fourth, working in an RKHS avoids the large data limit $M \rightarrow \infty$ typically required in $L^2$ settings, thereby reducing the SCI for computing Koopman spectral properties. Moreover, in the RKHS setting, we do not need to assume that the snapshots defined in \cref{snapshots_def} are drawn according to a quadrature rule associated with an $L^2(\mathcal{X},\omega)$ inner product. This makes the approach significantly more flexible and broadly applicable. Instead, the function space is determined by a user-specified kernel—an essential advantage in practical scenarios where the dataset is finite and externally provided.
\end{itemize}

Despite these advantages, there is currently no overarching framework for addressing the challenges outlined in the previous section when computing spectral properties in infinite-dimensional RKHSs.\footnote{While it has been attempted to verify kEDMD eigenpairs using residuals, it is known that trying to compute the standard $L^2$ residual fails without splitting the data into two subsets \cite{colbrook2021rigorousKoop,colbrook_another_2024}.} In particular, methods such as kEDMD suffer from spectral pollution and spectral invisibility. To overcome these issues, it is essential to connect the resulting approximations to the underlying Koopman operator rigorously.

\subsection{Contributions and outline}
\label{sec:contributions}

In response to these challenges, we develop the first general suite of rigorously convergent, data-driven algorithms for computing Koopman spectral properties on RKHSs. Crucially, we do not assume that $\mathcal{K}$ has a nontrivial finite-dimensional invariant subspace or a purely discrete spectrum—our methods require few assumptions. We provide some mathematical background in \cref{sec:maths_background} and summarize our contributions below.
\begin{itemize}
    \item In \cref{provconv_sect}, we develop the first provably convergent data-driven algorithms for computing spectral properties of Koopman and Perron--Frobenius operators on RKHSs. These algorithms avoid spectral pollution and invisibility, and include a-posteriori verification via error bounds (enabling assessment of kernel suitability). As shown in \cref{avoidquadapprox_sect}, working in an RKHS bypasses the large data limits required in $L^2$ settings, reducing the problem's SCI classification. We also provide a new connection: kEDMD is the transpose of a Galerkin approximation of the Perron--Frobenius operator on the relevant RKHS. We establish convergence for both noise-free (\cref{sec:initial_convergence}) and noisy (\cref{inexactupperbounds_sect}) data, and provide verified error control for observable predictions through Perron--Frobenius mode decompositions (\cref{sec:errorcontrolpfmd}).
          This method reveals a duality between $\koop$ and $\koop^*$ in observable prediction: $\koop$ requires solving a least-squares problem for each observable $g$, while $\koop^*$ requires solving one for each point $x_0$. As a result, Perron--Frobenius methods are advantageous when predicting many observables at few points (e.g., in high-dimensional systems), whereas Koopman-based methods are preferable when predicting a few observables at many points. Numerical examples are presented in \cref{sec:pseud_examples}.
    \item In \cref{sci_sect}, we show the optimality of our algorithms by proving that no algorithm can succeed in fewer limits for these problems. We do this through the construction of adversarial dynamical systems for which no algorithm can compute spectral properties. Our approach embeds interval exchange maps into the dynamics for Sobolev space RKHSs. This section also includes a self-contained introduction to the SCI hierarchy. In \cref{sec_extension_to random}, we extend these results to probabilistic algorithms, showing that no randomized method can succeed with probability greater than $2/3$. These results apply to any random algorithm or data collection, such as EDMD with random trajectories and algorithms using stochastic components, including SGD-based neural network training. Hence, the adversarial systems we construct are not pathological but illustrate fundamental computational barriers.
    \item In \cref{unisa_sect}, we address the challenge of continuous spectra by developing convergent algorithms to approximate spectral measures of Koopman and Perron–Frobenius operators on RKHSs. We first provide theorems that characterize when Koopman/Perron--Frobenius operators on RKHSs are self-adjoint or unitary. We then develop convergent, data-driven algorithms for computing their spectral measures by evaluating the resolvent $(\mathcal{K}^* - Iz)^{-1}$ to obtain smoothed approximations. Two numerical examples are provided: M\"obius maps on the Poincar\'e disk (unitary) and stochastic Koopman operators on infinite Markov chains (self-adjoint). These results allow for the effective capture and analysis of continuous spectrum components, which are beyond the reach of traditional eigenvalue-focused methods.
    \item In \cref{numexample_sect}, we demonstrate our methods in high-dimensional systems using real-world data and data from high-fidelity numerical simulations. Examples include turbulent channel flow, molecular dynamics of a binding protein, Antarctic sea ice concentration, and Northern Hemisphere sea surface height. These applications highlight the strength of the RKHS approach in handling high-dimensional state spaces.
\end{itemize}
In summary, we have developed the first general framework of provably convergent Koopman spectral algorithms on RKHSs, established their computational optimality, and demonstrated their effectiveness on real-world systems. These results apply to any RKHS. \cref{algorithms-table} shows the algorithms developed in this paper and \cref{sci-classification-table} shows the SCI classifications.  This work opens avenues for future research, such as extending these methods to other operator-theoretic problems or integrating them with machine learning approaches. Code and data accompanying this work are publicly available at \url{https://github.com/GustavConradie1/SpecRKHS}.

\begin{table}[t]
    \caption{The key algorithms that we develop in this paper and the sections providing the corresponding theoretical results. These come in the form of convergence proofs and/or error bounds (or control) on the output of the algorithms.}
    \begin{center}
        \begin{tabular}{ lllc }
            \toprule[\thick pt]
            Computed quantity                           & Algorithm                      & Name                 & Theoretical results         \\
            \midrule[\thick pt]
            Verified eigenvalues                        & \cref{evalverif_alg}           & SpecRKHS-Eig         & \cref{verepairs_sect}       \\
            Approximate point pseudospectrum of PF      & \cref{pspecadjoint_alg}        & SpecRKHS-PseudoPF
                                                        & \cref{sec:initial_convergence}                                                      \\
            Approximate point pseudospectrum of Koopman & \cref{pspeckoop_alg}           & SpecRKHS-PseudoKoop  & \cref{sec:fullpspecspec}    \\
            Future evolution of observables             & \cref{alg_verifiedPF}          & SpecRKHS-Obs         & \cref{sec:errorcontrolpfmd} \\
            Self-adjoint spectral measures              & \cref{saspecmeas_alg}          & SpecRKHS-SAdjMeasure  & \cref{compspecmeas_sec}     \\
            Unitary spectral measures                   & \cref{unispecmeas_alg}         & SpecRKHS-UniMeasure & \cref{compspecmeas_sec}     \\
            \bottomrule[\thick pt]
        \end{tabular}
        \label{algorithms-table}
    \end{center}
\end{table}

\begin{table}[t]
    \caption{A summary of the upper bounds (convergence results) and lower bounds (impossibility results) in terms of the SCI hierarchy for what is computationally possible for the different problem functions discussed in this paper. $\mathcal{X}$, $r$ and $\epsilon$ are as defined in \cref{rkhslowerbounds_thm}. The upper bounds are developed throughout \cref{provconv_sect}, and collected as a theorem in \cref{sci_sect}, where we also prove the lower bounds. \cref{random_no_help_thm} shows that our impossibility results extend to probabilistic algorithms (e.g., random sampling of trajectory data or execution of computations).  A result of the form $\not\in\Delta_1^G$ means that the problem cannot be solved with error control, $\not\in\Delta_{m+1}^G$ means that the problem cannot be solved in $m$ limits, the inclusions $\in\Sigma_{m}^G$ or $\in\Pi_{m}^G$ mean that the problem can be solved in $m$ limits with verification in the final limit.}
    \begin{center}
        \begin{tabular}{ lcccl }
            \toprule[\thick pt]
            Problem function                        & Upper bound     & Lower bound          & RKHS for lower bound & Reference for lower bound                                      \\
            \toprule[\thick pt]
            $\spec_{\mathrm{ap},\epsilon}(\koop^*)$ & $\in\Sigma_1^G$ & $\notin\Delta_1^G$   & $H^r(\mathcal{X})$   & \cref{notdelta1specapeps_thm}                                  \\
            $\spec_{\mathrm{ap},\epsilon}(\koop)$   & $\in\Sigma_2^G$ & $\notin\Delta_2^{G}$ & $h^r(\mathbb{N})$    & \cref{notdelta2specapepsdiscretesobolev_thm}                   \\
            $\spec_{\mathrm{ap}}(\koop^*)$          & $\in\Pi_2^G$    & $\notin\Delta_2^G$   & $H^r(\mathcal{X})$   & \cref{notdelta2specapinterval_thm,notdelta2specaprealline_thm} \\
            $\spec_{\mathrm{ap}}(\koop)$            & $\in \Pi_3^G$   & $\notin\Delta_3^G$   & $h^r(\mathbb{N})$    & \cref{notdelta3specapkoop_thm}                                 \\
            \bottomrule[\thick pt]
        \end{tabular}
        \label{sci-classification-table}
    \end{center}
\end{table}

\section{Mathematical background}
\label{sec:maths_background}

\subsection{Spectral properties of Koopman operators}
\label{specprop_sect}

The linearity of the Koopman operator makes it a powerful tool for analyzing dynamical systems. As noted in the introduction, the spectral features of the Koopman operator (such as eigenfunctions and pseudoeigenfunctions) are closely tied to the system's behavior and correspond to coherent observables~\cite{colbrook2023multiverse}.

Suppose that $\mathcal{H}$ is a separable Hilbert space with inner product $\langle\cdot,\cdot\rangle$ (conjugate linear in its second argument) and norm $\|\cdot\|$. If $(\lambda,\psi)$ is an eigenpair of $\koop$, then $\psi(x_n)=\koop^n\psi(x_0)=\lambda^n\psi(x_0)$ for $x_0\in\mathcal{X}$, $n\in\mathbb{N}$, and we have perfect coherence: the behavior of the observable $\psi$ is determined by the eigenvalue $\lambda$. Suppose instead that $(\lambda,\psi)$ is such that $\|\psi\|=1$ and $\|(\koop-\lambda I)\psi\|\leq\epsilon$ for some $\epsilon>0$, i.e., $\psi$ is an $\epsilon$-pseudoeigenfunction. Then
$$
    \|\koop^n\psi-\lambda^n\psi\|=\mathcal{O}(n\epsilon),\quad \text{as }\epsilon\to 0,
$$
i.e., we have approximate coherence. Often we can write a function as the sum of coherent or approximately coherent observables~\cite{brunton2021modern}. The Koopman mode decomposition (KMD), first introduced in \cite{mezic2005spectral}, generalizes this idea. For example, suppose that $\koop$ has an orthonormal basis of eigenfunctions $\{\varphi\}_{j=1}^{\infty}$ with corresponding eigenvalues $\lambda_j$. For a vector of observables $\m{g}\in\mathcal{H}^N$, we define $\m{v}_j=(\langle\varphi_j,g_1\rangle\,\cdots\, \langle\varphi_j,g_N\rangle)
$. Then for any $x_0$, we have
$$
    \m{g}(x_n)=\m{g}(F^n(x_0))=\koop^n\sum_{j=1}^{\infty}\varphi_j(x_0)\m{v}_j=\sum_{j=1}^{\infty}\lambda_j^n\varphi_j(x_0)\m{v}_j.
$$
The KMD triple $\{(\lambda_j,\varphi_j,\m{v}_j\}_{j=1}^{\infty}$ is a powerful tool to analyze the dynamical system. If $\m{g}(x)=x$, then we can predict the behavior of the dynamical system directly from its KMD. In general, the existence of a KMD relies on some spectral theorem that allows us to diagonalize $\koop$ \cite{colbrook2023multiverse}. We shall consider continuous spectra in \cref{unisa_sect} and see how these can lead to a more general KMD.

Let $T$ be a closed, densely defined operator on $\mathcal{H}$. The spectrum of $T$ is defined by
$$
    \spec(T)=\{\lambda\in\mathbb{C}:(T-\lambda I)\text{ is not boundedly invertible}\}.
$$
Note that if $T$ is not closed, then $\spec(T)=\mathbb{C}$; for a bounded operator, the spectrum is non-empty and compact, but for an unbounded operator, we only know that it is closed (and it may be empty). For $\epsilon>0$, the $\epsilon$-pseudospectrum of $T$ captures the sensitivity of $\spec(T)$ and is defined as
$$
    \spec_{\epsilon}(T)=\closure(\{z\in\mathbb{C}:\|(T-zI)^{-1}\|^{-1}<\epsilon\})=\closure\left(\bigcup_{E:\|E\|<\epsilon}\spec(T+E)\right),
$$
where $\closure$ denotes the closure of a set. Since $\spec(T)\subset\spec_{\epsilon}(T)$, pseudospectra can be a valuable tool for verifying spectral computations and ensuring that parts of the spectrum are not missed (note that any connected component of the pseudospectrum must intersect the spectrum). Pseudospectra can also help examine the transient behavior of a dynamical system  \cite[Sec.~IV]{trefethen2005spectra} that is not captured in the spectrum, which tends to capture asymptotic behavior.

We briefly introduce the injection modulus, which offers an alternative characterization of spectral sets that is more amenable to computation, and clarify the notions of convergence associated with spectral sets. The injection modulus of $T$ is
$\sigma_{\mathrm{inf}}(T)=\inf_{g\in\dom{T}}\|Tg\|/\|g\|$. If $T$ is a finite rectangular matrix acting between Euclidean spaces, then $\sigma_{\mathrm{inf}}(T)$ is its smallest singular value. In the general setting, the map $z\mapsto\sigma_{\inf}(T-zI)$ is continuous and it can be shown that for any $z\in\mathbb{C}$ \cite{colbrook3},
$$
    \|(T-zI)^{-1}\|^{-1}=\gamma(z,T)\coloneqq\min\{\sigma_{\mathrm{\inf}}(T-zI),\sigma_{\inf}(T^*-\overline{z}I)\}.
$$
This allows us to define the spectrum and pseudospectrum as
$$
    \spec(T)=\{z\in\mathbb{C}:\gamma(z,T)=0\},\quad\spec_{\epsilon}(T)=\closure(\{z\in\mathbb{C}:\gamma(z,T)<\epsilon\}).
$$
We may also define the approximate point spectrum and pseudospectrum as
$$
    \spec_{\mathrm{ap}}(T)=\{z\in\mathbb{C}:\sigma_{\inf}(T-zI)=0\},\quad \spec_{\mathrm{ap},\epsilon}(T)=\closure(\{z\in\mathbb{C}:\sigma_{\inf}(T-zI)<\epsilon\}).
$$
Hence, $\spec(T)=\spec_{\mathrm{ap}}(T)\cup\spec_{\mathrm{ap}}(T^*)$, and similarly for the pseudospectrum. The set $\spec_{\mathrm{ap},\epsilon}(T)$ is the set of $\lambda$ for which there exists an $\epsilon$-pseudoeigenfunction. Hence, $\spec_{\mathrm{ap},\epsilon}(\koop)$ and $\spec_{\mathrm{ap},\epsilon}(\koop^*)$ contain key information about dynamical systems. For bounded $T$, we define its surjective spectrum as
$$
    \spec_{\mathrm{su}}(T)=\{\lambda\in\mathbb{C}:\;T-\lambda I \text{ is not surjective}\},
$$
which satisfies $\spec_{\mathrm{ap}}(T)=\spec_{\mathrm{su}}(T^*)$ and $\spec_{\mathrm{ap}}(T^*)=\spec_{\mathrm{su}}(T)$ \cite[Lem.~1.30]{aiena_fredholm_2004}. Hence, to compute the approximate point spectrum of the Perron--Frobenius operator, we may instead compute the surjective spectrum of the Koopman operator. Accurately computing spectral properties of Koopman operators, such as eigenfunctions and pseudoeigenfunctions, is essential. Since these operators are infinite-dimensional and we only have access to finite information, careful methods are required to ensure reliable results. This issue will be addressed explicitly in \cref{verepairs_sect} and is a central theme of this paper.

\subsection{Galerkin methods on $L^2(\mathcal{X},\omega)$ and $M\rightarrow\infty$}
\label{sec:edmd}

The central idea of methods such as EDMD \cite{williams2015data} is a data-driven Galerkin approximation of the Koopman operator in the Hilbert space $\mathcal{H}=L^2(\mathcal{X},\omega)$ of observables, where $\omega$ is a positive measure on $\mathcal{X}$. Given a subspace $V_N=\spann\{g_1,\dots,g_N\}\subset L^2(\mathcal{X},\omega)$, EDMD computes an approximation of $\mathcal{P}_{V_N}\mathcal{K}\mathcal{P}_{V_N}^*$, where $\mathcal{P}_{V_N}:L^2(\mathcal{X},\omega)\rightarrow V_N$ denotes the corresponding orthogonal projection with adjoint $\smash{\mathcal{P}_{V_N}^*:V_N\rightarrow L^2(\mathcal{X},\omega)}$.

Let $G_{\mathrm{EDMD}},A_{\mathrm{EDMD}}\in\mathbb{C}^{N\times N}$ be the matrices defined by
$$
    [G_{\mathrm{EDMD}}]_{jk}=\langle g_k,g_j\rangle_{L^2},\qquad [A_{\mathrm{EDMD}}]_{jk}=\langle\koop g_k,g_j\rangle_{L^2},\quad 1\leq k,j\leq N.
$$
Then $\smash{\mathcal{P}_{V_N}\mathcal{K}\mathcal{P}_{V_N}^*}$ is represented by the matrix $G_{\mathrm{EDMD}}^{-1}A_{\mathrm{EDMD}}$ in the basis $\{g_1,\dots,g_N\}$. To approximate these matrices using snapshot data in \cref{snapshots_def}, i.e., pairs $\{(x^{(j)},y^{(j)})\}_{j=1}^M$ such that $y^{(j)}=F(x^{(j)})$, one assumes that the $\{x^{(j)}\}_{j=1}^M$ form a convergent quadrature rule for the $L^2$-inner product with quadrature weights $\{w_j\}_{j=1}^M$. The matrices are then recovered in a large data limit $M\rightarrow\infty$:
\begin{equation}\label{EDMD_quadrature_rule}
    \lim_{M\rightarrow\infty}\sum_{j=1}^Mw_jg_k(x^{(j)})\overline{g_j(x^{(j)})}=\langle g_k,g_j\rangle_{L^2}, \quad\lim_{M\rightarrow\infty}\sum_{j=1}^Mw_jg_k(y^{(j)})\overline{g_j(x^{(j)})}=\langle\koop g_k,g_j\rangle_{L^2},
\end{equation}
for $1\leq k,j\leq N$. Common choices of quadrature include random sampling, ergodic sampling, and high-order quadrature rules (e.g., Gauss--Legendre) \cite{colbrook2021rigorousKoop}. There is a large literature studying the limits in \cref{EDMD_quadrature_rule}; see \cite{aloisio2024spectral,nuske2023finite} for a small sample. We shall see in \cref{avoidquadapprox_sect} that working in an RKHS allows us to avoid having to take this limit.

If we define $\Psi_X,\Psi_Y\in\mathbb{C}^{N\times N}$ by $(\Psi_X)_{jk}=g_k(x^{(j)})$ and $(\Psi_Y)_{jk}=g_k(y^{(j)})$, then $\lim_{M\rightarrow\infty}\Psi_X^*W\Psi_X= G$ and $\lim_{M\rightarrow\infty}\Psi_X^*W\Psi_Y= A$. The approximation to $\mathcal{P}_{V_N}\koop\mathcal{P}_{V_N}^*$ is given by $\mathbb{K}=(\Psi_X^*W\Psi_X)^{-1}(\Psi_X^*W\Psi_Y)$.\footnote{One can use the pseudoinverse or a regularization technique if $\Psi_X^*W\Psi_X$ is not invertible.} We are interested in the spectral properties of $\koop$. Hence, we may compute the eigenvalues and eigenfunctions of $\mathbb{K}$, which we hope approximate a KMD of $\koop$. Under suitable conditions, $\mathcal{P}_{V_N}\koop\mathcal{P}_{V_N}^*$ converges to the Koopman operator $\koop$ as $N\rightarrow\infty$ in the strong operator topology \cite{korda2018convergence}. However, it is well-known that convergence in the strong operator topology does not imply convergence of spectral properties. In particular, the EDMD eigenvalues may completely miss regions of spectra (spectral invisibility), or have persisting spurious eigenvalues (spectral pollution), even as $M\rightarrow\infty$ and $N\rightarrow\infty$. This motivated the development of ResDMD \cite{colbrook2021rigorousKoop}, which provides guarantees of spectral convergence for Koopman operators on $L^2$ spaces.

\subsection{Reproducing kernel Hilbert spaces and kernelized methods}
\label{rkhskedmd_sect}

An RKHS on a set $\mathcal{X}$ is a Hilbert space of functions from $\mathcal{X}$ to $\mathbb{C}$ such that for every $x\in\mathcal{X}$, the pointwise evaluation map $E_x:\mathcal{H}\rightarrow\mathbb{C};E_x(f)=f(x)$ is continuous \cite{berlinet_reproducing_2004,paulsen_introduction_2016}. By the Riesz representation theorem, this continuity implies that for each $x \in \mathcal{X}$, there exists a unique element $\mathfrak{K}_x \in \mathcal{H}$ such that $f(x)=\langle f,\mathfrak{K}_x\rangle_{\mathfrak{K}}$ for all $f\in\mathcal{H}$. Throughout, we use $\langle\cdot,\cdot\rangle_{\mathfrak{K}}$ and $\|\cdot\|_\mathfrak{K}$ to denote the inner product and norm in $\mathcal{H}$. The function $\mathfrak{K}_x$ is called the reproducing kernel at $x$, and the map $\mathfrak{K} : \mathcal{X} \times \mathcal{X} \to \mathbb{C}$ defined by
\begin{equation}
    \label{kernel_inner_prod}
    \mathfrak{K}(x,y)=\langle \mathfrak{K}_x,\mathfrak{K}_y\rangle_\mathfrak{K}=\mathfrak{K}_x(y).
\end{equation}
is called the \textit{reproducing kernel} of $\mathcal{H}$.
The reproducing kernel $\mathfrak{K}$ is conjugate symmetric, meaning that $\mathfrak{K}(x,y)=\overline{\mathfrak{K}(y,x)}$, and positive definite:
$$
    \sum_{j=1}^n\sum_{k=1}^na_j\overline{a_k}\mathfrak{K}(x_j,x_k)\geq 0\quad \forall a\in\mathbb{C}^n,
    \{x_1,\dots,x_n\}\subset\mathcal{X}.
$$
Conversely, the Moore--Aronszajn theorem \cite{aronszajn_theory_1950} states that a conjugate symmetric, positive definite function (often called a kernel function) defines a unique RKHS and that $\mathrm{span}\{\mathfrak{K}_x:x\in\mathcal{X}\}$ is dense in this space.

The RKHS induced by a kernel function $\mathfrak{K}$ is often called the \emph{native space} for the kernel $\mathfrak{K}$, and is denoted $\mathcal{N}_{\mathfrak{K}}$. Let $\varphi:\mathcal{X}\rightarrow\mathbb{F}$ for some Hilbert space $\mathbb{F}$ and define $\mathfrak{K}(x,y)=\langle\varphi(x),\varphi(y)\rangle_{\mathbb{F}}$. Then $\mathfrak{K}$ is conjugate symmetric and positive definite, so it defines an RKHS $\mathcal{H}$ by the Moore--Aronszajn theorem. We say that $\varphi$ is the feature map of $\mathcal{H}$. Conversely, any RKHS has a feature map given by $\varphi(x)=\mathfrak{K}_x$. In the following section, we will see examples of kernels whose RKHS are Sobolev spaces. Other common examples are:
\begin{itemize}[leftmargin=*, itemsep=0pt, topsep=0pt]
    \item \textbf{Polynomial kernels} \cite[Lem.~4.7]{christmann_support_2008}: for $m\geq 0$, $d\geq 1$ integers, and $c\geq 0$ a real number, $\mathfrak{K}(x,y)=(x^Ty+c)^m$ for $x,y\in\mathbb{C}^d$ defines an RKHS on $\mathbb{R}^d$;
    \item \textbf{Radial basis function (RBF) kernels}, e.g., Gaussians \cite[Sec.~4.4]{christmann_support_2008}: one selects $\mathfrak{K}(x,y)=\varphi(\|x-y\|)$ for some radial function $\varphi$, e.g., $\varphi(r)=e^{-r^2}$;
    \item \textbf{The Hardy space on the unit disk $H^2(\mathbb{D})$} \cite[Sec.~1.4.1]{paulsen_introduction_2016}: this is the space of formal power series $\sum_{n=0}^{\infty}a_nz^n$ with finite norm defined by the inner product $\langle \sum_{n=0}^{\infty}a_nz^n,\sum_{n=0}^{\infty}b_nz^n\rangle=\sum_{n=0}^{\infty}a_n\overline{b_n}$; this space can be shown to be an RKHS with reproducing kernel $\mathfrak{K}(z,w)=(1-\overline{w}z)^{-1}$.
\end{itemize}

The kernel trick \cite{scholkopf_kernel_2000} is a technique in machine learning that allows algorithms to operate in a high-dimensional feature space without explicitly computing the coordinates in that space. Instead, it relies on \cref{kernel_inner_prod} to compute inner products. The kEDMD algorithm \cite{williams2015kernel} employs the kernel trick to handle dynamical systems where $N$ is much larger than $M$ to avoid the curse of dimensionality \cite{bishop_pattern_2016}. For example, if one wants to use a dictionary of polynomials up to degree $10$ on $\mathbb{R}^3$, this requires $10^6$ dictionary functions. Over a high-dimensional space such as $\mathbb{R}^{100}$, one needs $10^{20}$ dictionary functions. Such sizes are not practical for direct computation, but can be circumvented using a suitable polynomial kernel. The following compression proposition underlies kEDMD, and shows how one may compute the eigenvalues of a much larger matrix using a smaller, related one.

\begin{proposition}[Prop.~1 in \cite{williams2015kernel}]
    \label{kedmd_prop}
    Let $\Psi_X=\hat{Q}\hat{\Sigma} \hat{Z}^*$ be the SVD of $\Psi_X\in\mathbb{C}^{M\times N}$, where $\hat{Q}\in\mathbb{R}^{M\times M}$ is unitary, $\hat{\Sigma}\in\mathbb{R}^{M\times M}$ is diagonal with non-increasing nonnegative entries, and $\hat{Z}\in\mathbb{C}^{N\times M}$ is an isometry. Then $(\lambda,v)$ with $\lambda\neq 0$ is an eigenpair of
    $
        \hat{\mathbb{K}}=(\hat{\Sigma}^{\dagger}\hat{Q}^*)(\Psi_Y\Psi_X^*)(\hat{Q}\hat{\Sigma}^{\dagger})
    $
    if and only if $(\lambda,\hat{Z}v)$ is an eigenpair of $\mathbb{K}$.
\end{proposition}

Instead of the large $N\times N$ matrix $\mathbb{K}$, one can work with the smaller $M\times M$ matrix $\hat{\mathbb{K}}$. Define $\hat{G}=\Psi_X\Psi_X^*$, $\hat{A}=\Psi_Y\Psi_X^*$, and note that $\hat{G}=\hat{Q}\hat{\Sigma}^2\hat{Q}^*$. One can compute $\hat{Q}$ and $\hat{\Sigma}$ from the eigendecomposition of $\hat{G}$. The kernel trick provides an efficient way to compute $\hat{G}$ and $\hat{A}$. Let $\Psi:\mathcal{X}\rightarrow\mathbb{C}^N$ defined by $\Psi(x)=(\psi_1(x),\dots,\psi_N(x))^T$ and corresponding kernel $\mathfrak{K}(x,y)=\Psi(y)^*\Psi(x)$. Then we see that $\hat{G}_{jk}=\mathfrak{K}(x_j,x_k)$ and $\hat{A}_{jk}=\mathfrak{K}(y_j,x_k)$. Typically, rather than starting with a set of dictionary functions and defining a corresponding kernel, we begin with a kernel that implicitly defines the dictionary, as in \cref{kedmd_alg}. However, like EDMD, kEDMD can suffer from spectral pollution and invisibility. While \cite{colbrook2021rigorousKoop,colbrook_another_2024} provide methods for spectral verification of kEDMD for Koopman operators on $L^2$ spaces, there is currently no provably convergent technique for spectral properties of Koopman operators on RKHS spaces. The situation is further complicated by the fact that the kernel trick often obscures the underlying RKHS structure in analyses of kEDMD. We present a new interpretation of kEDMD in \cref{avoidquadapprox_sect}.

\begin{algorithm}[t]
    \caption{kEDMD: computing approximate eigenvalues and eigenvectors of Koopman operators from snapshot data and a kernel function~\cite{williams2015kernel}.}\label{kedmd_alg}
    \textbf{Input:} The snapshot data $\{x_i,y_i=F(x_i)\}_{i=1}^M$, a kernel function $\mathfrak{K}:\mathcal{X}\times\mathcal{X}\rightarrow\mathbb{C}$, and a rank $r\in\mathbb{N}$.

    \begin{algorithmic}[1]
        \STATE{Compute the matrices $\hat{G}_{jk}=\mathfrak{K}(x_j,x_k)$ and $\hat{A}_{jk}=\mathfrak{K}(y_j,x_k)$.}
        \STATE{Compute the eigendecomposition of $\hat{G}=\hat{Q}\hat{\Sigma}^2\hat{Q}^*$ to find $\hat{Q},\hat{\Sigma}$.}
        \STATE{For rank reduction, we may compute $\hat{\Sigma}_r=\hat{\Sigma}(1:r,1:r)$ and $\hat{Q}_r=\hat{Q}(:,1:r)$, the matrices of the $r$ dominant eigenvalues and eigenvectors respectively.}
        \STATE{Construct $\hat{\mathbb{K}}=(\hat{\Sigma}_r^{\dagger}\hat{Q}^*_r)\hat{A}(\hat{Q}_r\hat{\Sigma}^{\dagger}_r)$ and compute its eigenvalues $(\lambda_i,\m{g}_i)$ for $i=1,\dots,M$.}
    \end{algorithmic}
    \textbf{Output:} The $M$ non-zero eigenpairs of $\mathbb{K}$ $(\lambda_i,\hat{Z}^*\m{g}_i)\in\mathbb{C}\times\mathbb{C}^{N}$.
\end{algorithm}

\subsection{Koopman operators on RKHSs and Sobolev spaces}
\label{sobolevkoopman_sect}

The Koopman operator on an RKHS $\mathcal{H}$ is defined by
$$
    [\mathcal{K}g](x)=g(F(x)),\quad \mathcal{D}(\mathcal{K})=\{g\in\mathcal{H}: g \circ F\in \mathcal{H}\},
$$
and is always a closed operator (this relies on $\mathcal{H}$ being an RKHS). This means that its graph
$
    \{(g,Kg):g\in\mathcal{D}(K)\}
$
is a closed subset of $\mathcal{H}$.\footnote{Note this is different to other notions of closure such as closure under composition of the space.} However, we also want the Koopman operator to be densely defined. In this case, we define
$$
    \mathcal{D}(\mathcal{K}^*)=\left\{w\in \mathcal{H}:v\mapsto \langle \mathcal{K}v,w \rangle_\mathfrak{K} \text{ is a bounded functional on }\mathcal{D}(\mathcal{K})\right\}.
$$
If $w\in\mathcal{D}(\mathcal{K}^*)$, then, since $\mathcal{D}(\mathcal{K})$ is dense, there is a unique $u\in\mathcal{H}$ with
$$
    \langle \mathcal{K}v,w \rangle_\mathfrak{K}=\langle v,u\rangle_\mathfrak{K}\quad \forall v\in\mathcal{D}(\mathcal{K}).
$$
We denote $u$ by $\mathcal{K}^*w$, which defines a linear operator $\mathcal{K}^*$, the adjoint of $\mathcal{K}$ known as the Perron--Frobenius operator.

The choice of kernel affects the properties of the corresponding Koopman and Perron--Frobenius operators. If the full state observable $g:x\mapsto x$ is an element of the RKHS (or well approximated by elements of the RKHS), the analysis of $\mathcal{K}^*$ enables direct prediction of the state evolution (see \cref{sec:errorcontrolpfmd}). Bounded Koopman operators are particularly convenient to work with. However, given an RKHS, the class of maps $F$ for which $\koop$ is bounded is subtle. In particular, only affine dynamics yield bounded Koopman operators on spaces such as the Fock space \cite{carswell_compostion_2003}, Paley--Wiener spaces \cite{chacon_composition_2007}, and the native RKHS of the Gaussian RBF kernel \cite{gonzalez_kernel_2023}.

A class of RKHSs applicable to a broad range of dynamics, for which we shall also prove impossibility theorems, is given by Sobolev spaces \cite{ikeda_koopman_2024,kohne_linfty-error_2024}. Let $\mathcal{X}\subset\mathbb{R}^d$ be a non-empty open set. We define the derivative of a function $f$ at $x$ with respect to a multi-index $\alpha=(\alpha_1,\dots,\alpha_d)\in\{0,1,\dots\}^d$ to be
$$
    D^{\alpha}f(x)=\frac{\partial^{|\alpha|}f}{\partial x_1^{\alpha_1}\cdots\partial x_d^{\alpha_d}}(x), \quad |\alpha|=|\alpha_1|+\dots+|\alpha_d|,
$$
and assume that all partial derivatives commute. For $f\in L^2(\mathcal{X},\omega)$, we say that $g$ is the $\alpha$-weak derivative of $f$ if
$$
    \int_{\mathcal{X}}f(x)D^{\alpha}\varphi(x)\dd x=(-1)^{|\alpha|}\int_{\mathcal{X}}g(x)\varphi(x)\dd x\quad\forall \varphi\in C_c^{\infty}(\mathcal{X}).
$$
We denote the weak derivative of $f$ by $g=D^{\alpha}f$ (with $D^{0}f=f$). The Sobolev space $H^r(\mathcal{X})$ of order $r\in\mathbb{N}$ is
$$
    H^r(\mathcal{X})=\{f\in L^2(\mathcal{X}):D^{\alpha}f\in L^2(\mathcal{X})\text{ for all }|\alpha|\leq r\},
$$
with inner product $\langle f,g\rangle_{H^r} = \sum_{|\alpha|\leq r}\langle D^\alpha f,D^\alpha g\rangle_{L^2}$.
If the domain $\mathcal{X}$ satisfies the cone condition \cite[Def.~4.6]{adams_sobolev_2003}, the Sobolev embedding theorem \cite[Thm.~4.12]{adams_sobolev_2003} implies that if $r>d/2$ then $H^r(\mathcal{X})$ is an RKHS and a subspace of $C^0(\mathcal{X})$ equipped with the supremum norm $\|f\|_{C^0}=\sup_{x\in\mathcal{X}}|f(x)|$.
Indeed, there exists a constant $C>0$ such that for all $x\in\mathcal{X}$ and $f\in H^r(\mathcal{X})$, $|E_x(f)|=|f(x)|\leq \|f\|_{C^0}\leq C\|f\|_{H^r}$. Sobolev spaces are widely used in the theory of partial differential equations but will be used here for their properties as RKHSs. The cone condition is typical in practice and will be satisfied by all domains considered in this paper.

Having an explicit kernel for the RKHS is useful. Sobolev kernel functions are often given by relevant Green's functions \cite[Sec.~1.3.1]{paulsen_introduction_2016}. For example, with $\mathcal{X}=(a,b)\subset \mathbb{R}$ and on $H^1((a,b))$, $\mathfrak{K}_x$ satisfies $\langle f,\mathfrak{K}_x\rangle_{H^1}=\langle f,\mathfrak{K}_x\rangle_{L^2}+\langle f',\mathfrak{K}_x'\rangle_{L^2}=f(x)$. After integrating by parts, one solves $-\mathfrak{K}_x''+\mathfrak{K}_x=\delta_x$, yielding \cite[Ex.~12]{berlinet_reproducing_2004}
\begin{equation}
    \label{H1abreproducingkernel_eqn}
    \mathfrak{K}(x,y)=\begin{cases}
        \frac{\cosh(x-a)\cosh(b-y)}{\sinh(b-a)}, & a\leq x\leq y\leq b, \\
        \frac{\cosh(y-a)\cosh(b-x)}{\sinh(b-a)}, & a\leq y\leq x\leq b.
    \end{cases}
\end{equation}
The Wendland and Mat{\'e}rn kernels are a more convenient set of kernels whose native spaces are (potentially fractional) Sobolev spaces \cite{kohne_linfty-error_2024,wendland_scattered_2004}. Given $\varphi$ such that $t\mapsto t\varphi(t)$ is an element of $L^1([0,\infty))$, we define
$$
    (\mathcal{I}\phi)(r)=\int_r^{\infty}t\varphi(t)\dd t,\quad
    \varphi_l(r)=(1-r)_+^l=\begin{cases}
        (1-r)^l, & 0\leq r\leq 1,    \\
        0,       & \text{otherwise},
    \end{cases}
$$
and let $\varphi_{d,k}=\mathcal{I}^k\varphi_{\lfloor d/2\rfloor+k+1}$. The Wendland kernel for $d\in\mathbb{N}$ and $k\in\mathbb{N}\cup\{0\}$ is \cite[Def.~9.11]{wendland_scattered_2004}
\begin{equation}
    \label{wendlandkernel_defn}
    \mathfrak{K}^{\mathrm{W}}_{d,k}(x,y)=\varphi_{d,k}(\|x-y\|_2).
\end{equation}
Additionally, $\varphi_{d,k}$ is of the form \cite[Thms.~9.12 and 9.13]{wendland_scattered_2004}:
$$
    \varphi_{d,k}(r)=\begin{cases}
        p_{d,k}(r), & 0\leq r\leq 1, \\
        0,          & r\geq 1,
    \end{cases}
$$
where $p_{d,r}$ is a polynomial of degree $\lfloor d/2\rfloor+3k+1$ with rational coefficients that can be computed using a recurrence relation and $\varphi_{d,k}\in C^{2k}([0,\infty))$. Up to equivalent norms, the native space of the Wendland kernel (where if $k=0$ then $d\geq 3$) is $\mathcal{N}^{\mathrm{W}}_{d,k}(\mathbb{R}^d)=H^{(d+1)/2+k}(\mathbb{R}^d)$ \cite[Thm.~4.1]{kohne_linfty-error_2024}. The Mat{\'e}rn function (or Sobolev spline) on $\mathbb{R}^d$ for $n\in\mathbb{N}$, $n>d/2$ is defined by
\begin{equation}
    \label{defn:matern}
    \theta_{d,n}(x)=\cfrac{2^{1-n-d/2}}{\pi^{d/2}n!\sigma^{2n-d}}(\sigma\|x\|_2)^{n-d/2}K_{n-d/2}(\sigma\|x\|_2),\quad x\in\mathbb{R}^d.
\end{equation}
Here, $\sigma$ is a dimensionless scale factor, and $K_{\alpha}$ is the modified Bessel function of the second kind of order $\alpha$. The corresponding Mat{\'e}rn kernel is $\mathfrak{K}^{\mathrm{M}}_{d,n}(x,y)=\theta_{d,n}(x-y)$. We drop the constant in front of $\theta_{d,n}$ as it does not affect our algorithms. Up to equivalent norms, the native space of the Mat{\'e}rn kernel is $\mathcal{N}^{\mathrm{M}}_{d,n}(\mathbb{R}^d)=H^n(\mathbb{R}^d)$ \cite{bold_kernel-based_2024}.

\subsection{The meaning of convergent spectral computations} \label{sec_mean_spectral}

To develop algorithms that converge to these various spectral sets, we must make sense of the convergence of different sets computed by our algorithms. We wish to capture convergence without spectral pollution or invisibility. When considering bounded operators, the spectral sets are compact, so we use the Hausdorff metric space $(\mathcal{M}_{\mathrm{H}},d_{\mathrm{H}})$. This is the set of non-empty compact subsets subsets of $\mathbb{C}$ equipped with the metric
$$
    d_{\mathrm{H}}(X,Y)=\max\left\{\sup_{x\in X}d(x,Y),\sup_{y\in Y}d(X,y)\right\},
$$
where $d$ is the standard Euclidean metric. When considering unbounded operators, the spectral sets are closed, so we use the Attouch--Wets metric space \cite{beer_topologies_1993} $(\mathcal{M}_{\mathrm{AW}},d_{\mathrm{AW}})$ on $\mathbb{C}$. This is the set of non-empty closed subsets of $\mathbb{C}$ equipped with the metric
\begin{equation}
    \label{AW_metric_def}
    d_{\mathrm{AW}}(X,Y)=\sum_{n=1}^{\infty}2^{-n}\min\left\{1,\sup_{x\in B_m(0)}|d(x,X)-d(x,Y)|\right\},
\end{equation}
where $d$ is the standard Euclidean metric and $B_m(0)$ is the closed ball of radius $m$ and center $0$. To include the empty set, we say $X_n\rightarrow\emptyset$ as $n\rightarrow\infty$ if for any $m\in\mathbb{N}$, for all large $n$, $X_n\cap B_m(0)=\emptyset$.

It is known \cite{colbrook2020PhD} that for any non-empty closed sets $C,C_n\subset\mathbb{C}$, $d_{\mathrm{AW}}(C_n,C)\rightarrow 0$ if and only if for every $\delta>0$ and compact $K$, there exists $N\in\mathbb{N}$ such that for all $n>N$, $C_n\cap K\subset C+B_{\delta}(0)$ and $C\cap K\subset C_n+B_{\delta}(0)$. It suffices to consider $K=B_m(0)$ for $m\in\mathbb{N}$. Hence, this metric captures the notion of convergence whilst avoiding spectral pollution and invisibility.
Moreover, for any densely defined, closed operator $T$, $\spec_{\epsilon}(T)$ (resp. $\spec_{\mathrm{ap},\epsilon}(T)$) converges in the Attouch--Wets topology to $\spec(T)$ (resp. $\spec_{\mathrm{ap}}(T)$) as $\epsilon\rightarrow 0$.

\section{Provably convergent algorithms for spectral properties on RKHSs}
\label{provconv_sect}

We now develop the first provably generally convergent data-driven algorithms for computing spectral properties of Koopman and Perron--Frobenius operators on RKHSs and explore how to do so with verification (error bounds on the output). In particular, our algorithms avoid the issues of spectral pollution and invisibility. In \cref{sci_sect}, we prove that our algorithms are optimal by providing complementary lower bounds (impossibility results) for Koopman operators on Sobolev spaces. Given an RKHS $\mathcal{H}$ with kernel function $\mathfrak{K}$ and a dynamical system $(F,\mathcal{X})$, we define $\koop:\mathcal{D}(\koop)\rightarrow\mathcal{H}$ by $\koop g=g\circ F$; we assume throughout that $\koop$ is densely defined, so the Perron--Frobenius operator $\koop^*$ exists and is closed. Since $\koop$ is always closed, $\koop^*$ is also densely defined. A key advantage of RKHS methods will become apparent: the RKHS structure lets us completely bypass the need for a large data limit!

\subsection{The action of $\koop^*$ on the kernel functions and its consequences}
\label{avoidquadapprox_sect}

We first consider how $\koop^*$ acts on the kernel functions $\{\mathfrak{K}_x\;:x\in\mathcal{X}\}$. For any fixed $x\in\mathcal{X}$, the map
$$
    f\mapsto \langle \koop f,\mathfrak{K}_x\rangle_{\mathfrak{K}}=(\koop f)(x)=f(F(x))=\langle f,\mathfrak{K}_{F(x)}\rangle_{\mathfrak{K}},\quad f\in \mathcal{D}(\koop),$$
is a bounded linear functional on $\mathcal{D}(\koop)=\{g\in\mathcal{H}:g\circ F\in\mathcal{H}\}$, with bound $\|\mathfrak{K}_{F(x)}\|_{\mathfrak{K}}<\infty$. It follows that
\begin{equation}
    \label{koopstarkernelfunction_eqn}
    \mathfrak{K}_x\in\mathcal{D}(\koop^*)\quad\text{and}\quad
    \mathcal{K}^*\mathfrak{K}_x=\mathfrak{K}_{F(x)}\quad \forall x\in\mathcal{X}.
\end{equation}
This formula allows us to use the kernel $\mathfrak{K}$ to efficiently compute Galerkin matrices to approximate $\mathcal{K}^*$ on $\mathcal{H}$. Suppose that we have snapshots $\{(x^{(i)},y^{(i)}=F(x^{(i)}))\}_{i=1}^N$ and take $\{\mathfrak{K}_{x^{(i)}}\}_{i=1}^N\subset\mathcal{H}$ as our dictionary. In contrast to the $L^2$ case, here we take $M=N$. Then \cref{koopstarkernelfunction_eqn} implies that
$$
    \langle \mathfrak{K}_{x^{(i)}},\mathfrak{K}_{x^{(j)}}\rangle_{\mathfrak{K}}=\mathfrak{K}(x^{(i)},x^{(j)}),\quad
    \langle \koop^* \mathfrak{K}_{x^{(i)}},\mathfrak{K}_{x^{(j)}}\rangle_{\mathfrak{K}}=\langle \mathfrak{K}_{F(x^{(i)})},\mathfrak{K}_{x^{(j)}}\rangle_{\mathfrak{K}}=\mathfrak{K}(F(x^{(i)}),x^{(j)})=\mathfrak{K}(y^{(i)},x^{(j)}).
$$
Hence, we can immediately compute Galerkin approximations of $\mathcal{K}$ and $\mathcal{K}^*$ on the space $\mathcal{H}$. In particular, if we define
\begin{equation}
    \label{GA_defn}
    G_{jk}=\mathfrak{K}(x^{(k)},x^{(j)})=\langle \mathfrak{K}_{x^{(k)}},\mathfrak{K}_{x^{(j)}}\rangle_{\mathfrak{K}},\quad
    A_{jk}=\mathfrak{K}(y^{(k)},x^{(j)})=\langle \koop^*\mathfrak{K}_{x^{(k)}},\mathfrak{K}_{x^{(j)}}\rangle_{\mathfrak{K}}, \quad 1\leq j,k\leq N,
\end{equation}
then, assuming that $G$ is invertible (if not, we can use a regularization technique or take the pseudoinverse), the finite section approximation to $\koop^*$ with respect to the dictionary $\{\mathfrak{K}_{x^{(1)}},\dots,\mathfrak{K}_{x^{(N)}}\}$ is represented by the matrix $G^{-1}A$. Note that $A$ is the transpose of $\hat{A}$ defined in \cref{rkhskedmd_sect} since we are considering $\koop^*$ instead of $\koop$.

There are two key consequences of these connections:
\begin{itemize}[leftmargin=*,topsep=0pt]
    \item \textbf{Avoiding $M\rightarrow\infty$:}  Compared to standard EDMD methods—where computing the analog of $G$ and $A$ requires applying a quadrature rule to approximate $L^2$ inner products, as in \cref{EDMD_quadrature_rule}—the RKHS structure allows us to bypass this step entirely. In other words, we avoid having to take the large data limit $M\rightarrow\infty$ before taking $N\rightarrow\infty$, a step that is provably unavoidable when working in a standard $L^2$ space \cite{colbrook2024limits}. Additionally, the data, which is typically externally obtained, does not need to be drawn from a convergent quadrature rule, yielding much greater flexibility and reducing the need for assumptions such as ergodicity in the case of single trajectory data.
    \item \textbf{A new interpretation of kernelized EDMD:}
          Suppose that we take $r\leq M$ in \cref{kedmd_alg}. For notational convenience, let $U_r=\overline{\hat{Q}_r}$. We define the functions
          \begin{equation}
              \label{new_basis}
              u_j= \sum_{i=1}^N\sum_{\ell=1}^r\mathfrak{K}_{x^{(i)}} [U_r]_{i\ell}[\hat{\Sigma}_r^\dagger]_{\ell j}\in\mathcal{H},\quad 1\leq j\leq r.
          \end{equation}
          Using \cref{koopstarkernelfunction_eqn}, these functions satisfy the following orthogonality relations:
          $$
              \langle u_j,u_k \rangle_{\mathfrak{K}}=\left[\hat{\Sigma}_r^{\dagger}U_r^*\hat{G}^\top U_r\hat{\Sigma}_r^{\dagger}\right]_{kj}=\delta_{kj},\quad
              \langle \mathcal{K}^*u_j,u_k \rangle_{\mathfrak{K}}=\left[\hat{\Sigma}_r^{\dagger}U_r^*\hat{A}^\top U_r\hat{\Sigma}_r^{\dagger}\right]_{kj}=\left[\hat{\Kv}^\top\right]_{jk},
          $$
          where $\hat{G}$, $\hat{A}$, and $\hat{\Sigma}_r$ are as in in \cref{kedmd_alg}.
          It follows that $\hat{\Kv}^\top$ is a Galerkin approximation of $\mathcal{K}^*$ on $\mathcal{H}$ (the RKHS induced by the kernel $\mathfrak{K}$)! 
\end{itemize}
The first consequence enables us to prove that computing spectral properties of Koopman operators is typically strictly easier on an RKHS $\mathcal{H}$ than on an $L^2$ space (we shall prove this in \cref{sci_sect}). The second consequence allows us to extend the connection and define residuals directly in the space $\mathcal{H}$.

\subsection{Verifying eigenpairs and computing approximate point pseudospectra of $\mathcal{K}^*$}
\label{verepairs_sect}

As above, we first use $\{\mathfrak{K}_{x^{(i)}}\}_{i=1}^N$ as a dictionary. Following \cref{koopstarkernelfunction_eqn}, we have
$\langle\koop^*\mathfrak{K}_{x^{(k)}},\koop^*\mathfrak{K}_{x^{(j)}}\rangle_{\mathfrak{K}}=\mathfrak{K}(y^{(k)},y^{(j)})$. This motivates defining the following matrix:
\begin{equation}
    \label{R_defn}
    R_{jk}=\mathfrak{K}(y^{(k)},y^{(j)})=\langle \koop^*\mathfrak{K}_{x^{(k)}},\koop^* \mathfrak{K}_{x^{(j)}}\rangle_{\mathfrak{K}}, \quad 1\leq j,k\leq N,
\end{equation}
which allows us to perform verified spectral computations through residuals. Indeed, suppose that we have a candidate eigenpair $(\lambda,\m{g})$ for the Perron--Frobenius operator, where the components of $\m{g}$ are $\{g_i\}_{i=1}^N$ and correspond to the proposed eigenfunction $g=\sum_{i=1}^Ng_i\mathfrak{K}_{x^{(i)}}$. For now, we do not specify how such a candidate is produced. We want to verify whether $\lambda\in\spec_{\mathrm{ap},\epsilon}(\koop^*)$ by computing the relative residual
$$
    \res^*(\lambda,\m{g})=\frac{\|(\koop^*-\lambda I)g\|_{\mathfrak{K}}}{\|g\|_{\mathfrak{K}}}\geq
    \inf_{h\in\mathcal{D}(\koop^*)}\frac{\|(\koop^*-\lambda I)h\|_{\mathfrak{K}}}{\|h\|_{\mathfrak{K}}}
    =\sigma_{\mathrm{inf}}(\koop^*-\lambda I).
$$
If the residual is smaller than $\epsilon$, then $\lambda$ is guaranteed to lie inside the $\epsilon$-approximate point pseudospectrum, and $g$ is an $\epsilon$-pseudoeigenfunction. The squared residual can be re-written as
\begin{align*}
    \left[\mathrm{res}^*(\lambda,\m{g})\right]^2 & =\frac{\langle(\koop^*-\lambda I)g,(\koop^*-\lambda I)g\rangle_{\mathfrak{K}}}{\langle g,g\rangle_{\mathfrak{K}}}
    =\frac{\langle\koop^*g,\koop^*g\rangle_{\mathfrak{K}}-\lambda\langle g,\koop^*g\rangle_{\mathfrak{K}}-\overline{\lambda}\langle\koop^*g,g\rangle_{\mathfrak{K}}+|\lambda|^2\langle g,g\rangle_{\mathfrak{K}}}{\langle g,g\rangle_{\mathfrak{K}}}                            \\
                                                 & =\frac{\sum_{j,k=1}^Ng_j(\mathfrak{K}(y_j,y_k)-\lambda \mathfrak{K}(x_j,y_k)-\overline{\lambda} \mathfrak{K}(y_j,x_k)+|\lambda|^2\mathfrak{K}(x_j,x_k))\overline{g_k}}{\sum_{j,k=1}^Ng_j\mathfrak{K}(x_j,x_k)\overline{g_k}} \\
                                                 & =\frac{\sum_{j,k=1}^Ng_j(R-\lambda A^*-\overline{\lambda}A+|\lambda|^2G)_{kj}\overline{g_k}}{\sum_{j,k=1}^Ng_jG_{kj}\overline{g_k}}
    =\frac{\m{g}^*(R-\lambda A^*-\overline{\lambda}A+|\lambda|^2G)\m{g}}{\m{g}^*G\m{g}},
\end{align*}
which can be computed exactly without taking any limits. The associated algorithm is summarized in SpecRKHS-Eig (\cref{evalverif_alg}) and allows us to avoid spurious eigenvalues generated by naively applying kEDMD. Other methods can also be used to compute the candidate eigenpairs. Moreover, no large data limit is required in this process, similarly to \cref{avoidquadapprox_sect}.

\begin{algorithm}[t]
    \caption{SpecRKHS-Eig: Computing verified eigenvalues of Perron--Frobenius operators.}\label{evalverif_alg}
    \textbf{Input:} Snapshot data $\{(x^{(i)},y^{(i)}=F(x^{(i)}))\}_{i=1}^N$, kernel function $\mathfrak{K}:\mathcal{X}\times\mathcal{X}\rightarrow\mathbb{C}$, and target accuracy $\epsilon>0$.

    \begin{algorithmic}[1]
        \STATE{Compute the matrices $G$, $A$ and $R$ defined by $\smash{G_{jk}=\mathfrak{K}(x^{(k)},x^{(j)})}$, $\smash{A_{jk}=\mathfrak{K}(y^{(k)},x^{(j)})}$ and $\smash{R_{jk}=\mathfrak{K}(y^{(k)},y^{(j)})}$.\!\!\!\!\!}
        \STATE{Solve the generalized eigenvalue problem $A\m{g}^{(i)}=\lambda_iG\m{g}^{(i)}$ for $i=1,\dots,N$.}
        \STATE{Compute the residual $$\res^*(\lambda_i,\m{g}^{(i)})=\frac{(\m{g}^{(i)})^*[R-\lambda_i A^*-\overline{\lambda_i}A+|\lambda_i|^2G]\m{g}^{(i)}}{(\m{g}^{(i)})^*G\m{g}^{(i)}},$$ for each $i$ and discard the eigenpair if $\res^*(\lambda_i,g^{(i)})>\epsilon$.}
    \end{algorithmic}
    \textbf{Output:} Verified eigenpairs $(\lambda_i,\m{g}^{(i)})\in\mathbb{C}\times\mathbb{C}^{N}$ on the RKHS with kernel $\mathfrak{K}$.
\end{algorithm}

\begin{remark}[$L^2$ residual estimates]
    In some cases, it is possible to extract $L^2$ residual estimates from RKHS residual estimates. For example, the Sobolev space $H^1(\mathcal{X})$ satisfies $\|(\mathcal{K}^*-\lambda I)g\|_{L^2}\leq\|(\mathcal{K}^*-\lambda I)g\|_{H^1}$ for all $g\in\dom{\koop|_{H^1}}$. More generally, if an RKHS $\mathcal{H}$ with kernel ${\mathfrak{K}}$ satisfies
    $
        \|\mathfrak{K}\|_{L^2}^2=\int_{\mathcal{X}}\mathfrak{K}(x,x)\dd x<\infty,
    $
    then $\mathcal{H}\subset L^2(\mathcal{X})$ and for all $g\in\mathcal{H}$, $\|g\|_{L^2}\leq\|\mathfrak{K}\|_{L^2}\|g\|_{\mathfrak{K}}$ \cite[Thm.~4.26]{christmann_support_2008}.
    Hence, if we compute $\|g\|_{L^2}$ by quadrature, we can obtain upper bounds for $\|(\mathcal{K}^*-\lambda I)g\|_{L^2}/\|g\|_{L^2}$ using the RKHS methods, which provides a sufficient condition to confirm valid eigenpairs.
\end{remark}

Verifying eigenpairs addresses the problem of spectral pollution and spurious eigenvalues. However, applying SpecRKHS-Eig (\cref{evalverif_alg}) requires computing candidate eigenpairs. If these candidates are obtained via EDMD or kernelized EDMD, then verification alone does not resolve the issue of spectral invisibility—i.e., missing parts of the spectrum. To avoid this problem, we can adapt our algorithm to compute (approximate point) pseudospectra with SpecRKHS-PseudoPF (\cref{pspecadjoint_alg}). To determine if a point $z\in\mathbb{C}$ lies in $\spec_{\mathrm{ap},\epsilon}(\mathcal{K}^*)$, we consider
\begin{align}
    \sigma_{\inf}(\koop^*-zI)=\inf_{h\in\mathcal{D}(\koop^*)}\frac{\|(\koop^*-z I)h\|_{\mathfrak{K}}}{\|h\|_{\mathfrak{K}}}
     & \leq \inf_{g\in\spann\{\mathfrak{K}_{x^{(1)}},\dots,\mathfrak{K}_{x^{(N)}}\}}\frac{\|(\koop^*-z I)g\|_{\mathfrak{K}}}{\|g\|_{\mathfrak{K}}} \notag                    \\
     & =\inf_{\m{g}\in\mathbb{C}^n} \res^*(z,\m{g})=\inf_{\m{g}\in\mathbb{C}^n}\sqrt{\frac{\m{g}^*(R-z A^*-\overline{z}A+|z|^2G)\m{g}}{\m{g}^*G\m{g}}}.\label{minimized_res}
\end{align}
All of the matrices on the right-hand side are finite and can be computed exactly by evaluating kernel functions. Computing the infimum is equivalent to finding the smallest eigenvalue of the (finite-dimensional) generalized eigenvalue problem $(R-zA^*-\overline{z}A+|z|^2G)v=\lambda Gv$. If the right-hand side is less than $\epsilon$, then $z$ must lie in $\spec_{\mathrm{ap},\epsilon}(\koop^*)$. The above calculation can be performed for all points $z$ lying in some finite search grid:
$$
    \gridop(N)=\frac{1}{N}[\mathbb{Z}+i\mathbb{Z}]\cap\{z\in\mathbb{C}\;:\;|z|\leq N\}.
$$
Any sequence of finite grids with $\lim_{n\rightarrow\infty}\dist(\lambda,\gridop(N))$ for all $\lambda\in\mathbb{C}$ works, and if, for example, we know a priori that the spectrum lies in a given region, then we can restrict the grid to that region. As we take $N\rightarrow\infty$, we shall see in \cref{sec:initial_convergence} that the right-hand side of \cref{minimized_res} converges down to $\sigma_{\inf}(\koop^*-zI)$ locally uniformly to any compact set. Hence, we can prove that the output of SpecRKHS-PseudoPF (\cref{pspecadjoint_alg}) converges in the Attouch--Wets topology to the approximate point pseudospectrum (see \cref{sec:initial_convergence}).

\begin{algorithm}[t]
    \caption{SpecRKHS-PseudoPF: Computing the approximate point pseudospectrum of Perron--Frobenius operators.}\label{pspecadjoint_alg}
    \textbf{Input:} Snapshot data $\{(x^{(i)},y^{(i)}=F(x^{(i)}))\}_{i=1}^N$, kernel function $\mathfrak{K}:\mathcal{X}\times\mathcal{X}\rightarrow\mathbb{C}$, target accuracy $\epsilon>0$, and finite grid $\gridop(N)$.

    \begin{algorithmic}[1]
        \STATE{Compute the matrices $G$, $A$ and $R$ defined by $\smash{G_{jk}=\mathfrak{K}(x^{(k)},x^{(j)})}$, $\smash{A_{jk}=\mathfrak{K}(y^{(k)},x^{(j)})}$ and $\smash{R_{jk}=\mathfrak{K}(y^{(k)},y^{(j)})}$.\!\!\!\!\!}
        \FOR{$z_j\in\gridop(N)$}
        \STATE{Solve the generalized eigenvalue problem $\tau(z_j)=\min_{\m{g}\in\mathbb{C}^{N}}\res^*(z_j,\m{g})$.}
        \ENDFOR
    \end{algorithmic}
    \textbf{Output:} An approximation of the approximate point pseudospectrum $\{z\in\gridop(N)\;:\;\tau(z_j)<\epsilon\}$ and corresponding $\epsilon$-pseudoeigenfunctions ${\m{g}_j}$.
\end{algorithm}

\subsection{Low-rank compression with $r< M$}
\label{sec_compressed}

As discussed in \cref{rkhskedmd_sect}, kEDMD seeks to avoid the curse of dimensionality when $N$ is large. In the previous subsection, we considered the case of $N=M$. Here, we show that we can achieve further dimension reduction by using the basis in \cref{new_basis}:
$$
    u_j= \sum_{i=1}^N\sum_{\ell=1}^r\mathfrak{K}_{x^{(i)}} [U_r]_{i\ell}[\hat{\Sigma}_r^\dagger]_{\ell j}\in\mathcal{H},\quad 1\leq j\leq r.
$$
To motivate the discussion, consider the generalized eigenvalue problem:
\begin{equation}
    \label{geneigprob_eqn}
    A\m{g}=\lambda G\m{g}.
\end{equation}
Since $G$ is symmetric and positive-definite, there exists a diagonal matrix $\Sigma$ with positive non-increasing entries and a unitary matrix $U$ such that $G=U\Sigma^2U^*$. This is similar to \cref{kedmd_alg} but with $G_{jk}=\hat{G}_{kj}$ and $A_{jk}=\hat{A}_{kj}$. This transpose means that we can take $U=\overline{\hat{Q}}$ and $\Sigma=\hat\Sigma$, where the matrices $\hat{Q}$ and $\hat\Sigma$ are from \cref{kedmd_alg}.
Hence, \eqref{geneigprob_eqn} can be rewritten as
$$
    (U\Sigma^{-1})^*A(U\Sigma^{-1})(\Sigma U^*)\m{g}=\lambda (\Sigma U^*)\m{g}.
$$
Let $\Sigma_r$ be the diagonal matrix containing the $r$ dominant eigenvalues and $U_r$ be the matrix whose columns contain the first $r$ eigenvectors for $r\leq N$. Then we expect that the solutions $(\tilde{\lambda},\m{\tilde{g}})$ to the reduced-dimension eigenvalue problem
$$
    (U_r\Sigma^{-1}_r)^*A(U_r\Sigma^{-1}_r)(\Sigma_r U_r^*)\m{\tilde{g}}=\tilde{\lambda} (\Sigma_r U_r^*)\m{\tilde{g}}
$$
are close to an eigenpair for $\mathbb{K}^*=G^{-1}A$; additionally, $(U_r\Sigma^{-1}_r)^*A(U_r\Sigma^{-1}_r)\in\mathbb{C}^{r\times r}$ rather than $\mathbb{C}^{N\times N}$, so we can have a significant reduction in computation time.

Suppose $(\tilde{\lambda},\tilde{\m{g}})$ is an eigenpair of $(U_r\Sigma_r^{-1})^*A(U_r\Sigma_r^{-1})$. We verify this eigenpair using the full residual from SpecRKHS-Eig (\cref{evalverif_alg}). Since $(U_r\Sigma^{-1}_r)^*G(U_r\Sigma^{-1}_r)=I$, this yields
$$
    \res^*(\tilde{\lambda},U_r\Sigma^{-1}_r\m{\tilde{g}})=\sqrt{\frac{(U_r\Sigma^{-1}_r\m{\tilde{g}})^*(R-\tilde{\lambda} A^*-\overline{\tilde{\lambda}}A+|\tilde{\lambda}|^2G)(U_r\Sigma^{-1}_r\m{\tilde{g}})}{\m{\tilde{g}}^*\m{\tilde{g}}}}.
$$
Similarly, we may also compute pseudospectra in this compressed basis. We have that
\begin{equation*}
    \begin{split}
        \left[\tau_r(z_j)\right]^2=\min_{\m{\tilde{g}}\in\mathbb{C}^N}\left[\res^*(z_j,U_r\Sigma^{-1}_r\m{\tilde{g}})\right]^2=\min_{\m{g}\in\mathrm{im}(U_r\Sigma^{-1}_r)}\left[\res^*(z_j,\m{g})\right]^2\geq \min_{\m{g}\in\mathbb{C}^N}\left[\res^*(z_j,\m{g})\right]^2=\left[\tau(z_j)\right]^2,
    \end{split}
\end{equation*}
where $\mathrm{im}(B)$ denotes the image of a matrix $B$.
Hence, following \cref{verepairs_sect}, we find that $\tau_r(z_j)\geq \sigma_{\inf}(\koop^*-zI)$. If $\tau_r(z_j)<\epsilon$, then $z_j\in\spec_{\mathrm{ap},\epsilon}(\koop^*)$, and we may use this method to verify that points are in the approximate point pseudospectrum. Note that when $r=N$, we have $\mathrm{im}(U_N\Sigma_N^{-1})=\mathbb{C}^N$, and the two residuals are equal.

\subsection{Convergence results with noise-free snapshots}
\label{sec:initial_convergence}

We now prove the convergence of SpecRKHS-PseudoPF (\cref{pspecadjoint_alg}) to the approximate point pseudospectrum of $\koop^*$. We assume that the snapshot data are exact (i.e., $y^{(m)}=F(x^{(m)})$). Later in \cref{inexactupperbounds_sect}, we show that our algorithms can still be implemented and converge in the inexact setting (e.g., caused by noisy measurements). We also assume that the span of a countable subset of the kernel functions $\mathfrak{K}_x$ forms a core of $\koop^*$. This assumption ensures that even for unbounded operators, the operator's behavior—when restricted to a countable set of kernel functions—faithfully reflects that of the full operator. In the case of bounded operators, the following result shows that this assumption is automatically satisfied.

\begin{lemma}
    \label{kernelorthonormalbasis_lemma}
    Let $\mathcal{H}$ be an RKHS with kernel $\mathfrak{K}$ consisting of functions on $\mathcal{X}$. Suppose that there exists a countable subset $\mathcal{X}_0\subset\mathcal{X}$ such that for all $g\in\mathcal{H}$, $g\restriction_{\mathcal{X}_0}=0$ if and only if $g=0$. Then, we can construct an orthonormal basis of $\mathcal{H}$ from finite linear combinations of kernel functions at points in $\mathcal{X}_0$.
\end{lemma}

\begin{proof}
    We begin by enumerating the elements of $\mathcal{X}_0$ as $\{\hat{x}_1,\hat{x}_2,\dots\}$. Using a Gram--Schmidt procedure, we construct an orthonormal system from $\{\mathfrak{K}_{\hat{x}_1},\mathfrak{K}_{\hat{x}_2},\dots\}$, and claim that the corresponding orthonormal system, $\{f_1,f_2,\dots\}$, is complete. Let $f\in\mathcal{H}$ such that $\langle f,f_i\rangle_{\mathfrak{K}}=0$ for each $i$. Then, $f(\hat{x}_1)=0$ since $\langle f,f_1\rangle_{\mathfrak{K}}=0$ and $f_1$ is proportional to $\mathfrak{K}_{\hat{x}_1}$. Moreover, since $\langle f,f_2\rangle_{\mathfrak{K}}=0$ and $f_2$ is a linear combination of $\mathfrak{K}_{\hat{x}_1}$ and $\mathfrak{K}_{\hat{x}_2}$, we also have $f(\hat{x}_2)=0$. Pursing by induction, we find that $f|_{\mathcal{X}_0}=0$, and so by assumption $f=0$ identically, completing the proof.
\end{proof}

\begin{example}[Sobolev space $H^r(U)$] \label{ex_sob_space}
    Let $r>d/2$ and consider the Sobolev space $H^r(U)$ for $U\subset\mathbb{R}^d$ open. All elements of $H^r(U)$ are continuous so taking $\mathcal{X}_0=U\cap\mathbb{Q}^d$ suffices. Hence, if $\{q_1,q_2,\dots\}$ is a list of the rationals in some open set $U\subset \mathbb{R}^d$, then the span of $\{\mathfrak{K}_{q_1},\mathfrak{K}_{q_2},\dots\}$ forms a core for any bounded operator on $H^r(U)$ for $r>d/2$. Alternatively, on any weighted sequence space, as functions are defined at countably many points, $\mathcal{X}_0=\mathcal{X}$ satisfies the assumptions of \cref{kernelorthonormalbasis_lemma}.
\end{example}

Let $\{\mathfrak{K}_1,\mathfrak{K}_2,\dots\}$ be a countable subset of the kernel functions whose span forms a core of $\koop^*$, define $V_N=\spann\{\mathfrak{K}_1,\dots,\mathfrak{K}_N\}$, and let $\mathcal{P}_N:\mathcal{H}\rightarrow V_N$ be the orthogonal projection onto $V_N$. Note that
\begin{equation}
    \label{eqn:sigmainfproj}
    \sigma_{\inf}((\koop^*-zI)\mathcal{P}_N^*)=\inf_{g\in\spann\{\mathfrak{K}_{1},\dots,\mathfrak{K}_{N}\}}\frac{\|(\koop^*-zI)g\|_{\mathfrak{K}}}{\|g\|_{\mathfrak{K}}}.
\end{equation}
We begin by proving that $\sigma_{\inf}((\koop^*-zI)\mathcal{P}_N^*)$ converges down to $\sigma_{\inf}(\koop^*-zI)$ as $N\rightarrow\infty$.

\begin{lemma}\label{sigmainfconv}
    The function $z\mapsto\sigma_{\inf}((\koop^*-zI)\mathcal{P}_N^*)$ is non-increasing in $N$ and converges down to $\sigma_{\inf}(\mathcal{K}^*-zI)$ as $N\rightarrow\infty$ uniformly on compact subsets of $\mathbb{C}$.
\end{lemma}

\begin{proof}
    This result is a standard `folklore result' from operator theory, but we include a proof for completeness. As $V_{N}\subset V_{N+1}$, $\sigma_{\inf}((\koop^*-zI)\mathcal{P}_N^*)$ is a non-increasing function of $N$.
    Additionally, as $V_{N}\subset\mathcal{D}(\mathcal{K}^*)$, $\sigma_{\inf}((\koop^*-zI)\mathcal{P}_N^*)\geq \sigma_{\inf}(\mathcal{K}^*-zI)$ for all $z\in\mathbb{C}$ and $N$. Fix $0<\epsilon<1$ and $z\in\mathbb{C}$; by definition of the injection modulus there exists $g\in\mathcal{D}(\mathcal{K}^*)$ such that
    \begin{equation}\label{eq1lem1}
        \|g\|_{\mathfrak{K}}=1 \quad\text{and}\quad \|(\mathcal{K}^*-z I)g\|_{\mathfrak{K}}\leq \sigma_{\inf}(\mathcal{K}^*-z I)+\epsilon.
    \end{equation}
    As $\cup_{N=1}^{\infty}V_N$ forms a core, there exists a sequence $g_N\in V_N$ such that $g_N\rightarrow g$ and $\mathcal{K}^*g_N\rightarrow\mathcal{K}^*g$ as $N\rightarrow\infty$. Hence, for sufficiently large $N_0$, there exists $g_{N_0}\in V_{N_0}$ such that
    \begin{equation}\label{eq2lem1}
        \|g-g_{N_0}\|_{\mathfrak{K}}\leq\epsilon \quad\text{and}\quad \|(\mathcal{K}^*-z I)g_{N_0}\|_{\mathfrak{K}}\leq\|(\mathcal{K}^*-z I)g\|_{\mathfrak{K}}+\epsilon.
    \end{equation}
    Then by \cref{eq1lem1,eq2lem1} and the triangle inequality, we have that $\|g_{N_0}\|_{\mathfrak{K}}\geq 1-\epsilon$ and
    $$
        \sigma_{\inf}((\koop^*-zI)\mathcal{P}_{N_0}^*)\leq\frac{\|(\mathcal{K}^*-z I)g_{N_0}\|_{\mathfrak{K}}}{\|g_{N_0}\|_{\mathfrak{K}}}\leq\frac{\|(\mathcal{K}^*-z I)g\|_{\mathfrak{K}}+\epsilon}{\|g_{N_0}\|_{\mathfrak{K}}}\leq\frac{\sigma_{\inf}(\mathcal{K}^*-z I)+2\epsilon}{1-\epsilon}.
    $$
    Since $\epsilon$ with $0<\epsilon<1$ was arbitrary, $\lim_{N\rightarrow\infty}\sigma_{\inf}((\koop^*-zI)\mathcal{P}_N^*)=\sigma_{\inf}(\mathcal{K}^*-zI)$ pointwise, so as the convergence is monotonic and all functions involved are continuous, by Dini's theorem the convergence is uniform on compact subsets of $\mathbb{C}$.
\end{proof}

The output of SpecRKHS-PseudoPF (\cref{pspecadjoint_alg}) is the set of points in a grid $\gridop(N)$ such that $\sigma_{\inf}((\koop^*-zI)\mathcal{P}_N^*)<\epsilon$, i.e., we define
\begin{equation}
    \label{one_lim_needed}
    \Gamma_{N}^{\epsilon}(\mathcal{K}^*)=\{z\in\gridop(N)\;:\;\sigma_{\inf}((\koop^*-zI)\mathcal{P}_N^*)<\epsilon\}.
\end{equation}
The following theorem shows that $\Gamma_{N}^{\epsilon}(\mathcal{K}^*)$ converges to the approximate point pseudospectrum of $\koop^*$. This result guarantees that our method yields accurate results with sufficient (noise-free) data: no spurious eigenvalues in the limit and no missed spectra (within the RKHS setting). Moreover, this convergence occurs in a single limit as $N\rightarrow\infty$. Thus, computing spectra on an RKHS is fundamentally easier than on $L^2$, a fact we formalize in \cref{sci_sect}.

\begin{theorem}[Convergence to the approximate point pseudospectrum of $\koop^*$]\label{adjointapspecconv_thm}
    Let $\epsilon>0$, then for all $N\in\mathbb{N}$, $\Gamma_{N}^{\epsilon}(\mathcal{K}^*)\subset\spec_{\mathrm{ap},\epsilon}(\mathcal{K}^*)$ and
    $
        \lim_{N\rightarrow\infty}d_{\mathrm{AW}}\left(\Gamma_{N}^{\epsilon}(\mathcal{K}^*),\spec_{\mathrm{ap},\epsilon}(\mathcal{K}^*)\right)=0.
    $
\end{theorem}

\begin{proof}
    Throughout this proof, let $\epsilon>0$.
    For any $z\in\Gamma_{N}^{\epsilon}(\mathcal{K}^*)$, by \cref{sigmainfconv} we have $\sigma_{\inf}(\mathcal{K}^*-zI)\leq\sigma_{\inf}((\koop^*-zI)\mathcal{P}_N^*)<\epsilon$. It follows that $\Gamma_{N}^{\epsilon}(\mathcal{K}^*)\subset\spec_{\mathrm{ap},\epsilon}(\mathcal{K}^*)$ for all $N\in\mathbb{N}$.

    To prove convergence, if $\spec_{\mathrm{ap},\epsilon}(\mathcal{K}^*)=\emptyset$ we are done, so assume that $\spec_{\mathrm{ap},\epsilon}(\mathcal{K}^*)\neq\emptyset$. Then there exists $z\in\mathbb{C}$ such that $\sigma_{\inf}(\mathcal{K}^*-z I)<\epsilon$.
    Since $\lim_{N\rightarrow\infty}\dist(z,\gridop(N))=0$ and the map $z\rightarrow\sigma_{\inf}(\mathcal{K}^*-zI)$ is continuous, \cref{sigmainfconv} implies that $\Gamma_{N}^{\epsilon}(\mathcal{K}^*)\neq \emptyset$ for large $N$. Hence, by the characterization of the Attouch--Wets topology for non-empty sets, as $\Gamma_{N}^{\epsilon}(\mathcal{K}^*)\subset\spec_{\mathrm{ap},\epsilon}(\mathcal{K}^*)$, it suffices to show that for every $\delta>0$ and all large $m\in\mathbb{N}$, $\spec_{\mathrm{ap},\epsilon}(\mathcal{K}^*)\cap\;B_m(0)\subset \Gamma_{N}^{\epsilon}(\mathcal{K}^*)+B_{\delta}(0)$ for sufficiently large $N$.

    Suppose, for a contradiction, that this final statement is false. By taking subsequences if necessary, we may assume that there exists $\delta>0$ and a sequence $z_N\in\spec_{\mathrm{ap},\epsilon}(\mathcal{K}^*)\cap B_m(0)$ such that $\dist(z_N,\Gamma_N^{\epsilon}(\mathcal{K}^*))\geq\delta$ for all $N$ and $z_N\rightarrow z$. As $\spec_{\mathrm{ap},\epsilon}(\mathcal{K}^*)$ is closed, $z\in\spec_{\mathrm{ap},\epsilon}(\mathcal{K}^*)\cap B_m(0)$. Since $\spec_{\mathrm{ap},\epsilon}(\mathcal{K}^*)=\closure(\{z\in\mathbb{C}\;:\;\sigma_{\inf}(\mathcal{K}^*-z I)<\epsilon\})$, there exists $w\in B_{\delta/2}(z)$ such that $\sigma_{\inf}(\mathcal{K}^*-w I)<\epsilon$. Then as $\lim_{N\rightarrow\infty}\dist(w,\gridop(N))=0$, there exist $w_N\in\gridop(N)$ such that $|w_N-w|\rightarrow 0$ as $N\rightarrow\infty$. By the triangle inequality,
    \begin{align*}
         & \sigma_{\inf}((\koop^*-w_NI)\mathcal{P}_N^*)                                                                                                                                                      \\
         & \quad\leq|\sigma_{\inf}((\koop^*-w_NI)\mathcal{P}_N^*)-\sigma_{\inf}(\mathcal{K}^*-w_N I)|+|\sigma_{\inf}(\mathcal{K}^*-w_N I)-\sigma_{\inf}(\mathcal{K}^*-wI)|+\sigma_{\inf}(\mathcal{K}^*-w I).
    \end{align*}
    The first term on the right-hand side converges to $0$ by \cref{sigmainfconv} since the points $w_N$ are contained in some compact set, the second term converges to $0$ by continuity as $w_N\rightarrow w$, and the third term is $<\epsilon$ by assumption. Hence, for large $N$, $w_N\in\Gamma_N^{\epsilon}(\mathcal{K}^*)$. However, $|w_N-z|\leq |w_N-w|+|w-z|\leq |w_N-w|+\delta/2<\delta$ for large $N$, giving the desired contradiction.
\end{proof}

\subsection{Computing the full pseudospectrum and spectrum}
\label{sec:fullpspecspec}

We now turn to computing the full pseudospectrum and spectrum of the Perron--Frobenius operator (and, hence, the Koopman operator). Recall that the pseudospectrum can be expressed as
$$
    \spec_{\epsilon}(\koop)=\spec_{\epsilon}(\koop^*)= \closure(\{z\in\mathbb{C}\;:\;\injmod(\koop^*-\overline{z} I)<\epsilon\})\cup \closure(\{z\in\mathbb{C}\;:\;\injmod(\koop-z I)<\epsilon\}).
$$
We aim to estimate $\sigma_{\mathrm{inf}}(\koop-z I)$ using the techniques of \cref{sec:initial_convergence}, which requires computing the following matrix:
$$
    L_{jk}=\langle\koop \mathfrak{K}_{x^{(k)}},\koop \mathfrak{K}_{x^{(j)}}\rangle_{\mathfrak{K}}, \quad 1\leq j,k\leq N.
$$
Unfortunately, there does not appear to be a straightforward way to express $\langle\koop \mathfrak{K}_{x_i},\koop \mathfrak{K}_{x_j}\rangle_{\mathfrak{K}}$ directly in terms of the kernel $\mathfrak{K}$. However, there are at least two methods to overcome this issue.

\subsubsection{Quadrature}

If the samples are chosen to form a convergent quadrature rule with weights $\{w_j\}_{1\leq j\leq M}$, then $L$ can be estimated by quadrature. For example, if we are working in $H^1(\mathbb{R})$, then
\begin{equation*}
    \begin{split}
        L_{jk} & =\lim_{M\rightarrow\infty}\sum_{i=1}^Mw_i\left(\koop \mathfrak{K}_{x^{(k)}}(x^{(i)})\overline{\koop \mathfrak{K}_{x^{(j)}}(x^{(i)})}+(\koop \mathfrak{K}_{x^{(k)}})'(x^{(i)})\overline{(\koop \mathfrak{K}_{x^{(j)}})'(x^{(i)})}\right) \\
               & =\lim_{M\rightarrow\infty}\sum_{i=1}^Mw_i\left(\mathfrak{K}(x^{(k)},y^{(i)})\mathfrak{K}(y^{(i)},x^{(j)})+|F'(x^{(i)})|^2\mathfrak{K}_{x^{(k)}}'(y^{(i)})\overline{\mathfrak{K}_{x^{(j)}}'(y^{(i)})}\right).
    \end{split}
\end{equation*}
The first part of the sum satisfies $\sum_{i=1}^Mw_i\mathfrak{K}(x^{(k)},y^{(i)})\mathfrak{K}(y^{(i)},x^{(j)})=\sum_{i=1}^M(A^*)_{ik}W_{ii}A_{ji}=(AWA^*)_{jk}$, where $W=\diag(w)$. In the spaces associated with the Wendland or Mat{\'e}rn kernels, which are equivalent to Sobolev spaces, estimating the $H^r$-inner products as above provides an estimate of the inner products up to a given constant.

\subsubsection{Rectangular truncations}

\begin{figure}[t]
    \centering
    \includegraphics[scale=0.5]{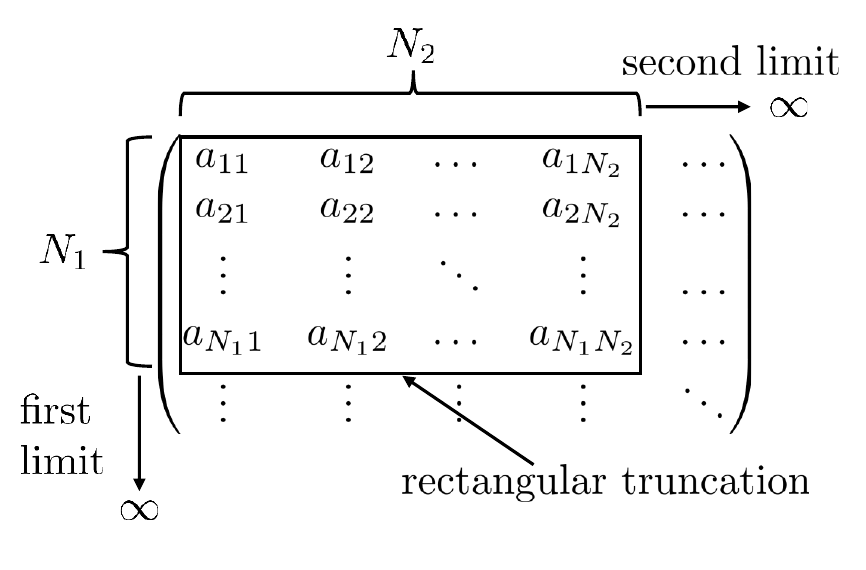}
    \caption{To compute $\sigma_{\inf}(\koop-z I)$, we begin with a $N_1\times N_2$ truncation of $\koop-zI$ from which we may compute $\sigma_{\inf}(\mathcal{P}_{N_1}(\koop-zI)\mathcal{P}_{N_2}^*)$, take the limit as $N_1\rightarrow\infty$ to compute $\sigma_{\inf}((\koop-zI)\mathcal{P}_{N_2}^*)$, and then take the limit as $N_2\rightarrow\infty$ to get $\sigma_{\inf}(\koop-z I)$. These two limits do not commute and generally cannot be combined into a single limit (e.g., by adaptively choosing $N_2$ dependent on $N_1$).}
    \label{fig:rectangular_trunc}
\end{figure}

An alternative and much simpler method is to use rectangular truncations and compute $\sigma_{\inf}(\mathcal{P}_{N_1}(\koop-zI)\mathcal{P}_{N_2}^*)$ for $N_1,N_2\in\mathbb{N}$ and any $z\in\mathbb{C}$. From there we prove that as we take $N_1\rightarrow\infty$ this converges to $\sigma_{\inf}((\koop-zI)\mathcal{P}_{N_2}^*)$ and then \cref{sigmainfconv} says this converges to $\sigma_{\inf}(\koop-zI)$ as $N_2\rightarrow\infty$. \cref{fig:rectangular_trunc} shows the idea. We define
$$
    L_{N_1}(z)=AG^{-1}A^*-{z} A-\overline{z} A^*+|z|^2G\in\mathbb{C}^{N_1\times N_1}.
$$
We assume that $\{\mathfrak{K}_1,\mathfrak{K}_2,\dots\}$ is a countable subset of the kernel functions whose span forms a core of $\koop$ and $\koop^*$, and let $V_n=\spann\{\mathfrak{K}_1,\dots,\mathfrak{K}_n\}$ and $\mathcal{P}_n$ be the orthogonal projection onto $V_n$.

\begin{lemma}
    \label{rectcompute_lemma}
    If $N_1\geq N_2$ and $g=\sum_{j=1}^{N_2}g_j \mathfrak{K}_{j}\in V_{N_2}\subset V_{N_1}$, then
    $$
        \m{g}^*L_{N_1}(z)\m{g}=\|\mathcal{P}_{{N_1}}^*\mathcal{P}_{{N_1}}(\mathcal{K}-{z}I)\mathcal{P}_{N_2}^*\mathcal{P}_{N_2}g\|_{\mathfrak{K}}^2.
    $$
\end{lemma}

\begin{proof}
    Since $N_1\geq N_2$,  $\mathfrak{K}_{j}=\mathcal{P}_{{N_1}}^*\mathcal{P}_{{N_1}}\mathfrak{K}_{j}$ for $1\leq j\leq N_2$, and $\langle \mathfrak{K}_{k},\koop ^*\mathfrak{K}_{j}\rangle_{\mathfrak{K}}=\langle \koop \mathfrak{K}_{k}, \mathfrak{K}_{j}\rangle_{\mathfrak{K}}=\langle\koop \mathfrak{K}_{k},\mathcal{P}_{{N_1}}^*\mathcal{P}_{N_1}\mathfrak{K}_{j}\rangle_{\mathfrak{K}}=\langle \mathcal{P}_{{N_1}}^*\mathcal{P}_{N_1}\mathcal{K}\mathfrak{K}_{k},\mathfrak{K}_{j}\rangle_{\mathfrak{K}}$. It follows that
    \begin{align*}
        \sum_{j,k=1}^{N_2}\overline{g_j}(AG^{-1}A^*)_{jk}g_k & =\sum_{j,k=1}^{N_2}\sum_{m,n=1}^{N_1}\overline{g_j}A_{jm}(G^{-1})_{mn}(A^*)_{nk}{g_k}                                                                                                                                                                                                                       \\
                                                             & =\sum_{j,k=1}^{N_2}\sum_{m,n=1}^{N_1}\overline{g_j}\langle \koop^*\mathfrak{K}_{m},\mathfrak{K}_{j}\rangle_{\mathfrak{K}}(G^{-1})_{mn}\langle \mathfrak{K}_{k},\koop^*\mathfrak{K}_{n}\rangle_{\mathfrak{K}}{g_k}                                                                                           \\
                                                             & =\sum_{j,k=1}^{N_2}\sum_{m,n=1}^{N_1}\overline{g_j}\langle \mathfrak{K}_{m},\mathcal{P}_{{N_1}}^*\mathcal{P}_{{N_1}}\mathcal{K}\mathfrak{K}_{j}\rangle_{\mathfrak{K}}(G^{-1})_{mn}\langle \mathcal{P}_{{N_1}}^*\mathcal{P}_{{N_1}}\mathcal{K}\mathfrak{K}_{k},\mathfrak{K}_{n}\rangle_{\mathfrak{K}} {g_k}.
    \end{align*}
    If we define coefficients $a_{ij}$ such that $\mathcal{P}_{{N_1}}^*\mathcal{P}_{{N_1}}\mathcal{K}\mathfrak{K}_{i}=\sum_{j=1}^{N_1}a_{ij}\mathfrak{K}_{j}$, then $\sum_{j,k=1}^{N_2}\overline{g_j}(AG^{-1}A^*)_{jk}{g_k}$ is equal to
    \begin{align*}
        \sum_{j,k=1}^{N_2}\sum_{m,n,p,q=1}^{N_1}\overline{g_j}\langle \mathfrak{K}_{m},a_{jp}\mathfrak{K}_{p}\rangle_{\mathfrak{K}}(G^{-1})_{mn}\langle a_{kq}\mathfrak{K}_{q},\mathfrak{K}_{n}\rangle_{\mathfrak{K}} {g_k} & =\sum_{j,k=1}^{N_2}\sum_{m,n,p,q=1}^{N_1}\overline{g_j}\overline{a_{jp}}G_{pm}(G^{-1})_{mn}G_{nq}{a_{kq}}{g_k} \\
                                                                                                                                                                                                                            & =\sum_{j,k=1}^{N_2}\sum_{m,p=1}^{N_1}\overline{g_ja_{jp}}G_{pm}{a_{km}}{g_k}.
    \end{align*}
    Conversely,
    $$
        \langle\mathcal{P}_{{N_1}}^*\mathcal{P}_{{N_1}}\koop g,\mathcal{P}_{{N_1}}^*\mathcal{P}_{{N_1}}\koop g\rangle_{\mathfrak{K}}=\sum_{j,k=1}^{N_2}\sum_{m,p=1}^{N_1}\langle g_ka_{km}\mathfrak{K}_{m},g_ja_{jp}\mathfrak{K}_{p}\rangle_{\mathfrak{K}}=\sum_{j,k=1}^{N_2}\sum_{m,p=1}^{N_1}\overline{g_ja_{jp}}G_{pm}{a_{km}}{g_k}.
    $$
    Thus, we have agreement. The remaining terms are similar.
\end{proof}

It follows that
$$
    \sigma_{\inf}(\mathcal{P}_{N_1}(\koop^*-zI)\mathcal{P}_{N_2}^*)=\min_{g\in V_{N_2}}\m{g}^*L_{N_1}(z)\m{g},
$$
which can be computed as $L_{N_1}(z)$ is a finite matrix. This can be obtained by solving a generalized eigenvalue problem, and \cref{inexactupperbounds_sect} will discuss how this can be done using arithmetic algorithms and error control. We now want to show that $\lim_{N_2\rightarrow\infty}\lim_{N_1\rightarrow\infty}\sigma_{\inf}(\mathcal{P}_{N_1}(\koop-zI)\mathcal{P}_{N_2}^*)=\sigma_{\inf}(\koop-zI)$ using the following lemma. Note that \cref{sigmainfconv} previously showed that $\lim_{N_2\rightarrow\infty}\sigma_{\mathrm{\inf}}((\koop-zI)\mathcal{P}_{N_2}^*)=\sigma_{\mathrm{inf}}(\koop-zI)$.

\begin{lemma}
    \label{n2n1limit_lemma}
    The function
    $z\mapsto\sigma_{\inf}(\mathcal{P}_{N_1}(\koop-zI)\mathcal{P}_{N_2}^*)$ is non-decreasing in $N_1$ and converges up to $\sigma_{\inf}((\koop-zI)\mathcal{P}_{N_2}^*)$ as $N_1\rightarrow\infty$ uniformly on compact subsets of $\mathbb{C}$.
\end{lemma}

\begin{proof}
    As $N_1$ increases, the range of $\mathcal{P}_{N_1}$ increases, so the injection modulus cannot decrease. Then, by Dini's theorem, the convergence is uniform on compact subsets of $\mathbb{C}$ if we can prove pointwise convergence. Let $z\in\mathbb{C}$, by strong convergence of orthogonal projections to the identity (using the fact that the kernel functions form a core) and the finite dimensionality of the range of $\mathcal{P}_{N_2}$, we have
    $$
        \lim_{N_1\rightarrow\infty}\|\mathcal{P}_{N_1}^*\mathcal{P}_{N_1}(\koop-zI)\mathcal{P}_{N_2}^*-(\koop-zI)\mathcal{P}_{N_2}^*\|_{\mathfrak{K}}=0.
    $$
    Since $\sigma_{\inf}$ is continuous with respect to the operator norm topology on the set of bounded operators (which we are dealing with due to the projections, without assuming $\koop$ is bounded), we have the desired pointwise convergence.
\end{proof}

\begin{algorithm}[t]
    \caption{SpecRKHS-PseudoKoop: Computing the approximate point pseudospectrum of Koopman operators.}\label{pspeckoop_alg}
    \textbf{Input:} Parameters $N_1\geq N_2$, snapshot data $\{(x^{(i)},y^{(i)}=F(x^{(i)}))\}_{i=1}^{N_1}$, kernel function $K:\mathcal{X}\times\mathcal{X}\rightarrow\mathbb{C}$, target accuracy $\epsilon>0$, and finite grid $\gridop(N_2)$.

    \begin{algorithmic}[1]
        \STATE{Compute the matrices $G$, $A$ and $R$ defined by $\smash{G_{jk}=\mathfrak{K}(x^{(k)},x^{(j)})}$, $\smash{A_{jk}=\mathfrak{K}(y^{(k)},x^{(j)})}$ and $\smash{R_{jk}=\mathfrak{K}(y^{(k)},y^{(j)})}$.\!\!\!\!\!}
        \FOR {$z_j\in\gridop(N_2)$}
        \STATE{Define $L_{N_1}(z_j)=AG^{-1}A^*-z_jA-\overline{z_j}A^*+|z_j|^2G$ and take its $N_2\times N_2$ truncation, $L_{N_1,N_2}(z)$.}
        \STATE{Define $G_{N_1,N_2}$ as the $N_2\times N_2$ truncation of $G$.}
        \STATE{Solve the generalized eigenvalue problem $L_{N_1,N_2}(z_j)\m{g}=\lambda G_{N_1,N_2}\m{g}$ for the smallest eigenvalue $\tau(z_j)$.}
        \ENDFOR
    \end{algorithmic}
    \textbf{Output:} An approximation of the approximate point pseudospectrum $\hat{\Gamma}_{N_2,N_1}^{\epsilon}(\koop)=\{z_j\in\gridop(N_2)\;:\;\tau(z_j)+1/N_2<\epsilon\}$ that converges as $N_1\rightarrow\infty$, and then as $N_2\rightarrow\infty$.
\end{algorithm}

To construct an algorithm converging to the approximate point pseudospectrum of the Koopman operator, one may try to choose $\Gamma_{N_2,N_1}^{\epsilon}(\koop)=\{z\in\gridop(N_2):\sigma_{\inf}(\mathcal{P}_{N_1}(\koop-zI)\mathcal{P}_{N_2}^*)<\epsilon\}$. This may not converge as the resolvent norm of an unbounded operator can be constant on open sets \cite{shargorodsky_level_2008}. Instead, we define
\begin{equation} \label{eq_alg_app_point_sp}
    \hat{\Gamma}_{N_2,N_1}^{\epsilon}(\koop)=\{z\in\gridop(N_2):\sigma_{\inf}(\mathcal{P}_{N_1}(\koop-zI)\mathcal{P}_{N_2}^*)+1/N_2\leq \epsilon\}.
\end{equation}
\cref{pspeckoop_alg} summarizes the algorithm, which we call SpecRKHS-PseudoKoop, for computing $\spec_{\mathrm{ap},\epsilon}(\koop)$.

We first show that $\hat{\Gamma}_{N_2,N_1}^{\epsilon}(\koop)$ converges as $N_1\rightarrow\infty$ to
\begin{equation} \label{eq_conv_Gamma_pseudo}
    \hat{\Gamma}_{N_2}^{\epsilon}(\koop)=\{z\in\gridop(N_2):\sigma_{\inf}((\koop-zI)\mathcal{P}_{N_2}^*)+1/N_2\leq\epsilon\}.
\end{equation}
Note the additional $1/N_2$ terms in \cref{eq_alg_app_point_sp,eq_conv_Gamma_pseudo} compared to the previous attempt to enable a non-strict inequality $\leq$ rather than a strict one, which allows us to prove the following lemma.

\begin{lemma}[Convergence as $N_1\to\infty$]
    \label{apspeclimit1_lemma}
    Let $\epsilon>0$ and $N_2\geq 1$, then $\lim_{N_1\rightarrow\infty}\hat{\Gamma}_{N_2,N_1}^{\epsilon}(\mathcal{K})=\hat{\Gamma}_{N_2}^{\epsilon}(\mathcal{K})$.
\end{lemma}

\begin{proof}
    We show that for $N_1$ large enough, $\hat{\Gamma}_{N_2,N_1}^{\epsilon}(\koop)=\hat{\Gamma}_{N_2}^{\epsilon}(\koop)$. Let $z\in\gridop(N_2)$, and suppose $z\in\hat{\Gamma}_{N_2}^{\epsilon}(\mathcal{K})$, so $\sigma_{\inf}((\koop-zI)\mathcal{P}_{N_2}^*)+1/N_2\leq\epsilon$. By the first part of \cref{n2n1limit_lemma}, $\sigma_{\inf}(\mathcal{P}_{N_1}(\koop-zI)\mathcal{P}_{N_2}^*)+1/N_2\leq\epsilon$ for all $N_1$, therefore $z\in\hat{\Gamma}_{N_2,N_1}^{\epsilon}(\mathcal{K})$ for all $N_1$. Conversely, suppose that $z\notin\hat{\Gamma}_{N_2}^{\epsilon}(\mathcal{K})$, such that $\sigma_{\inf}((\koop-zI)\mathcal{P}_{N_2}^*)+1/N_2>\epsilon$. Following the second part of \cref{n2n1limit_lemma}, $\lim_{N_1\rightarrow\infty}\sigma_{\inf}(\mathcal{P}_{N_1}(\koop-zI)\mathcal{P}_{N_2}^*)=\sigma_{\inf}((\koop-zI)\mathcal{P}_{N_2}^*)$ and so for $N_1$ large enough (and monotonicity), $\sigma_{\inf}(\mathcal{P}_{N_1}(\koop-zI)\mathcal{P}_{N_2}^*)+1/N_2>\epsilon$ and $z\notin\hat{\Gamma}_{N_2,N_1}^{\epsilon}(\mathcal{K})$. Moreover, since $\gridop(N_2)$ is finite (so we can pick $N_1$ sufficiently large for all $\lambda\notin\hat{\Gamma}_{N_2}^{\epsilon}(\mathcal{K})$) the claim holds.
\end{proof}

Finally, analogously to \cref{adjointapspecconv_thm}, we have the following result, whose proof is similar to before.

\begin{theorem}[Convergence to the approximate point pseudospectrum of $\koop$]\label{koopapspecconv_thm}
    For all $N_2$, $\hat{\Gamma}_{N_2}^{\epsilon}(\mathcal{K})\subset{\spec_{\mathrm{ap},\epsilon}(\mathcal{K})}$ and
    $
        \lim_{N_2\rightarrow\infty}d_{\mathrm{AW}}(\hat{\Gamma}^{\epsilon}_{N_2}(\mathcal{K}),{\spec_{\mathrm{ap},\epsilon}(\mathcal{K})})=0.
    $
\end{theorem}

The combination of \cref{adjointapspecconv_thm,koopapspecconv_thm} shows that the full pseudospectrum, and hence the whole spectrum, can be recovered.

\begin{corollary}[Convergence to the full spectrum of $\mathcal{K}$]
    Let $\hat{\Gamma}_{N_2,N_1}^{\epsilon}(\koop)$ and $\Gamma_{N_2}^{\epsilon}(\mathcal{K}^*)$ be respectively defined by \cref{eq_alg_app_point_sp,one_lim_needed}, then
    $$\lim_{\epsilon\rightarrow 0}\lim_{N_2\rightarrow\infty}\lim_{N_1\rightarrow\infty}\left[\Gamma_{N_2}^{\epsilon}(\mathcal{K}^*)\cup\left\{\overline{z}\in\mathbb{C}:z\in\hat{\Gamma}^{\epsilon}_{N_2,N_1}(\mathcal{K})\right\}\right]=\spec(\mathcal{K}^*)=\left\{\overline{z}\in\mathbb{C}:z\in\spec(\koop)\right\}.$$
\end{corollary}

\begin{remark}
    As we require the image of $\mathcal{P}_{{N_2}}^*$ to be finite-dimensional in the proof of \cref{n2n1limit_lemma}, and that $\gridop(N_2)$ is finite even as $N_1\rightarrow\infty$ in \cref{apspeclimit1_lemma}, we cannot take the limits $N_1$ and $N_2$ simultaneously to infinity. In general, we require an extra limit to compute the pseudospectrum compared to the approximate point pseudospectrum because we cannot estimate $\langle\mathcal{K}g,\mathcal{K}g\rangle_{\mathfrak{K}}$ directly from the kernels.
\end{remark}

\subsection{Perron--Frobenius mode decompositions: Error control for observable predictions}
\label{sec:errorcontrolpfmd}

We can also use residuals to estimate the evolution of observables according to the dynamical system \cref{the_dynamical_system}, starting from the initial state $x_0\in\mathcal{X}$. Suppose that $\koop$ is bounded. For all $n\in\mathbb{N}$ and $g\in\mathcal{H}$, we have
$$
    g(x_n)=[\koop^ng](x_0)=\langle \koop^ng,\mathfrak{K}_{x_0}\rangle_\mathfrak{K}=\langle g,(\koop^*)^n\mathfrak{K}_{x_0}\rangle_\mathfrak{K},
$$
so it suffices to compute $(\koop^*)^n\mathfrak{K}_{x_0}$. Suppose that for some $\epsilon>0$ we have a set of verified approximate eigenpairs of $\koop^*$, $\{\lambda_i,\psi_i\}_{i=1}^{N_{\mathrm{ver}}}$, such that for each $i$,
$$
    \|(\koop^*-\lambda_i I)\psi_i\|_{\mathfrak{K}}\leq\epsilon,\quad\text{for }i=1,\ldots,N_{\mathrm{ver}}.
$$
Then,
\begin{equation} \label{ineq_K_eps}
    \|(\koop^*)^n\psi_i-\lambda_i^n \psi_i\|_{\mathfrak{K}}\leq \sum_{j=1}^n\|\lambda_i^{n-j}(\koop^*)^{j-1}(\koop^*-\lambda_i I)\psi_i\|_\mathfrak{K}\leq\epsilon\sum_{j=1}^n|\lambda_i|^{n-j}\|\koop^*\|_\mathfrak{K}^{j-1}\quad\forall n\in\mathbb{N}.
\end{equation}
In most cases, $|\lambda_i|\leq 1$. For example, we shall see in \cref{sec_duffing} that on an RKHS with radial kernel $\mathfrak{K}(x,y)=\varphi(\|x-y\|)$, if $\koop$ is densely defined, then it is bounded with $\|\koop\|_{\mathfrak{K}}=\|\koop^*\|_{\mathfrak{K}}=1$, so that $|\lambda_i|\leq 1$ for all $i$.

For any fixed $x_0\in\mathcal{X}$, there exist $\delta>0$ and coefficients $c_i$ such that
\begin{equation} \label{ineq_K_delta}
    \left\|\mathfrak{K}_{x_0}-\sum_{i=1}^{N_{\mathrm{ver}}}c_i\psi_i\right\|_{\mathfrak{K}}\leq\delta,
\end{equation}
where the coefficients are typically computed by solving a least squares problem, and all the inner products can be straightforwardly evaluated as they are between kernel functions. By the triangle inequality, we obtain
\begin{align*}
    \left\|(\koop^*)^n\mathfrak{K}_{x_0}-\sum_{i=1}^{N_{\mathrm{ver}}}c_i\lambda_i^n\psi_i\right\|_{\mathfrak{K}} & \leq
    \|(\koop^*)^n\|\left\|\mathfrak{K}_{x_0}-\sum_{i=1}^{N_{\mathrm{ver}}}c_i\psi_i\right\|_{\mathfrak{K}}+\sum_{i=1}^{N_{\mathrm{ver}}}|c_i|\left\|((\koop^*)^n-\lambda_i^n I)\psi_i\right\|_{\mathfrak{K}}                                                    \\
                                                                                                                  & \leq\|\koop^*\|_\mathfrak{K}^n\delta+\epsilon\sum_{i=1}^{N_{\mathrm{ver}}}|c_i|\sum_{j=1}^n|\lambda_i|^{n-j}\|\koop^*\|_\mathfrak{K}^{j-1},
\end{align*}
where we used \cref{ineq_K_delta} to bound the first term and \cref{ineq_K_eps} for the second term. Then, our approximation to $[\koop^ng](x_0)$ is given by
$\sum_{i=1}^{N_{\mathrm{ver}}}\overline{c_i}\overline{\lambda_i^n}\langle g,\psi_i\rangle_\mathfrak{K}$, where the conjugates arise from the definition of the inner product, $\langle g,\psi_i\rangle_\mathfrak{K}$ can be straightforwardly computed by simply evaluating $g$ as the $\psi_i$ are linear combinations of kernel functions, and the error is given by
\begin{equation}
    \label{error_pfmd}
    \left|g(F^n(x_0))-\sum_{i=1}^{N_{\mathrm{ver}}}\overline{c_i}\overline{\lambda_i^n}\langle g,\psi_i\rangle_\mathfrak{K}\right|\leq \|g\|_{\mathfrak{K}}\left(\delta\|\koop^*\|_\mathfrak{K}^n+\epsilon\sum_{i=1}^{N_{\mathrm{ver}}}|c_i|\sum_{j=1}^n|\lambda_i|^{n-j}\|\koop^*\|_\mathfrak{K}^{j-1}\right).
\end{equation}
Hence, spectral analysis of the Perron--Frobenius operator yields error control for predicting the future evolution of observables! We call this algorithm SpecRKHS-Obs and summarize it in \cref{errorcontrolpred_alg}.

\begin{algorithm}[t]
    \caption{SpecRKHS-Obs: Observables prediction with verified Perron--Frobenius mode decompositions.}\label{errorcontrolpred_alg}
    \textbf{Input:} Snapshot data $\{(x^{(i)},y^{(i)}=F(x^{(i)}))\}_{i=1}^N$, kernel function $\mathfrak{K}:\mathcal{X}\times\mathcal{X}\rightarrow\mathbb{C}$, $\epsilon>0$, verified approximate eigenpairs $\{\lambda_i,\psi_i\}_{i=1}^{N_{\mathrm{ver}}}$ with $\lambda_i\in\mathbb{C}$, $\psi_i\in\spann\{\mathfrak{K}_{x_1},\dots,\mathfrak{K}_{x_N}\}$ and $\|(\koop^*-\lambda_i)\psi_i\|\leq\epsilon$ (generated by, e.g., SpecRKHS-Eig (\cref{evalverif_alg})), starting point $x_0$, pairs of observables and numbers of time steps $\{g_k,n_k\}_{k=1}^N$ with each $g_k\in\mathcal{D}(\koop^*)$ and $n_k\in\mathbb{N}$.

    \begin{algorithmic}[1]
        \STATE{Solve the least squares problem $\min_{c_i}\|\mathfrak{K}_{x_0}-\sum_{i=1}^{N_{\mathrm{ver}}}c_i\psi_i\|_{\mathfrak{K}}$ for coefficients $\{c_i\}_{i=1}^{N_{\mathrm{ver}}}$.}
        \STATE{For each pair $(g_k,n_k)$, compute $\Phi(g_k,n_k)=\sum_{i=1}^{N_{\mathrm{ver}}}\overline{c_i}\overline{\lambda_i^{n_k}}\langle g_k,\psi_i\rangle $.}
    \end{algorithmic}
    \textbf{Output:} Approximations $\Phi(g_k,n_k)$ to  $[\koop^{n_k}g_k](x_0)$ with error that can be computed (if $\koop$ is bounded) by \eqref{error_pfmd}.
    \label{alg_verifiedPF}
\end{algorithm}

SpecRKHS-Obs constructs \textit{verified} Perron--Frobenius mode decompositions.
There are two key components of the error: $\delta$, which is determined by the richness of the space of pseudoeigenfunctions (i.e., how well linear combinations of the pseudoeigenfunctions approximate $\mathfrak{K}_{x_0}$), and $\epsilon$, which measures the accuracy of the pseudoeigenfunctions. This method also reveals a duality between $\koop$ and $\koop^*$ when predicting observables. For $\koop$, we have to solve a least squares problem for every different observable $g$, but then we can evaluate $g$ straightforwardly at any given point. For $\koop^*$, we have to solve a least squares problem at every point $x_0$ and corresponding kernel function $\mathfrak{K}_{x_0}$, but then can straightforwardly evaluate every observable at that point $x_0$. Hence, Perron--Frobenius methods may be preferred when the state space dimension $d$ is large (with several $g$ to predict), but the number of starting points is small. On the other hand, Koopman-based methods may be preferred when we want to predict the evolution of a small number of observables at a large number of starting points.

\subsection{Handling noise and approximate snapshots}
\label{inexactupperbounds_sect}

In applications, we may not have access to exact values $y^{(m)}=F(x^{(m)})$ as they may require physical measurements with an error tolerance. Moreover, we may not have exact values for the kernel functions due to some numerical errors with quadrature or the evaluation of special functions. This section shows how our algorithms can handle inexact inputs. We control the error produced at each stage and obtain a convergent algorithm as the input error decreases to zero. Additionally, we show that our algorithms can be implemented using finitely many arithmetic operations over $\mathbb{Q}$ (or another suitable discrete set).

Since we are dealing with inexact snapshots and kernel evaluations, we assume in this section that $(\mathcal{X},d)$ is a metric space and that the kernel $\mathfrak{K}$ has a modulus of continuity, defined as follows.

\begin{definition}[Kernel modulus of continuity]
    \label{moduluscty_defn}
    Let $\omega:\mathbb{R}_{\geq 0}\rightarrow\mathbb{R}_{\geq 0}$ be an increasing continuous function with $\omega(0)=0$ and let $(\mathcal{X},d)$ be a metric space. A kernel function $\mathfrak{K}:\mathcal{X}\times\mathcal{X}\rightarrow\mathbb{C}$ has a  modulus of continuity $\omega$ if
    $$
        |\mathfrak{K}(x,y)-\mathfrak{K}(\tilde{x},\tilde{y})|\leq \omega(d(x,\tilde{x})+d(y,\tilde{y}))\quad\forall x,y,\tilde{x},\tilde{y}\in\mathcal{X}.
    $$
\end{definition}

\noindent
Recall that to compute $\spec_{\mathrm{ap},\epsilon}(\koop^*)$ in SpecRKHS-PseudoPF (\cref{pspecadjoint_alg}), we approximated $\sigma_{\inf}((\koop^*-z I)\mathcal{P}_N^*)$ and then took the limit $N\rightarrow\infty$ to compute $\sigma_{\inf}(\koop^*-z I)$. The following theorem shows that we can compute $\sigma_{\inf}((\koop^*-zI)\mathcal{P}_N^*)$ to any desired accuracy in only finitely many arithmetic operations and comparisons, assuming that we have access to arbitrarily accurate snapshot data (see \cref{sec:inexact_input}; note that even though we may take a limit to obtain exact snapshot data, the challenge is to do the computation without taking any additional limits).

\begin{theorem}[Computation of $\sigma_{\inf}((\koop^*-zI)\mathcal{P}_N^*)$]\label{siginfcompute_thm}
    Let $(\mathcal{X},d)$ be a metric space, $\mathfrak{K}$ have modulus of continuity $\omega$, and $\{(x^{(m)},y^{(m)}=F(x^{(m)}))\}_{m=1}^N$ be a collection of `true' snapshot date. Suppose that for any $\epsilon>0$, we have access to snapshot data $\{(\tilde{x}^{(m)},\tilde{y}^{(m)})\}_{m=1}^N$ such that
    $$
        d(\tilde{x}^{(m)},x^{(m)})\leq\epsilon\quad\text{and}\quad d(\tilde{y}^{(m)},y^{(m)})\leq\epsilon,\quad 1\leq m\leq N.
    $$
    Suppose also that we can compute $\epsilon$-accurate approximations of $\mathfrak{K}(z,z')$ in finitely many arithmetic operations and comparisons for any $z,z'\in \{\tilde{x}^{(m)},\tilde{y}^{(m)}\}_{m=1}^N$. Then, $\sigma_{\inf}((\koop^*-zI)\mathcal{P}_N^*)$ can be computed to any desired accuracy in finitely many arithmetic operations and comparisons.
\end{theorem}

We introduce some necessary results before proving \cref{siginfcompute_thm}. As discussed in \cref{verepairs_sect}, $\sigma_{\inf}((\koop^*-zI)\mathcal{P}_N^*)$ is exactly the square root of the smallest eigenvalue of the generalized eigenvalue problem
$$
    H(z)g=(R-z A^*-\overline{z}A+|z|^2G)g=\lambda Gg,
$$
where $H(z)$ is self-adjoint for all $z\in\mathbb{C}$, and $G$ is self-adjoint and positive definite (as $G_{jk}=\mathfrak{K}(x_k,x_j)$). Hence, we first consider computability and perturbation bounds for eigenvalues. The following lemma~\cite[Cor.~6.9]{colbrook3} states that one may compute eigenvalues of a self-adjoint matrix to any desired accuracy in finitely many arithmetic operations and comparisons.

\begin{lemma}[Cor.~6.9 in \cite{colbrook3}]
    \label{saevalcompute_lemma}
    Let $B\in\mathbb{C}^{n\times n}$ be a self-adjoint matrix, with eigenvalues $\lambda_1\leq\dots\leq\lambda_n$ (including multiplicity). Given any $\epsilon>0$, one can compute $\epsilon$-accurate approximations to $\lambda_1,\dots,\lambda_n$ in finitely many arithmetic operations and comparisons.
\end{lemma}

We extend \cref{saevalcompute_lemma} to generalized eigenvalue problems, as they apply to \cref{pspecadjoint_alg,pspeckoop_alg}. While we use \cref{sagenevalcompute_lemma} to convert the generalized eigenvalue problem to a standard eigenvalue problem for computational purposes, we will work with the original generalized eigenvalue problem to analyze the error.

\begin{corollary}[Approximation of generalized eigenvalues]
    \label{sagenevalcompute_lemma}
    Let $B\in\mathbb{C}^{n\times n}$ be a self-adjoint matrix and $C\in\mathbb{C}^{n\times n}$ be a self-adjoint and positive-definite matrix. Let $\mu_1,\dots,\mu_n$ denote the eigenvalues (including multiplicity) of the generalized eigenvalue problem $Bv=\mu Cv$. Given any $\epsilon>0$, one can compute $\epsilon$-accurate approximations to $\mu_1,\dots,\mu_n$ in finitely many arithmetic operations and comparisons.
\end{corollary}

\begin{proof}
    Using the Cholesky factorization of $C$, we can write $C=LL^*$, where $L$ is lower triangular with real and positive diagonal entries (hence it is invertible) \cite[Cor.~7.2.9]{horn_matrix_2012}. Then, the generalized eigenvalue problem $Bv=\mu Cv$ can be rewritten as $L^{-1}B(L^*)^{-1}L^*v=\mu L^*v$. Note that $L^{-1}B(L^*)^{-1}$ is self-adjoint because $B$ is self-adjoint. Moreover, since $L^*$ is invertible, this problem is equivalent to computing the eigenvalues of the matrix $L^{-1}B(L^*)^{-1}$. Both the Cholesky factorization and matrix inversion can be performed in finitely many arithmetic operations and comparisons, so the result follows from \cref{saevalcompute_lemma}.
\end{proof}

To analyze the errors in solving the eigenvalue problems arising from errors in the matrices, we use the following two theorems, where $\|B\|_2=\sigma_{\max}(B)$ denotes the spectral norm of $B$, and $\|B\|_F$ the Frobenius norm of $B$.

\begin{theorem}[Thm.~1.1 in \cite{nakatsukasa_absolute_2010}]
    \label{weyl_thm}
    Let $G,\tilde{G}\in\mathbb{C}^{n\times n}$ be self-adjoint matrices. Suppose that $G$ has eigenvalues $\lambda_1\leq\lambda_2\leq\dots\leq\lambda_n$ with multiplicity and $\tilde{G}$ has eigenvalues $\tilde{\lambda}_1\leq\tilde{\lambda}_2\leq\dots\leq\tilde{\lambda}_n$. Then for $i=1,\dots,n$,
    $$
        |\lambda_i-\tilde{\lambda}_i|\leq\|G-\tilde{G}\|_2.
    $$
\end{theorem}

\begin{theorem}[Thm.~2.2 in \cite{nakatsukasa_absolute_2010}]
    \label{weylgeneralized_thm}
    Let $G,H\in\mathbb{C}^{n\times n}$ be self-adjoint, and $G$ be positive definite. Suppose the generalized eigenvalue problem $Hg=\mu Gg$ has eigenvalues $\mu_1\leq\mu_2\leq\dots\leq\mu_n$ with multiplicity. If $\tilde{G}, \tilde{H}$ are Hermitian and $\|G-\tilde{G}\|_2<\sigma_{\inf}(G)$, then the eigenvalues of the generalized eigenvalue problem $\tilde{H}(z)g=\tilde{\mu}\tilde{G}g$, $\tilde{\mu}_1\leq\tilde{\mu}_2\leq\dots\leq\tilde{\mu_n}$, satisfy
    $$
        |\mu_i-\tilde{\mu}_i|\leq\frac{\|H-\tilde{H}\|_2}{\sigma_{\inf}(G)}+\frac{\|H\|_2+\|H-\tilde{H}\|_2}{\sigma_{\inf}(G)(\sigma_{\inf}(G)-\|G-\tilde{G}\|_2)}\|G-\tilde{G}\|_2,\quad 1\leq i\leq n.
    $$
\end{theorem}

We can now prove \cref{siginfcompute_thm}.

\begin{proof}[Proof of \cref{siginfcompute_thm}]
    Define $G_{jk}=\mathfrak{K}(x^{(k)},x^{(j)})$ and $\tilde{G}_{jk}=\mathfrak{K}(\tilde{x}^{(k)},\tilde{x}^{(j)})$. Since they are self-adjoint, we only need to compute the $N(N+1)/2$ upper triangular entries of these matrices. Note that
    \begin{align*}
        \|G-\tilde{G}\|_2^2\leq\|G-\tilde{G}\|_F^2 & =\sum_{j=1}^N\sum_{k=1}^N|\mathfrak{K}(x^{(k)},x^{(j)})-\mathfrak{K}(\tilde{x}^{(k)},\tilde{x}^{(j)})|^2                       \\
                                                   & \leq \sum_{j=1}^N\sum_{k=1}^N[\omega(d(x^{(k)},\tilde{x}^{(k)})+d(x^{(j)},\tilde{x}^{(j)}))]^2\leq N^2[\omega({2}\epsilon)]^2,
    \end{align*}
    so that $\|G-\tilde{G}\|_2\leq N\omega({2}\epsilon)$. To use \cref{weylgeneralized_thm}, we want to find $C>0$ such that $\sigma_{\inf}(G)>C$ so we can choose $\|G-\tilde{G}\|$ smaller than $\sigma_{\inf}(G)$. Following \cref{weyl_thm} and the triangle inequality, we have
    $$
        \sigma_{\inf}(G)=\lambda_1=|\tilde{\lambda}_1+(\lambda_1-\tilde{\lambda}_1)|\geq |\tilde{\lambda}_1|-|\lambda_1-\tilde{\lambda}_1|\geq |\tilde{\lambda}_1|-N\omega({2}\epsilon).
    $$
    Suppose we compute an $N\omega(2\epsilon)$-accurate approximation $\hat{\lambda}_1$ to $\tilde{\lambda}_1$; then $\sigma_{\inf}(G)\geq|\tilde{\lambda}_1|-N\omega(2\epsilon)\geq |\hat{\lambda}_1|-2N\omega(2\epsilon)\coloneqq C$ and so if $C>0$ we have the desired lower bound. Also, note that $C\geq \lambda_1-4N\omega(2\epsilon)$ and so for all suitably small $\epsilon C$ will be positive. We can then choose $\epsilon$ sufficiently small such that $\sigma_{\inf}(G)>C>N\omega({2}\epsilon)\geq \|G-\tilde{G}\|_2$ so the conditions of \cref{weylgeneralized_thm} are satisfied.

    Let $A_{jk}=\mathfrak{K}(y^{(k)},x^{(j)})$, $\tilde{A}_{jk}=\mathfrak{K}(\tilde{y}^{(k)},\tilde{x}^{(j)})$, $R_{jk}=\mathfrak{K}(y^{(k)},y^{(j)})$, and $\tilde{R}_{jk}=\mathfrak{K}(\tilde{y}^{(k)},\tilde{y}^{(j)})$. Following similar arguments as before, we have $\|A-\tilde{A}\|_2\leq N\omega({2}\epsilon)$ and $\|R-\tilde{R}\|_2\leq N\omega({2}\epsilon)$.
    Fix $z\in\mathbb{C}$ and note that
    $$
        |\sigma_{\inf}((\koop^*-zI)\mathcal{P}_N^*)-\sigma_{\inf}((\koop^*-\tilde{z}I)\mathcal{P}_N^*)|\leq|z-\tilde{z}|\quad\forall\tilde{z}\in\mathbb{C}.
    $$
    So without loss of generality, we may take $z\in\mathbb{Q}+i\mathbb{Q}$.
    Then we define
    $$
        H(z)=R-zA^*-\overline{z}A+|z|^2G,\quad \tilde{H}(z)=\tilde{R}-z\tilde{A}^*-\overline{z}\tilde{A}+|z|^2\tilde{G},
    $$
    which is automatically self-adjoint. Then, we observe that $\|H(z)-\tilde{H}(z)\|_2\leq(1+|z|)^2N\omega({2}\epsilon)$, which gives
    $$
        \|H(z)\|_2\leq\|\tilde{H}(z)\|_2+\|H(z)-\tilde{H}(z)\|_2\leq \|\tilde{H}(z)\|_F+(1+|z|)^2N\omega({2}\epsilon).
    $$
    Therefore, we can compute an upper bound $\|H(z)\|_2<C'$ in finitely many arithmetic operations and comparisons. Note that, to compute $\|H(z)\|_F$, we can compute a square root to any desired accuracy in finitely many arithmetic operations and comparisons. Then, by \cref{weylgeneralized_thm}, we find that
    $$
        |\mu_1-\tilde{\mu}_1|\leq \frac{(1+|z|)^2N\omega({2}\epsilon)}{C}+\frac{C'+(1+|z|^2)N\omega({2}\epsilon)}{C^2}N\omega({2}\epsilon).
    $$
    Finally, by taking $\epsilon$ sufficiently small and using \cref{saevalcompute_lemma}, we can approximate $\mu_1=\sqrt{\sigma_{\inf}((\koop^*-zI)\mathcal{P}_N^*)}$ to any desired accuracy in finitely many arithmetic operations and comparisons.
\end{proof}

\cref{siginfcompute_thm} applies to most kernels used in practice. For example, the assumptions of \cref{siginfcompute_thm} are satisfied by the Wendland kernels, defined in \cref{wendlandkernel_defn}.

\begin{lemma}
    \label{wendlandcompute_lemma}
    The Wendland kernels satisfy the assumptions in \cref{siginfcompute_thm}.
\end{lemma}

\begin{proof}
    Any Wendland kernel with $d\in\mathbb{N}$ and $k\in\mathbb{N}\cup\{0\}$ can be written as
    $$
        \mathfrak{K}(x,y)=\varphi(\|x-y\|), \text{ where }\varphi(r)=\begin{cases}
            p(r), & 0\leq r\leq 1, \\
            0,    & r>1,
        \end{cases}
    $$
    where $p(r)$ is a polynomial in $r$ whose rational coefficients can be computed recursively \cite[Thm.~9.12]{wendland_scattered_2004}. We show that the Wendland kernel is Lipschitz continuous with a computable constant, which implies the result. Note that by the triangle inequality, for any $x,y,\tilde{x},\tilde{y}\in\mathcal{X}$,
    $$
        |r-\tilde{r}|\coloneqq|\|x-y\|-\|\tilde{x}-\tilde{y}\||\leq \|x-\tilde{x}\|+\|y-\tilde{y}\|.
    $$
    If $\varphi$ is Lipschitz with constant $C>0$, this implies that
    $$
        |\mathfrak{K}(x,y)-\mathfrak{K}(\tilde{x},\tilde{y})|=|\varphi(r)-\varphi(\tilde{r})|\leq C|r-\tilde{r}|\leq C\left(\|x-\tilde{x}\|+\|y-\tilde{y}\|\right).
    $$
    Let $p(r)=\sum_{l=0}^{\lfloor{d/2}\rfloor+3k+1}a_lr^l$ where $a_l\in\mathbb{Q}$ for all $l$, then
    $  \sup_{r\in[0,1]}|p'(r)|\leq\sum_{l=1}^{\lfloor{d/2}\rfloor+3k+1}|la_l|=C\in\mathbb{Q}$.
    For $k\geq 1$, $\varphi$ has a continuous first derivative \cite[Thm.~9.13]{wendland_scattered_2004}, and by the mean value theorem, $|\varphi(r)-\varphi(\tilde{r})|\leq C|r-\tilde{r}|$. Hence, $\mathfrak{K}$ is Lipschitz continuous. If $k=0$, then $d\geq 3$ and $\varphi(r)=(1-r)_+^{\lfloor{d/2}\rfloor}$ and so for $r,\tilde{r}<1$,
    $$
        |\mathfrak{K}(x,y)-\mathfrak{K}(\tilde{x},\tilde{y})|=|(1-r)^{\lfloor{d/2}\rfloor}-(1-\tilde r)^{\lfloor{d/2}\rfloor}|\leq \lfloor{d/2}\rfloor|r-\tilde{r}|.
    $$
    If instead $r>1$ and $\tilde{r}<1$ (one could also choose the reverse without loss of generality), then
    $$
        |\mathfrak{K}(x,y)-\mathfrak{K}(\tilde{x},\tilde{y})|=|\mathfrak{K}(\tilde{x},\tilde{y})|=(1-\tilde{r})^{\lfloor{d/2}\rfloor}\leq 1-\tilde{r}\leq |r-\tilde{r}|.
    $$
    Finally, if both $r,\tilde{r}>1$, the left-hand side vanishes and any constant works. This achieves to prove the case $k=0$. In the end, we note that, given inputs $x,y\in\mathcal{X}$ and arbitrarily close rational approximations of $x,y\in\mathcal{X}$, we can compute arbitrarily close rational approximations to the square root $\|x-y\|_2$, by using the fact that the square root is strictly increasing. Then, since we have a known Lipschitz constant for $\mathfrak{K}$, $\mathfrak{K}(x,y)$ can be computed to any desired accuracy.
\end{proof}

Let $\tau_N(z,\koop)$ be an approximation of $\sigma_{\inf}((\koop^*-zI)\mathcal{P}_N^*)$ with accuracy $1/N$, which we can compute by \cref{siginfcompute_thm}, and define
$$
    \tilde{\Gamma}_N^{\epsilon}(\koop^*)=\{z\in\gridop(N):\tau_N(z,\koop)+1/N<\epsilon\}\subset\spec_{\mathrm{ap},\epsilon}(\koop^*).
$$
By minor adjustments to the proof of \cref{adjointapspecconv_thm}, we find that $\lim_{N\rightarrow\infty}\tilde{\Gamma}_N^{\epsilon}(\koop^*)=\spec_{\mathrm{ap},\epsilon}(\koop^*)$, with convergence from below. Similarly, following \cref{rectcompute_lemma} we can compute $\sigma_{\inf}(\mathcal{P}_{N_1}(\koop-zI)\mathcal{P}_{N_2}^*)$ by computing the smallest eigenvalue of the $N_2\times N_2$ truncation of the generalized eigenvalue problem $[AG^{-1}A^*-zA-\overline{z}A^*+|z|^2G]\m{g}=\lambda G\m{g}.$ Using similar arguments, we can compute this to any desired accuracy under the same assumptions.

\subsection{Numerical examples}
\label{sec:pseud_examples}

We now provide three numerical examples of our algorithms from this section. Examples, including real-world experimental data, in high dimensions are discussed later in \cref{numexample_sect}.

\subsubsection{The Gauss iterated map}
\label{sec:gauss_map_example}

We first consider the Gauss iterated map on the domain $\mathcal{X}=(-1,0)$, defined by
$$F(x)=\exp(-\alpha x^2)+\beta,$$
where $\alpha=2$ and $\beta=-1-\exp(-2)$ are chosen so that $F(\mathcal{X})\subset\mathcal{X}$. The Sobolev space $H^1(\mathcal{X})$ has kernel
$$
    \mathfrak{K}(x,y)=\begin{cases}
        \frac{\cosh(x+1)\cosh(y)}{\sinh(1)}, \quad & \text{if } x\leq y, \\
        \frac{\cosh(x)\cosh(y+1)}{\sinh(1)}, \quad & \text{if } x\geq y.
    \end{cases}
$$

We first claim that the Gauss iterated map gives rise to a bounded Koopman operator on $H^1(\mathcal{X})$. Since $F$ is not a bijection and $1/F'$ is unbounded, we cannot use the results discussed in \cref{sobolevkoopman_sect} to prove this. Instead, we argue directly. By a simple density argument (e.g., \cite[Ex.~IV.19]{ikeda_koopman_2024}), it suffices to show that $\koop$ is bounded on $C^1(\mathcal{X})$ functions. Suppose $\koop\restriction_{H^1(\mathcal{X})\cap C^1(\mathcal{X})}$ is bounded, then $\koop$ extends to a bounded operator on the entire Sobolev space, $\tilde{\koop}:H^1(\mathcal{X})\rightarrow H^1(\mathcal{X})$, by density of $C^1(\mathcal{X})$. Indeed, given $g\in H^1(\mathcal{X})$, there exists $g_m\in C^1(\mathcal{X})\cup H^1(\mathcal{X})$ such that $g_m\rightarrow g$ as $m\rightarrow\infty$. Since $\tilde{\koop}$ is bounded $\koop g_m=\tilde{\koop}g_m\rightarrow\tilde{\koop}g$, but as $\koop$ is closed, we must have $g\in\mathcal{D}(\koop)$ and $g\circ F=\koop g=\tilde{\koop}g$. Let $g\in C^1(\mathcal{X})$, then
\begin{align*}
    \|\koop g\|^2_{H^1} & =\int_{-1}^0g(F(x))^2\dd x+\int_{-1}^0F'(x)^2g'(F(x))^2\dd x
    \leq \|g\|_{L^{\infty}}^2+\int_{-1}^{\exp(-2)}F'(F^{-1}(u))g'(u)^2\dd u                                                                                                                                                                                           \\
                        & \leq \sup_{x\in\mathcal{X}}|\langle g,\mathfrak{K}_x\rangle_{H_1}|^2+\|F'\|_{L^{\infty}}\|g'\|_{L^2}^2\leq \left(\sup_{x\in \mathcal{X}}\|\mathfrak{K}_x\|_{H^1}\right)^2\|g\|_{H^1}^2+\|F'\|_{L^{\infty}}\|g\|_{H^1}^2\leq C\|g\|_{H^1}^2,
\end{align*}
for some constant $C>0$, which can be explicitly computed, and the result follows.

Interestingly, while the Koopman operator on $H^1(\mathcal{X})$ is bounded, the Koopman operator on $L^2(\mathcal{X})$ is unbounded. This follows from the facts that $1/F'(x)$ diverges as $x\rightarrow 0$ and functions in $L^2$ can have arbitrarily large values near the origin independently of their $L^2$ norm, unlike in the Sobolev space case. In particular, for any $g\in L^2(\mathcal{X})$,
$$
    \|g\circ F\|_{L^2}^2  =\int_{-1}^0g(F(x))^2\dd x =\int_{-1}^{-\exp(-2)}g(u)^2/F'(F^{-1}(u))\dd u,
$$
and so by choosing $g$ of unit norm but highly peaked around $u=\exp(-2)$, we can make $\|g\circ F\|_{L^2}$ arbitrarily large.

We approximate the Koopman operator associated with the Gauss iterated map in the Sobolev space $H^1(\mathcal{X})$ with kEDMD and compare it with EDMD. For snapshot data, we evaluate $F$ at $201$ Chebyshev nodes on the interval $(-1,0)$, and take $\{w_i\}_{i=1}^{201}$ to be the corresponding Clenshaw--Curtis quadrature weights. When implementing EDMD we use a dictionary of Chebyshev polynomials of the first kind transformed to the interval $(-1,0)$, i.e., $T_0(2x+1),T_1(2x+1),\dots,T_{200}(2x+1)$, and to compute the matrices $G$, $A$ and $R$ in \eqref{GA_defn} and \eqref{R_defn} we use the kernel defined in \cref{H1abreproducingkernel_eqn} with $(a,b)=(-1,0)$. These yield finite section approximations of the Koopman operator $\mathbb{K}_{\mathrm{EDMD}}=(\Psi_X^*W\Psi_X)^{-1}(\Psi_X^*W\Psi_Y)$ (see \cref{sec:edmd}) and $\mathbb{K}_{\mathrm{kEDMD}}=G^{-1}A$ respectively. Since $g(x)=x$ is contained in the span of our EDMD dictionary (as $T_1(2x+1)=2x+1$) we can immediately use the action of $\mathbb{K}_{\mathrm{EDMD}}$ on this observable to predict the trajectory of the state. For kEDMD, we compute a vector $c$, which is the best projection of the kernel functions onto the state $x$ in the least squares sense; we then use $\mathbb{K}_{\mathrm{kEDMD}}$ to predict the evolution of the kernel functions and then project back onto the state space. We take $c$ to be the solution of the least squares problem $\min_{c\in\mathbb{C}^{N}}\|X-c^TG\|$, where $X=(x_1,\dots,x_N)$, the solution of which is given by $c=G^{-1}X^T$. The left plot of \cref{fig:gauss_map} shows the results of these two predictors, illustrating that kEDMD on a Sobolev space much more accurately predicts the evolution of the dynamical system compared to EDMD. This is due to both the increase in accuracy from bypassing the large data limit and the pointwise, rather than $L^2$, convergence of observables.

\begin{figure}[t]
    \centering
    \includegraphics[height=0.35\textwidth]{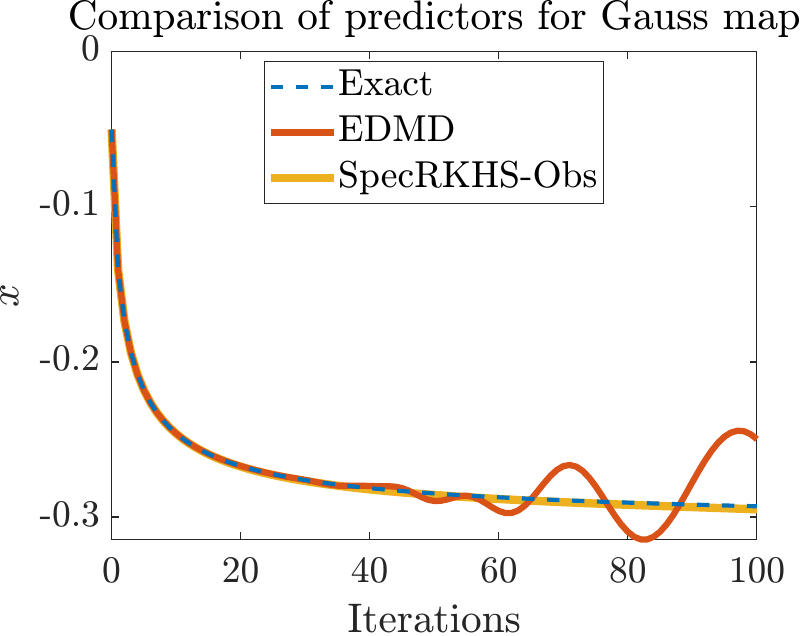}\hspace{1cm}
    \includegraphics[height=0.35\textwidth]{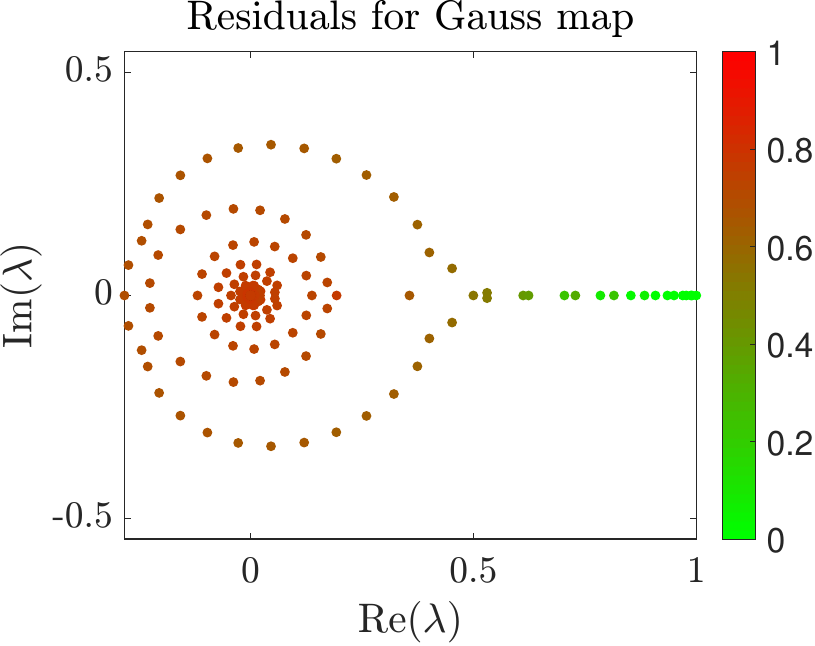}
    \caption{Gauss map. Left: Relative forecast errors for the Gauss map for kEDMD compared to EDMD. Right: The eigenvalues outputted by kEDMD, with the color showing the size of their residuals computed by SpecRKHS-Eig (\cref{evalverif_alg}).}
    \label{fig:gauss_map}
\end{figure}

For the right image of \cref{fig:gauss_map}, we compute the eigenpairs of $\mathbb{K}_{\mathrm{kEDMD}}=G^{-1}A$ and use SpecRKHS-Eig (\cref{evalverif_alg}) and the matrix $R$ to compute the corresponding residuals. These are shown on a scatter plot with color depending on the size of the residual. There is heavy spectral pollution, with the eigenvalues around $\lambda=0$ having residuals close to $1$, indicating that they are spurious; by contrast, the eigenvalues near $\lambda=1$ have comparatively small residuals and yield accurate pseudoeigenfunctions. By \cref{sigmainfconv} the residuals that we compute converge down to $\sigma_{\inf}(\koop^*-\lambda I)$, and so any eigenvalues with computed residuals less than $\epsilon$ must lie in the $\epsilon$-approximate point pseudospectrum, without needing to take any limits to confirm.

\subsubsection{The Duffing oscillator} \label{sec_duffing}

As a second (and more complicated) example, we consider the unforced damped Duffing oscillator defined by
$$
    \ddot{u}+\delta\dot{u}+\alpha u+\beta u^3=0,\quad \alpha,\beta,\delta>0.
$$
We discretize the system at times $t_n = n\Delta t$, where $\Delta t$ is the time step, and define the state $x_n=(u(t_n),\dot{u}(t_n))\in\mathcal{X}=\mathbb{R}^2$ and evolution function $F$ such that $F(x)$ gives the solution at time $\Delta t$ to the Duffing oscillator with initial condition $x\in\mathbb{R}^2$ at time $t_0=0$. We consider the Koopman operator defined on the native space of the Wendland kernel $\varphi_{2,k}$ for $k\geq 1$, which is equivalent to the fractional Sobolev space $H^{3/2+k}(\mathbb{R}^2)$ \cite[Thm.~4.1]{kohne_linfty-error_2024}.

We first prove that the corresponding Koopman operator is bounded whenever $\Delta t<1/(4\delta)$. Indeed, the Hamiltonian of the system is
$H(t)=\frac{1}{2}\dot{u}^2+\frac{1}{2}\alpha u^2+\frac{1}{4}\beta u^4$, and $dH/dt=-\delta[\dot{u}]^2$. Hence,
$$
    H(t_{k+1})=H(t_k)-\delta\int_{t_k}^{t_{k+1}}\dot{u}^2(s)\dd s\geq H(t_k)-2\delta (t_{k+1}-t_k)H(t_k),
$$
and for $\Delta t=t_{k+1}-t_k<1/(4\delta)$, we have $H(t_k)\geq H(t_{k+1})\geq \frac{1}{2}H(t_k)$. Let $g\in C_c^\infty(\mathbb{R}^2)$ be a compactly supported smooth function. By \cite[Cor.~4.1]{hartman_ordinary_2002}, the solution at a given time $t>0$ is a smooth function of the initial conditions, and so $\koop g$ is also smooth. Additionally, since $H(t_{k+1})\geq \frac{1}{2}H(t_k)$ and $\alpha,\beta>0$, the distance $x\in\mathbb{R}^2$ can move towards the origin after application of $F$ is bounded below, and so as $g$ is compactly supported so too is $\koop g$. Since compactly supported smooth functions are dense in Sobolev spaces \cite{meyers_h_1964}, including fractional Sobolev spaces \cite{di_nezza_hitchhikers_2012}, $\koop$ is densely defined, and $\koop^*$ is well-defined. As $\mathfrak{K}$ is a radial kernel, we have $\|\mathfrak{K}_x\|^2_{\mathfrak{K}}=\langle \mathfrak{K}_x,\mathfrak{K}_x\rangle_{\mathfrak{K}}=\mathfrak{K}_x(x)=\varphi_{2,k}(0)$ for all $x\in\mathcal{X}$, and $\|\koop^*\mathfrak{K}_x\|_{\mathfrak{K}}=\|\mathfrak{K}_{F(x)}\|_{\mathfrak{K}}=\|\mathfrak{K}_x\|_{\mathfrak{K}}=\varphi_{2,k}(0)$ for all $x\in\mathcal{X}$. We conclude by density that $\koop^*$ (and hence $\koop$) is bounded with $\|\koop\|_{\mathfrak{K}}=\|\koop^*\|_{\mathfrak{K}}=1$.

\begin{remark}[Argument for radial kernel functions]
    This argument implies that on the native space of a radial kernel function, any densely defined operator is immediately bounded. Similar energy-type arguments (where we can also use Lyapunov functions in the place of $H$) can be used to show that a number of Koopman operators arising from the solution to ordinary differential equations are bounded.
\end{remark}

For a specific example, take $\alpha=\beta=1$ and $\delta=0.2$. The snapshot data consists of trajectories running for $0.3$ seconds, each sampled every $\Delta t=0.01$ seconds. Note that $\Delta t<1/(4\delta)$ as desired. The trajectories start from random initial conditions uniformly and independently sampled in the square $[-1,1]^2$ and are computed using MATLAB's ode45 routine.
We want to analyze which kernel is best for the Duffing oscillator out of three Wendland kernels with $d=2$ and $k=1,2,3$ and three Mat{\'e}rn kernels with $d=2$ and $n=2,3,4$. By `best', we mean the one that achieves the most accurate predictive model analogously to how we use kEDMD to predict the evolution of states under the Gauss map.

\begin{table}[t]
    \caption{Mean relative errors for the predictions generated by kEDMD on the Duffing oscillator over $10$ test trajectories with Wendland and Mat{\'e}rn kernels.}
    \begin{center}
        \begin{tabular}{ cc }
            \toprule[\thick pt]
            Kernel          & Mean relative error  \\
            \midrule[\thick pt]
            Wendland, $k=1$ & $9.09\times 10^{-5}$ \\
            Wendland, $k=2$ & $3.62\times 10^{-4}$ \\
            Wendland, $k=3$ & $7.00\times 10^{-2}$ \\
            \bottomrule[\thick pt]
        \end{tabular}
        \hspace{1cm}
        \begin{tabular}{ cc }
            \toprule[\thick pt]
            Kernel            & Mean relative error  \\
            \midrule[\thick pt]
            Mat{\'e}rn, $n=2$ & $2.17\times 10^{-4}$ \\
            Mat{\'e}rn, $n=3$ & $3.50\times 10^{-5}$ \\
            Mat{\'e}rn, $n=4$ & $4.77\times 10^{-3}$ \\
            \bottomrule[\thick pt]
        \end{tabular}
        \label{table:duffingerrors}
    \end{center}
\end{table}

\begin{figure}[t]
    \centering
    \includegraphics[height=0.4\textwidth]{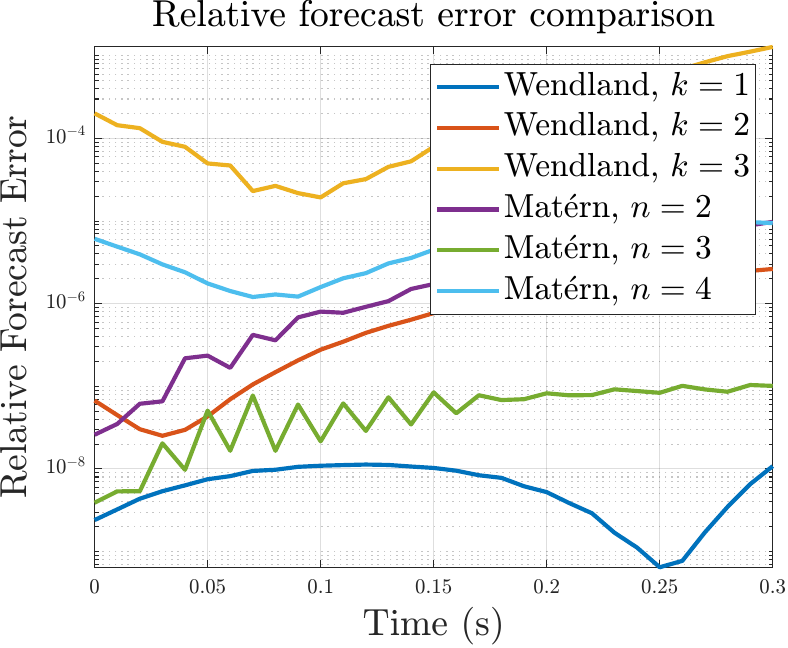}\hfill
    \includegraphics[height=0.4\textwidth]{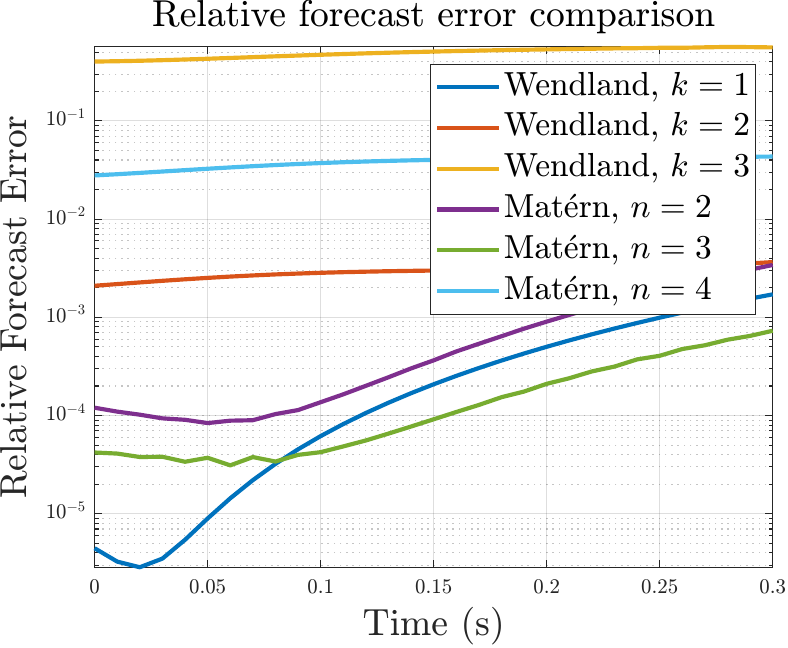}
    \caption{Duffing oscillator. Comparison of relative errors for predicting trajectories of the Duffing oscillator for six different kernels.}
    \label{fig:duffingkernelcomparison}
\end{figure}

For each kernel, the choice of scale factor $\sigma$ for $r$ impacts the conditioning of the matrices $G$, $A$, and $R$ and the method's accuracy. To tune this hyperparameter, we use an adaptive grid search. We begin by generating (independently and identically) $50$ trajectories, the first $40$ of which we use to train Koopman models for each kernel and $\sigma$ pair, which we then use to predict the trajectories beginning from the starting value of the remaining $10$ trajectories. We then compute the average error compared to numerical integration. We adaptively vary the set of $\sigma$ we try for each kernel until minimal improvement is achieved. Once a $\sigma$ for each kernel is achieved, we compare the kernels against each other. We generate (independently and uniformly to each other and the first set) another set of $50$ trajectories and use $40$ of these $50$ trajectories to train Koopman models to predict the future evolution of the system for the six kernels, which we then compare to the remaining trajectory we generated and compute the mean errors. \cref{table:duffingerrors} shows the results, and some illustrative relative error plots are shown in \cref{fig:duffingkernelcomparison}. We see that the `best' kernel for analyzing the dynamics of the Duffing oscillator is the Mat{\'e}rn kernel with $n=3$ given by
\begin{equation*}
    \mathfrak{K}(x,y)=\begin{cases}
        1,                                       & x=y,              \\
        (\sigma\|x-y\|_2)^2K_2(\sigma\|x-y\|_2), & \text{otherwise},
    \end{cases}
\end{equation*}
where $K_2$ is the modified Bessel function of the second kind, the native space of which is equivalent to the Sobolev space $H^3(\mathbb{R}^2)$.

\begin{figure}[htbp]
    \centering
    \includegraphics[height=0.4\textwidth]{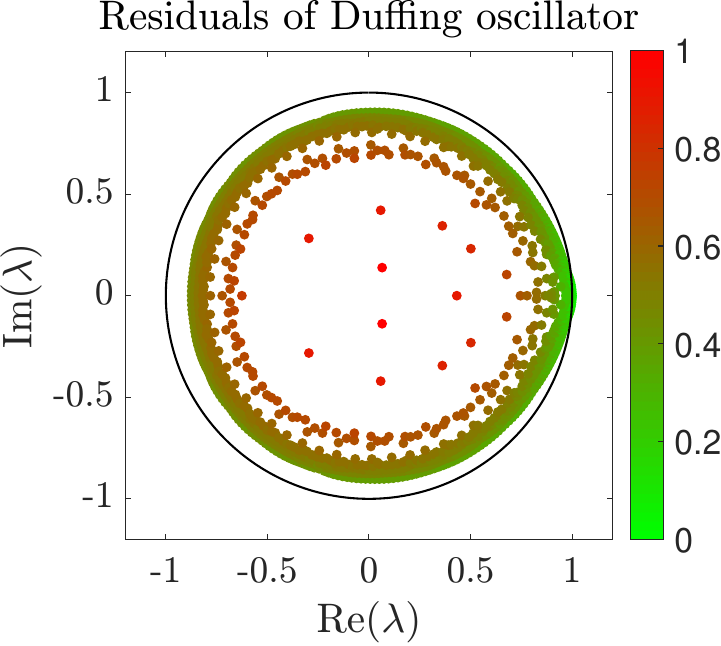}\hfill
    \includegraphics[height=0.4\textwidth]{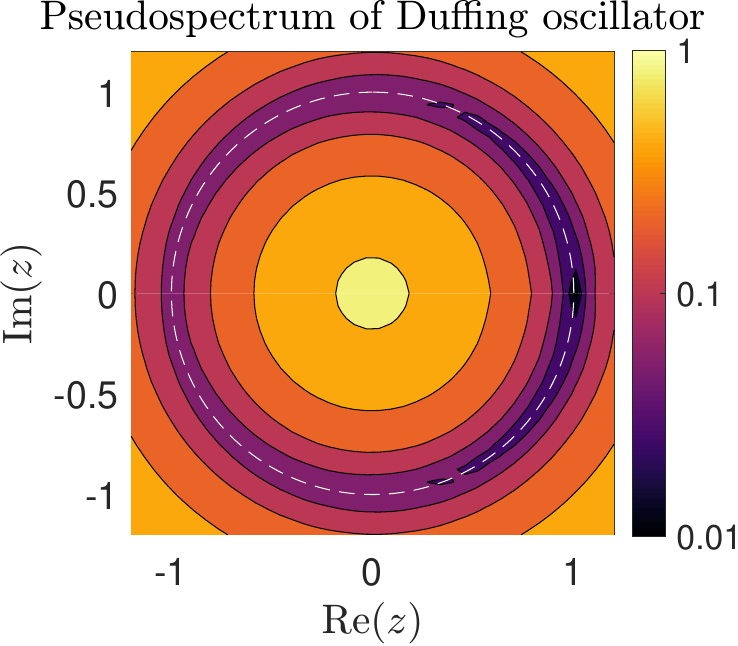}
    \caption{Duffing oscillator. Left: The eigenvalues outputted by kEDMD, with the color showing the size of their residuals computed by SpecRKHS-Eig (\cref{evalverif_alg}). Right: Pseudospectrum of the Duffing oscillator computed using SpecRKHS-PseudoPF (\cref{pspecadjoint_alg}). The unit circle is highlighted in both panels.}
    \label{fig:duffingoscillator}
\end{figure}

Using this kernel with $\sigma=6$, we generate a further $40$ independent trajectories and compute the spectrum of the Perron--Frobenius operator. We use SpecRKHS-Eig (\cref{evalverif_alg}) to compute the residuals of eigenpairs outputted by kEDMD and use SpecRKHS-PseudoPF (\cref{pspecadjoint_alg}) to compute the pseudospectrum. The results are shown in \cref{fig:duffingoscillator}. As for the Gauss map, there is spectral pollution in the eigenvalues outputted by kEDMD, with only those near the unit circle, and in particular near $\lambda=1$, having small residuals. This shows agreement with the pseudospectral plot.

\subsubsection{The Lorenz system}

\begin{figure}[t]
    \centering
    \includegraphics[height=0.4\textwidth]{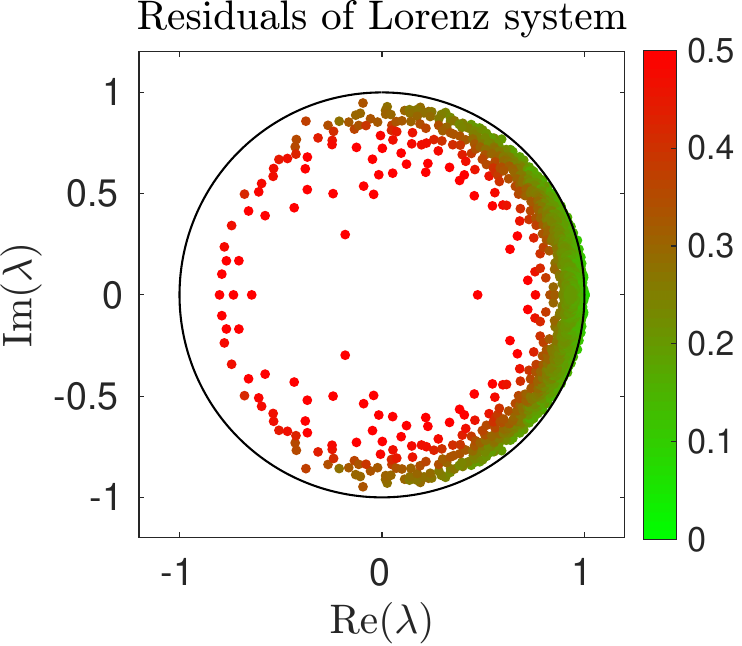}\hfill
    \includegraphics[height=0.4\textwidth]{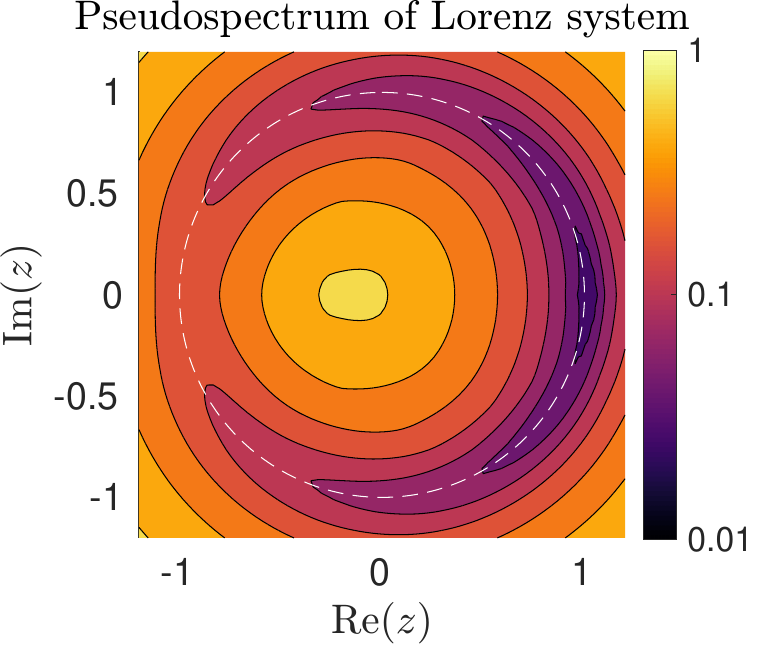}
    \caption{Lorenz system. Left: The eigenvalues outputted by kEDMD, with the color showing the size of their residuals computed by SpecRKHS-Eig (\cref{evalverif_alg}). Right: Pseudospectrum of the Lorenz system computed using SpecRKHS-PseudoPF (\cref{pspecadjoint_alg}). The unit circle is highlighted in both panels.}
    \label{fig:lorenzsystem}
\end{figure}

\begin{figure}[htbp]
    \centering
    \includegraphics[width=0.48\linewidth]{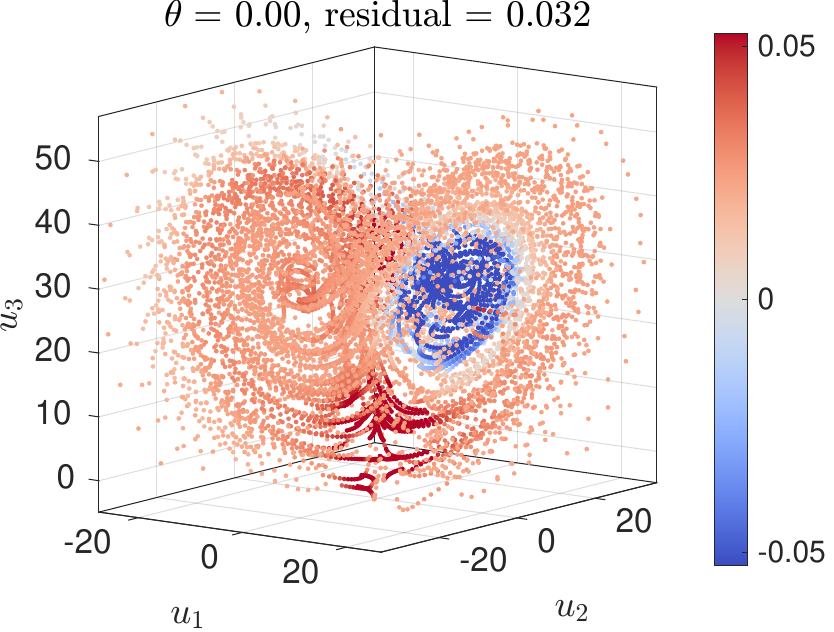}\hfill
    \includegraphics[width=0.48\linewidth]{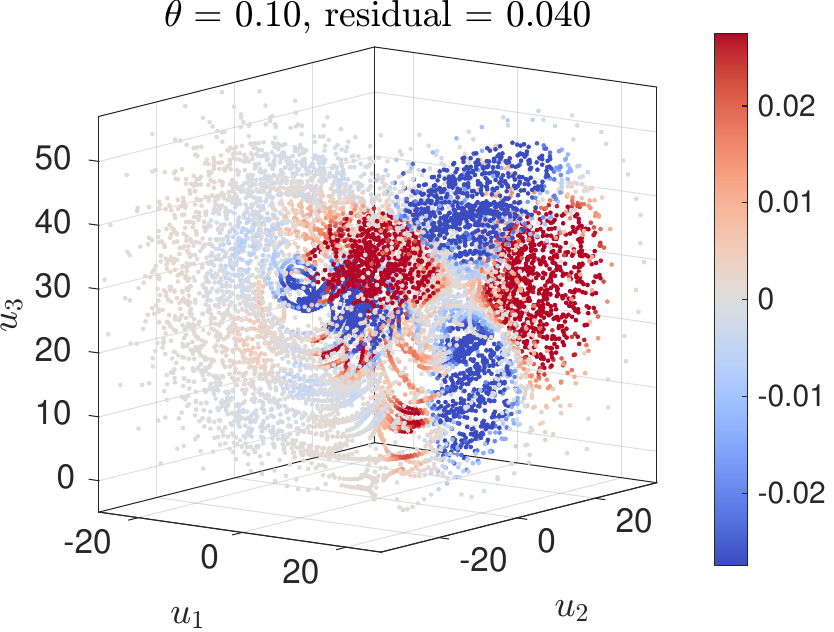}\\ \vspace{0.15cm}
    \includegraphics[width=0.48\linewidth]{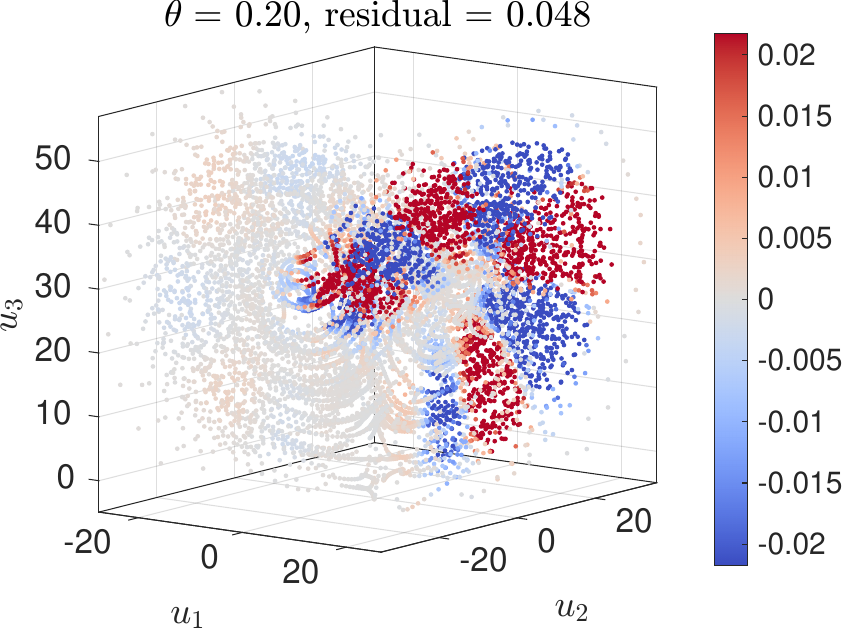}\hfill
    \includegraphics[width=0.48\linewidth]{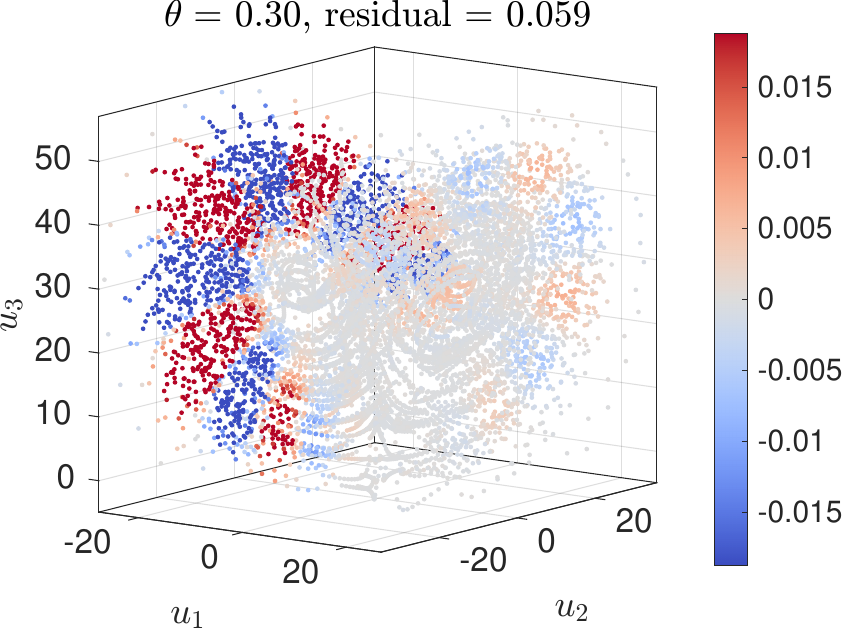}\\ \vspace{0.15cm}
    \includegraphics[width=0.48\linewidth]{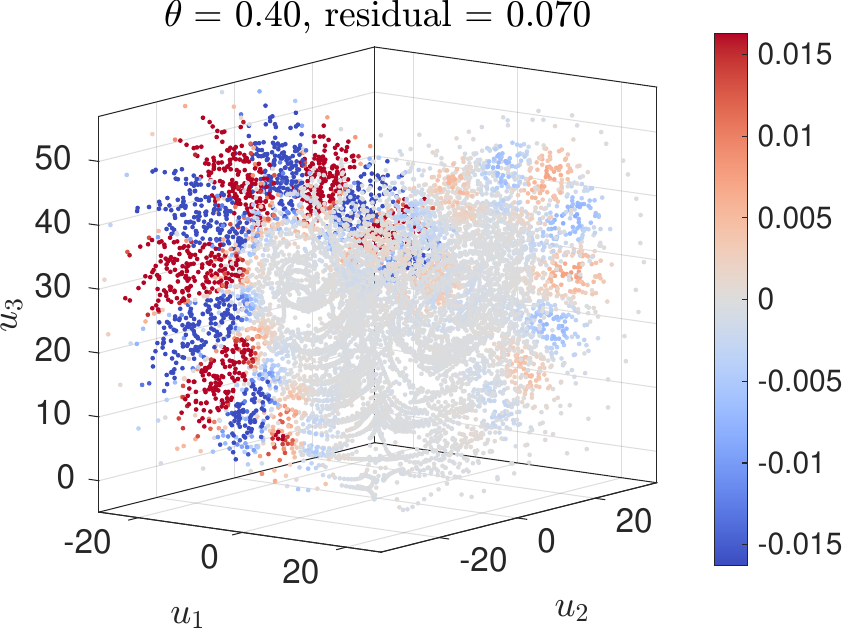}\hfill
    \includegraphics[width=0.48\linewidth]{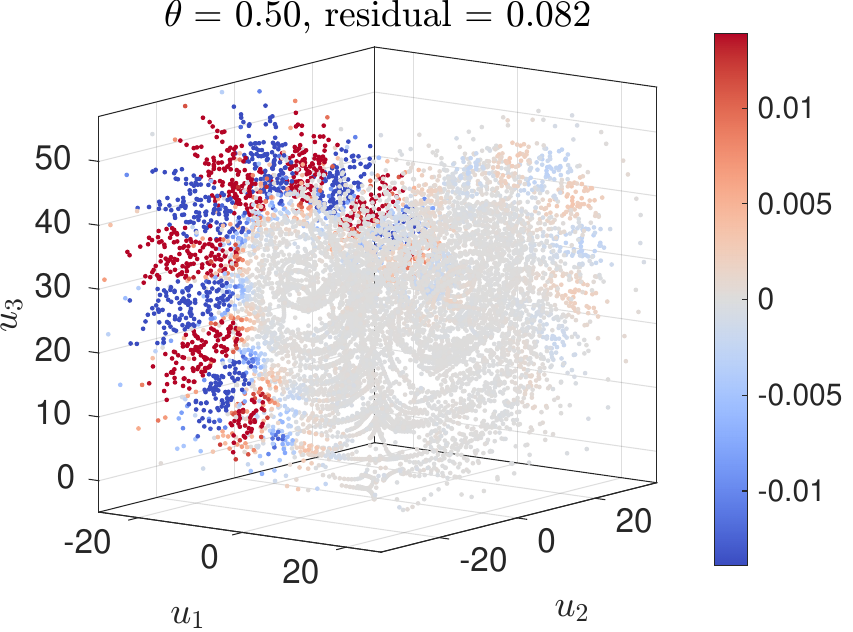}
    \caption{Lorenz system. Six pseudoeigenfunctions for the chaotic Lorenz system with eigenvalue $\lambda=e^{i\pi\theta}$ computing using SpecRKHS-PseudoPF (\cref{pspecadjoint_alg}) plotted along the trajectory data used to generate the Koopman model.}
    \label{fig:lorenzpefuns}
\end{figure}

We now consider the Lorenz system defined by
$$
    \dot{u_1}=\sigma(u_2-u_1),\quad
    \dot{u_2}=u_1(\rho-u_3)-u_2,\quad
    \dot{u_3}=u_1u_2-\beta u_3,
$$
with classical parameters $\rho=28$, $\sigma=10$ and $\beta=8/3$ and state $x=(u_1,u_2,u_3)\in\mathbb{R}^3$. The system is chaotic, and almost all initial points end up in a (fractal) invariant set known as the Lorenz attractor. We consider the Perron--Frobenius operator associated with this Lorenz system on $H^2(\mathbb{R}^3)$. This RKHS is the native space of the Wendland kernel, which we scale to $\mathfrak{K}(x,y)=\varphi_{3,0}(\|x-y\|/10)$ \cite[Thm.~4.1]{kohne_linfty-error_2024} where
\begin{equation}
    \varphi_{3,0}(\|x-y\|/10)=\begin{cases}
        (1-\|x-y\|/10)^2, & \text{if }\|x-y\|\leq 10, \\
        0,                & \text{otherwise}.
    \end{cases}
\end{equation}
We generate $500$ trajectories starting from randomly generated initial conditions inside the cuboid $[-25,25]\times [-25,25]\times [0,50]$ and integrate them forward in time using MATLAB's ode45 for $0.2$ seconds, sampling every $0.01$ seconds to obtain $10000$ snapshots total across the trajectories. We then use the compressed basis described in \cref{sec_compressed} with $r=1000$ basis functions.

Using \cref{evalverif_alg,pspecadjoint_alg}, we compute residuals for the kEDMD eigenvalues as well as the approximate point pseudospectrum of the Perron--Frobenius operator and show the results in \cref{fig:lorenzsystem}. As is typical, kEDMD exhibits heavy spectral pollution that can be spotted using residuals; most of the eigenvalues inside the unit circle are spurious. This is echoed in the pseudospectral plot.

We can also examine the pseudoeigenfunctions of the Lorenz system to study its behavior. Applying SpecRKHS-PseudoPF (\cref{pspecadjoint_alg}) to the points $z_j=\exp(i\pi \times j/10)$ for $j=0,1,\ldots,5$, we compute the corresponding pseudoeigenfunctions $g_j$ at these points along with their corresponding residuals. We plot the pseudoeigenfunctions along the trajectory data points in \cref{fig:lorenzpefuns}, where the geometric shape of the attractor can be seen. However, we stress that our function space is not supported just on the attractor. The residual increases moving from $z_0$ to $z_5$, as the pseudoeigenfunctions become more oscillatory and harder to approximate in our chosen subspace. The number of oscillations increases as the angle increases. For example, there are approximately twice as many oscillations for $g_2$ than $g_1$, which is expected from $z_2=z_1^2$. Finally, we notice that the oscillations are prominent on the wings of the attractor.

\section{The boundaries of Koopman computations on RKHSs}
\label{sci_sect}

In this section, we prove the optimality of our algorithms for computing spectral properties by classifying the associated computational problems within the Solvability Complexity Index (SCI) hierarchy~\cite{colbrook2020PhD,ben2015can,Hansen_JAMS}. The SCI hierarchy classifies computational problems based on the number of limits necessary for their solution and refinements, such as error control. For example, by taking $\epsilon\downarrow 0$ in \cref{adjointapspecconv_thm}, we can compute the approximate point spectrum $\spec_{\mathrm{ap}}(\koop^*)$ in two successive limits: we shall prove that there is no alternative algorithm that can converge in a single limit. As discussed in \cref{avoidquadapprox_sect}, our approach circumvents the need for quadrature approximations by directly evaluating the kernel. This typically lowers the SCI of the algorithms by one. The RKHS structure makes computing spectral properties significantly easier compared to the standard $L^2$ case due to the action of $\koop^*$ on the kernel functions.

The lower bounds (impossibility results) we prove match the upper bounds (convergence of algorithms), and involve constructing families of adversarial dynamical systems such that no algorithm can converge to compute their spectral properties. These lower bounds are derived under the assumption of exact information, emphasizing that the computational complexity arises not from sampling error but from the fundamental structure of the problems. Furthermore, we extend our impossibility results to probabilistic algorithms in \cref{sec_extension_to random}. This scenario covers techniques such as EDMD with randomly sampled trajectories and algorithms employing randomness in their computations, such as stochastic gradient descent for neural network training. Consequently, the impossibility results established directly imply that randomized algorithms cannot succeed with probability exceeding $2/3$. Thus, our adversarial dynamical systems are neither rare nor pathological but fundamental examples illustrating inherent computational limitations.

\subsection{The SCI hierarchy: Classifying difficulty and proving optimality}
\label{scihierarchyintro_sect}

First, we provide some background material on the necessary aspects of the SCI hierarchy. These are included in a compact self-contained format. For further discussion, we point the reader to \cite{colbrook2020PhD,ben2015can,Hansen_JAMS}.

\subsubsection{Computational problems and towers of algorithms}

To classify the difficulty of a computational problem, such as computing the spectral properties of Koopman operators, we must first precisely define a computational problem.

\begin{definition}[Computational problem]\label{compprob_defn}
    A computational problem is a quadruple $\{\Xi,\Omega,\mathcal{M},\Lambda\}$ consisting of an input class $\Omega$, a metric space $(\mathcal{M},d)$, a problem function $\Xi:\Omega\rightarrow\mathcal{M}$ (what we want to compute), and a set of evaluation functions $\Lambda$ consisting of functions from $\Omega$ to $\mathbb{C}$ that we allow algorithms access to. We require that if $\zeta_1,\zeta_2\in\Omega$ with $\Xi(\zeta_1)\neq\Xi(\zeta_2)$, then there exists $f\in \Lambda$ such that $f(\zeta_1)\neq f(\zeta_2)$ (otherwise it would be impossible to compute both $\Xi(\zeta_1)$ and $\Xi(\zeta_2)$ correctly).
\end{definition}

\begin{example}
    In the context of computing spectra of Koopman operators associated to dynamics $(F,\mathcal{X})$ on an RKHS $\mathcal{H}$, one may choose $\Omega=\{F:\mathcal{X}\rightarrow\mathcal{X}\text{ s.t. }\koop_F\text{ is densely defined on } \mathcal{H}\}$, $(\mathcal{M},d)=(\mathcal{M}_{\mathrm{AW}},d_{\mathrm{AW}})$, $\Xi(F)=\spec(\koop_F)$ where the spectrum is computed on $\mathcal{H}$, and $\Lambda_{\mathcal{X}}=\left\{F\mapsto F(\hat{x}_j):j=1,2,\dots\right\}$ for $\{\hat{x}_j\}_{j=1}^{\infty}$, a dense subset of $\mathcal{X}$.
\end{example}

We must also clarify what we mean by an algorithm. To make our impossibility results as strong as possible, we begin with a general definition that captures the essential features of any deterministic algorithm (we shall prove in \cref{sec_extension_to random} that our impossibility results extend to probabilistic algorithms). A deterministic algorithm can only use finitely many evaluation functions to compute its output and must be consistent. If two inputs yield the same information read by the algorithm, then the algorithm must produce the same output for both.

\begin{definition}[General algorithm]\label{genalg_defn}
    A general algorithm for a computational problem $\{\Xi,\Omega,\mathcal{M},\Lambda\}$ is a map $\Gamma:\Omega\to \mathcal{M}$ with the following property: for any $\zeta\in\Omega$, there exists a non-empty finite subset of evaluations $\Lambda_\Gamma(\zeta) \subset\Lambda$ such that if $\zeta'\in\Omega$ with $f(\zeta)=f(\zeta')$ for every $f\in\Lambda_\Gamma(\zeta)$, then $\Lambda_\Gamma(\zeta)=\Lambda_\Gamma(\zeta')$ and $\Gamma(\zeta)=\Gamma(\zeta')$.
\end{definition}

For infinite-dimensional computational problems, successive limits are often needed to compute the desired output, which motivates the notion of towers of algorithms.

\begin{definition}[Tower of algorithms]
    \label{toweralg_defn}
    A tower of algorithms of height $k\in\mathbb{N}$ for a computational problem $\{\Xi,\Omega,\mathcal{M},\Lambda\}$ is a set of functions $\Gamma_{n_k,\dots,n_1},\dots,\Gamma_{n_1}:\Omega\rightarrow\mathcal{M}$, $n_k,\dots,n_1\in\mathbb{N}$, where each $\Gamma_{n_k,\dots,n_1}$ is a general algorithm and for every $\zeta\in\Omega$, we have
    $$
        \lim_{n_1\rightarrow\infty}\Gamma_{n_k,\dots,n_2,n_1}(\zeta)=\Gamma_{n_k,\dots,n_2}(\zeta),\quad \dots,\quad \lim_{n_{k-1}\rightarrow\infty}\Gamma_{n_k,n_{k-1}}(\zeta)=\Gamma_{n_k}(\zeta),\quad\lim_{n_k\rightarrow\infty}\Gamma_{n_k}(\zeta)=\Xi(\zeta).
    $$
    where the convergence holds in $(\mathcal{M},d)$. We say that a tower of algorithms is general, denoted by $\alpha=G$, if we impose no further restrictions, and arithmetic, denoted by $\alpha=A$, if each $\Gamma_{n_k,\dots,n_1}$ can be computed using only $\Lambda$ and finitely many arithmetic operators and comparisons.
\end{definition}

We can now define the SCI of a computational problem.

\begin{definition}[Solvability Complexity Index]
    \label{sci_defn}
    The SCI of a computational problem $\{\Xi,\Omega,\mathcal{M},\Lambda\}$ with respect to the type $\alpha$ is the smallest integer $k$ such that there exists a tower of algorithms of type $\alpha$ and height $k$ for the computational problem, and is denoted by $\mathrm{SCI}(\Xi,\Omega,\mathcal{M},\Lambda)_{\alpha}=k$. If there is no tower of any height for the computational problem then we say that $\mathrm{SCI}(\Xi,\Omega,\mathcal{M},\Lambda)_{\alpha}=\infty$; if there exists a general algorithm $\Gamma$ of type $\alpha$ such that $\Gamma=\Xi$, then we say that $\mathrm{SCI}(\Xi,\Omega,\mathcal{M},\Lambda)_{\alpha}=0$.
\end{definition}

This definition allows us to order and classify problems using the SCI hierarchy. We let $\mathcal{T}_{\alpha}$ denote the collection of type $\alpha$ towers of algorithms for a given computational problem.

\begin{definition}[Class of computational problems]
    \label{scihier_defn}
    Let $\alpha\in\{G,A\}$, we define the following classes of computational problems:
    \begin{align*}
        \Delta_0^{\alpha}     & =\{\{\Xi,\Omega,\mathcal{M},\Lambda\}:\mathrm{SCI}(\Xi,\Omega,\mathcal{M},\Lambda)_{\alpha}=0\},                                                                      \\
        \Delta_1^{\alpha}     & =\{\{\Xi,\Omega,\mathcal{M},\Lambda\}:\exists\{\Gamma_n\}_{n=1}^{\infty}\in\mathcal{T}_{\alpha} \text{ s.t. } \forall \zeta\in\Omega, d(\Gamma_n(\zeta),\Xi(\zeta))\leq 2^{-n}\}, \\
        \Delta_{m+1}^{\alpha} & =\{\{\Xi,\Omega,\mathcal{M},\Lambda\}:\mathrm{SCI}(\Xi,\Omega,\mathcal{M},\Lambda)_{\alpha}\leq m\}, \quad m\geq 1.
    \end{align*}
\end{definition}

\begin{remark}
    The term $2^{-n}$ in the definition for $\Delta_1^{\alpha}$ can be replaced by any positive sequence $\{a_n\}$ that is computable from $\Lambda$ and has $\lim_{n\rightarrow\infty}a_n=0$.
\end{remark}

\subsubsection{Error control}
\label{error_control_sci_sect}

The class $\Delta_1^{\alpha}$ defined in \cref{scihier_defn} provides error control: in addition to a convergence algorithm, we have a bound on how close any given output is to the correct output. However, this form of error control is often too strong in practice. For example, when computing eigenvalues of elliptic operators using the finite element method, we may be able to compute error bounds on a finite number of eigenvalues (how close the output is to the spectrum) but not on all eigenvalues at once (how close the spectrum is to the output). In the results of \cref{provconv_sect}, we frequently see that the output is either approximately contained in or approximately contains the desired spectral set. This motivates the following refinement of the SCI hierarchy, which can be interpreted as providing `half' an error control.

\begin{definition}[Refinement of SCI hierarchy]
    \label{sigmapi_defn}
    Consider the Hausdorff or Attouch--Wets metric  $(\mathcal{M},d)$ on $\mathbb{C}$, set $\Sigma_0^{\alpha}=\Pi_0^{\alpha}=\Sigma_0^{\alpha}$, and for $m\in\mathbb{N}$ define
    \begin{align*}
        \Sigma_m^{\alpha} & =\left\{\{\Xi,\Omega,\mathcal{M},\Lambda\}\in\Delta_{m+1}^{\alpha}:\exists\{\Gamma_{n_m,\dots,n_1}\}\in\mathcal{T}_{\alpha},\exists\{X_{n_m}(\zeta)\}\subset\mathbb{C}\text{ s.t. }\forall \zeta\in\Omega\;\Gamma_{n_m}(\zeta)\subset X_{n_m}(\zeta),\right. \\
                          & \quad\quad\lim_{n_m\rightarrow\infty}\Gamma_{n_m}(\zeta)=\Xi(\zeta),d(X_{n_m}(\zeta),\Xi(\zeta))\leq 2^{-n_m}\left.\right\},                                                                                                                                 \\
        \Pi_m^{\alpha}    & =\left\{\{\Xi,\Omega,\mathcal{M},\Lambda\}\in\Delta_{m+1}^{\alpha}:\exists\{\Gamma_{n_m,\dots,n_1}\}\in\mathcal{T}_{\alpha},\exists\{X_{n_m}(\zeta)\}\subset\mathbb{C}\text{ s.t. }\forall \zeta\in\Omega\;\Xi(\zeta)\subset X_{n_m}(\zeta),\right.          \\
                          & \quad\quad\lim_{n_m\rightarrow\infty}\Gamma_{n_m}(\zeta)=\Xi(\zeta),d(X_{n_m}(\zeta),\Gamma_{n_m}(\zeta))\leq 2^{-n_m}\left.\right\}.
    \end{align*}
\end{definition}

Similarly to \cref{scihier_defn}, one can replace the term $2^{-n}$ with any sequence that converges to zero. For the Attouch--Wets metric, which includes the empty set, for $\Sigma_m^{\alpha}$ we require that $\Gamma_{n_m}(\zeta)=\emptyset$ if $\Xi(\zeta)=\emptyset$, and for $\Pi_m^{\alpha}$ we may take $X_{n_m}(\zeta)=\Gamma_{n_m}(\zeta)$ if $\Xi(\zeta)=\emptyset$. The classes $\Delta_1$, $\Sigma_1$, and $\Pi_1$ are crucial, as they allow the verification of the output of our algorithms. In addition to being vital for real-world applications, such algorithms have also been used in computed-assisted proofs, such as Kepler's conjecture (Hilbert's 18th problem) \cite{bastounis_extended_2022,hales_formal_2017}.

\subsubsection{Inexact input}
\label{sec:inexact_input}

The SCI can accommodate the scenario where we do not know $f(\zeta)$ exactly for some $f\in\Lambda$ and $\zeta\in\Omega$. For example, in the case of Koopman operators, we may only have access to the snapshot data to a certain tolerance (as considered in \cref{inexactupperbounds_sect}). In this case, we must alter our definition of a computational problem.

\begin{definition}[Computational problem with $\Delta_1$-information]\label{inexactcompprob_defn}
    Let $\{\Xi,\Omega,\mathcal{M},\Lambda\}$ be a computational problem, where $\Lambda=\{f_j:\Omega\rightarrow\mathbb{C}\}_{j\in I}$ for some index set $I$. The corresponding problem with $\Delta_1$-information is denoted by $\{\Xi,\Omega,\mathcal{M},\Lambda\}^{\Delta_1}=\{\Xi^{\Delta_1},\Omega^{\Delta_1},\mathcal{M},\Lambda^{\Delta_1}\}$, where $\Omega^{\Delta_1}$ is the class of all possible $\tilde{\zeta}=\left\{f_{j,n}(\zeta):j\in I,n\in\mathbb{N}\right\}$ such that $\zeta\in\Omega$ and $|f_{j,n}(\zeta)-f_j(\zeta)|\leq 2^{-n}$ for all $\zeta\in\Omega$, $\Xi^{\Delta_1}(\tilde{\zeta})=\Xi(\zeta)$, and $\Lambda^{\Delta_1}=\{\tilde{f}_{j,n}\}_{(j,n)\in I\times \mathbb{N}}$, where $\tilde{f}_{j,n}(\tilde{\zeta})=f_{j,n}(\zeta)$.
\end{definition}

The key idea is to replace an element $\zeta\in\Omega$ with all possible sets of approximations of its evaluation functions to arbitrary accuracy. This is natural as for any $\zeta$, there are many different choices of inputs $f_{j,n}(\zeta)$ that converge to $f_j(\zeta)$ in the desired way, and we want towers of algorithms that converge for all possible choices. This approach is common when using interval arithmetic \cite{tucker2011validated}, for example. \cref{inexactcompprob_defn} can be expanded in the obvious manner to $\Delta_m$ information for $m\geq 2$.

\subsection{Summary of results: Upper bounds and lower bounds}
\label{scisummary_sect}

Given any RKHS $\mathcal{H}$ of functions on $\mathcal{X}$ (the state space) with kernel $\mathfrak{K}$, we define the input class
$$
    \Omega_{\mathcal{H}}=\left\{ F:\mathcal{X}\rightarrow\mathcal{X}\text{ s.t. }\koop \text{ is densely defined on }\mathcal{H},\text{ span of countably many kernel functions is a core of }\koop, \koop^*\right\}.
$$
We work with the Attouch--Wets metric $\mathcal{M}_{\mathrm{AW}}$ defined in \cref{AW_metric_def}. Let $\{\hat{x}_j\}_{i=1}^{\infty}$ be a dense subset of $\mathcal{X}$, and define the evaluation set $\Lambda_{\mathcal{X}}=\left\{F\mapsto F(\hat{x}_j):j=1,2,\dots\right\}$. We employ the notation $\Xi_{\spec_{\mathrm{ap}}}$ (respectively, $\Xi_{\spec_{\mathrm{ap}}^*}$) for the problem function of computing $\spec_{\mathrm{ap}}(\koop)$ (respectively, $\spec_{\mathrm{ap}}(\koop^*)$), and $\Xi_{\spec_{\mathrm{ap},\epsilon}}$ (respectively, $\Xi_{\spec_{\mathrm{ap},\epsilon}^*}$) for the problem function of computing $\spec_{\mathrm{ap},\epsilon}(\koop)$ (respectively, $\spec_{\mathrm{ap},\epsilon}(\koop^*)$) for a given $\epsilon>0$.

\subsubsection{Upper bounds (convergent algorithms)}

We first summarize the results of \cref{provconv_sect} in the language of the SCI hierarchy. Our convergence results proven above lead to the following upper bounds in the SCI hierarchy.

\begin{theorem}[Upper bounds for spectral computations]\label{rkhsupperbounds_thm}
    Let $\epsilon>0$, we have the following classifications:
    \begin{align*}
         & \{\Xi_{\spec_{\mathrm{ap},\epsilon}^*},\Omega_{\mathcal{H}},\mathcal{M}_{\mathrm{AW}},\Lambda_{\mathcal{X}}\}\in\Sigma_1^G,\quad
        \{\Xi_{\spec_{\mathrm{ap}}^*},\Omega_{\mathcal{H}},\mathcal{M}_{\mathrm{AW}},\Lambda_{\mathcal{X}}\}\in\Pi_2^G,                     \\
         & \{\Xi_{\spec_{\mathrm{ap},\epsilon}},\Omega_{\mathcal{H}},\mathcal{M}_{\mathrm{AW}},\Lambda_{\mathcal{X}}\}\in\Sigma_2^G,\quad
        \{\Xi_{\spec_{\mathrm{ap}}},\Omega_{\mathcal{H}},\mathcal{M}_{\mathrm{AW}},\Lambda_{\mathcal{X}}\}\in\Pi_3^G.
    \end{align*}
    If, in addition, $\mathfrak{K}$ satisfies the assumptions of \cref{siginfcompute_thm}, all of the above classifications hold with inexact information and with general algorithms replaced by arithmetic algorithms.
\end{theorem}

The results for the computations of spectral properties of $\koop^*$ and $\koop$ are proven in \cref{sec:initial_convergence,sec:fullpspecspec} respectively, while the extension to inexact information and arithmetic algorithms can be found in \cref{inexactupperbounds_sect}. As discussed in \cref{error_control_sci_sect}, the $\Sigma_1^G$ result is particularly powerful, as it allows us to verify the existence of the approximately coherent observables described by the approximate point pseudospectrum. In general, we require an additional limit to compute spectra compared to pseudospectra since $\spec_{\mathrm{ap},\epsilon}(\koop)\rightarrow\spec_{\mathrm{ap}}(\koop)$ as $\epsilon\rightarrow 0$. Moreover, computing spectral properties of $\koop$ is generally harder than computing those of $\koop^*$, which will be explored more in \cref{discretesobimposs_sect}.

\subsubsection{Lower bounds (impossibility results)}

In the rest of the section, we prove impossibility results that show the optimality of \cref{rkhsupperbounds_thm}. To obtain stronger impossibility results, given a metric space $\mathcal{X}$ and an RKHS $\mathcal{H}$ of functions on $\mathcal{X}$ with kernel $\mathfrak{K}$, we define $\Omega_{\mathcal{H}}^B=\left\{F:\mathcal{X}\rightarrow\mathcal{X},\text{ s.t. } \koop \text{ is bounded on }\mathcal{H}\right\}$, and use the Hausdorff metric space in this case, with the same evaluation functions.\footnote{Bounded operators automatically have countable subsets of the kernel functions whose span forms a core of $\koop$ and $\koop^*$ by \cref{kernelorthonormalbasis_lemma}.} This yields a stronger result as if there does not exist a tower of height $k$ for the space of bounded operators, there cannot exist a tower of height $k$ for the larger class of unbounded operators (the Hausdorff and Attouch--Wets topologies agree if all sets are contained in a common compact set).

We prove our impossibility results on two types of spaces. The first is the Sobolev space $H^r(\mathcal{X})$ for a suitable open subset $\mathcal{X}\subset\mathbb{R}^d$ and $r>d/2$ a positive integer, as discussed in \cref{sobolevkoopman_sect}. The second is the discrete Sobolev space $h^r(\mathbb{N})$ defined for $r\in\mathbb{R}$ by \cite{de_hoop_convergence_2023}
\begin{equation}
    \label{discretesobolevspace_defn}
    h^r(\mathbb{N})=\left\{v:\mathbb{N}\rightarrow\mathbb{C}:\sum_{j=1}^{\infty}j^{2r}|v_j|^2<\infty\right\}.
\end{equation}
These can be seen as alterations of $\ell^2(\mathbb{N})$ to impose more or less `smoothness' on the elements of the sequence. For example, we can regard the sequence as Fourier coefficients for a function lying in a Sobolev space with certain smoothness \cite[Thm.~4.13]{haroske_distributions_2007}. \cref{rkhslowerbounds_thm,sci-classification-table} summarize the key impossibility results we prove.

\begin{theorem}[Lower bounds for spectral computations]\label{rkhslowerbounds_thm} Let $\epsilon>0$.
    Consider the Sobolev space $H^r(\mathcal{X})$ defined in \cref{sobolevkoopman_sect}, where $\mathcal{X}\subset\mathbb{R}^d$ is a product of finite open intervals and $r>d/2$ is an integer. We have the following classifications:
    $$
        \left\{\Xi_{\spec_{\mathrm{ap},\epsilon}^*},\Omega_{H^r(\mathcal{X})}^B,\mathcal{M}_{\mathrm{H}},\Lambda_{\mathcal{X}}\right\}\notin \Delta_1^G, \quad
        \left\{\Xi_{\spec^*_{\mathrm{ap}}},\Omega_{H^r(\mathcal{X})},\mathcal{M}_{\mathrm{AW}},\Lambda_{\mathcal{X}}\right\}\notin\Delta_2^G.
    $$
    If instead $\mathcal{X}\subset\mathbb{R}^d$ is a product of finite open intervals and the real line, with at least one copy of the real line, we have the following classifications, the second of which is improved from the previous case:
    $$
        \left\{\Xi_{\spec_{\mathrm{ap},\epsilon}^*},\Omega_{H^r(\mathcal{X})}^B,\mathcal{M}_{\mathrm{H}},\Lambda_{\mathcal{X}}\right\}\notin \Delta_1^G, \quad
        \left\{\Xi_{\spec^*_{\mathrm{ap}}},\Omega_{H^r(\mathcal{X})}^B,\mathcal{M}_{\mathrm{H}},\Lambda_{\mathcal{X}}\right\}\notin\Delta_2^G.
    $$
    Finally, for the discrete Sobolev space $h^r(\mathbb{N})$ defined in \cref{discretesobolevspace_defn} with $r>0$ not necessarily an integer, there exists $R>0$ such that for all $0<\epsilon<R$ the following classifications hold:
    $$
        \left\{\Xi_{\spec_{\mathrm{ap},\epsilon}},\Omega_{h^r(\mathbb{N})},\mathcal{M}_{\mathrm{H}},\Lambda_{\mathbb{N}}\right\}\notin\Delta_2^G, \quad
        \left\{\Xi_{\spec_{\mathrm{ap}}},\Omega_{h^r(\mathbb{N})}^B,\mathcal{M}_{\mathrm{H}},\Lambda_{\mathbb{N}}\right\}\notin\Delta_3^{G}.
    $$
\end{theorem}

The results on $h^r(\mathbb{N})$ are proven in \cref{discretesobimposs_sect}, the results for $H^r(\mathcal{X})$ on finite intervals are proven in \cref{sobolevintervalimposs_sect}, and the extension to the real line is discussed in \cref{sobolevreallineimposs_sect}. Combining with \cref{rkhsupperbounds_thm}, these results tell us that we can compute (for any $\epsilon>0$) $\spec_{\mathrm{ap},\epsilon}(\koop^*)$ in one limit with half an error control (i.e., we have verification that the points we compute are indeed inside the approximate point pseudospectrum), but that full error control is impossible (we can never tell how close we are to finding all points in the approximate point pseudospectrum). Moreover, in general, the additional $\epsilon\rightarrow 0$ limit is required for computing the $\spec_{\mathrm{ap},\epsilon}(\koop^*)$ and so two limits are required, with convergence from above in the second limit. Furthermore, computing the approximate point spectrum and pseudospectrum of $\koop$ requires an additional limit compared to $\koop^*$, leading to two limits being necessary for $\spec_{\mathrm{ap},\epsilon}(\koop)$, with convergence from below in the second limit; as before, the final $\epsilon\rightarrow0$ limit is required to compute $\spec_{\mathrm{ap}}(\koop)$, so this computation requires three limits with convergence from above in the third limit.

\begin{remark}
    The proof techniques exploit different properties of Koopman and Perron--Frobenius operators. For the results on $\koop^*$, we exploit the non-normality of Perron--Frobenius operators. Interval exchange maps as defined in \cref{sobolevintervalimposs_sect} provide a versatile example of such cases. For the results on $\koop$, we exploit the lack of knowledge of off-diagonal decay of the matrix representation of the Koopman operator. By contrast, for the Perron--Frobenius operator, the formula $\koop^*\mathfrak{K}_x=\mathfrak{K}_{F(x)}$ limits this.
\end{remark}

\subsection{Impossibility results for $h^r(\mathbb{N})$}
\label{discretesobimposs_sect}

In this section, we show that computing $\spec_{\mathrm{ap},\epsilon}(\koop)$ requires an additional limit compared to computing $\spec_{\mathrm{ap},\epsilon}(\koop^*)$. That is, spectral computations for the Koopman operator are more difficult than for the Perron--Frobenius operator in the RKHS setting. We prove this result for the discrete Sobolev space $h^r(\mathbb{N})$, but first present a version on $\ell^2(\mathbb{N})$ to build intuition. Since there exists a bijection between $\mathbb{N}^d$ and $\mathbb{N}$ for any $d\geq 1$, this result extends to $\ell^2(\mathbb{N}^d)$.

\subsubsection{Proof for $\ell^2(\mathbb{N})$}

The space $\ell^2(\mathbb{N})$ is an RKHS with kernel functions $\mathfrak{K}_i(j)=\delta_i(j)=\delta_{ij}$. These form an orthonormal basis with $\koop^*\delta_i=\koop^*\delta_{F(i)}$ for every $i\in\mathbb{N}$. When written as an infinite matrix with respect to this basis, $\koop^*$ has exactly one $1$ in every column, and all remaining elements are $0$. Similarly, $\koop$ can be represented as an infinite matrix with exactly one $1$ in every row and the remaining elements all $0$. Conversely, any infinite matrix with exactly one $1$ in every column (respectively row) and the remaining elements all $0$ defines a Perron--Frobenius (respectively Koopman) operator and a corresponding $F$. These operators are densely defined provided that $F$ does not map infinitely many states to any given states, and for such operators, the kernel functions automatically form a core of $\koop$ and $\koop^*$; we denote the set of such $F$ by $\Omega_{\ell^2(\mathbb{N})}$, to distinguish them from $\Omega_{h^r(\mathbb{N})}$ later.

The key intuition of the proof below is that while for $\koop^*$ we have $\koop^*\delta_j=\delta_{F(j)}$, for $\koop$ we have the more complicated equation $\koop\delta_j=\sum_{i\in F^{-1}(\{j\})}\delta_i$; the former relates to the forward dynamics and be computed without needing a limit, but the latter relates to the backward dynamics and requires a limit to compute. For the proof, we also introduce the notation $\mathbb{T}=\{z\in\mathbb{C}:|z|=1\}$ and $\mathbb{D}=\{z\in\mathbb{C}:|z|\leq 1\}$ for the unit circle and disk, respectively.

\begin{theorem}[Approximate point pseudospectrum lower bound on $\ell^2(\mathbb{N})$]\label{notdelta2specapepskoop_thm}
    For all $0<\epsilon<1$,
    $$
        \left\{\Xi_{\spec_{\mathrm{ap},\epsilon}},\Omega_{\ell^2(\mathbb{N})}^B,\mathcal{M}_{\mathrm{H}},\Lambda_{\mathbb{N}}\right\}\notin\Delta_2^G.
    $$
\end{theorem}

\begin{proof}
    Define the operators
    $$
        C=\begin{pmatrix}
            0 & 1 &        &        & \\
              & 0 & 1                 \\
              &   & \ddots & \ddots
        \end{pmatrix},\quad
        A_n=\begin{pmatrix}
            0      & 1 &        &        & 0 \\
            \vdots & 0 & 1                   \\
                   &   & \ddots & \ddots     \\
            0      &   &        & 0      & 1 \\
            1      & 0 & \dots  &        & 0
        \end{pmatrix}\in\mathbb{C}^{n\times n},
    $$
    where $C:\ell^2(\mathbb{N})\to \ell^2(\mathbb{N})$ is expressed in the canonical basis.
    Since $C^*$ is a non-invertible isometry, $\spec_{\mathrm{ap}}(C^*)=\mathbb{T}$ and $\spec(C^*)=\mathbb{D}$. In particular, the surjective spectrum satisfies $\spec_{\mathrm{su}}(C^*)=\mathbb{D}$, hence $\spec(C)=\spec_{\mathrm{ap}}(C)=\mathbb{D}$. However, $\spec(A_n)=\spec_{\mathrm{ap}}(A_n)=\{z\in\mathbb{C}:z^n=1\}$, and $\spec_{\epsilon}(A_n)=\spec_{\mathrm{ap},\epsilon}(A_n)=\{z\in\mathbb{C}:z^n=1\}+B_{\epsilon}(0)$ because $A_n$ is unitary. Define $A=\bigoplus_{r=1}^{\infty}A_{l_r}$ for $l_r\in\mathbb{N}\setminus\{1\}$ to be chosen below. Since $A$ is unitary, $\spec(A)=\spec_{\mathrm{ap}}(A)\subset\mathbb{T}$, hence $\spec_{\epsilon}(A)=\spec_{\mathrm{ap},\epsilon}(A)\subset\mathbb{T}+B_{\epsilon}(0)$. Note also that all of the above infinite matrices correspond to the Koopman operators on discrete dynamical systems; we denote the evolution function corresponding to $B:\ell^2(\mathbb{N})\rightarrow\ell^2(\mathbb{N})$ by $F_B$.

    We proceed by contradiction and assume the existence of general algorithms $\{\Gamma_n\}$ such that
    $$
        \lim_{n\rightarrow\infty}\Gamma_n(F)=\spec_{\mathrm{ap},\epsilon}(\koop_F)\quad \forall F\in\Omega_{\ell^2(\mathbb{N})}.
    $$
    Let $0<\epsilon<\eta<1$. Since $0\in\spec_{\mathrm{ap}}(C)\subset\spec_{\mathrm{ap},\epsilon}(C)$, \cref{genalg_defn} implies that there exists $n_1>0$ and $0<N_1<\infty$ such that
    $$
        \dist(\Gamma_{n_1}(F_C),0)<1-\eta,\quad\Lambda_{\Gamma_{n_1}}(F_C)\subset \{F_C(i):i=1,\ldots, N_1\}
    $$
    Hence, if we take $l_1=N_1+1$, by the consistency of general algorithms,
    $
        \dist(\Gamma_{n_1}(F_A),0)<1-\eta.
    $

    Now, we proceed by induction. Suppose that $l_1,\dots,l_k$ have been chosen and let $C_k=A_{l_1}\oplus \cdots\oplus A_{l_k}\oplus C$. Then $0\in\spec_{\mathrm{ap}}(C_k)\subset\spec_{\mathrm{ap},\epsilon}(C_k)$ and so by \cref{genalg_defn}, there exists $n_k>n_{k-1}$ and $N_{k-1}<N_k<\infty$ such that
    $$
        \dist(\Gamma_{n_k}(F_{C_k}),0)<1-\eta,\quad\Lambda_{\Gamma_{n_k}}(F_{C_k})\subset\{F_{C_k}(i):i=1,\ldots, N_k\}.
    $$
    Then we may select $l_{k+1}\geq 2$ so that, by the consistency of general algorithms,
    $
        \dist(\Gamma_{n_k}(F_A),0)<1-\eta.
    $
    By induction, this then holds for all $k$. However,
    $$
        \dist(\spec_{\mathrm{ap},\epsilon}(A),0)\geq 1-\epsilon>1-\eta,
    $$
    and so $\Gamma_{n_k}(F_A)$ cannot converge to $\spec_{\mathrm{ap},\epsilon}(\koop_{F_A})$, the required contradiction.
\end{proof}

\subsubsection{Proof for $h^r(\mathbb{N})$}

Next, we prove the lower bounds in \cref{rkhslowerbounds_thm} for $h^r(\mathbb{N})$ and begin by proving that we need strictly more than two limits to compute the approximate point spectrum. The space $h^r(\mathbb{N})$ has an orthonormal basis given by \begin{equation}
    \label{fl_defn}
    f_l=l^{-r}e_l=l^{-r}(0,\dots,0,1,0,\dots),
\end{equation}
where the only non-zero element is in the $l$th position. It is an RKHS with kernel functions given by $\mathfrak{K}_l=l^{-2r}e_l=l^{-r}f_l$ and
$$\langle v,\mathfrak{K}_l\rangle_{h^r}=\sum_{j=1}^{\infty}j^{2r}v_j\overline{(\mathfrak{K}_l)_j}=l^{2r}v_ll^{-2r}=v_l=v(l)\quad \forall v\in h^r(\mathbb{N}).
$$
The Koopman operator $\koop_F:h^r(\mathbb{N})\rightarrow h^r(\mathbb{N})$ is densely defined provided that $F$ only maps finitely many states to any given state. We denote the set of such functions by $\Omega_{h^r(\mathbb{N})}$. For any such $F$, we have
$$
    \koop_F^*f_l=l^r\koop_F^* \mathfrak{K}_l=l^r\mathfrak{K}_{F(l)}=(l/F(l))^{r}f_{F(l)}.
$$
Hence, the matrix corresponding to $\koop_F^*$ with respect to the orthonormal basis $\{f_j\}_{j\in\mathbb{N}}$ has exactly one non-zero element in the $l$th column, which lies in the $F(l)$th row and has value $(l/(F(l))^r$. Conversely, any such a matrix defines a Perron--Frobenius operator and a corresponding $F$; similarly, its transpose defines a Koopman operator.

We first prove the following technical lemma.

\begin{lemma}
    \label{nearlyjordanblocks_lemma}
    Given $n\in\mathbb{N}$ and a strictly increasing sequence of positive integers $\{a_j\}_{j=1}^n$, define the matrix
    $$
        B_n=B_n(\{a_j\})=\begin{pmatrix}
            1                                                                                                                                                                                                      \\
            ((a_1+1)/a_1)^r & 0                                                                                                                                                                                    \\
                            & ((a_2+1)/a_2)^r & 0                                                                                                                                                                  \\
                            &                 & \makebox[\widthof{\(\left(\frac{a_n + 1}{a_n}\right)^r\)}][c]{\(\ddots\)} & \ddots                                                                                 \\
                            &                 &                                                                           & ((a_n+1)/a_n)^r & 0
        \end{pmatrix}\in\mathbb{C}^{(n+1)\times (n+1)}.
    $$
    Similarly, given an infinite strictly increasing sequence of positive integers $\{a_j\}_{j=1}^{\infty}$, we define
    $$
        B_{\infty}=B_{\infty}(\{a_j\})=\begin{pmatrix}
            1                                                                                                                                                                                         \\
            ((a_1+1)/a_1)^r & 0                                                                                                                                                                       \\
                            & ((a_2+1)/a_2)^r & 0                                                                                                                                                     \\
                            &                 & \makebox[\widthof{\(\left(\frac{a_n + 1}{a_n}\right)^r\)}][c]{\(\ddots\)} & \makebox[\widthof{\(\left(\frac{a_n + 1}{a_n}\right)^r\)}][c]{\(\ddots\)}
        \end{pmatrix}\in \mathcal{B}(h^r(\mathbb{N)}).
    $$
    Then, for any strictly increasing sequence of positive integers $\{a_j\}_{j=1}^{\infty}$, $\spec_{\mathrm{ap}}(B_{\infty}(\{a_j\}))=\mathbb{T}$. Given a sequence $\{l_k\}_{k=1}^{\infty}\subset\mathbb{N}$ and a strictly increasing sequence of positive integers $\{a_j\}_{j=1}^{\infty}$, let $m_1=0$ and $m_r=\sum_{k=1}^{r-1}l_k$ for $r\geq 2$, and define $B(\{l_k\},\{a_j\})=\bigoplus_{k=1}^{\infty}B_{l_k}(\{a_j\}_{j=m_k+1}^{m_{k+1}})$. Then, if $\{l_k\}$ is bounded, $\spec_{\mathrm{ap}}(B(\{l_k\},\{a_j\}))=\{0,1\}$, while if $\{l_k\}$ is unbounded, $\spec_{\mathrm{ap}}(B(\{l_k\},\{a_j\}))=\mathbb{D}$.
\end{lemma}

\begin{proof}
    As we have written operators on $h^r(\mathbb{N})$ as infinite matrices with respect to their orthonormal bases, we may work on $\ell^2$ throughout; note in particular that even though $h^r(\mathbb{N})$ is isomorphic to $\ell^2(\mathbb{N})$ as such, their Koopman operators are different. By direct calculation, we have that
    $$
        B_{\infty}^n=\begin{pmatrix}
            1                                    & 0                                           & 0      & \dots  \\
            ((a_1+1)/a_1)^r                      & 0                                           & 0      & \dots  \\
            ((a_2+1)/a_2)^r((a_1+1)/a_1)^r       & 0                                           & 0      & \dots  \\
            \vdots                               & \vdots                                      & \vdots          \\
            ((a_n+1)/a_n)^r\cdots((a_1+1)/a_1)^r & 0                                           & 0      & \dots  \\
            0                                    & ((a_{n+1}+1)/a_{n+1})^r\dots((a_2+1)/a_2)^r & 0      & \dots  \\
            \vdots                               & \ddots                                      & \ddots & \ddots
        \end{pmatrix}
    $$
    and so obtain the very loose upper bound $\|B_{\infty}^n\|\leq\|B_{\infty}^ne_1\|\leq(n+1)((a_n+1)/a_n)^r\cdots ((a_1+1)/a_1)^r$. Then as $(a_n+1)/a_n\rightarrow1$ as $n\rightarrow\infty$, $\limsup_{n\rightarrow\infty}\|B^n_{\infty}\|^{1/n}\leq 1$ and so by Gelfand's spectral radius formula, $\spec(B_{\infty})\subset \mathbb{D}$. Next, note that $$\|B_{\infty}u\|^2=(1+((a_1+1)/a_1)^r)^2|u_1|^2+\sum_{k=2}^{\infty}((a_k+1)/a_k)^{2r}|u_k|^2\geq \|u\|^2$$ for any $u\in \ell^2(\mathbb{N})$ and so for any $|z|<1$,
    $$
        \|(B_{\infty}-zI)u\|\geq \|B_{\infty}u\|-|z|\|u\|\geq (1-|z|)\|u\|,
    $$
    and $B_{\infty}-zI$ is bounded below. Hence, for all $z\in\mathbb{C}$ such that $|z|<1$, we have $z\notin\spec_{\mathrm{ap}}(B_{\infty})$, which implies that $\partial\spec(B_{\infty})\subset \spec_{\mathrm{ap}}(B_\infty)\subset \mathbb{T}$. However, $0\in \spec(B_{\infty})$ (since $B_{\infty}$ is not surjective) so as $\spec(B_{\infty})\subset\mathbb{D}$ we must have that $\spec(B_{\infty})=\mathbb{D}$ and $\spec_{\mathrm{ap}}(B_{\infty})=\mathbb{T}$.

    Next, note that since the relevant matrices are lower triangular, $\spec_{\mathrm{ap}}(B_{l_k})=\spec(B_{l_k})=\{0,1\}$ for any $l_k\in\mathbb{N}$, and so if $z\notin \{0,1\}$ and $\{l_k\}$ is bounded then for each $k$, $(B_{l_k}-zI)^{-1}$ exists and has uniformly bounded norm. Hence, $\spec_{\mathrm{ap}}(B(\{l_k\}))=\spec(B(\{l_k\}))=\{0,1\}$. Finally, suppose that $\{l_k\}$ is unbounded. Note that for $n\leq l_k-1$,
    $$
        B_{l_k}(\{a_j\}_{j=m_k+1}^{m_{k+1}})^n=\begin{pmatrix}
            1                                                                      & 0                                                                          & 0      & \dots                                                                               &   &        & 0      \\
            \left(\frac{a_{m_k+1}+1}{a_{m_k+1}}\right)^r                           & 0                                                                          & 0      & \dots                                                                               &   &        & 0      \\
            \displaystyle\prod_{s=1}^2\left(\frac{a_{m_k+s}+1}{a_{m_k+s}}\right)^r & 0                                                                          & 0      & \dots                                                                               &   &        & 0      \\
            \vdots                                                                 & \vdots                                                                     & \vdots & \vdots                                                                              &   &        & \vdots \\
            \displaystyle\prod_{s=1}^n\left(\frac{a_{m_k+s}+1}{a_{m_k+s}}\right)^r & 0                                                                          & 0      & \dots                                                                               &   &        & 0      \\
            0                                                                      & \displaystyle\prod_{s=2}^{n+1}\left(\frac{a_{m_k+s}+1}{a_{m_k+s}}\right)^r & 0      & \dots                                                                               &   &        & \vdots \\

            \vdots                                                                 & \ddots                                                                     & \ddots & \ddots                                                                              &   &        & \vdots \\
            0                                                                      & \dots                                                                      & 0      & {\displaystyle \prod_{s=l_k+1-n}^{l_k}}\left(\frac{a_{m_k+s}+1}{a_{m_k+s}}\right)^r & 0 & \cdots & 0
        \end{pmatrix}
    $$
    while for $n\geq l_k$,
    $$
        B_{l_k}(\{a_j\}_{j=m_k+1}^{m_{k+1}})^n=\begin{pmatrix}
            1                                                                & 0      &  & \dots  &  & 0      \\
            ((a_{m_k+1}+1)/a_{m_k+1})^r                                      & 0      &  & \dots  &  & 0      \\
            ((a_{m_k+2}+1)/a_{m_k+2})^r((a_{m_k+1}+1)/a_{m_k+1})^r           & 0      &  & \dots  &  & 0      \\
            \vdots                                                           & \vdots &  & \ddots &  & \vdots \\
            ((a_{m_{k+1}}+1)/a_{m_{k+1}})^r\cdots((a_{m_k+1}+1)/a_{m_k+1})^r & 0      &  & \cdots &  & 0
        \end{pmatrix}.
    $$
    Using the fact that the sequence $\{(a_k+1)/a_k\}$ converges downwards to $1$ as $k\rightarrow\infty$, we find that
    $$
        \|B(\{l_k\},\{a_j\})^n\|^{1/n}\leq \left((n+1)\prod_{j=1}^n\left(\frac{a_j+1}{a_j}\right)^r\right)^{1/n}\rightarrow1,\text{ as }n\rightarrow\infty,
    $$
    and so by Gelfand's spectral radius formula, we have that $\spec(B(\{l_k\},\{a_j\}))\subset\mathbb{D}$. Now take $u_n=z^{l_n-1}f_{m_n+1}+((a_{m_n+1}+1)/a_{m_n})^rz^{l_n-2}f_{m_n+2}+\dots+((a_{m_{n}+1}+1)/a_{m_n+1})^r\cdots ((a_{m_{n+1}}+1)/a_{m_{n+1}})^rf_{m_{n+1}}$ for $|z|<1$. Then $(B(\{l_k\},\{a_j\})-zI)u_n=(1-z)z^{n-1}f_{m_n+1}$ and so $\|(B(\{l_k\},\{a_j\})-zI)u_n\|/\|u_n\|\rightarrow 0$ as $n\rightarrow\infty$. Hence for $\{l_k\}$ unbounded, we see that $z\in \spec_{\mathrm{ap}}(B(\{l_k\},\{a_j\}))$ for $|z|<1$. Hence as the approximate point spectrum is closed $\spec_{\mathrm{ap}}(B(\{l_k\},\{a_j\}))=\mathbb{D}$.
\end{proof}

We are now ready to state the main result of this section, which says that we cannot compute the approximate point spectrum of bounded Koopman operators in two limits, even with exact information. To prove this, we combine \cref{nearlyjordanblocks_lemma} with a suitably difficult canonical combinatorial problem within the SCI hierarchy, showing that any algorithm to compute the approximate point spectrum in two limits would also solve the combinatorial problem in two limits, yielding a contradiction. In the proof, we also use the (immediate) fact that for any bounded operators $\{A_i\}_{i\in\mathbb{N}}$ on a Hilbert space $\mathcal{H}$,
$$
    \bigcup_{i=1}^{\infty}\spec_{\mathrm{ap}}(A_i)\subset \spec_{\mathrm{ap}}\left(\bigoplus_{i=1}^{\infty}A_i\right).
$$

\begin{theorem}[Approximate point spectrum lower bound on $h^r(\mathbb{N})$]
    \label{notdelta3specapkoop_thm}
    The approximate point spectrum of bounded Koopman operators cannot be computed in two limits, even with exact information, i.e.,
    $$
        \left\{\Xi_{\spec_{\mathrm{ap}}},\Omega_{h^r(\mathbb{N})}^B,\mathcal{M}_{\mathrm{H}},\Lambda_{\mathbb{N}}\right\}\notin\Delta_3^{G}.
    $$
\end{theorem}

\begin{proof}
    Let $\Omega'$ denote the class of all infinite matrices $a=\{a_{i,j}\}_{i,j\in\mathbb{N}}$ with entries $a_{i,j}\in\{0,1\}$. Define the problem function $\Xi_{2,Q}$ by
    $$
        \Xi_{2,Q}(\{a_{i,j}\})=\begin{cases}
            1, & \text{if }\sum_{i,j}a_{i,j}=\infty\text{ for all but finitely many }j, \\
            0, & \text{otherwise}.
        \end{cases}
    $$
    It is known that if $\Lambda'$ is the set of component-wise evaluation of $\{a_{i,j}\}$ then $\{\Xi_{2,Q},\Omega',[0,1],\Lambda'\}\notin\Delta_3^G$ \cite{colbrook4}. We prove that we can reduce this decision problem to the problem of computing the approximate point spectrum of suitable Koopman operators, establishing the result.

    There exists a computable bijection between $\mathbb{N}$ and $\bigoplus_{l=1}^{\infty}\mathbb{N}$; using this, we can treat a dynamical system on $\bigoplus_{l=1}^{\infty}\mathbb{N}$ as one on $\mathbb{N}$. Given a sequence $\{c_i\}_{i\in\mathbb{N}}$, where each $c_i\in\{0,1\}$, define a corresponding $F_{\{c_i\}}:\mathbb{N}\rightarrow\mathbb{N}$ by
    $$
        F_{\{c_i\}}(k)=l\text{ if }\begin{cases}
            l<k, c_l=c_k=1,\text{ and }c_n=0\text{ for all }l<n<k,      \\
            l=k,c_k=1,\text{ and }c_n=0\text{ for all }n<k, \text{ or,} \\
            l=k \text{ and }c_k=0.
        \end{cases}
    $$
    Using the notation of \cref{nearlyjordanblocks_lemma}, the bounded Koopman operator corresponding to $F_{\{c_i\}}$ acts as $B_{\sum_i c_i}(\{b_i\})$ (where $\sum_i c_i$ may be infinite) on the closure of the span of $\{f_i:c_i=1\}$ (as defined in \eqref{fl_defn}), where $b_1<b_2<\cdots$ are such that $c_{b_i}=1$ for all $i$, and $c_n=0$ for $n\notin\{b_i\}$, and the identity on the orthogonal complement of this subspace. Given $a\in\Omega'$, we define $\{c_i^{(j)}\}_{i\in\mathbb{N}}$ for $j\in\mathbb{N}$ by
    $$
        c_i^{(j)}=c_i^{(j)}(a)=\begin{cases}
            1, & \text{if }i\leq j, \\ a_{i+1-j}, & \text{otherwise;}
        \end{cases}.
    $$
    This indexing choice ensures that if we have infinitely many columns with finitely many $1$s, the matrices $B_{l_k}$ they produce have unbounded $l_k$. Using the bijection between $\mathbb{N}$ and $\bigoplus_{l=1}^{\infty}\mathbb{N}$, we can treat $F_a:\bigoplus_{l=1}^{\infty}\mathbb{N}\rightarrow\bigoplus_{l=1}^{\infty}\mathbb{N}$ defined by $F_a(k_1,k_2,\dots)=(F_{\{c_i^{(1)}\}}(k_1),F_{\{c_i^{(2)}\}}(k_2),\dots)$ as $F_a:\mathbb{N}\rightarrow\mathbb{N}$. If $\Xi_{2,Q}(a)=1$, then $\koop_{F_a}$ is a direct sum of identity operators, a finite number of matrices $B_{l_k}$ and (infinitely) countably many copies of $B_{\infty}$ so by \cref{nearlyjordanblocks_lemma}, $\{0\}\cup\mathbb{T}\subset \spec_{\mathrm{ap}}(\koop_{F_a})$. Additionally, since by the proof of \cref{nearlyjordanblocks_lemma} each copy of $B_{\infty}$ has $B_{\infty}-zI$ bounded below by $(1-|z|)$ for $|z|<1$, we see that $\koop_{F_a}$ is bounded below for all $0<|z|<1$. Moreover, the Gelfand spectral radius formula arguments in the proof of \cref{nearlyjordanblocks_lemma} still carry through to the direct sum case as the bounds on the matrix norm are uniform, and so we have that $\spec_{\mathrm{ap}}(\koop_{F_a})=\{0\}\cup\mathbb{T}$. If instead $\Xi_{2,Q}(a)=0$, then $\koop_{F_a}$ is a direct sum of identity operators, (finitely or infinitely) countably many matrices $B_{\infty}$ and finite matrices $B_{l_k}$ with $\{l_k\}$ unbounded, and so by \cref{nearlyjordanblocks_lemma} $\mathbb{D}\subset \spec_{\mathrm{ap}}(\koop_{F_a})$, and so by Gelfand spectral radius formula arguments we have that $\spec_{\mathrm{ap}}(\koop_{F_a})=\mathbb{D}$.

    Finally, suppose by contradiction that $\{\Gamma_{n_2,n_1}\}$ is a $\Delta_3^{G}$-tower of algorithms for $\{\Xi_{\spec_{\mathrm{ap}}},\Omega_{h^r(\mathbb{N})},\mathcal{M}_{\mathrm{H}},\Lambda_{\mathbb{N}}\}$. Then, given $a\in \Omega'$, we can construct the corresponding Koopman operator $\koop_{F_a}$ as discussed above, and its matrix entries can be computed in only finitely many calls to $\Lambda'$ with input $a$. Hence,
    $$
        \tilde{\Gamma}_{n_2,n_1}(a)=\min\{2\dist(1/2,\Gamma_{n_2,n_1}({F_a})),1\}
    $$
    defines a general algorithm that maps into the desired output space $[0,1]$ and uses the evaluation functions $\Lambda'$, and has
    $$
        \lim_{n_2\rightarrow\infty}\lim_{n_1\rightarrow\infty}\tilde{\Gamma}_{n_2,n_1}(a)=\Xi_{2,Q}(a),
    $$ yielding the desired contradiction.
\end{proof}

The proof of \cref{notdelta3specapkoop_thm} fails when considering $\koop^*$ instead of $\koop$ because it is not true that $\|B_{\infty}^*u\|\geq\|u\|$ for all $u\in h^r(\mathbb{N})$ (e.g., $u=f_1-f_2$ for $f_i$ as defined in \eqref{fl_defn}). Moreover,  $\spec_{\mathrm{ap}}(B^*_{\infty})=\mathbb{D}$ rather than $\mathbb{T}$, as $f_1\notin \mathrm{im}(B_{\infty}-zI)$ for any $|z|<1$, which causes the subsequent proof to fail. As the approximate point pseudospectrum converges to the approximate point spectrum as $\epsilon\rightarrow 0$, \cref{notdelta3specapkoop_thm} implies the following result for $\Xi_{\spec_{\mathrm{ap},\epsilon}}$.

\begin{theorem}[Approximate point pseudospectrum lower bound on $h^r(\mathbb{N})$]\label{notdelta2specapepsdiscretesobolev_thm}
    There exists $R>0$ such that for all $0<\epsilon<R$,
    $$
        \left\{\Xi_{\spec_{\mathrm{ap},\epsilon}},\Omega_{h^r(\mathbb{N})},\mathcal{M}_{\mathrm{H}},\Lambda_{\mathbb{N}}\right\}\notin\Delta_2^G.
    $$
\end{theorem}

\begin{proof}
    Suppose that such $R$ does not exist by contradiction. Then, there exists a sequence $\epsilon_m\downarrow 0$ with corresponding towers of algorithms $\{\Gamma_n^{\epsilon_m}\}$ such that for all $F\in\Omega_{h^r(\mathbb{N})}$ and $m\in\mathbb{N}$, $\lim_{n\rightarrow\infty}\Gamma_n^{\epsilon_m}(F)=\spec_{\mathrm{ap},\epsilon_m}(\koop_F)$, where the convergence is in the Hausdorff metric. Define a tower of algorithms $\{\Gamma_{n_2.n_1}\}$ by $\Gamma_{n_2,n_1}(F)=\Gamma_{n_2}^{\epsilon_{n_1}}(F)$ for each $F\in\Omega_{h^r(\mathbb{N})}$. Then, $\lim_{n_2\rightarrow\infty}\lim_{n_1\rightarrow\infty}\Gamma_{n_2,n_1}(F)=\spec_{\mathrm{ap}}(\koop_F)$, contradicting \cref{notdelta3specapkoop_thm}.
\end{proof}

\subsection{Impossibility results for $H^r((0,1)^d)$}
\label{sobolevintervalimposs_sect}

We now prove impossibility results for the space $H^r((0,1)^d)$. Our proofs use particular constructions of interval exchange maps, exploiting the non-normality of the associated operators. An interval exchange map on $(0,1)$ is defined as follows. Suppose that $\{c_n\}_{n=0}^{\infty}\subset[1/2,1)$ is a strictly increasing sequence, where $c_0=1/2$ and $\lim_{n\rightarrow\infty}c_n=1$. Set $c_{-n}=1-c_n$ for $n\in\mathbb{N}$ and define a collection of intervals $\{I_n\}_{n=-\infty}^{\infty}$ by $I_n=[c_n,c_{n+1})$. An interval exchange map corresponding to $\{c_n\}$ is a continuous, strictly increasing (hence bijective) map $F:(0,1)\rightarrow(0,1)$ such that $F(c_n)=c_{n+1}$ and $F(I_n)=I_{n+1}$ for all $n\in\mathbb{Z}$. Often we define the sequence $\{c_n\}_{n=0}^{\infty}$ implicitly by defining the set of intervals $\{I_n\}_{n=0}^\infty$.

\subsubsection{Computing the approximate point spectrum of $\koop^*$ is not in $\Delta_2^G$}

We prove that computing the approximate point spectrum of $\koop^*$ is not in $\Delta_2^G$ and focus on the space $H^r((0,1))$. The extension to $H^r((0,1)^d)$ is straightforward and discussed later.
To prove the main theorem of this section, we use the following lemma, which bounds the spectrum of a particular type of interval exchange map.

\begin{lemma}
    \label{iemspecinterval_lemma}
    Let $F:(0,1)\rightarrow(0,1)$ be a $C^{\infty}$ interval exchange map on $(0,1)$ with intervals $\{I_n\}_{n=-\infty}^{\infty}$, such that $F$ has $C^1$ inverse and there exist $1/2<x_0,x_1<1$ for which $F'(x)=\alpha\in (0,1)$ for all $x>x_0$ and $F'(x)=1/\alpha$ for all $x<1-x_1$. The Koopman operator $\koop_F$ defined on the Sobolev space $H^r((0,1))$ for some $r\in\mathbb{N}$ is bounded and satisfies
    $$
        \{z\in\mathbb{C}:\alpha^{1/2}\leq |z|\leq \alpha^{-1/2}\}\subset\spec_{\mathrm{ap}}(\koop_F^*)\subset\spec(\koop^*_F)\subset\{z\in\mathbb{C}:\alpha^{r-1/2}\leq |z|\leq \alpha^{-(r-1/2)}\}.
    $$
\end{lemma}

\begin{proof}
    The Sobolev norm $\|f\|_{H^r}^2$ is equivalent to $\|f\|_{L^2}^2+\|f^{(r)}\|_{L^2}^2$ on $(0,1)$, i.e., we may consider only the smallest and largest derivatives \cite[Thm.~5.2]{adams_sobolev_2003}. Moreover, there exists a constant $C$ such that for all $g\in H^r((0,1))$,
    $\|g\|_{L^{\infty}}\leq C\|g\|_{H^r}
    $ \cite[Thm.~4, Sec.~5.6.2]{evans_partial_2010}.  By the inverse function theorem, since $F$ is $C^{\infty}$ and strictly increasing, $F^{-1}$ exists and is $C^{\infty}$. Note also that all derivatives of $F$ and $F^{-1}$ are bounded on $(0,1)$ (by the compactness of $[1-x_1,x_0]$); additionally, $F'$ and $(F^{-1})'$ are bounded both above and below for the same reason and that $F$ is strictly increasing. We will repeatedly use the Fa{\'a} di Bruno's formula \cite{faa_di_bruno_note_1857}, which provides a generalization of the chain rule to arbitrarily many function compositions. It states that for any $r\in\mathbb{N}\cup\{0\}$ and $l\in\mathbb{N}$, and for any $g\in C^r((0,1))$ and $x\in(0,1)$, using the notation $F^1=F$ and $F^l=F\circ F^{l-1}$ for $l\geq 1$ for the composition of $F$ with itself $l$ times, we have that
    \begin{equation} \label{eq_faa_bruno}
        \frac{d^r}{dx^r}(g\circ F^l)(x)=\sum_{m_1+\cdots+rm_r=r}\frac{r!}{m_1!1!^{m_1}\cdots m_r!r!^{m_r}}g^{(m_1+\cdots +m_r)}(F^l(x))\prod_{j=1}^r((F^l)^{(j)}(x))^{m_j}.
    \end{equation}
    We begin by showing that $\koop_F$ is bounded on $H^r((0,1))$. In particular, for any $g\in C^{\infty}((0,1))$, we have (here $\lesssim$ indicates inequality up to a constant independent of $g$),
    \begin{equation}
        \begin{aligned}\label{iec_boundedness}
            \|\koop_F g\|_{H^r}^2 & \lesssim \|\koop_F g\|_{L^2}^2+\|(\koop_F  g)^{(r)}\|_{L^2}^2 \\ &\lesssim \|g\|^2_{H^r}+\sum_{m_1+\cdots+rm_r=r}\int_0^1|g^{(m_1+\dots+m_r)}(F(x))|^2\prod_{j=1}^r|F^{(j)}(x)|^{2m_j}\dd x,
        \end{aligned}
    \end{equation}
    where the inequality $\|\koop_F g\|_{L^2}\leq \|g\|$ follows as $\int_0^1|g(F(x))|^2\;dx \leq 1\cdot \|g\|_{\infty}^2\lesssim\|g\|_{H^r}^2$. Then using that the derivatives of $F$ and $F^{-1}$ are bounded, we have that
    \begin{equation*}
        \begin{aligned}
            \|\koop_F g\|_{H^r}^2 & \lesssim \|g\|^2_{H^r}+\sum_{m_1+\dots+rm_r=r}\sup_{x\in   (0,1)}\{\prod_{j=1}^r|F^{(j)}(x)|^{2m_j}\}\int_0^1|g^{(m_1+\dots+m_r)}(F(x))|^2\dd x \\
                                  & \lesssim \|g\|_{H^r}^2+\sum_{m_1+\dots+rm_r=r}\int_0^1|g^{(m_1+\dots+m_r)}(u)|^2(F^{-1})'(u)\dd u\lesssim \|g\|^2_{H^r}.
        \end{aligned}
    \end{equation*}
    The boundedness of $\koop_F$ (and hence $\koop_F^*$) on $H^r((0,1))$ follows by the density of $C^{\infty}((0,1))$ and a closedness argument \cite[Ex.~IV.19]{ikeda_koopman_2024}.

    We now study the approximate point spectrum of $\koop_F$ and $\koop_F^*.$ For any $l\in\mathbb{N}$ and $g\in C^{\infty}((0,1))$, where $\lesssim$ indicates inequality up to a constant not depending on $g$ or $l$, by the same arguments as for \eqref{iec_boundedness}, we have that
    \begin{equation} \label{eq_koop_l_g}
        \begin{aligned}
            \|\koop_F^l g\|_{H^r}^2 & \lesssim \|\koop_F^l g\|_{L^2}^2+\|(\koop_F^l g)^{(r)}\|_{L^2}^2                                                                 \\
                                    & \lesssim \|g\|^2_{H^r}+\sum_{m_1+\cdots+rm_r=r}\int_0^1|g^{(m_1+\dots+m_r)}(F^l(x))|^2\prod_{j=1}^r|(F^l)^{(j)}(x)|^{2m_j}\dd x.
        \end{aligned}
    \end{equation}
    Splitting the above integral inside the sum into the intervals $I_n$ yields
    \begin{equation*}
        \begin{split}
            \int_0^1|g^{(m_1+\dots+m_r)}(F^l(x))|^2\prod_{j=1}^r|(F^l)^{(j)}(x)|^{2m_j}\dd x=\sum_{n\in\mathbb{Z}}\int_{I_n}|g^{(m_1+\dots+m_r)}(F^l(x)|^2\prod_{j=1}^r|(F^l)^{(j)}(x)|^{2m_j}\dd x \\
            =\sum_{n\in\mathbb{Z}}\int_{I_{n+l}}|g^{(m_1+\dots+m_r)}(u)|^2(\prod_{j=1}^r|(F^l)^{(j)}(F^{-l}(u))|^{m_j})^2[(F^l)'(F^{-l}(u))]^{-1}\dd u.
        \end{split}
    \end{equation*}
    Following repeated applications of the product and chain rules and the fact that $m_1+\dots+rm_r=r$, we find that $\prod_{j=1}^r|(F^l)^{(j)}(x)|^{m_j}$ is a sum of at most $l^r$ terms, where each term is a product of at most $lr$ copies of shifted $F'$ terms, at most $r$ of which have argument in the interval $I_{n+j-1}$ for $1\leq j\leq l$, and at most $r$ higher derivative terms. Let $x_n=\min\{1,\sup_{t\in I_n}|F'(t)|\}$ and $y_n=\inf_{t\in I_n}|F'(t)|$, then
    \begin{equation*}
        \begin{split}
             & \sum_{n\in\mathbb{Z}}\int_{I_{n+l}}|g^{(m_1+\dots+m_r)}(u)|^2(\prod_{j=1}^r|(F^l)^{(j)}(F^{-l}(u))|^{m_j})^2[(F^l)'(F^{-l}(u))]^{-1}\dd u                                                                               \\
             & \quad\quad\quad\lesssim\sum_{n\in\mathbb{Z}}l^{2r}\prod_{j=1}^{l}x_{n+j-1}^{2r}y_{n+j-1}\int_{I_n}|g^{(m_1+\dots+m_r)}(u)|^2\dd u\lesssim l^{2r}\|g\|_{H^r}^2\sup_{n\in\mathbb{Z}}\prod_{j=1}^lx_{n+j-1}^{2r}y_{n+j-1}.
        \end{split}
    \end{equation*}
    Combining this equation with \cref{eq_koop_l_g}, if $\|g\|_{H^r}=1$, we have the following inequality:
    \begin{equation} \label{eq_prod_zj}
        \|\koop_F^lg\|_{H^r}^{1/l}\leq C^{1/2l}l^{r/l}\left(\sup_{n\in\mathbb{Z}}\prod_{j=1}^lz_{n+j}\right)^{1/2l},
    \end{equation}
    for some constant $C>0$ not depending on $l$ or $g$, and where $z_{n}=\min\{x_n^{2r}y_n,1\}$. Next, by the AM-GM inequality we note that $(\sup_{n\in\mathbb{Z}}\prod_{j=1}^lz_{n+j})^{1/2l}\leq(\sup_{n\in\mathbb{Z}}l^{-1}\sum_{j=1}^l z_{n+j})^{1/2}$,
    and for any $L\in\mathbb{N}$, we have
    $$
        \limsup_{l\rightarrow\infty}\sup_{n\in\mathbb{Z}}\frac{1}{l}\sum_{j=1}^lz_{n+j}\leq \limsup_{l\rightarrow\infty}\frac{1}{l}\sum_{j=1}^l\sup_{|n|>L}z_{n}=\sup_{|n|>L}z_n.
    $$
    Therefore, since for all sufficiently large $t$ $F'(t)=\alpha$ and for all sufficiently small $t$ $F'(t)=1/\alpha$, after taking the limit as $L\rightarrow\infty$ by definition of $z_n$ we have that
    $$
        \limsup_{l\rightarrow\infty}\sup_{n\in\mathbb{Z}}\frac{1}{l}\sum_{j=1}^l z_{n+j}\leq \limsup_{|l|\rightarrow\infty}z_l=\alpha^{-(2r-1)}.
    $$
    Hence, following \cref{eq_prod_zj}, we have
    $\lim_{l\rightarrow\infty}\|\koop_F^l\|^{1/l}_{H^r}\leq \alpha^{-(r-1/2)}.
    $
    By Gelfand's formula for the spectral radius, we conclude that $\sup_{z\in\spec(\koop_F)}|z|\leq\alpha^{-(r-1/2)}$. The same argument applies if we replace $l$ by $-l$ but with $F$ replaced by $F^{-1}$ and vice versa. Hence,
    $$
        \inf_{z\in\spec(\koop_F)}|z|=\frac{1}{\sup_{z\in\spec(\koop_F^{-1})}|z|}\geq\frac{1}{\alpha^{-(r-1/2)}}=\alpha^{r-1/2},
    $$
    which implies that $\spec(\koop_F^*)\subset\{z\in\mathbb{C}:\alpha^{r-1/2}\leq|z|\leq\alpha^{-(r-1/2)}\}$ since $\spec(\koop_F)=\spec(\koop_F^*)$.

    For the other inclusion, fix $z\in\mathbb{C}$ such that $\sqrt{\alpha}< |z|< 1/\sqrt{\alpha}$ and $|z|\neq 1$. Suppose for contradiction that $\koop_{F}-zI$ is surjective, and choose $h\in H^r((0,1))$ such that $\operatorname{supp}h\subset I_0$. Then, there exists $g\in H^r((0,1))$ such that $\koop_F g-z g=h$. Let $f=g|_{I_0}$. Then taking $x\in I_0$ we have that $g(F(x))=zf(x)+h(x)$ and so as $F(x)\in I_1$ and $F$ is bijective, we see that
    \begin{equation}
        g|_{I_1}=z f\circ F^{-1}+h\circ F^{-1}.
    \end{equation}
    Next, take $x\in I_1$. Then $F(x)\in I_2$ and $\operatorname{supp}h\subset I_0$ so $g|_{I_2}(F(x))=zg|_{I_1}(x)$ and so combining with the previous expression for $g|_{I_1}$ we get that $g|_{I_2}=z^2f\circ F^{-2}+zh\circ F^{-2}$. By induction, setting $p(x)=z f(x)+h(x)$ for all $x\in I_0$, we have that $g|_{I_n}=z^{n-1}(p\circ F^{-n})$ for all $n\geq 1$. We first assume that $\int_{I_0}|p'(u)|^2\dd u>0$ and estimate $\|g\|_{H^r}$ as follows:
    \begin{align*}
        \|g\|_{H^r}  \geq \|g'\|_{L^2} & =\sum_{n=-\infty}^{\infty}\int_{I_n}|z|^{2(n-1)}|(p\circ F^{-n})'(x)|^2\dd x                                 \\
                                       & \geq\sum_{n=1}^{\infty}\int_{I_n}|z|^{2(n-1)}|(p\circ F^{-n})'(x)|^2\dd x                                    \\
                                       & =\sum_{n=1}^{\infty}|z|^{2(n-1)}\int_{I_n}|p'(F^{-n}(x))(F^{-1})'(F^{1-n}(x))\cdots (F^{-1})'(x)|^2\dd x     \\
                                       & =\sum_{n=1}^{\infty}|z|^{2(n-1)}\int_{I_0}|p'(u)|^2(F^{-1})'(F(u))\cdots (F^{-1})'(F^n(u))\dd u              \\
                                       & \geq|z|^{-2}\int_{I_0}|p'(u)|^2\dd u\sum_{n=1}^{\infty}|z|^{2n}\prod_{j=1}^{n}\inf_{t\in I_j}|(F^{-1})'(t)|.
    \end{align*}
    Then by assumption, there exists $J\in\mathbb{N}$ such that for all $j\geq\tilde J$, $F'|_{I_j}=\alpha$ and $(F^{-1})'|_{I_j}=1/\alpha$, and so as $1/|z|<1/\sqrt{\alpha}$, we have that $\inf_{t\in I_j}|(F^{-1})'(t)|>1/|z|^2$. Additionally, as $(F^{-1})'$ is bounded away from $0$, $C\coloneq |z|^{-2}\int_{I_0}|p'(u)|^2\dd u\prod_{j=1}^{J-1}\inf_{t\in I_j}|(F^{-1})'(t)|>0$. Then, discarding the positive terms corresponding to $n=1$ to $J-1$, we have that
    $$
        \|g\|_{H^r}\geq C\sum_{n=J}^{\infty}|z|^{2n}\prod_{j=J}^n \inf_{t\in I_j}|(F^{-1})'(t)| \geq C\sum_{n=1}^{\infty}|z|^{2n}/|z|^{2(n-J+1)}=\infty,
    $$
    which implies that $g\notin H^r((0,1))$, and yields the contradiction. If instead we have $\int_{I_0}|p'(u)|^2\dd u=0$, then $p|_{I_0}$ is constant and so $g$ is piecewise constant on $\bigcup_{n=1}^{\infty}I_n$. Moreover, for $|z|\neq 1$, $g$ is not weakly differentiable and hence $g\notin H^r((0,1))$, giving the same contradiction. Then, since $\spec_{\mathrm{su}}(\koop_{F})$ is closed,
    $\{z\in\mathbb{C}:\sqrt{\alpha}\leq |z|\leq 1/\sqrt{\alpha}\}\subset \spec_{\mathrm{su}}(\koop_{F})=\spec_{\mathrm{ap}}(\koop_F^*).
    $
\end{proof}

We now prove our main impossibility result on $H^r((0,1))$.

\begin{theorem}[Approximate point spectrum lower bound on $H^r((0,1))$]
    \label{notdelta2specapinterval_thm}
    Let $r\in\mathbb{N}$; then there is no sequence of general algorithms that can compute the approximate point spectrum of any densely defined Koopman operator on $H^r((0,1))$ in one limit, i.e.,
    $$
        \left\{\Xi_{\spec^*_{\mathrm{ap}}},\Omega_{H^r((0,1))},\mathcal{M}_{\mathrm{AW}},\Lambda_{(0,1)}\right\}\notin\Delta_2^G.
    $$
\end{theorem}

\begin{proof}
    Suppose, by contradiction, that $\{\Gamma_n\}$ is a $\Delta_2^G$-tower for $\{\Xi_{\spec^*_{\mathrm{ap}}},\Omega_{H^r((0,1))},\mathcal{M}_{\mathrm{AW}},\Lambda_{(0,1)}\}$ satisfying
    \begin{equation} \label{eq_spga_f}
        \lim_{n\rightarrow\infty}\Gamma_n(F)=\spec_{\mathrm{ap}}(\koop^*_F)\quad \forall F\in\Omega_{H^r((0,1))}.
    \end{equation}
    Throughout the proof, fix $\alpha_1,\alpha_2\in (0,1)$ and $\tau>0$ such that $(1+2\tau)\alpha_1^{-(r-1/2)}< (1-2\tau)\alpha_2^{-1/2}$, and let $K\subset\mathbb{C}$ be a compact set such that $B_{2/\alpha_2}(0)\subset K$.

    We first define an interval exchange map $F_1$ with $|I_n^{(1)}|=\beta_1\alpha_1^n$ for $n=0,1,\dots$, where $\beta_1$ is chosen such that
    $$
        \sum_{n=-\infty}^{\infty}|I_n^{(1)}|=\beta_1\left(1+2\sum_{n=1}^{\infty}\alpha_1^n\right)=\beta_1\left(1+2\cfrac{\alpha_1}{1-\alpha_1}\right)=1.
    $$
    We select $F_1(x)=\alpha x$ on $\smash{I_j^{(1)}}$ for $j\geq 1$, $F_1(x)=x/\alpha$ on $I_j$ for $j\leq -1$ and choose it to be a smooth transition function $g^{(1)}$ on $I_0$ such that $F_1$ is globally smooth and strictly increasing; since $g^{(1)}$ must be continuous at its endpoints in $I_0$ we automatically have an interval exchange map. In particular,
    $$
        F_1(x)=\begin{cases}
            x/\alpha_1,             & x\in \bigcup_{n=-\infty}^{-1}I_n^{(1)}, \\
            g^{(1)}(x),             & x\in I_{0}^{(1)},                       \\
            \alpha_1x+(1-\alpha_1), & x\in\bigcup_{n=1}^{\infty}I_n^{(1)}.
        \end{cases}
    $$
    Following \cref{iemspecinterval_lemma}, we have
    $$
        \{z\in\mathbb{C}:\alpha_1^{1/2}\leq |z|\leq \alpha_1^{-1/2)}\}\subset\spec_{\mathrm{ap}}(\koop_{F_1}^*)\subset\spec(\koop^*_{F_1})\subset\{z\in\mathbb{C}:\alpha_1^{r-1/2}\leq |z|\leq \alpha_1^{-(r-1/2)}\}.
    $$
    It follows that by combining \cref{genalg_defn} and the definition of convergence in the Attouch--Wets topology (see \cref{sec_mean_spectral}), there exist $n_1>0$ and $N_1>1$ such that
    $$
        \sup\{|z|:{z\in K\cap \Gamma_{n_1}({F_1})}\}\leq (1+\tau)\alpha_1^{-(r-1/2)},\quad\Lambda_{\Gamma_{n_1}}(F_1)\subset\bigcup_{n=-N_1}^{n=N_1-1}I_n^{(1)}.
    $$
    Next, we choose intervals $\{I_n^{(2)}\}$ such that
    $$
        |I_n^{(2)}|=\begin{cases}
            |I_n^{(1)}|,                     & 0\leq n\leq N_1, \\
            \beta^{(2)}\alpha_2^{n-(N_1+1)}, & n>N_1,
        \end{cases}
    $$
    where as before $\beta_2$ is chosen such that $\sum_{n=-\infty}^{\infty}|I_n^{(2)}|=1$. Then we define $F_2(x)$ to be
    $$
        F_2(x)=\begin{cases}
            x/\alpha_2,             & x\in\bigcup_{n=-\infty}^{-N_1-2}I_n^{(2)}, \\
            g_1^{(2)}(x),           & x\in I_{-N_1-1}^{(2)},                     \\
            F_1(x),                 & x\in\bigcup_{n=-N_1}^{N_1-1}I_n^{(2)},     \\
            g_2^{(2)}(x),           & x\in I_{N_1}^{(2)},                        \\
            \alpha_2x+(1-\alpha_2), & x\in\bigcup_{n=N_1+1}^{\infty}I_n^{(2)},
        \end{cases}
    $$
    where $\smash{g_1^{(2)}}$ and $\smash{g_2^{(2)}}$ are transition functions such that $F_2$ is globally smooth and strictly increasing. By \cref{iemspecinterval_lemma} again, we know that
    $$
        \{z\in\mathbb{C}:\alpha_2^{1/2}\leq |z|\leq \alpha_2^{-1/2)}\}\subset\spec_{\mathrm{ap}}(\koop_{F_2}^*)\subset\spec(\koop^*_{F_2})\subset\{z\in\mathbb{C}:\alpha_2^{r-1/2}\leq |z|\leq \alpha_2^{-(r-1/2)}\}
    $$
    and so by \cref{genalg_defn} of a general algorithm, there exists $n_2>n_1$, $N_2>N_1+2$ such that
    $$
        \sup\{|z|:{z\in K\cap \Gamma_{n_2}({F_2})}\}\geq (1-\tau)\alpha_2^{-1/2},\quad\Lambda_{\Gamma_{n_2}}(F_2)\subset\bigcup_{n=-N_2}^{n=N_2-1}I_n^{(2)}.
    $$
    We may now proceed by induction and construct a sequence of maps $F_l$ of the form given in \cref{iemspecinterval_lemma} and $n_l,N_l\in\mathbb{N}$ for all $l\in\mathbb{N}$. In particular, suppose that we have some $F_k$ with intervals $\smash{\{I_n^{(k)}\}}$, $n_k$ and $N_k$ and want to construct $F_{k+1}$. Take $i=1$ if $k$ is even and $i=2$ if $k$ is odd. Let $\smash{\{I_n^{(k+1)}\}}$ be such that
    $$
        |I_n^{(k+1)}|=\begin{cases}
            |I_n^{(k)}|,                       & 0\leq n\leq N_k, \\
            \beta^{(k+1)}\alpha_i^{n-(N_1+1)}, & n>N_k,
        \end{cases}
    $$
    where as before $\beta^{(k+1)}$ is chosen such that $\sum_{n=-\infty}^{\infty}|I_n^{(k+1)}|=1$. Then we define $F_{k+1}(x)$ to be
    $$
        F_{k+1}(x)=\begin{cases}
            x/\alpha_i,             & x\in\bigcup_{n=-\infty}^{-N_k-2}I_n^{(k+1)}, \\
            g_1^{(k+1)}(x),         & x\in I_{-N_k-1}^{(k+1)},                     \\
            F_k(x),                 & x\in\bigcup_{n=-N_k}^{N_k-1}I_n^{(k+1)},     \\
            g_2^{(k+1)}(x),         & x\in I_{N_k}^{(k+1)},                        \\
            \alpha_ix+(1-\alpha_i), & x\in\bigcup_{n=N_k+1}^{\infty}I_n^{(k+1)},
        \end{cases}
    $$
    where $\smash{g_1^{(k+1)}}$ and $\smash{g_2^{(k+1)}}$ are transition functions such that $F_2$ is globally smooth and strictly increasing. By the same arguments as previously, if $k$ is odd:
    \begin{equation}
        \label{eqn:koddprob}
        \sup\{|z|:{z\in K\cap \Gamma_{n_k}({F_k})}\}\leq (1+\tau)\alpha_1^{-(r-1/2)},\quad\Lambda_{\Gamma_{k}}(F_k)\subset\bigcup_{n=-N_k}^{n=N_k-1}I_n^{(k)},
    \end{equation}
    while if $k$ is even:
    \begin{equation}
        \label{eqn:kevenprob}
        \sup\{|z|:{z\in K\cap \Gamma_{n_k}({F_k})}\}\geq (1-\tau)\alpha_2^{-1/2},\quad\Lambda_{\Gamma_{n_k}}(F_k)\subset\bigcup_{n=-N_k}^{n=N_k-1}I_n^{(k)}.
    \end{equation}
    Next, we define the function $F$ by $F(x)=\lim_{k\rightarrow\infty}F_k(x)$, where the limit exists by construction as for all $m\geq k$,
    $$\bigcup_{n=-N_k}^{N_k-1}I_n^{(m)}=\bigcup_{n=-N_k}^{N_k-1}I_n^{(k)},\quad F_m|_{\bigcup_{n=-N_k}^{N_k-1}I_n^{(m)}}=F_k|_{\bigcup_{n=-N_k}^{N_k-1}I_n^{(k)}};$$ note that this also implies that $F$ is a smooth, strictly increasing, interval exchange map on $(0,1)$. However, $\sup_{x\in I_{N_k}^{(k+1)}}|(g_2^{(k+1)})''(x)|$ is unbounded as $k\rightarrow\infty$, since at either end of the interval $I_{N_k}^{(k+1)}$ $(g_2^{(k+1)})'(x)$ is $\alpha_1$ and $\alpha_2$ in some order but $|I_{N_k}^{(k+1)}|\rightarrow 0$ as $k\rightarrow\infty$; hence $\koop_F$ is unbounded whenever $r>1$. It is, however, densely defined by considering compactly supported smooth functions, which are dense in $H^r((0,1))$.

    By consistency of general algorithms, $\Gamma_{n_k}(F)=\Gamma_{n_k}({F_k})$ for all $k$. Since $(1+2\tau)\alpha_1^{-(r-1/2)}<(1-2\tau)\alpha_2^{-1/2}$, it follows that $\Gamma_{n_k}(F)$ cannot converge, the required contradiction.
\end{proof}

\begin{remark}[Extensions to higher dimensions and other domains]
    It is straightforward to extend this proof to $\mathcal{X}=(0,1)^d$, under the assumption that $d<r/2$ so $H^r((0,1)^d)$ is still an RKHS. Indeed, we may adapt \cref{iemspecinterval_lemma} to functions of the form $\hat{F}(x_1,\dots,x_d)=(F(x_1),c_2,\dots,c_d)$ where $F$ satisfies the conditions of the original \cref{iemspecinterval_lemma} and $\{c_i\}_{i=2}^d$ are fixed constants. The estimates of $\|\koop^lg\|_{H^r}^2$ for $l\geq 1$ then hold exactly (up to a constant) as $\hat{F}$ only depends on $x_1$ and the constant functions lie in the Sobolev space. For finding the surjective spectrum, if we take $h\in H^r((0,1)^d)$ such that $\operatorname{supp} h\subset I_0\times (0,1)^{d-1}$ and $\partial_ih=0$ for $i=2,\dots,d$, then we must also have $\partial_ig=0$ for $i=2,\dots,d$ and the same arguments goes through again. \cref{notdelta2specapinterval_thm} then goes through using the functions $\hat{F}$. By translating and scaling, it is immediate to show that the result holds for any rectangular domain. A similar trick of taking $\hat{F}$ to be radial shows that the result also holds on spherical domains.
\end{remark}

\subsubsection{Computing the approximate point pseudospectrum of $\koop^*$ is not in $\Delta_1^G$}

The same family of adversarial dynamical systems can also be used to prove that computing $\spec_{\mathrm{ap},\epsilon}(\koop^*)$ does not lie in $\Delta_1^G$. We prove it on the class of bounded operators with the Hausdorff metric, which implies it for unbounded operators in the Attouch--Wets metric. As before, this result holds even with exact information and can be adapted to more general rectangular domains.

\begin{theorem}[Approximate point pseudospectrum lower bound]\label{notdelta1specapeps_thm}
    For all $\epsilon>0$,
    $$
        \left\{\Xi_{\spec_{\mathrm{ap},\epsilon}^*},\Omega_{H^r((0,1))}^B,\mathcal{M}_{\mathrm{H}},\Lambda_{(0,1)}\right\}\notin \Delta_1^G.
    $$
\end{theorem}

\begin{proof} Take $\epsilon>0$, fix $\alpha\in (0,1)$ and define an interval exchange map $F:(0,1)\rightarrow (0,1)$ with $|I_n|=\beta\alpha^n$ for $n\in\mathbb{N}\cup\{0\}$, where $\beta$ is chosen so that
    $$
        \sum_{n=-\infty}^{\infty}|I_n|=\beta\left(1+2\sum_{n=1}^{\infty}\alpha^n\right)=1.
    $$
    We take $F(x)=\alpha x$ on $I_j$ for $j\geq 1$, $F(x)=x/\alpha$ on $I_j$ for $j\leq -1$ and $F(x)=g(x)$ on $I_{0}$, where $g$ is a smooth transition function such that $F$ is globally smooth and strictly increasing. In particular, we have that
    $$
        F(x)=\begin{cases}
            x/\alpha,            & x\in \bigcup_{n=-\infty}^{-1}I_n, \\
            g(x),                & x\in I_{0},                       \\
            \alpha x+(1-\alpha), & x\in\bigcup_{n=1}^{\infty}I_n.
        \end{cases}
    $$
    Then, as $\koop_{F}$ is bounded, for $|z|>\|\koop_{F}^*\|$, by a standard Neumann series argument, we have that $\koop_F^*-zI=-z(I-\koop_F/z)$ is invertible with inverse given by $-1/z(I+\koop_F/z+\koop_F^2/z^2+\cdots)$, and so
    \begin{equation} \label{eq_inv_koop}
        \|(\koop_{F}^*-z I)^{-1}\|\leq\frac{1/|z|}{1-\|\koop_{F}^*\|/|z|}=\frac{1}{|z|-\|\koop_{F}^*\|}.
    \end{equation}
    Therefore, there exists $N>0$ such that for all $|z|\geq N$, $\|(\koop_{F}^*-zI)^{-1}\|\leq 1/\epsilon$.

    Suppose for contradiction that $\{\Gamma_n\}$ is a $\Delta_1^G$-tower for $\{\Xi_{\spec_{\mathrm{ap},\epsilon}^*},\Omega_{H^r((0,1))}^B,\mathcal{M}_{\mathrm{H}},\Lambda_{(0,1)}\}$ satisfying
    $$
        d_{\mathrm{H}}(\Gamma_n(F),\spec_{\mathrm{ap},\epsilon}(\koop_F^*))\leq 2^{-n}\quad \forall n\in\mathbb{N},F\in\Omega_{H^r((0,1))}.
    $$
    By \cref{genalg_defn} of a general algorithm, there exists $N_1>1$ such that
    \begin{equation}
        \label{eqn:evalfnsdelta1}
        d_{\mathrm{H}}(\Gamma_{n_1}(F),\spec_{\mathrm{ap},\epsilon}(\koop_F^*))\leq 2^{-n_1},\quad\Lambda_{\Gamma_{n_1}}(F)\subset\bigcup_{n=-N_1}^{N_1-1}I_n.
    \end{equation}
    Now, choose $\hat{\alpha}\in\mathbb{R}$ such that $\hat{\alpha}^{-1/2}>N+1$, and define intervals $\hat{I}_n$ such that
    $$
        |\hat{I}_n|=\begin{cases}
            |I_n|,                                & 0\leq n\leq N_1, \\
            \hat{\beta} \hat{\alpha}^{n-(N_1+1)}, & n>N_1,
        \end{cases}
    $$
    where $\hat{\beta}$ is chosen such that $\sum_{n=-\infty}^{\infty}|\hat{I}_n|=1$. Then, we define the map $\hat{F}:(0,1)\rightarrow(0,1)$ by
    $$
        \hat{F}(x)=\begin{cases}
            x/\hat{\alpha},                 & x\in\bigcup_{n=-\infty}^{-N_1-2}I_n, \\
            g_1(x),                         & x\in I_{-N_1-1},                     \\
            F(x),                           & x\in\bigcup_{n=-N_1}^{N_1-1}I_n,     \\
            g_2(x),                         & x\in I_{N_1},                        \\
            \hat{\alpha}x+(1-\hat{\alpha}), & x\in\bigcup_{n=N_1+1}^{\infty}I_n,
        \end{cases}
    $$
    where $g_1$ and $g_2$ are smooth transition functions such that $\hat{F}$ is globally smooth and strictly increasing. Applying \cref{iemspecinterval_lemma} again yields
    \begin{equation}
        \label{eqn:spectrumboundalphahat}
        \{z\in\mathbb{C}:\hat{\alpha}^{1/2}\leq |z|\leq \hat{\alpha}^{-1/2)}\}\subset \spec_{\mathrm{ap}}(\koop_{\hat{F}}^*).
    \end{equation}
    Since $\bigcup_{n=-N_1}^{N_1-1}\hat{I}_n=\bigcup_{n=-N_1}^{N_1-1}I_n$ and $F|_{\bigcup_{n=-N_1}^{N_1-1}I_n}=\hat{F}|_{\bigcup_{n=-N_1}^{N_1-1}\hat{I}_n}$, by \eqref{eqn:evalfnsdelta1} and the consistency of general algorithms, we have
    \begin{equation}\label{eqn:FFhatagreement}
        \Gamma_{1}(\koop_F)=\Gamma_{1}(\koop_{\hat{F}}).
    \end{equation}
    Let $z\in\mathbb{C}$ such that $|z|\geq N$, then for all $\|E\|<\epsilon$, $\|(\koop_F^*-zI)^{-1}\|\|E\|<1$ (cf.~\cref{eq_inv_koop}). Then by the same Neumann series argument as before, $\koop^*_F+E-zI=(\koop_F^*-zI)(I+(\koop_F^*-zI)^{-1}E)$ is invertible and so $z\notin\spec(\koop_{F}^*+E)$, and hence $z\notin\spec_{\mathrm{ap},\epsilon}(\koop_{F}^*)$. This implies that \begin{equation}\label{eqn:spectrumboundN}\spec_{\mathrm{ap},\epsilon}(\koop_{F}^*)\subset \{z\in\mathbb{C}:|z|\leq N\}.\end{equation}
    However, using \eqref{eqn:spectrumboundalphahat}, \eqref{eqn:spectrumboundN} and \eqref{eqn:FFhatagreement}, the following inequalities hold:
    \begin{align*}
        1<\hat{\alpha}^{-1/2}-N & \leq d_{\mathrm{H}}(\{z\in\mathbb{C}:|z|\leq N\},\spec_{\mathrm{ap},\epsilon}(\koop^*_{\hat{F}}))\leq d_{\mathrm{H}}(\spec_{\mathrm{ap},\epsilon}(\koop^*_{F}),\spec_{\mathrm{ap},\epsilon}(\koop_{\hat{F}}^*)) \\&
        \leq d_{\mathrm{H}}(\Gamma_1(F),\spec_{\mathrm{ap},\epsilon}(\koop^*_{F}))+d_{\mathrm{H}}(\Gamma_1(F),\spec_{\mathrm{ap},\epsilon}(\koop^*_{\hat{F}}))                                                                                    \\
                                & =d_{\mathrm{H}}(\Gamma_1(F),\spec_{\mathrm{ap},\epsilon}(\koop^*_{F}))+d_{\mathrm{H}}(\Gamma_1(\hat{F}),\spec_{\mathrm{ap},\epsilon}(\koop^*_{\hat{F}}))\leq 1/2+1/2=1,
    \end{align*}
    giving the desired contradiction.
\end{proof}

\subsection{Impossibility results for $H^r(\mathbb{R}^d)$}
\label{sobolevreallineimposs_sect}

We now adapt the proof of \cref{iemspecinterval_lemma} to consider the class of bounded (instead of unbounded)  operators and rely on interval exchange maps on $\mathbb{R}$ (instead of $(0,1)$). An interval exchange map on $\mathbb{R}$ is defined as follows. Suppose that $\{c_n\}_{n=0}^{\infty}\subset[0,\infty)$ is a strictly increasing sequence, where $c_0=0$ and $\lim_{n\rightarrow\infty}c_n=\infty$. Set $c_{-n}=-c_n$ for $n\in\mathbb{N}$ and define a collection of intervals $\{I_n\}_{n=-\infty}^{\infty}$ by $I_n=[c_n,c_{n+1})$. An interval exchange map corresponding to $\{c_n\}$ is a continuous, strictly increasing (and hence bijective) map $F:(0,1)\rightarrow(0,1)$ such that $F(c_n)=c_{n+1}$ and $F(I_n)=I_{n+1}$ for all $n\in\mathbb{Z}$.

\begin{lemma}
    \label{iemspecrealline_lemma}
    Let $F$ be a $C^{\infty}$, invertible interval exchange map on $\mathbb{R}$ with intervals $\{I_n\}_{n=-\infty}^{\infty}$ such that there exists $x_0,x_1>0$ for which $F'(x)=1/\alpha\in (1,\infty)$ for all $x>x_0$ and $F'(x)=\alpha\in (0,1)$ for all $x<-x_1$. The Koopman operator $\koop_F$ defined on the Sobolev space $H^r(\mathbb{R})$ for some $r\in\mathbb{N}$ is bounded and satisfies
    $$
        \{z\in\mathbb{C}:\alpha^{1/2}\leq |z|\leq \alpha^{-1/2)}\}\subset\spec_{\mathrm{ap}}(\koop_F^*)\subset\spec(\koop^*_F)\subset\{z\in\mathbb{C}:\alpha^{r-1/2}\leq |z|\leq \alpha^{-(r-1/2)}\}.
    $$
\end{lemma}

\begin{remark}
    Notice that $\alpha$ and $1/\alpha$ are switched in the definition of $F$ compared to \cref{iemspecinterval_lemma}. This is necessary due to how the functions are constructed in the proof of \cref{notdelta2specaprealline_thm}.
\end{remark}

\begin{proof}[Proof of \cref{iemspecrealline_lemma}]
    The Sobolev norm $\|f\|_{H^r}^2$ is equivalent to $\|f\|_{L^2}^2+\|f^{(r)}\|_{L^2}^2$ on $\mathbb{R}$, i.e., we may consider only the smallest and largest derivatives \cite[Thm.~5.2]{adams_sobolev_2003}. By the inverse function theorem, since $F$ is $C^{\infty}$ and strictly increasing, $F^{-1}$ exists and is $C^{\infty}$. Note also that all derivatives of $F$ and $F^{-1}$ are bounded on $\mathbb{R}$ (by compactness of the interval $[-x_1,x_0]$); additionally, by compactness and that $F$ is strictly increasing, $F'$ and $(F^{-1})'$ are bounded both above and below on $\mathbb{R}$. We first show that $\koop_F$ is bounded. For any $g\in C_c^{\infty}(\mathbb{R})$, where $\lesssim$ indicates inequality up to a constant not depending on $g$, we have
    \begin{align*}
        \|\koop_F g\|_{H^r}^2 & \lesssim \|\koop_F g\|_{L^2}^2+\|(\koop_F g)^{(r)}\|_{L^2}^2                                                                                                                             \\
                              & \lesssim \int_{\mathbb{R}}|g(F(x))|^2\dd x+\sum_{m_1+\cdots+rm_r=r}\int_\mathbb{R}|g^{(m_1+\dots+m_r)}(F(x))|^2\prod_{j=1}^r|F^{(j)}(x)|^{2m_j}\dd x                                     \\
                              & \lesssim \int_{\mathbb{R}}|g(F(x))|^2\dd x+\sum_{m_1+\dots+rm_r=r}\sup_{x\in   \mathbb{R}}\left(\prod_{j=1}^r|F^{(j)}(x)|^{2m_j}\right)\int_\mathbb{R}|g^{(m_1+\dots+m_r)}(F(x))|^2\dd x \\
                              & \lesssim \int_{\mathbb{R}}|g(u)|^2(F^{-1})'(u)\dd u+\sum_{m_1+\dots+rm_r=r}\int_\mathbb{R}|g^{(m_1+\dots+m_r)}(u)|^2(F^{-1})'(u)\dd u\lesssim \|g\|^2_{H^r}.
    \end{align*}
    It follows by the density of $C^{\infty}_c(\mathbb{R})$ and a closedness argument \cite[Ex.~IV.19]{ikeda_koopman_2024} that $\koop_F$, and hence $\koop_F^*$, is bounded on $H^r(\mathbb{R})$. Next, for any $l\in\mathbb{N}$ and $g\in C^{\infty}_c(\mathbb{R})$, where $\lesssim$ indicates inequality up to a constant not depending on $g$ or $l$, we have
    \begin{equation} \label{eq_sum_int}
        \begin{aligned}
            \|\koop_F^l g\|_{H^r}^2 & \lesssim \|\koop_F^l g\|_{L^2}^2+\|(\koop_F^l g)^{(r)}\|_{L^2}^2                                                                                                           \\
                                    & \lesssim \int_{\mathbb{R}}|g(F^l(x))|^2\dd x+\sum_{m_1+\cdots+rm_r=r}\int_\mathbb{R}|g^{(m_1+\dots+m_r)}(F^l(x))|^2\left(\prod_{j=1}^r|(F^l)^{(j)}(x)|^{2m_j}\right)\dd x.
        \end{aligned}
    \end{equation}
    Splitting up the integral in \cref{eq_sum_int} into the intervals $I_n$ yields
    \[
        \begin{split}
            \int_\mathbb{R}|g^{(m_1+\dots+m_r)}(F^l(x))|^2\prod_{j=1}^r|(F^l)^{(j)}(x)|^{2m_j}\dd x =\sum_{n\in\mathbb{Z}}\int_{I_n}|g^{(m_1+\dots+m_r)}(F^l(x)|^2\prod_{j=1}^r|(F^l)^{(j)}(x)|^{2m_j}\dd x \\
            =\sum_{n\in\mathbb{Z}}\int_{I_{n+l}}|g^{(m_1+\dots+m_r)}(u)|^2\left(\prod_{j=1}^r|(F^l)^{(j)}(F^{-l}(u))|^{m_j}\right)^2((F^l)'(F^{-l}(u)))^{-1}\dd u
        \end{split}
    \]
    Repeated applications of the product and chain rules and that $m_1+\dots+rm_r=r$ imply that $\prod_{j=1}^r|(F^l)^{(j)}(x)|^{m_j}$ is a sum of at most $l^r$ terms, where each term is a product of at most $lr$ copies of shifted $F'$ terms, at most $r$ of which have argument in the interval $I_{n+j-1}$ for $1\leq j\leq l$, and at most $r$ higher derivative terms. Let $x_n=\min\{1,\sup_{t\in I_n}|F'(t)|\}$ and let $y_n=\inf_{t\in I_n}|F'(t)|$. Then, we have that
    \begin{equation*}
        \begin{split}
             & \sum_{n\in\mathbb{Z}}\int_{I_{n+l}}|g^{(m_1+\dots+m_r)}(u)|^2\left(\prod_{j=1}^r|(F^l)^{(j)}(F^{-l}(u))|^{m_j}\right)^2((F^l)'(F^{-l}(u)))^{-1}\dd u                                                                    \\
             & \quad\quad\quad\lesssim\sum_{n\in\mathbb{Z}}l^{2r}\prod_{j=1}^{l}x_{n+j-1}^{2r}y_{n+j-1}\int_{I_n}|g^{(m_1+\dots+m_r)}(u)|^2\dd u\lesssim l^{2r}\|g\|_{H^r}^2\sup_{n\in\mathbb{Z}}\prod_{j=1}^lx_{n+j-1}^{2r}y_{n+j-1}.
        \end{split}
    \end{equation*}
    Similarly, for the simpler term in \cref{eq_sum_int}, we have
    \[
        \int_{\mathbb{R}}|g(F^l(x))|^2\dd x=\sum_{n\in\mathbb{Z}}\int_{I_n}|g(F^l(x))|^2\dd x=\sum_{n\in\mathbb{Z}}\int_{I_{n+l}}|g(u)|^2((F^l)'(F^{-l}(u)))^{-1}\dd u\lesssim \|g\|_{H^r}^2\sup_{n\in\mathbb{Z}}\prod_{j=1}^ly_{n+j-1}.
    \]
    Therefore, if $\|g\|_{H^r}=1$, we have for some constant $C$ not depending on $l$ or $g$,
    $$
        \|\koop_F^lg\|_{H^r}^{1/l}\leq C^{1/2l}l^{r/l}\left(\sup_{n\in\mathbb{Z}}\prod_{j=1}^lz_{n+j}\right)^{1/2l},
    $$
    where $z_{n}=\min\{x_n^{2r},1\}y_n$. Next, by the AM-GM inequality, we obtain
    \[
        \left(\sup_{n\in\mathbb{Z}}\prod_{j=1}^lz_{n+j}\right)^{1/2l}\leq\sqrt{\sup_{n\in\mathbb{Z}}\frac{1}{l}\sum_{j=1}^lz_{n+j}}.\]
    For any $L\in\mathbb{N}$,
    $$
        \limsup_{l\rightarrow\infty}\sup_{n\in\mathbb{Z}}\frac{1}{l}\sum_{j=1}^lz_{n+j}\leq \limsup_{l\rightarrow\infty}\frac{1}{l}\sum_{j=1}^l\sup_{|n|>L}z_{n}=\sup_{|n|>L}z_n,
    $$
    and so taking $L\rightarrow\infty$ we find that
    $$
        \limsup_{l\rightarrow\infty}\sup_{n\in\mathbb{Z}}\frac{1}{l}\sum_{j=1}z_{n+j}\leq \limsup_{|l|\rightarrow\infty}z_l=\alpha^{-(2r-1)},
    $$
    which implies that $\lim_{l\rightarrow\infty}\|\koop_F^l\|^{1/l}_{H^r}\leq \alpha^{-(r-1/2)}$. If we replace $l$ by $-l$, the same argument applies but with $F$ replaced by $F^{-1}$ and vice versa. Finally, by Gelfand's spectral radius formula, we have
    $
        \spec(\koop_F^*)\subset\{z\in\mathbb{C}:\alpha^{r-1/2}\leq|z|\leq\alpha^{-(r-1/2)}\}.
    $

    For the other inclusion, we fix $z\in\mathbb{C}$ such that $\sqrt{\alpha}< |z|< 1/\sqrt{\alpha}$ and $|z|\neq 1$. Suppose by contradiction that $\koop_{F}-zI$ is surjective and choose $h\in H^r(\mathbb{R})$ such that $\mathrm{supp}\;h\subset I_0$. Then, there exists $g\in H^r(\mathbb{R})$ such that $\koop_F g-z g=h$, and we define $f=g|_{I_0}$. After evaluating at $x\in I_0$, $g|_{I_1}(F(x))=zf(x)+h(x)$ and so $g|_{I_1}=z f\circ F^{-1}+h\circ F^{-1}$. Since $\mathrm{supp}\;h\subset I_0$, by evaluating at $x\in I_1$, $g|_{I_2}(F(x))=zg|_{I_1}(x)+h(x)=zg|_{I_1}(x)$ and so $g|_{I_2}=z^2f\circ F^{-2}+zh\circ F^{-2}$. By induction, setting $p(x)=z f(x)+h(x)$ for all $x\in I_0$, we have that
    $
        g|_{I_n}=z^{n-1}(p\circ F^{-n})
    $
    for all $n\geq 1$. Now, as in the proof of \cref{iemspecinterval_lemma}, we want to show that $g\notin H^r(\mathbb{R})$. Due to the switched role of $\alpha$ and $1/\alpha$, rather than considering $\|g'\|_{L^2}$ we instead consider $\|g\|_{L^2}$. Assume that $\int_{I_0}|p'(u)|^2\dd u>0$, we estimate
    $\|g\|_{L^2}$ as follows:
    \begin{align*}
        \|g\|_{L^2} &
        \geq\sum_{n=1}^{\infty}\int_{I_n}|z|^{2(n-1)}|(p(F^{-n}(x))|^2\dd x                              =\sum_{n=1}^{\infty}|z|^{2(n-1)}\int_{I_0}|p'(u)|^2((F^{-1})'(F(u))\cdots (F^{-1})'(F^n(u)))^{-1}\dd u \\
                    & \geq|z|^{-2}\int_{I_0}|p'(u)|^2\dd u\sum_{n=1}^{\infty}|z|^{2n}\left(\prod_{j=1}^{n}\sup_{t\in I_j}|(F^{-1})'(t)|\right)^{-1}.
    \end{align*}
    By assumption, there exists $J\in\mathbb{N}$ such that for all $j\geq J$, $(F^{-1})'|_{I_j}=\alpha$ and $1/|z|<1/\sqrt{\alpha}$, so for all $j\geq J$ $(\sup_{t\in I_j}|(F^{-1})'(t)|)^{-1}>1/|z|^2$. Additionally, as $(F^{-1})'$ is bounded away from $0$, we have
    $$
        \|g\|_{L^2}\geq C\sum_{n=J}^{\infty}|z|^{2n}\prod_{j=J}^n\inf_{t\in I_j}|(F^{-1})'(t)|\geq C\sum_{n=1}^{\infty}|z|^{2n}/|z|^{2(n-J+1)}=\infty,
    $$
    where
    $$C\coloneq |z|^{-2}\int_{I_0}|p'(u)|^2\dd u\prod_{j=1}^{J-1}\inf_{t\in I_j}|(F^{-1})'(t)|>0.$$
    Therefore, $g\notin H^r(\mathbb{R})$, which yields the desired contradiction. If instead we have $\int_{I_0}|p'(u)|^2\dd u=0$ then $p|_{I_0}$ is constant and so $g$ is piecewise constant on $\bigcup_{n=1}^{\infty}I_n$ and for $|z|\neq 1$, $g$ is not weakly differentiable and hence $g\notin H^r(\mathbb{R})$ also. Then, since $\spec_{\mathrm{su}}(\koop_{F})$ is closed, we have
    $\{z\in\mathbb{C}:\;\sqrt{\alpha}\leq |z|\leq 1/\sqrt{\alpha}\}\subset \spec_{\mathrm{su}}(\koop_{F})=\spec_{\mathrm{ap}}(\koop_F^*).
    $
\end{proof}

Now, we will use the above lemma to prove our desired impossibility result.

\begin{theorem}[Approximate point spectrum lower bound on $H^r(\mathbb{R})$]
    \label{notdelta2specaprealline_thm}
    Let $r\in\mathbb{N}$; then there is no sequence of general algorithms that can compute the approximate point spectrum of any bounded Koopman operator on $H^r(\mathbb{R})$ in one limit, i.e.,
    $$
        \left\{\Xi_{\spec^*_{\mathrm{ap}}},\Omega_{H^r(\mathbb{R})}^B,\mathcal{M}_{\mathrm{H}},\Lambda_{\mathbb{R}}\right\}\notin\Delta_2^G.
    $$
\end{theorem}

\begin{proof}
    Suppose by contradiction, that $\{\Gamma_n\}$ is a $\Delta_2^G$-tower for $\{\Xi_{\spec^*_{\mathrm{ap}}},\Omega_{H^r(\mathbb{R})}^B,\mathcal{M}_{\mathrm{H}},\Lambda_{\mathbb{R}}\}$ satisfying
    $$
        \lim_{n\rightarrow\infty}\Gamma_n(F)=\spec_{\mathrm{ap}}(\koop_F^*)\quad
        \forall F\in\Omega_{H^r(\mathbb{R})}^B.
    $$
    Fix $\alpha_1,\alpha_2\in (0,1)$ and $\tau>0$ such that $(1+2\tau)\alpha_1^{-(r-1/2)}< (1-2\tau)\alpha_2^{-1/2}$.

    We first define an interval exchange map $F_1:\mathbb{R}\rightarrow\mathbb{R}$ with $I_0^{(1)}=[0,1)$, $I_1^{(1)}=[1,1/\alpha_1)$, $I_2^{(2)}=[1/\alpha_1,1/\alpha_1^2)$ and so on (by symmetry $I_{-1}^{(1)}=[-1,0)$, $I_{-2}^{(1)}=[-1/\alpha_1,-1)$, $\ldots$). We select $F_1$ to be affine on all intervals except $I_{-2}$, $I_{-1}$ and $I_{0}$, where we define it to be a smooth transition function $g^{(1)}$ such that $F_1$ is globally smooth and strictly increasing; so that $F_1$ is still an interval exchange map, we require $g^{(1)}(-1/\alpha_1)=-1$, $g^{(1)}(-1)=0$, $g^{(1)}(0)=1$ and $g^{(1)}(1)=1/\alpha$ (the first and last conditions are also imposed by continuity). In particular, we have
    $$
        F_1(x)=\begin{cases}
            \alpha_1x,           & x\in \bigcup_{n=-\infty}^{-3}I_n, \\
            g^{(1)}(x),          & x\in I_{-2}\cup I_{-1}\cup I_{0}, \\
            \frac{1}{\alpha_1}x, & x\in\bigcup_{n=1}^{\infty}I_n.
        \end{cases}
    $$
    Following \cref{iemspecrealline_lemma}, we have
    $$
        \{z\in\mathbb{C}:\alpha_1^{1/2}\leq |z|\leq \alpha_1^{-1/2)}\}\subset\spec_{\mathrm{ap}}(\koop_{F_1}^*)\subset\spec(\koop^*_{F_1})\subset\{z\in\mathbb{C}:\alpha_1^{r-1/2}\leq |z|\leq \alpha_1^{-(r-1/2)}\}.
    $$
    Then, by \cref{genalg_defn} of a general algorithm, there exists $n_1>0$, $N_1>3$ such that
    $$
        \sup\{|z|:{z\in \Gamma_{n_1}(F_1)}\}\leq (1+\tau)\alpha_1^{-(r-1/2)},\quad\Lambda_{\Gamma_{n_1}}(F_1)\subset\bigcup_{n=-N_1}^{n=N_1-1}I_n^{(1)}.
    $$
    Next, we choose a new sequence $\{c_n^{(2)}\}_{n=0}^\infty$ such that $I_n^{(2)}=I_n^{(1)}$ for $-N_1\leq n\leq N_1$, $I_{N_1+1}^{(2)}=[1/(\alpha_1^{N_1}),1/(\alpha_1^{N_1}\alpha_2))$, $I_{N_1+2}^{(2)}=[1/(\alpha_1^{N_1}\alpha_2),1/(\alpha_1^{N_1}\alpha_2^2))$ and so on, and $I_{-N_1-1}=[-1/(\alpha_1^{N_1-1}\alpha_2),-1/(\alpha_1^{N_1-1}))$ and so on.
    Then, we define the function $F_2:\mathbb{R}\rightarrow\mathbb{R}$ by
    $$
        F_2(x)=\begin{cases}
            \alpha_2x,           & x\in\bigcup_{n=-\infty}^{-N_1-2}I_n, \\
            g_1^{(2)}(x),        & x\in I_{-N_1-1},                     \\
            F_1(x),              & x\in\bigcup_{n=-N_1}^{N_1-1}I_n,     \\
            g_2^{(2)}(x),        & x\in I_{N_1},                        \\
            \frac{1}{\alpha_2}x, & x\in\bigcup_{n=N_1+1}^{\infty}I_n,
        \end{cases}
    $$
    where $\smash{g_1^{(2)}}$ and $\smash{g_2^{(2)}}$ are transition functions such that $F_2$ is globally smooth and strictly increasing. Applying \cref{iemspecrealline_lemma} again yields
    $$
        \{z\in\mathbb{C}:\alpha_2^{1/2}\leq |z|\leq \alpha_2^{-1/2)}\}\subset\spec_{\mathrm{ap}}(\koop_{F_2}^*)\subset\spec(\koop^*_{F_2})\subset\{z\in\mathbb{C}:\alpha_2^{r-1/2}\leq |z|\leq \alpha_2^{-(r-1/2)}\}
    $$
    and so by \cref{genalg_defn} of a general algorithm, there exists $n_2>n_1$, $N_2>N_1+2$ such that
    $$
        \sup\{|z|:{z\in \Gamma_{n_2}(F_2)}\}\geq (1-\tau)\alpha_2^{-1/2},\quad\Lambda_{\Gamma_{n_2}}(F_2)\subset\bigcup_{n=-N_2}^{n=N_2-1}I_n^{(2)}.
    $$
    We may proceed in the same way by induction and construct a sequence of functions $\{F_k\}_k$ (of the form given in \cref{iemspecrealline_lemma}), $n_k$ and $N_k$ such that for all $k$, $\bigcup_{n=-N_k}^{N_k-1}I_n^{(k)}=\bigcup_{n=-N_k}^{N_k-1}I_n^{(k+1)}$, $F_{k+1}|_{\bigcup_{n=-N_k}^{N_k-1}I_n^{(k)}}=F_k|_{\bigcup_{n=-N_k}^{N_k-1}I_n^{(k)}}$, and if $k$ is odd:
    $$
        \sup\{|z|:{z\in \Gamma_{n_k}(F_k)}\}\leq (1+\tau)\alpha_1^{-(r-1/2)},\quad\Lambda_{\Gamma_{k}}(F_k)\subset\bigcup_{n=-N_k}^{n=N_k-1}I_n^{(k)},
    $$
    while if $k$ is even:
    $$
        \sup\{|z|:{z\in \Gamma_{n_k}(F_k)}\}\geq (1-\tau)\alpha_2^{-1/2},\quad\Lambda_{\Gamma_{n_k}}(F_k)\subset\bigcup_{n=-N_k}^{n=N_k-1}I_n^{(k)}.
    $$
    Moreover, since the functions $g_1^{(k)}$ and $g_2^{(k)}$ are chosen to smoothly connect the functions $\alpha_1 x$ and $\alpha_2 x$ in either order or the functions $x/\alpha_1$ and $x/\alpha_2$ in either order over intervals of increasing size, we can choose them to be rescalings of one another, which ensures that the derivatives of $F_k$ are uniformly bounded in $k$ and also that the $(F_k^{-1})'$ are uniformly bounded. In particular, consider the function $g_2^{(2)}(x)$ which transitions between $x/\alpha_1$ and $x/\alpha_2$ on the interval $I_{N_1}^{(1)}=[1/(\alpha_1^{N_1-1}),1/(\alpha_1^{N_1}))$ with $g_2^{(2)}(1/(\alpha_1^{N_1-1}))=1/\alpha_1^{N_1}$ and $g_2^{(2)}(1/(\alpha_1^{N_1}))=1/(\alpha_1^{N_1}\alpha_2)$. Suppose that in a future step of the induction we need to transition again between $x/\alpha_1$ and $x/\alpha_2$ on an interval $I=[\beta,\beta/\alpha_1)$ with a transition function $g$ such that $g(\beta)=g(\beta/\alpha_1)$ and $g(\beta/\alpha_1)=g(\beta/(\alpha_1,\alpha_2)$ (where $\beta>1/\alpha_1^{N_1-1}$ is a suitable product of $1/\alpha_1$ and $1/\alpha_2$ constructed by the induction). Then we see that $g(x)=\beta\alpha_1^{N_1-1}g_2^{(2)}(x/(\beta \alpha_1^{N_1-1}))$ works, and $g'(x)=(g_2^{(2)})'(x/(\beta\alpha_1^{N_1-1}))$, so $\sup_{x\in I}g'(x)=\sup_{x\in I_{N_1}^{(1)}}(g_2^{(2)})'(x)$ and by the inverse function theorem we also have that $\sup_{x\in I}(g^{-1})'(x)=\sup_{x\in I_{N_1}^{(1)}}((g_2^{(2)})^{-1})'(x)$; additionally, $g''(x)=(1/(\beta/\alpha_1^{N_1-1}))(g_2^{(2)})''(x/(\beta\alpha_1^{N_1-1}))$, and so $\sup_{x\in I}g''(x)<\sup_{x\in I_{N_1}^{(1)}}(g_2^{(2)})''(x)$, and the same holds for all higher derivatives.

    Next, we define the function $F$ by $F(x)=\lim_{k\rightarrow\infty}F_k(x)$, where the limit exists by construction as for all $m\geq k$,
    $$\bigcup_{n=-N_k}^{N_k-1}I_n^{(m)}=\bigcup_{n=-N_k}^{N_k-1}I_n^{(k)},\quad F_m|_{\bigcup_{n=-N_k}^{N_k-1}I_n^{(m)}}=F_k|_{\bigcup_{n=-N_k}^{N_k-1}I_n^{(k)}};$$ it is also clearly a smooth, strictly increasing, interval exchange map. Moreover, all derivatives of $F$ are bounded, and $(F^{-1})'$ is also bounded. Hence, it is straightforward to see that $\koop_F$ is bounded by using the same arguments as in the start of \cref{iemspecrealline_lemma}.

    By consistency of general algorithms, $\Gamma_{n_k}(F)=\Gamma_{n_k}({F_k})$ for all $k$. Since $(1+2\tau)\alpha_1^{-(r-1/2)}<(1-2\tau)\alpha_2^{-1/2}$, it follows that $\Gamma_{n_k}(F)$ cannot converge, the required contradiction.
\end{proof}

\begin{remark}[Extensions to higher dimensions and other domains]
    The above results can be directly extended to $H^r(\mathbb{R}^d)$ for $r>d/2$ by taking $\hat{F}(x_1,\dots,x_d)=(F(x_1),x_2,\dots,x_d)$. When estimating the bounds on $\|\koop^l g\|_{H^r}^2$, we obtain a number (depending on $r$ and $d$) of additional terms that incorporate the derivatives in $x_2,\dots,x_d$. However, these terms all satisfy the same bound as the corresponding term where all derivatives fall on $x_1$, and it does not affect the result. For the surjective spectrum, the same argument works by taking $\mathrm{supp}\;h\subset I_0\times\mathbb{R}^{d-1}$. Note that this also works for more general rectangular domains, where some sides can be the whole real line and others a finite interval, as long as at least one side is the whole real line.
\end{remark}

\begin{remark}
    Similarly to \cref{notdelta1specapeps_thm}, we may use the same family of adversarial dynamical systems to directly prove that the problem of computing $\spec_{\mathrm{ap},\epsilon}(\koop^*)$ does not lie in $\Delta_1^G$ on $H^r(\mathbb{R}^d)$ or more general rectangular domains.
\end{remark}

\subsection{Random algorithms and/or data do not help}
\label{sec_extension_to random}

We now show that the above impossibility results extend to probabilistic algorithms.
To study such algorithms in the most general setting—and thereby obtain the strongest possible impossibility results—we allow general algorithms access to an infinite sequence of fair coin tosses. In the context of computing spectral properties of Koopman and Perron--Frobenius operators from snapshots, such randomization could be used, for example, to select data or in the computations itself.

Define $\mathcal{C}=\{0,1\}^{\mathbb{N}}$ equipped with the product topology where each $\{0,1\}$ has the discrete topology, and let $p_j$ denote the projection map $p_j(a)=a_j$ for $a\in \mathcal{C}$, $j\in\mathbb{N}$. Finally, define the probability measure $\mathbb{P}$ to be that of a fair coin, i.e., for $n\in\mathbb{N}$ and $a_1,\dots,a_n\in\{0,1\}$, $\mathbb{P}\left(\{b\in \mathcal{C}:b_j=a_j,j=1,\dots,n\}\right)=2^{-n}$. We may now define a probabilistic computational problem and a probabilistic general algorithm.

\begin{remark}[$\mathrm{NH}$ point]
    In the definitions below, we include a new point $\mathrm{NH}$ to be interpreted as `non-halting.' For example, consider the algorithm that, when given access to arbitrarily many tosses of a potentially biased coin, is asked whether the coin comes up heads with probability $0$. If the coin comes up heads at any point, the algorithm outputs a negative answer, but if the coin comes up tails infinitely many times, it never outputs an answer. `Non-halting' allows for the possibility that the algorithm may never output an answer.
\end{remark}

\begin{definition}[Probabilistic computational problem]
    \label{probcompprob_defn}
    The probabilistic computational problem corresponding to a computational problem $\{\Xi,\Omega,\mathcal{M},\Lambda\}$ is defined as $\{\Xi^{\mathbb{P}},\Omega^{\mathbb{P}},\mathcal{M}^{\mathbb{P}},\Lambda^{\mathbb{P}}\}$, where $\Omega^{\mathbb{P}}=\Omega\times \mathcal{C}$, $\mathcal{M}^{\mathbb{P}}=\mathcal{M}\cup\{\mathrm{NH}\}$, where $\mathrm{NH}$ is added as an isolated point to $(\mathcal{M},d)$, $\Xi^{\mathbb{P}}(\zeta,a)=\Xi(\zeta)$ for $(\zeta,a)\in\Omega^{\mathbb{P}}$, and $\Lambda^{\mathbb{P}}=\{\tilde{f}:\tilde{f}(\zeta,a)\coloneqq f(\zeta),f\in\Lambda,(\zeta,a)\in\Omega^{\mathbb{P}}\}\cup\{p_j:p_j(\zeta,a)\coloneqq a_j,j\in\mathbb{N},(\zeta,a)\in\Omega^{\mathbb{P}}\}$.
\end{definition}

\begin{definition}[Probabilistic general algorithm]
    \label{probgenalg_defn}
    Given a computational problem $\{\Xi,\Omega,\mathcal{M},\Lambda\}$, a probabilistic general algorithm is a map $\Gamma:\Omega^\mathbb{P}\to \mathcal{M}\cup\{\mathrm{NH}\}$ with the following property. For any $(\zeta,a)\in\Omega^\mathbb{P}$, there exists a non-empty subset of evaluations $\Lambda_\Gamma(\zeta,a) \subset\Lambda^\mathbb{P}$ such that:
    \begin{itemize}[leftmargin=0.7cm]
        \item If $(\zeta',b)\in\Omega^\mathbb{P}$ with $f(\zeta,a)=f(\zeta',b)$ for every $f\in\Lambda_\Gamma(\zeta,a)$, then $\Lambda_\Gamma(\zeta,a)=\Lambda_\Gamma(\zeta',b)$ and $\Gamma(\zeta,a)=\Gamma(\zeta',b)$;
        \item If $\Gamma(\zeta,a)\neq \mathrm{NH}$, then $\Lambda_\Gamma(\zeta,a)$ is finite.
    \end{itemize}
    We will refer to a `sequence of probabilistic general algorithms' as an SPGA.
\end{definition}

This definition of an SPGA includes any standard model of computation that possesses the capability of flipping coins (e.g., probabilistic Turing or BSS machines)\footnote{One could also consider other, even continuous, probability distributions. In the case of BSS machines, machines that can pick numbers uniformly at random in $[0,1]$ are no more computationally powerful \cite[Sec.~17.5]{BCSS}. Hence, we do not consider such scenarios, which are also unrealistic in practice.}, thus providing a broad class of algorithms for our impossibility results. One can show, (e.g., see \cite[Appendix]{colbrook2024limits}) that if $\Gamma$ is a probabilistic general algorithm  for a computational problem $\{\Xi,\Omega,\mathcal{M},\Lambda\}$, then for any fixed $\zeta\in\Omega$, the map $a\mapsto \Gamma(\zeta,a)$ is measurable. Hence, we may define
$$
    \mathbb{P}\left(\lim_{n\rightarrow\infty}\Gamma_n(\zeta)=\Xi(\zeta)\right)=\mathbb{P}\left(\left\{a\in \mathcal{C}:\lim_{n\rightarrow\infty}\Gamma_n(\zeta,a)=\Xi^{\mathbb{P}}(\zeta,a)\right\}\right).
$$
This allows us define the probabilistic SCI classes $\Delta_1^{\mathbb{P}}$ and $\Delta_2^{\mathbb{P}}$.

\begin{definition}[Probabilistic SCI class]
    \label{probsci_lemma}
    A computational problem $\{\Xi,\Omega,\mathcal{M},\Lambda\}$ belongs to $\Delta_1^{\mathbb{P}}$ if there exists an SPGA $\{\Gamma_n\}$ with
    $$
        \mathbb{P}\left(d(\Gamma_{n}(\zeta),\Xi(\zeta))\leq 2^{-n}\right)> 2/3\quad \forall n\in\mathbb{N}, \zeta\in\Omega.
    $$
    A computational problem $\{\Xi,\Omega,\mathcal{M},\Lambda\}$ belongs to $\Delta_2^{\mathbb{P}}$ if there exists an SPGA  $\{\Gamma_n\}$ with
    $$
        \mathbb{P}\left(\lim_{m\rightarrow\infty}\Gamma_{m}(\zeta)=\Xi(\zeta)\right)> 2/3\text{ and }
        \mathbb{P}\left(\Gamma_{n}(\zeta)\neq \mathrm{NH}\right)> 2/3\quad \forall n\in\mathbb{N}, \zeta\in\Omega.
    $$
\end{definition}

\begin{remark}
    The choice of $2/3$ in the above definition is standard in complexity theory but also occurs due to the inclusion of the additional point $\mathrm{NH}$.
\end{remark}

It is straightforward to show that if a computational problem does not lie in $\Delta_j^{\mathbb{P}}$, then it does not lie in $\Delta_j^G$, for $j\in\{1,2\}$. Remarkably, the converse is also true. The idea of the proof of the following theorem is that since a probabilistic general algorithm only ever uses finitely many evaluation functions when it does not output $\mathrm{NH}$, we can construct a finite set of evaluation functions such that, with high probability, the probabilistic general algorithm gives output close to the problem function and uses only evaluation functions from that finite set. Then, by deterministically providing coin toss inputs to the finite set of evaluation functions, we can find an element of the metric space which is close to the output of the probabilistic general algorithm with high probability; the general algorithm then outputs this element. Then since the intersection of these high probabilistic events has non-zero probability (i.e., both occur for some input sequence of coin tosses), the resulting deterministic general algorithm has the desired behaviour for a $\Delta_i$ tower, $i=1,2$. This result is of independent interest in the SCI hierarchy.

\begin{theorem}[Equivalence between $\Delta_j^G$ and $\Delta_j^{\mathbb{P}}$]
    \label{random_no_help_thm}
    Let $\{\Xi,\Omega,\mathcal{M},\Lambda\}$ be a computational problem, where $\Lambda$ is countable, and let $j=1$ or $2$. Then $\{\Xi,\Omega,\mathcal{M},\Lambda\}\in \Delta_j^{\mathbb{P}}$ if and only if $\{\Xi,\Omega,\mathcal{M},\Lambda\}\in \Delta_j^G$. In particular, if $\{\Xi,\Omega,\mathcal{M},\Lambda\}\not\in \Delta_j^G$, then $\{\Xi,\Omega,\mathcal{M},\Lambda\}\notin \Delta_j^{\mathbb{P}}$.
\end{theorem}

We will use the following lemma~\cite{colbrook2024limits}, which provides a convenient form of the consistency properties of \cref{probgenalg_defn}, as part of the proof of \cref{random_no_help_thm}.

\begin{lemma}[Consistency lemma, {\cite{colbrook2024limits}}]\label{probconsistency_lemma}
    Let $\Gamma$ be a probabilistic general algorithm for a computational problem $\{\Xi,\Omega,\mathcal{M},\Lambda\}$, $\zeta\in\Omega$, and $S\subset\Lambda^{\mathbb{P}}$. Then, for any $\zeta'$ such that $f(\zeta)=f(\zeta')$ for any $f\in S\cup\Lambda$,
    $$
        \{a\in \mathcal{C}:\Lambda_{\Gamma}(\zeta,a)\subset S\}=\{a\in \mathcal{C}:\Lambda_{\Gamma}(\zeta',a)\subset S\},
    $$
    and for any $a$ in this set, $\Gamma(\zeta,a)=\Gamma(\zeta',a)$.
\end{lemma}

\begin{proof}[Proof of \cref{random_no_help_thm}]
    \textbf{Equivalence of $\Delta_1^{\mathbb{P}}$ and $\Delta_1^{G}$ when $\Lambda$ is countable:} It is clear that if a problem lies in $\Delta_1^{G}$, then it lies in $\Delta_1^{\mathbb{P}}$. For the converse, suppose that $\{\Xi,\Omega,\mathcal{M},\Lambda\}\in\Delta_1^{\mathbb{P}}$. Let $n\in\mathbb{N}$, then there exists a probabilistic general algorithm $\Gamma'$ such that
    \begin{equation}
        \label{delta_1P_assumption}
        \mathbb{P}\left(\left\{a\in\mathcal{C}:d(\Gamma'(\zeta,a),\Xi(\zeta))\leq 2^{-(n+1)}\right\}\right)> 2/3\quad \forall \zeta \in\Omega.
    \end{equation}
Since the set of finite subsets of $\Lambda^\mathbb{P}$ is countable, there exists an increasing sequence of finite sets $S_{1}\subset S_{2}\subset S_{3}\subset \cdots\subset \Lambda^\mathbb{P}$ such that if $S\subset\Lambda^\mathbb{P}$ is finite, then $S\subset S_M$ for sufficiently large $M$. Let $\zeta\in\Omega$ be an input, for any finite set $S\subset\Lambda^\mathbb{P}$, define the set
    $$
        U_S=\left\{a\in\mathcal{C}:\Lambda_{\Gamma'}(\zeta,a)\subset S\right\}.
    $$
    We claim that $U_S$ is open and hence measurable. To see this, suppose that $a\in U_S$. Since $\Lambda_{\Gamma'}(\zeta, a)\subset S$ is finite, \cref{probgenalg_defn} implies that there exists $N\in\mathbb{N}$ such that if $b\in\mathcal{C}$ with $p_n(b) = p_n(a)$ for all $n\leq N$, then $\Lambda_{\Gamma'}(\zeta, b) = \Lambda_{\Gamma'}(\zeta, a)$ and $\Gamma'(\zeta,b) = \Gamma'(\zeta,a)$. In particular, $b\in U_S$. Hence, an open neighborhood of $a$ exists inside $U_S$, and the claim follows.

    If $a\in\mathcal{C}$ has $\Gamma'(\zeta,a)\neq\mathrm{NH}$, then $\Lambda_{\Gamma'}(\zeta,a)$ is finite. It follows that
    $$
        \left\{a\in\mathcal{C}:\Gamma'(\zeta,a)\neq\mathrm{NH}\right\}\subset\bigcup_{S\subset\Lambda^\mathbb{P},S\text{ finite}}U_S=:U,
    $$
    where $U$ is open and hence measurable. Then, by \cref{delta_1P_assumption} and that $\mathrm{NH}$ is an isolated point, it follows that $\mathbb{P}(U)>2/3$. Moreover, the partition of $\Lambda^\mathbb{P}$ defined above implies that $U=\bigcup_{m=1}^\infty U_{S_{m}}$. Since the sets $\{U_{S_m}\}_m$ are increasing and $\mathbb{P}(U)>2/3$, we may choose $m$ minimal such that $\mathbb{P}(U_{S_{m}})>2/3$. We claim that with the choice $S(\zeta)=S_m$, the map $\zeta\mapsto S(\zeta)$ is a general algorithm with $\Lambda_{S}(\zeta)=\Lambda\cap S(\zeta)$. To see this, suppose that $\zeta'\in\Omega$ with $f(\zeta)=f(\zeta')$ for all $f\in \Lambda\cap S(\zeta)$. \cref{probconsistency_lemma} implies that
    $$
        \mathbb{P}\left(\left\{a\in\mathcal{C}:\Lambda_{\Gamma'}(\zeta',a)\subset S(\zeta)\right\}\right)=\mathbb{P}\left(\left\{a\in\mathcal{C}:\Lambda_{\Gamma'}(\zeta,a)\subset S(\zeta)\right\}\right)> 2/3.
    $$
    Since we chose the set $S_m$ with $m$ minimal in the above construction, it follows that $S(\zeta')\subset S(\zeta)$. Since $f(\zeta)=f(\zeta')$ for all $f\in \Lambda\cap S(\zeta')\subset \Lambda\cap S(\zeta)$, we apply \cref{probconsistency_lemma} again (with the roles of $\zeta$ and $\zeta'$ reversed) to see that
    $$
        \mathbb{P}\left(\left\{a\in\mathcal{C}:\Lambda_{\Gamma'}(\zeta,a)\subset S(\zeta')\right\}\right)=\mathbb{P}\left(\left\{a\in\mathcal{C}:\Lambda_{\Gamma'}(\zeta',a)\subset S(\zeta')\right\}\right)> 2/3.
    $$
    Again, using the minimality of $m$, we see that $ S(\zeta)\subset S(\zeta')$. It follows that $S(\zeta)=S(\zeta')$ and $\Lambda_{S}(\zeta)=\Lambda_{S}(\zeta')$. Hence, $S$ is a general algorithm as claimed.

    Given $\zeta\in\Omega$, let
    $$
        E_1(\zeta)=\left\{a\in\mathcal{C}:d(\Gamma'(\zeta,a),\Xi(\zeta))\leq 2^{-(n+1)}\right\}.
    $$
    Since $\mathbb{P}(E_1(\zeta))>2/3$ and $\mathbb{P}(U_{S(\zeta)})>2/3$, it follows immediately that $\mathbb{P}(E_1(\zeta)\cap U_{S(\zeta)})>1/3$. Let $\Gamma(\zeta)$ be any member of $\mathcal{M}$ that satisfies
    $$
        \mathbb{P}\left(\left\{a\in\mathcal{C}:\Lambda_{\Gamma'}(\zeta,a)\subset S(\zeta),d(\Gamma(\zeta),\Gamma'(\zeta,a))\leq 2^{-(n+1)}\right\}\right)> 1/3.
    $$
    Such a choice $\Gamma(\zeta)$ must always exist (since the above is satisfied by $\Xi(\zeta)$ since $\mathbb{P}(E_1(\zeta)\cap U_{S(\zeta)})>1/3$). Since $\Gamma(\zeta)$ only depends on finitely many $f(\zeta)$ such that $f\in\Lambda\cap S(\zeta)$, and $S(\zeta)$ is a general algorithm, we may choose $\Gamma(\zeta)$ consistently as $\zeta$ varies so that it is a general algorithm. Note that
    $$
        \mathbb{P}\left(E_1(\zeta)\cap\left\{a\in\mathcal{C}:\Lambda_{\Gamma'}(\zeta,a)\subset S(\zeta),d(\Gamma(\zeta),\Gamma'(\zeta,a))\leq 2^{-(n+1)}\right\}\right)>2/3+1/3-1=0.
    $$
    Hence, there exists $a$ in this intersection so that
    $$
        d(\Gamma(\zeta),\Xi(\zeta))\leq d(\Gamma(\zeta),\Gamma'(\zeta,a))+d(\Gamma'(\zeta,a),\Xi(\zeta))\leq 2^{-n}.
    $$
    Since $n\in\mathbb{N}$ was arbitrary, it follows that $\{\Xi,\Omega,\mathcal{M},\Lambda\}\in\Delta_1^{G}$.

    \vspace{2mm}

    \noindent
    \textbf{Equivalence of $\Delta_2^{\mathbb{P}}$ and $\Delta_2^{G}$ when $\Lambda$ is countable:} It is clear that if a problem lies in $\Delta_2^{G}$, then it lies in $\Delta_2^{\mathbb{P}}$. For the converse, suppose that $\{\Xi,\Omega,\mathcal{M},\Lambda\}\in\Delta_2^{\mathbb{P}}$. Then there exists an SPGA $\{\Gamma_n'\}$ such that
    $$
        \mathbb{P}\left(\left\{a\in\mathcal{C}:\lim_{m\rightarrow\infty}\Gamma_{m}'(\zeta)=\Xi(\zeta)\right\}\right)> 2/3\text{ and }
        \mathbb{P}\left(\left\{a\in\mathcal{C}:\Gamma_{n}'(\zeta,a)\neq \mathrm{NH}\right\}\right)> 2/3\quad \forall n\in\mathbb{N}, \zeta\in\Omega.
    $$
    Let $N\in\mathbb{N}$. Arguing as above, there exist general algorithms (now dependent on $n$) $\{S_n\}$ such that each $S_n(\zeta)$ is a finite subset of $\Lambda^{\mathbb{P}}$ that satisfy
    $$
        \mathbb{P}\left(\left\{a\in\mathcal{C}:\Lambda_{\Gamma_n'}(\zeta,a)\subset S_n(\zeta)\right\}\right)>2/3.
    $$
    We now consider the set
    $$
        R_n(\zeta)=\left\{x\in\mathcal{M}:\mathbb{P}\left(\left\{a\in\mathcal{C}:\Lambda_{\Gamma_n'}(\zeta,a)\subset S_n(\zeta),d(\Gamma_n'(\zeta,a),x)\leq 2^{-N}\right\}\right)>1/3\right\}.
    $$
    Let $x_0$ be a fixed member of $\mathcal{M}$. If $R_n(\zeta)=\emptyset$, we set $\Gamma_n^N(\zeta)=x_0$. Otherwise, let $\Gamma_n^N(\zeta)=x$ for some element $x$ of $R_n(\zeta)$. This choice can be executed so that $\Gamma_n^N$ is a general algorithm.

    We have
    $$
        \left\{a\in\mathcal{C}:\lim_{m\rightarrow\infty}\Gamma_{m}'(\zeta,a)=\Xi(\zeta)\right\}\subset
        \bigcup_{M=1}^\infty \left\{a\in\mathcal{C}:d(\Gamma_{m}'(\zeta,a),\Xi(\zeta))\leq 2^{-N}\text{ for all }m\geq M\right\}.
    $$
    It follows that there exists some $M$ so that the set
    $$
        E_M(\zeta)=\left\{a\in\mathcal{C}:d(\Gamma_{m}'(\zeta,a),\Xi(\zeta))\leq 2^{-N}\text{ for all }m\geq M\right\}
    $$
    has $\mathbb{P}(E_M(\zeta))>2/3$. If $n\geq M$, then $\mathbb{P}(E_M(\zeta)\cap U_{S_n(\zeta)})>1/3$. In particular, $\Xi(\zeta)\in R_n(\zeta)$ so that $R_n(\zeta)\neq\emptyset$. Moreover,
    $$
        \mathbb{P}\left(E_M(\zeta)\cap\left\{a\in\mathcal{C}:\Lambda_{\Gamma_n'}(\zeta,a)\subset S_n(\zeta),d(\Gamma_n'(\zeta,a),\Gamma_n^N(\zeta))\leq 2^{-N}\right\}\right)>2/3+1/3-1=0.
    $$
    If $a$ lies in this intersection, then
    $$
        d(\Gamma_n^N(\zeta),\Xi(\zeta))\leq d(\Gamma_n^N(\zeta),\Gamma_n'(\zeta,a))+d(\Gamma_n'(\zeta,a),\Xi(\zeta))\leq 2^{1-N}.
    $$
    It follows that
    \begin{equation}
        \label{prob_osc_bound}
        \limsup_{n\rightarrow\infty}d(\Gamma_n^N(\zeta),\Xi(\zeta))\leq 2^{1-N}.
    \end{equation}
    We must now alter $\{\Gamma_n^N\}$ to obtain a $\Delta_2^G$-tower.

    Given $m$, choose $N=N(m)\leq m$ maximal such that for all $1\leq l\leq N$,
    $$
        d(\Gamma_m^N(\zeta),\Gamma_m^l(\zeta))\leq 4(2^{-N}+2^{-l}).
    $$
    If no such $N$ exists then set $\Gamma_{m}(\zeta)=x_0$, otherwise set $\Gamma_{m}(\zeta)=
        \Gamma_m^{N(m)}(\zeta)$. The bound in \cref{prob_osc_bound} implies that $\lim_{m\rightarrow\infty}N(m)=\infty$. Furthermore, for any fixed $q\in\mathbb{N}$, since for sufficiently large $m$ $N(m)\geq q$, it follows that
    \begin{align*}
        \limsup_{m\rightarrow\infty}d(\Gamma_{m}(\zeta),\Xi(\zeta)) & \leq
        \limsup_{m\rightarrow\infty}d(\Gamma_m^{N(m)}(\zeta),\Gamma_m^{q}(\zeta))
        +\limsup_{m\rightarrow\infty}d(\Gamma_m^{q}(\zeta),\Xi(\zeta))                                                                              \\
                                                            & \leq \limsup_{m\rightarrow\infty}4(2^{-N(m)}+2^{-q}) + 2^{1-q}=6\cdot 2^{-q}.
    \end{align*}
    Since $q$ was arbitrary, we have the desired convergence. It follows that $\{\Xi,\Omega,\mathcal{M},\Lambda\}\in\Delta_2^{G}$.
\end{proof}

\section{Spectral measures of unitary and self-adjoint operators on RKHSs}
\label{unisa_sect}

Spectral computations often simplify when the underlying operators are normal; typically, this reduces the SCI of the method by $1$ by converting towers of algorithms for the pseudospectrum into towers using the same number of limits for the spectrum \cite{colbrook4,colbrook3}. In particular, if $T$ is a normal operator, its spectrum is stable since
$$
    \spec_\epsilon(T)=\spec(T)+B_{\epsilon}(0)=\left\{z\in\mathbb{C}:\mathrm{dist}(z,\spec(T))\leq \epsilon\right\}.
$$
Moreover, normal operators are guaranteed to be diagonalizable via projection-valued spectral measures and the spectral theorem (see \cref{sec:specmeasintro}), even if they lack a basis of eigenvectors. Hence, knowing when an operator is normal is important. In the RKHS case, we can sometimes choose a specific kernel to make an operator normal. In this section, we provide a combination of necessary and sufficient conditions for Koopman operators defined on RKHSs to be unitary and self-adjoint and provide examples of such systems. We design efficient, data-driven algorithms to compute the corresponding scalar-valued spectral measures for unitary and self-adjoint operators, which recover the full projection-valued spectral measures.

\subsection{Spectral measures}
\label{sec:specmeasintro}

In finite dimensions, any normal matrix $T\in\mathbb{C}^{n\times n}$ has an orthonormal basis of eigenvectors $\{v_j\}_{j=1}^n$ with eigenvalues $\{\lambda_j\}_{j=1}^n$ such that $Tv_j=\lambda_jv_j$ for $j=1,\dots,n$ and
$$
    v=\left(\sum_{j=1}^nv_jv_j^*\right)v,\qquad \text{and}\qquad Tv=\left(\sum_{j=1}^n\lambda_j v_jv_j^*\right)v\qquad \forall v\in\mathbb{C}^n.
$$
However, in infinite dimensions, a normal operator on a Hilbert space $\mathcal{H}$ may not have a basis of eigenvectors. Indeed, it may not even have any eigenvectors, e.g., $\partial_x^2$ on $L^2(\mathbb{R})$. Nevertheless, projection-valued measures provide an alternative method of diagonalizing an operator.

If $X$ is a set and $\mathcal{F}$ a $\sigma$-algebra of subsets of $X$, then a projection-valued measure $\mathcal{E}$ for $(X,\mathcal{F},\mathcal{H})$ is a map from $\mathcal{F}$ to the set of orthogonal projections on $\mathcal{H}$ such that:
\begin{itemize}
    \item $\mathcal{E}(X)=I$ and $\mathcal{E}(\emptyset)=0$;
    \item $\mathcal{E}(E_1\cap E_2)=\mathcal{E}(E_1)\mathcal{E}(E_2)$ for all $E_1,E_2\in\mathcal{F}$;
    \item if $\{E_n\}_{n=1}^{\infty}\subset\mathcal{F}$ are pairwise disjoint, then $\mathcal{E}(\cup_{n=1}^{\infty}E_n)=\sum_{n=1}^{\infty}\mathcal{E}(E_n)$.
\end{itemize}
The spectral theorem for normal operators~\cite[Thm.~X.4.11]{conway_course_2007} states that if $T$ is a normal operator on $\mathcal{H}$, then there is a unique projection-valued spectral measure $\mathcal{E}$ on the Borel subsets of $\mathbb{C}$ such that if $S\cap \spec(T)=\emptyset$ then $\mathcal{E}(S)=0$, if $U$ is open with $U\cap\spec(T)\neq \emptyset$ then $\mathcal{E}(U)\neq 0$ and
$$
    v_1=\left(\int_{\spec(T)}1\dd\mathcal{E}(\lambda)\right)v_1\quad\forall v_1\in\mathcal{H}\qquad \text{and}\qquad Tv_2=\left(\int_{\spec(T)}\lambda\dd\mathcal{E}(\lambda)\right)\quad \forall v_2\in\mathcal{D}(T).
$$
These relations directly generalize the finite-dimensional case.

It is sufficient to study the corresponding scalar-valued spectral measures defined for $v,w\in\mathcal{H}$ by
$$
    \mu_{v,w}(S)=\langle\mathcal{E}(S)v,w\rangle_{\mathcal{H}}=\langle\mathcal{E}(S)v,\mathcal{E}(S)w\rangle_{\mathcal{H}},
$$
where $S$ a Borel subset of $\mathbb{C}$. If $v=w$, we denote $\mu_{v,v}=\mu_v$. The support of the spectral measure $\mu_v$ is a subset of $\spec(T)$ and $\mu_v(\mathbb{C})=\|v\|^2$. If $T$ is self-adjoint, its spectral measure is supported on $\mathbb{R}$; if it is unitary, its spectral measure is supported on the complex unit circle $\mathbb{T}$. For visualization purposes, we reparametrize spectral measures on $\mathbb{T}$ by $\xi_{v,w}(S)=\mu_{v,w}(\{e^{i\theta}:\theta\in S\})$ for $S\subset [-\pi,\pi]_{\mathrm{per}}$.

Spectral measures lead to generalizations of the KMD presented in \cref{specprop_sect}~\cite{mezic2005spectral}. For example, suppose that $\koop_F$ is a unitary operator and its spectrum consists only of point spectra (i.e., eigenvalues, denoted $\spec_{\mathrm{p}}(\koop_F)$) and absolutely continuous spectrum. Then, given a vector-valued observable $\m{g}$, a starting point $x_0$, and number of time steps $n$, we have~\cite{colbrook2025rigged}:
\begin{align*}
    \m{g}(x_n)=\m{g}(F^n(x_0)) & =\koop_F^n\left(\sum_{\lambda\in\spec_{\mathrm{p}}(\koop_F)}\varphi_{\lambda}(x_0)\m{v}_{\lambda}+\int_{[-\pi,\pi]_{\mathrm{per}}}\mathbb{\varphi}_{\theta,\m{g}}(x_0)\dd\theta\right) \\&=\sum_{\lambda\in\spec_{\mathrm{p}}(\koop)}\lambda^n\varphi_{\lambda}(x_0)\m{v}_{\lambda}+\int_{[-\pi,\pi]_{\mathrm{per}}}e^{in\theta}\mathbb{\varphi}_{\theta,\m{g}}(x_0)\dd\theta.
\end{align*}
Here, $\koop_F\varphi_{\lambda}=\lambda\varphi_{\lambda}$ for each $\lambda\in\spec_{\mathrm{p}}(\koop_F)$, $\m{v}_{\lambda}$ are the corresponding Koopman modes and $\varphi_{\theta,\m{g}}$ is the vector-valued Radon--Nikodym derivative of $\mathrm{d}\mathcal{E}(\theta)\m{g}$; a continuous family of generalized eigenfunctions.
Later in \cref{compspecmeas_sec}, we will discuss practical algorithms for computing spectral measures of unitary and self-adjoint operators. However, in the next section, we first examine conditions under which Koopman operators on RKHSs are unitary and self-adjoint.

\subsection{Characterizing unitary and self-adjoint Koopman operators on RKHSs}

In the $L^2$ setting, the Koopman operator is unitary if and only if the underlying dynamics are measure-preserving and invertible (up to measure zero). The following theorem characterizes when Koopman operators on RKHSs are unitary. Note that we require pointwise invertibility here since functions in an RKHS are defined pointwise. Moreover, $\koop_F$ is unitary, self-adjoint, or normal if and only if $\koop^*_F$ is unitary, self-adjoint, or normal, respectively.

\begin{theorem}[Unitary Koopman operator on an RKHS]
    \label{unisuff_thm}
    Let $\koop_F$ be the Koopman operator associated with the dynamics $(F,\mathcal{X})$ on an RKHS $\mathcal{H}$ with kernel $\mathfrak{K}$.
    \begin{enumerate}[nosep]
        \item[(i)] If $F$ is a bijection and $\mathfrak{K}(F(x),F(y))=\mathfrak{K}(x,y)$ for all $x,y\in\mathcal{X}$, then $\koop_F$ is unitary.
        \item[(ii)] If $\koop_F$ is unitary, then $\mathfrak{K}(F(x),F(y))=\mathfrak{K}(x,y)$ for all $x,y\in\mathcal{X}$.
    \end{enumerate}
\end{theorem}

\begin{proof}
    We first prove (i) so assume that $F$ is a bijection and $\mathfrak{K}(F(x),F(y))=\mathfrak{K}(x,y)$ for all $x,y\in\mathcal{X}$. For any such $x,y$, we observe that
    $(\mathfrak{K}_x\circ F)(y)=\mathfrak{K}_x(F(y))=\mathfrak{K}(x,F(y))=\mathfrak{K}(F^{-1}(x),y)=\mathfrak{K}_{F^{-1}(x)}(y)$. Therefore, for all $x\in\mathcal{X}$, $\mathfrak{K}_x\in\mathcal{D}(\koop_F)$ and $\koop_F \mathfrak{K}_x=\mathfrak{K}_{F^{-1}(x)}$. As the span of the kernel functions is dense in the RKHS, $\koop$ is densely defined. Additionally, for any $x\in \mathcal{X}$, $\|\koop_F \mathfrak{K}_x\|^2_{\mathfrak{K}}=\mathfrak{K}(F^{-1}(x),F^{-1}(x))=\mathfrak{K}(x,x)=\|\mathfrak{K}_x\|^2_{\mathfrak{K}}$ by assumption. Therefore, $\koop_F$ is bounded on $\spann\{\mathfrak{K}_x:x\in\mathcal{X}\}$ and, since this is dense, by closedness $\koop_F$ is a bounded operator on all of $\mathcal{H}$. For any $x\in\mathcal{X}$, $\koop_F\koop_F^* \mathfrak{K}_x=\koop_F \mathfrak{K}_{F(x)}=\mathfrak{K}_{F^{-1}(F(x))}=\mathfrak{K}_x$ by \cref{koopstarkernelfunction_eqn}, and similarly $\koop_F^*\koop_F \mathfrak{K}_x=\koop_F^*\mathfrak{K}_{F^{-1}(x)}=\mathfrak{K}_{F(F^{-1}(x))}=\mathfrak{K}_x$. Therefore, for all $g\in\spann\{\mathfrak{K}_x:x\in\mathcal{X}\}$ we have $\koop_F\koop_F^*g=\koop_F^*\koop_F g=g$. Hence, by density, we find that $\koop_F\koop_F^*=\koop_F^*\koop_F=I$, which implies that $\koop_F$ is unitary.

    We now consider the partial converse in (ii) and assume that $\koop_F$ is unitary. For all $x,y\in\mathcal{X}$, $\langle \koop_F^*\mathfrak{K}_x,\koop_F^*\mathfrak{K}_y\rangle_{\mathfrak{K}}=\langle \mathfrak{K}_x,\mathfrak{K}_y\rangle_{\mathfrak{K}}$ and so $\mathfrak{K}(F(x),F(y))=\mathfrak{K}(x,y)$.
\end{proof}

For many dynamical systems of interest, \cref{unisuff_thm} significantly constrains the set of kernels that can lead to unitary Koopman operators. We have the following corollary, which we will use later.

\begin{corollary}
    \label{uniiffisom_corol}
    Consider the Koopman operator $\koop_F$ associated to the dynamics $(F,\mathcal{X})$ on an RKHS $\mathcal{H}$ with kernel $\mathfrak{K}$. If $F$ is bijective and $\mathfrak{K}(x,y)=\mathfrak{K}_0(\|x-y\|)$ is radial with $\mathfrak{K}_0$ injective, then $\koop_F$ is unitary if and only if $F$ is an isometry.
\end{corollary}

\begin{proof}
    By assumption, $\mathfrak{K}(F(x),F(y))=\mathfrak{K}(x,y)$ for all $x,y\in\mathcal{X}$ if and only if $\|F(x)-F(y)\|=\|x-y\|$ for all $x,y\in\mathcal{X}$, if and only if $F$ is an isometry. The result follows by \cref{unisuff_thm}.
\end{proof}

\begin{example}
    \label{gaussrbf_exmpl}
    Consider the Gaussian RBF on $\mathbb{R}^n$. Then, \cref{uniiffisom_corol} states that Koopman operators on the native space of the Gaussian RBF are unitary only when the dynamics are a composition of rotations, reflections, and translations of $\mathbb{R}^n$. Note that isometries are always injective, and are always surjective on compact sets.
\end{example}

We can also obtain similar conditions for $\koop_F$ to be self-adjoint.

\begin{theorem}[Self-adjoint Koopman operator on an RKHS]
    \label{sakoop_thm}
    Let $\koop_F$ be a bounded Koopman operator associated with the dynamics $(F,\mathcal{X})$ on an RKHS $\mathcal{H}$ with kernel $\mathfrak{K}$. Then, $\koop_F$ is self-adjoint if and only if $\mathfrak{K}(x,F(y))=\mathfrak{K}(F(x),y)$ for all $x,y\in\mathcal{X}$.
\end{theorem}

\begin{proof}
    This follows as for all $x,y\in\mathcal{X}$, $\mathcal{K}_F\mathfrak{K}_x(y)=\mathfrak{K}_x(F(y))=\mathfrak{K}(x,F(y))$ while $\mathcal{K}^*_F\mathfrak{K}_x(y)=\mathfrak{K}_{F(x)}(y)=\mathfrak{K}(F(x),y)$, and the kernel functions are dense in $\mathcal{H}$.
\end{proof}

\begin{remark}
    When $\koop_F$ is densely defined rather than just bounded, \cref{sakoop_thm} only shows that $\koop_F$ is symmetric. To get self-adjointness, we need to show that $\mathcal{D}(\koop^*_F)\subset \mathcal{D}(\koop_F)$ (as $\koop_F$ is symmetric, $\mathcal{D}(\koop_F)\subset \mathcal{D}(\koop^*_F)$ already).
\end{remark}

\begin{example}
    \label{linearker_example}
    Consider the RKHS with linear kernel $\mathfrak{K}(x,y)=xy$ and functions defined on $\mathbb{R}$. Then, the only continuous functions $F$ leading to self-adjoint operators are of the form $F(x)=cx$ for $c\in\mathbb{R}$. Indeed,
    \cref{sakoop_thm} yields the functional equation $F(x)y=xF(y)$. Substituting into this, we have $F(x+y)y=(x+y)F(y)$, and $F(x+y)x=(x+y)F(x)$. Combining these, we obtain $(x+y)(F(x+y)-F(x)-F(y))=0$, and so for $x+y\neq 0$, $F(x+y)=F(x)+F(y)$. We also have $F(x)(-x)=xF(-x)$, so $F(x)=-F(-x)$ for $x\neq 0$ and $F(0)(1)=(0)F(1)$ so $F(0)=0$. Hence, $F(x+y)=F(x)+F(y)$ for all $x,y\in\mathbb{R}$. This is Cauchy's functional equation, which has a unique solution $F(x)=F(1)x$. Suppose instead that we consider the polynomial kernel $\mathfrak{K}(x,y)=(\alpha xy+1)^d$ for $\alpha\in\mathbb{R}$ and $d\in\mathbb{N}$ odd. Then, as $x\mapsto x^d$ is bijective, we obtain the same functional equation $F(x)y=xF(y)$ and the same result.
\end{example}

\subsection{Computing spectral measures}
\label{compspecmeas_sec}

To compute scalar-valued spectral measures for self-adjoint or unitary Koopman operators, we use the resolvent-based framework developed in~\cite{colbrook2021rigorousKoop,colbrook_computing_2021,colbrook2019computing,colbrook2025rigged}, whose key theoretical tool is Stone's formula~\cite[Thm.~VII.13]{reed_methods_1980}.

\subsubsection{Self-adjoint operators}

In \cite{colbrook_computing_2021}, the authors developed high-order methods for computing spectral measures of general self-adjoint operators. We extend these to self-adjoint Koopman operators on RKHSs. Stone's formula ensures that for a self-adjoint operator $T$ on a Hilbert space $\mathcal{H}$, and any $a,b\in\mathbb{R}$ with $a<b$ and $v\in\mathcal{H}$, we have
\begin{equation}
    \label{selfadjoint_stones_eqn}
    \lim_{\epsilon\downarrow 0}\left(\frac{1}{2\pi i}\int_a^b(T-(x+i\epsilon)I)^{-1}-(T-(x-i\epsilon)I)^{-1}\dd x\right)v=\frac{1}{2}\left[\mathcal{E}((a,b))+\mathcal{E}([a,b])\right]v.
\end{equation}
By the functional calculus,
$$
    (T-(x+i\epsilon)I)^{-1}=\int_{\mathbb{R}}\frac{\dd\mathcal{E}(y)}{y-(x+i\epsilon)}.
$$
This allows us to rewrite the integrand in \cref{selfadjoint_stones_eqn} as a convolution of $\mathcal{E}$ with a Poisson's kernel for the upper half-plane $P_{\epsilon}^{\mathbb{H}}(x)=\pi^{-1}\epsilon/(\epsilon^2+x^2)$. It follows that
\begin{equation}
    \label{spectralmeas_computation_eqn}
    \lim_{\epsilon\downarrow 0}\int_a^b\frac{1}{\pi}\mathrm{Im}\left(\langle (T-(x+i\epsilon)I)^{-1}v,v\rangle_{\mathcal{H}}\right)\dd x=\lim_{\epsilon\downarrow0}\int_a^b[P_\epsilon^{\mathbb{H}} *\mu_v](x)\dd x=\frac{1}{2}[\mu_v((a,b))+\mu_v([a,b])],
\end{equation}
where $\mathrm{Re}(z)$ and $\mathrm{Im}(z)$ denote the real and imaginary part of a complex number $z$, respectively. The left-hand side of \cref{spectralmeas_computation_eqn} can be computed, allowing us to compute spectral measures.

However, turning this equation into an algorithm leads to a slow convergence rate when $\epsilon\rightarrow 0$. In particular, we cannot achieve a convergence faster than $O(\epsilon\log(\epsilon^{-1}))$. We use convolution with higher-order kernels rather than the Poisson kernel to speed up the convergence. We consider $m$th order kernels of the form
$$
    K(x)=\frac{1}{2\pi i}\sum_{j=1}^{m}\frac{\alpha_j}{x-a_j}-\frac{1}{2\pi i}\sum_{j=1}^{m}\frac{\overline{\alpha_j}}{x-\overline{a_j}},\quad x\in\mathbb{R},
$$
where $\{a_i\}_{i=1}^m$ are poles in the upper-half plane; typical choices are $a_j=\frac{2j}{m+1}-1+i$ for $1\leq j\leq m$. We enforce the following conditions on an $m$th order kernel: $\int_{\mathbb{R}}K(x)\dd x=1$, $K(x)x^j$ is integrable with $\int_{\mathbb{R}}K(x)x^j\dd x=0$ for $j=1,\dots,m-1$ and $K(x)=\mathcal{O}(|x|^{-(m+1)})$. This yields the following Vandermonde system for the residues:
\begin{equation}
    \label{vandermonde}
    \begin{pmatrix}
        1         & \dots  & 1           \\
        a_1       & \dots  & a_{m}       \\
        \vdots    & \ddots & \vdots      \\
        a_1^{m-1} & \dots  & a_{m}^{m-1}
    \end{pmatrix}\begin{pmatrix}
        \alpha_1 \\\alpha_2\\\vdots\\\alpha_{m}
    \end{pmatrix}=\begin{pmatrix}
        1 \\0\\\vdots\\0
    \end{pmatrix},
\end{equation}
which always has a unique solution. A higher-order Stone's formula can be obtained for these higher-order kernels. The full procedure is summarized in \cref{saspecmeas_alg}, which we call SpecRKHS-SAdjMeasure. We have written the algorithm using kEDMD (\cref{kedmd_alg}). Recall from \cref{avoidquadapprox_sect} that $\hat{\Kv}^\top$ is a Galerkin approximation of $\mathcal{K}^*$ on $\mathcal{H}$ with respect to the orthonormal basis $\{u_j\}_{j=1}^r$ defined in \cref{new_basis}.

\begin{algorithm}[t]
    \caption{SpecRKHS-SAdjMeasure: Computing the scalar-valued spectral measures of self-adjoint Perron--Frobenius operators $\koop_F^*$ at points $\{x_k\}_{k=1}^P\subset\mathbb{R}$ with respect to $g=\sum_{j=1}^rg_ju_j\in\mathcal{H}$.}\label{saspecmeas_alg}
    \textbf{Input:} Snapshot data $\{(x^{(i)},y^{(i)}=F(x^{(i)}))\}_{i=1}^N$, kernel function $\mathfrak{K}:\mathcal{X}\times\mathcal{X}\rightarrow\mathbb{C}$, sample points $\{x_k\}_{k=1}^P\subset\mathbb{R}$, poles $\{a_j\}_{j=1}^m\subset\{z\in\mathbb{C}:\mathrm{Im}(z)>0\}$, $\epsilon>0$, rank $r\leq N$, and $\m{g}\in\mathbb{C}^r$.

    \begin{algorithmic}[1]
        \STATE{Use kEDMD (\cref{kedmd_alg}) to compute the matrix $\hat{\Kv}$.}
        \STATE{Compute the eigenvectors $V$ and eigenvalues $\Lambda$ of $\hat{\Kv}^\top$, and define $\m{h}=V^{-1}\m{g}$, $\m{h}'=V^*\m{g}$.}
        \STATE{Solve the Vandermonde system \eqref{vandermonde} for the residues $\alpha_1,\dots,\alpha_m\in\mathbb{C}$.}
        \FOR{$k=1,\dots,P$}
        \STATE{Compute $\m{w}_j^{\epsilon}=(\Lambda-(x_k-\epsilon a_j)I)^{-1}\m{h}$ for $1\leq j\leq m$.}
        \STATE{Set $[K_{\epsilon}^{\mathbb{T}}*\mu_g](x_k)=-\frac{1}{\pi}\mathrm{Im}\left(\sum_{j=1}^m\alpha_j{\m{h}'}^*\m{w}_j^{\epsilon}\right)$.}
        \ENDFOR
    \end{algorithmic}
    \textbf{Output:} The approximate spectral measure $[K_{\epsilon}^{\mathbb{T}}*\mu_g](x_k)$ for $k=1,\dots,P$.
\end{algorithm}

\subsubsection{Unitary operators}

In \cite{colbrook2025rigged} (see also the earlier paper \cite{colbrook2021rigorousKoop}), the authors extended the methods of \cite{colbrook_computing_2021} to unitary operators. This paper was written in the context of Koopman operators on $L^2$ spaces, but the techniques apply immediately to general unitary operators. We extend them to unitary Koopman operators on RKHSs. For a unitary operator $T$, we define the Carath\'eodory function as
$$
    F_{\mathcal{E}}(z)=\int_{[-\pi,\pi]_{\mathrm{per}}}\frac{e^{i\varphi}+z}{e^{i\varphi}-z}\dd\mathcal{E}(\varphi)=(T+zI)^{-1}(T-zI),\quad |z|\neq 1.
$$
Then, for any $a,b\in[-\pi,\pi]_{\mathrm{per}}$ with $-\pi\leq a<b<\pi$ and $v\in\mathcal{H}$, the unitary Stone's formula states that
$$
    \lim_{\epsilon\downarrow 0}\left(\frac{1}{4\pi}\int_a^bF_{\mathcal{E}}((1+\epsilon)^{-1}e^{i\theta})-F_{\mathcal{E}}((1+\epsilon)e^{i\theta})\dd\theta\right)v=\frac{1}{2}[\mathcal{E}((a,b))+\mathcal{E}([a,b])]v.
$$
This can also be expressed as a convolution with the Poisson kernel for the unit disk, which motivates the development of higher-order kernels and a variant on Stone's formula for those, leading to \cref{unispecmeas_alg}, which we call SpecRKHS-UniMeasure.

\begin{algorithm}[t]
    \caption{SpecRKHS-UniMeasure: Computing the scalar-valued spectral measures of unitary Perron--Frobenius operators $\koop_F^*$ at points $\{\theta_k\}_{k=1}^P\subset[-\pi,\pi]_{\mathrm{per}}$ with respect to $g=\sum_{j=1}^rg_ju_j\in\mathcal{H}$.}\label{unispecmeas_alg}
    \textbf{Input:} Snapshot data $\{(x^{(i)},y^{(i)}=F(x^{(i)}))\}_{i=1}^N$, kernel function $\mathfrak{K}:\mathcal{X}\times\mathcal{X}\rightarrow\mathbb{C}$, sample points $\{\theta_k\}_{k=1}^P\subset[-\pi,\pi]_{\mathrm{per}}$, poles $\{a_j\}_{j=1}^m\subset\{z\in\mathbb{C}:\mathrm{Im}(z)>0\}$, $\epsilon>0$, rank $r\leq N$, and $\m{g}\in\mathbb{C}^r$.

    \begin{algorithmic}[1]
        \STATE{Use kEDMD (\cref{kedmd_alg}) to compute the matrix $\hat{\Kv}$.}
        \STATE{Compute the eigenvectors $V$ and eigenvalues $\Lambda$ of $\hat{\Kv}^\top$, and define $\m{h}=V^{-1}\m{g}$, $\m{h}'=V^*\m{g}$.}
        \STATE{Solve the Vandermonde system \eqref{vandermonde} the residues $\alpha_1,\dots,\alpha_m\in\mathbb{C}$.}
        \FOR{$k=1,\dots,P$}
        \STATE{Set $z_j(\epsilon)=e^{i\theta_k-i\epsilon a_j}$ for $1\leq j\leq m$.}
        \STATE{Compute $\m{w}_j^{\epsilon}=(\Lambda-z_j(\epsilon)I)^{-1}(\Lambda+z_j(\epsilon)I)\m{h}$ for $1\leq j\leq m$.}
        \STATE{Set $[K_{\epsilon}^{\mathbb{T}}*\xi_g](\theta_k)=\frac{-1}{2\pi}\mathrm{Re}\left(\sum_{j=1}^m\alpha_j{\m{h}'}^*\m{w}_j^{\epsilon}\right)$.}
        \ENDFOR
    \end{algorithmic}
    \textbf{Output:} The approximate spectral measure $[K_{\epsilon}^{\mathbb{T}}*\xi_g](\theta_k)$ for $k=1,\dots,P$.
\end{algorithm}

\begin{remark}
    An alternative method for computing spectral measures of unitary operators is presented in~\cite{colbrook2021rigorousKoop} using auto-correlations to compute Fourier coefficients and then filtering. This method is straightforward in RKHSs since we can iterate $\koop_F^*\mathfrak{K}_x=\mathfrak{K}_{F(x)}$ to easily compute $\langle(\koop_F^*)^n\mathfrak{K}_{x_i},(\koop_F^*)^m\mathfrak{K}_{x_j}\rangle_{\mathfrak{K}}$ for general $n,m\in\mathbb{Z}$, but we do not provide details in this paper.
\end{remark}

\subsubsection{Convergence}

Convergence theorems and rates are provided in \cite{colbrook_computing_2021,colbrook2021rigorousKoop,colbrook2025rigged}. These extend immediately to \cref{saspecmeas_alg,unispecmeas_alg}, assuming we have a convergent approximation of the resolvent. We have written the above algorithms using the finite section method, i.e., assuming that $\mathcal{P}_N^*(\mathcal{P}_N\koop_F^*\mathcal{P}_N^*-zI)^{-1}\mathcal{P}_Ng\rightarrow (\koop_F-zI)^{-1}g$ as $N\rightarrow\infty$. For $\koop$ unitary, this holds for $|z|>1$ \cite[Thm.~5.1]{colbrook2021rigorousKoop}, and so the algorithm converges. For $\koop$ self-adjoint, the finite section method for the resolvent does not necessarily converge but often does in practice. However, we can use the matrices $G$, $A$, and $R$ from \cref{GA_defn,R_defn} to solve suitable normal equations and compute the resolvent with error control \cite[Thm.~B.2]{colbrook2019computing}.

\subsection{Numerical examples}

\subsubsection{Unitary example: M\"obius maps}
\label{sec:mobiusmaps}

M\"obius maps on the Poincar\'e disk $D=\{z\in\mathbb{C}:|z|<1\}$ are an example of unitary operators on an RKHS. Any M\"obius map of the form
$$
    T(z)=\frac{az+b}{\overline{b}z+\overline{a}}, \quad |a|^2-|b|^2=1,
$$
maps $D$ onto itself bijectively. We denote the set of such maps by $\mathcal{M}(D)$. In the hyperbolic metric associated with the Poincar\'e disk and distance function
$$
    d(x,y)=2\tanh^{-1}\left|\frac{y-x}{1-\overline{x}y}\right|,\quad x,y\in D,
$$
M\"obius maps in $\mathcal{M}(D)$ act as an isometry~\cite{anderson_hyperbolic_2006}. That is, for all $T\in \mathcal{M}(D)$ and $x,y\in D$,
$$
    d(T(x),T(y))=d(x,y).
$$
Hence, we consider the dynamics on $D$ given by a fixed $F\in\mathcal{M}(D)$ and choose our kernel to be
\begin{equation}
    \label{unitarymobiusmapkernel_eqn}
    \mathfrak{K}(x,y)=\exp(-\sigma d(x,y)^2),
\end{equation}
for some suitable scaling constant $\sigma>0$. After applying \cref{uniiffisom_corol} with the Euclidean distance replaced by the hyperbolic distance, we find that the associated Koopman operator $\mathcal{K}$ is unitary. If $\xi_v$ and $\xi_v'$ are the spectral measures of $\koop_F^*$ and $\koop_F$, then since $\koop_F^*=\koop_F^{-1}$, we obtain that $\xi_{v}(S)=\xi_{v}'(-S)$ for any Borel subset $S\subset[-\pi,\pi]_{\mathrm{per}}$. This means that it does not matter whether we compute the spectral measures of the Koopman operator or its adjoint.

\begin{figure}[t]
    \centering
    \includegraphics[width=0.45\linewidth]{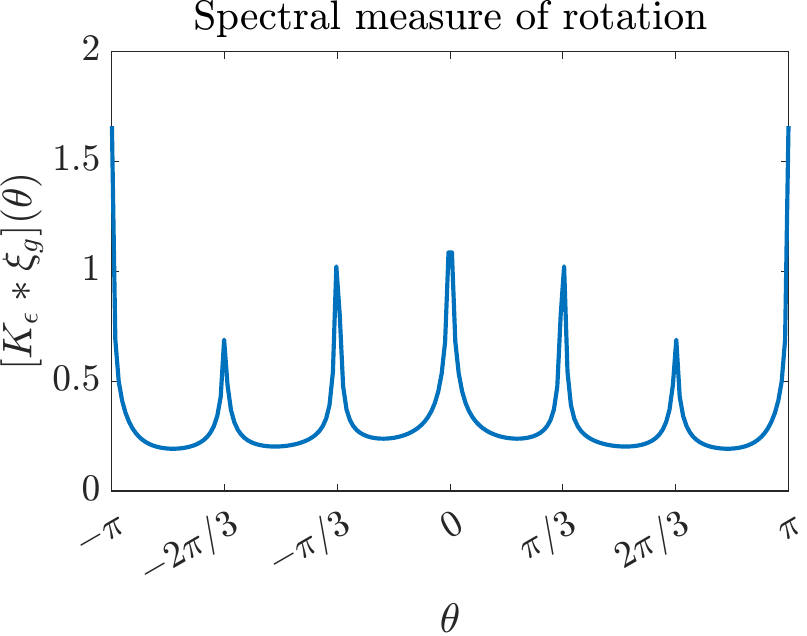}\hfill
    \includegraphics[width=0.45\linewidth]{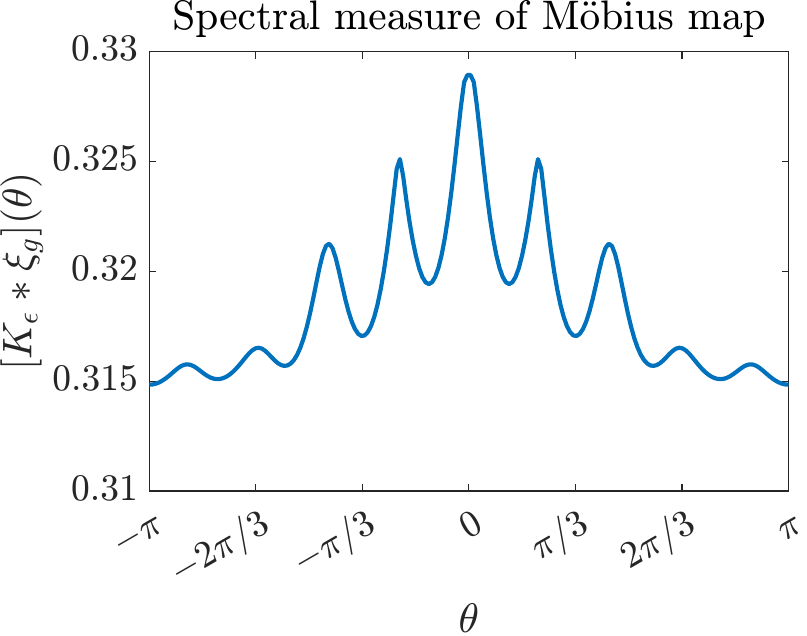}
    \caption{M{\"o}bius map. Left: Approximation to spectral measure for $T_1$; the rotation on the unit disk by $\pi/3$, plotted on a logarithmic scale. Right: Approximation to the spectral measure for the M{\"o}bius map defined by $T_2(z)=(az+b)/(\overline{b}z+\overline{a})$ with $a=\sqrt{2}e^{i\pi\sqrt{3}}$ and $b=e^{i\pi 9/7}$, now on a linear scale.}
    \label{fig:uni_rotation}
\end{figure}

We apply SpecRKHS-UniMeasure (\cref{unispecmeas_alg}) to the two M\"obius maps
\begin{equation}
    T_1(z)=e^{i\pi/3},\qquad T_2(z)=\frac{\sqrt{2}e^{i\pi\sqrt{3}}z+e^{i\pi 9/7}}{e^{-i\pi 9/7}z+\sqrt{2}e^{-i\pi\sqrt{3}}},
\end{equation}
with the kernel defined in \cref{unitarymobiusmapkernel_eqn}. We use $N=400$ pairs of snapshot data randomly generated with distribution $R^{\alpha}e^{2\pi i\Theta}$, where $R,\Theta$ are independent uniform random distributions on $(0,1)$. If $\alpha=1/2$, then the data are uniformly distributed in the disk, but we use $\alpha=1/4$ to improve the conditioning of $G$. We also select $\sigma$ to ensure the matrix $G$ is well-conditioned and take $\sigma=5$. We take $g$ to be an element of $V_N$ with uniformly randomly generated coefficients in $(0,1)$ and normalized such that $\|g\|_{\mathfrak{K}}^2=g^*Gg=1$. We use a high-order rational kernel of order $m=6$ with poles that have imaginary parts at $\pm i$ and real parts equally spaced, i.e., we take $a_j=2j/(m+1)-1+i$ for $j\in\{1,\ldots,m\}$, and we take $\epsilon=0.01$.

\cref{fig:uni_rotation} plots the results, where we have sampled the smoothed spectral measure in $[-\pi,\pi]$. For the rotation by $e^{i\pi/3}$, the spectral measure features sharp peaks at $e^{ki\pi/3}$ for $k\in\{-2,-1,0,1,2,3\}$, corresponding to eigenvalues at these points, as expected. For the more general M{\"o}bius map $T_2$, there is a continuous spectrum throughout the circle, with a potential embedded eigenvalue at $\theta=0$. To demonstrate the improved convergence provided by high-order kernels, we analyze the convergence of SpecRKHS-UniMeasure (\cref{unispecmeas_alg}) applied to the system given by the M{\"o}bius map $T_2$ as we vary $m$. We compute $K_{\epsilon}*\xi_g$ at a single point $\pi/3$ while varying $m$ and $\epsilon$ and keeping $N$ fixed, and plot the relative errors compared to the true value of $\rho_g(\pi/3)$. The results are shown in \cref{fig:uni_error_comp}. As $\epsilon$ decreases, convergence is significantly faster as we increase $m$. The observed error is $\mathcal{O}(\epsilon^m)$ for each $m$, as can be seen by the dotted black lines in \cref{fig:uni_error_comp}, in accordance with the theoretical convergence rates (up to logarithmic factors); for $m=6$, the stagnation for small $\epsilon$ is due to reaching machine precision.

\begin{figure}[t]
    \centering
    \includegraphics[width=0.45\linewidth]{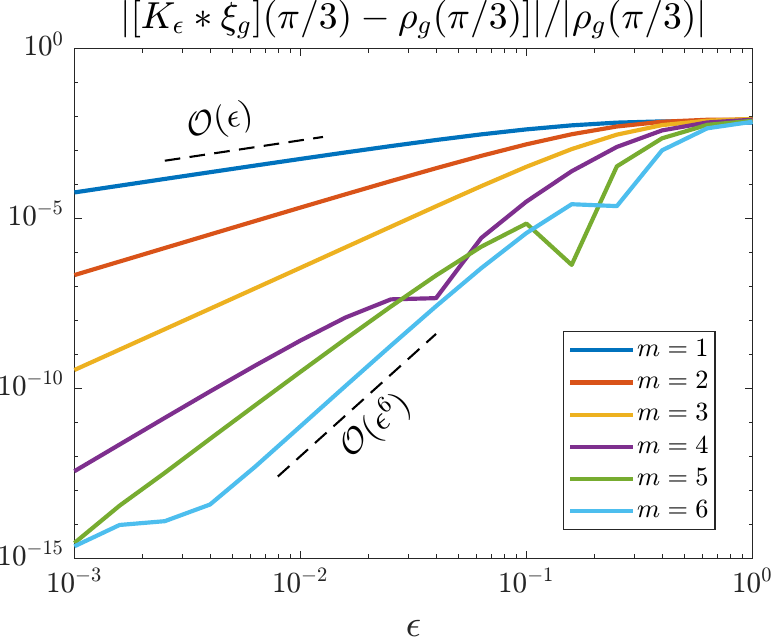}
    \caption{M{\"o}bius map. Analysis of the rate of convergence of SpecRKHS-UniMeasure (\cref{unispecmeas_alg}) to the true spectral measure for the M{\"o}bius map $T_2$ at a single point $z=\pi/3$. $N$ is fixed and we vary $\epsilon$ and $m$. The black dotted line show agreement with the theoretical convergence rate of $\mathcal{O}(\epsilon^m)$ (up to logarithmic factors).}
    \label{fig:uni_error_comp}
\end{figure}

\subsubsection{Self-adjoint example: Stochastic Koopman operators on RKHSs}

Stochastic Koopman operators~\cite{colbrook_beyond_2024,klus_eigendecompositions_2020,kostic_learning_2022,kostic_sharp_2023} provide examples of self-adjoint operators and constitute a further extension of our approach. We consider discrete-time stochastic systems of the form:
$$
    x_{n+1}=F(x_n,\tau_{n+1}),
$$
where $x_n\in\mathcal{X}$ (the state space) and $\{\tau_{n}\}\in\mathcal{X}_s$ are i.i.d. random variables with distribution $\rho$. The corresponding (stochastic) Koopman operator $\koop_F$ is defined on an appropriate function space $\mathcal{H}$ by
$$
    \koop_F g=\mathbb{E}_{\tau}[g\circ F_{\tau}]=\int_{\mathcal{X}_s}g\circ F_{\tau}\dd\rho(\tau),\quad
    \text{where }F_{\tau}(x)=F(x,\tau),\,g\in\mathcal{H}.
$$
Suppose that $\mathcal{H}$ is an RKHS with kernel $\mathfrak{K}$. Then, for all $x,y\in\mathcal{X}$,
$$
    \langle\koop_F^*\mathfrak{K}_x,\mathfrak{K}_y\rangle_{\mathfrak{K}}=\langle \mathfrak{K}_x,\koop_F\mathfrak{K}_y\rangle_{\mathfrak{K}}=\overline{\langle\koop_F \mathfrak{K}_y,\mathfrak{K}_x\rangle_{\mathfrak{K}}}=\overline{(\koop_F \mathfrak{K}_y)(x)}=\overline{\mathbb{E}_{\tau}[\mathfrak{K}_y(F_\tau(x))]}=
    \overline{\mathbb{E}_{\tau}[\langle \mathfrak{K}_y,\mathfrak{K}_{F_{\tau}(x)}\rangle_{\mathfrak{K}}]}.
$$
It follows that
$$
    \langle \koop_F^*\mathfrak{K}_x,\mathfrak{K}_y\rangle_{\mathfrak{K}}=\mathbb{E}_{\tau}[\langle \mathfrak{K}_{F_{\tau}(x)},\mathfrak{K}_y\rangle_{\mathfrak{K}}]
$$
This formula allows us to compute this inner product by taking an average over realizations of the system.

As an example of a stochastic process that can produce self-adjoint Koopman operators, we consider discrete-time Markov chains on infinite graphs \cite{klus_dynamical_2024}. A discrete-time stochastic process $\{x_n\}_{n\in\mathbb{N}}$ on a finite or countable state space $\mathcal{X}=\{s_1,s_2,\dots\}$ is said to have the Markov property, and in this case called a discrete-time Markov chain, if for all $n\in\mathbb{N}$ and for all $s_0,\dots,s_{n+1}\in \mathcal{X}$,
$$
    \mathbb{P}\left(x_{n+1}=s_{n+1}\mid x_n=s_n,\dots, x_0=s_0\right)=\mathbb{P}\left(x_{n+1}=s_{n+1}\mid x_n=s_n\right).
$$
A Markov chain is called time-homogeneous if for all $n\in\mathbb{N}$ and for all $x,y\in\mathcal{X}$, we have
$\mathbb{P}\left(x_{n+1}=y \mid x_n=x\right)=\mathbb{P}\left(x_n=y \mid x_{n-1}=x\right)$. In this case, the transition probability $p_{xy}$ from state $x$ to state $y$ is defined by
$p_{xy}=p(x,y)=\mathbb{P}\left(x_{n+1}=y\mid x_n=x\right)$. The Markov chain is said to be reversible if there exists a positive function $\pi:\mathcal{X}\mapsto\mathbb{R}$ such that for all $x,y\in\mathcal{X}$, $\pi(x)p(x,y)=\pi(y)p(y,x)$. Given such a function $\pi$, we define a $\pi$-weighted $L^2$ space of functions over $\mathcal{X}$ as
$$
    L^2_{\pi}(\mathcal{X})\coloneqq\left\{g:\mathcal{X}\rightarrow\mathbb{C}:\sum_{x\in\mathcal{X}}\pi(x)|g(x)|^2<\infty\right\},
$$
which is an RKHS with kernel function $\mathfrak{K}(x,y)=\delta_{xy}/\pi(x)$. Rescaling the kernel functions thus gives an orthonormal basis.

A Markov chain defines a Koopman operator associated with the function $F:\mathcal{X}\times\mathcal{X}_s\rightarrow\mathcal{X}$ where $\mathcal{X}_s=[0,1]$, the $\tau_n$ are independent, uniform distributions on $[0,1]$ and for all $n\in\mathbb{N}$,
$$
    F(s_n,\tau)=s_k,\quad \text{if }\sum_{l=1}^{k-1}p(s_n,s_l)\leq\tau<\sum_{l=1}^kp(s_n,s_l),
$$
where, if $k=1$, we take the first sum to be $0$, and if $\mathcal{X}$ is finite, we take the last strict inequality not to be strict. The associated stochastic Koopman operator is given by
$$
    \koop_Fg(x)=\sum_{y\in\mathcal{X}}p(x,y)g(y),
$$
and we assume that it is bounded.
To find conditions on $\koop_F$ to be self-adjoint, note that
$$
    \langle \koop_F^*\mathfrak{K}_x,\mathfrak{K}_y\rangle_{\mathfrak{K}}=
    \mathbb{E}_{\tau}[\langle \mathfrak{K}_{F_{\tau}(x)},\mathfrak{K}_{y}\rangle_{\mathfrak{K}}]=\sum_{s\in\mathcal{X}}p(x,s)\langle \mathfrak{K}_{s},\mathfrak{K}_{y}\rangle_{\mathfrak{K}}
    =\sum_{s\in\mathcal{X}}p(x,s)\delta_{sy}/\pi(s)=p(x,y)/\pi(y).
$$
Since $p$ and $\pi$ are real-valued, we may swap the roles of $x$ and $y$ to see that
$$
    \langle \koop_F\mathfrak{K}_x,\mathfrak{K}_y\rangle_{\mathfrak{K}}=
    \langle \mathfrak{K}_y,\koop_F\mathfrak{K}_x\rangle_{\mathfrak{K}}=
    \langle \koop_F^*\mathfrak{K}_y,\mathfrak{K}_x\rangle_{\mathfrak{K}}=p(y,x)/\pi(x).
$$
By density of the functions $\mathfrak{K}_x$, it follows that $\koop_F$ is self-adjoint if and only if
$$
    \pi(x)p(x,y)=\pi(y)p(y,x)\quad \forall x,y\in\mathcal{X}.
$$
These are known as the detailed balance equations.

As a simple numerical example, we consider the state space $\mathcal{X}=\mathbb{Z}$ with transition probabilities:
\begin{equation}
    \label{stochasticexample_eqn}
    p_1(n,n-1)=p_1(n,n)=p_1(n,n+1)=1/3,\quad p_1(n,m) = 0 \text{ otherwise}.
\end{equation}
These probabilities satisfy the detailed balance equations for $\pi(n)=1$. With respect to the orthonormal basis of kernel functions on $L^2(\mathcal{X})$, $\koop_F$ is represented by an infinite matrix that is tridiagonal, symmetric, and (bi-infinite) Toeplitz with all non-zero entries of values $1/3$. Hence, it has symbol $f(z)=(z+1+z^{-1})/3$ and so its spectrum is $\spec(\koop_F)=f(\mathbb{T})=[-1/3,1]$ \cite[Sec.~1.6]{bottcher_toeplitz_2000}. Moreover, we can analytically compute the spectral measures with respect to $g\in L^2(\mathcal{X})$. Indeed, $L^2(\mathcal{X})$ is isomorphic to the space
$$
    L^2(\mathbb{T})=\left\{\sum_{n=-\infty}^{\infty}a_ne^{in\theta}:\sum_{n=-\infty}^{\infty}|a_n|^2<\infty\right\},
$$
with inner product given by $\langle f,g\rangle=\frac{1}{2\pi}\int_{\mathbb{T}}f(\theta)\overline{g(\theta)}\dd\theta$.
On this space, $\koop_F$ is given by multiplication by $(e^{i\theta}+1+e^{-i\theta})/3$. Now, for $f,g\in L^2(\mathbb{T})$, using the spectral theorem initially and changing variables to $\lambda=(1+2\cos\theta)/3$ and considering both branches of the inverse, we have
\begin{equation*}
    \begin{split}
         & \int_{\mathbb{R}}\lambda^n\dd\mu_{f,g}(\lambda)=\langle \koop^nf,g\rangle=\frac{1}{2\pi}\int_{-\pi}^{\pi}\left[\frac{1}{3}\left(1+2\cos\theta\right)\right]^nf(\theta)\overline{g(\theta)}\dd\theta                                                                 \\
         & =\sum_{i=1}^2\frac{3}{4\pi}\int^1_{-1/3}\lambda^nf\left((-1)^i\cos^{-1}\left(\frac{3\lambda-1}{2}\right)\right)\overline{g\left((-1)^i\cos^{-1}\left(\frac{3\lambda-1}{2}\right)\right)}\sin\left(\cos^{-1}\left(\frac{3\lambda-1}{2}\right)\right)^{-1}\dd\lambda.
    \end{split}
\end{equation*}
As moments define a measure on a bounded interval \cite{feller_introduction_1991} (even if we did not know the spectrum explicitly, we know that $\koop$ is bounded, so has bounded spectrum), we see that the spectral measure $\mu_{f,g}$ of $\koop_F$ is absolutely continuous with Radon--Nikodym derivative
$$
    \rho_{f,g}(\lambda)=\chi_{[-1/3,1]}(\lambda)\sum_{i=1}^2\frac{3}{4\pi}f\!\left((-1)^i\cos^{-1}\!\left(\frac{3\lambda-1}{2}\right)\!\right)\overline{g\!\left((-1)^i\cos^{-1}\!\left(\frac{3\lambda-1}{2}\right)\!\right)}\sin\!\left(\!\cos^{-1}\!\left(\frac{3\lambda-1}{2}\right)\!\right)^{-1},
$$
where $\chi_S$ denotes the indicator function of a set $S$.
We take $f(\theta)=g(\theta)=\sin\theta$, which corresponds to $(\delta_1-\delta_{-1})/2$ in $L^2(\mathcal{X})$ and yields
\begin{equation}
    \label{radonnikodymexact_eqn}
    \rho_{f}(\lambda)=\chi_{[-1/3,1]}(\lambda)\frac{3}{4\pi}\left(6\lambda+3-9\lambda^2\right)^{1/2}.
\end{equation}

\begin{figure}[t]
    \centering
    \includegraphics[width=0.45\linewidth]{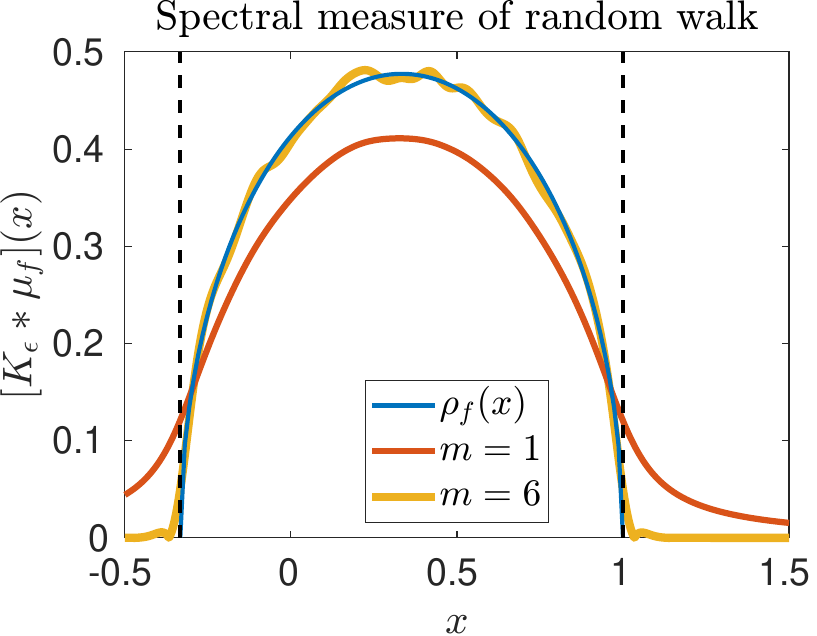}\hfill
    \includegraphics[width=0.45\linewidth]{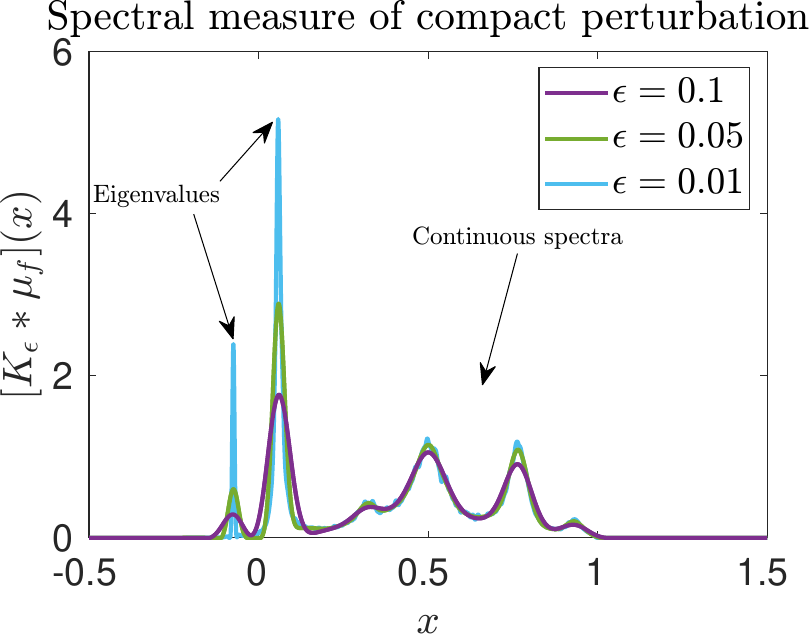}
    \caption{Random walk. Left: Approximation to spectral measure of system defined by \cref{stochasticexample_eqn} using SpecRKHS-SAdjMeasure (\cref{saspecmeas_alg}). The exact expression of the spectral measure $\rho_f$ is given in \cref{radonnikodymexact_eqn}, and the boundary of the spectrum is shown by the dashed vertical lines. We have shown two values of $m$ (order of the convolution kernel). Right: Approximation to the spectral measure of a random perturbation of the same system. Here, we vary $\epsilon$ and observe convergence to the continuous spectra and eigenvalues.}
    \label{fig:unispecmeas}
\end{figure}

We apply SpecRKHS-SAdjMeasure (\cref{saspecmeas_alg}) to numerically compute the spectral measure of $\koop$ with respect to $f$. For the snapshot data, we take $x^{(j)}=j$ for $j\in\{-5000,\ldots,5000\}$, and for each $x^{(j)}$, we sample $F_{\tau}(x^{(j)})$ $10^5$ times to construct a finite-section approximation of the stochastic Koopman operator. As in the case of the M{\"o}bius map, we use rational kernels of order $m$ with poles $a_k=2k/(m+1)-1+i$ for $k\in\{1,\ldots,m\}$. The left panel of \cref{fig:unispecmeas} shows the smoothed spectral measures for $m = 1$ and $m = 6$. The improved accuracy of the higher-order kernel is evident: for $m = 6$, the result closely matches the exact expression in \cref{radonnikodymexact_eqn}. Moreover, if the probabilistically sampled matrices $G$ and $A$ are replaced with their exact values, the computed spectral measure becomes indistinguishable from the true one, indicating that the observed oscillations are due to the finite number of samples used to approximate the expectation.

We also apply our algorithm to a randomly generated perturbation of \cref{stochasticexample_eqn}, which induces a compact perturbation of the Koopman operator and thus preserves its essential spectrum~\cite[Thm.~5.35]{kato_perturbation_1995}. Specifically, for each $i\in\{-4,\ldots,4\}$, we randomly generate independent and uniformly distributed values $a_{ij}\in(-1/8,1/8)$ for $j\in\{1,2\}$ and define a perturbed Markov chain via
\begin{equation}
    p_2(i,i-1)=1/4+a_{i1},\quad
    p_2(i,i)=1/2-a_{i1}-a_{i2},\quad
    p_2(i,i+1)=1/4+a_{i2},
\end{equation}
with $p_2=p_1$ elsewhere. We use the same snapshot data, number of samples, and rational kernel as for the unperturbed system. The right panel of \cref{fig:unispecmeas} shows the smoothed spectral measures for $m = 6$ and various values of $\epsilon$. As $\epsilon \downarrow 0$, the smoothed spectral measures converge to the (now perturbed) Radon--Nikodym derivative in regions of absolutely continuous spectrum. Conversely, in regions where the smoothed spectral measures diverge as $\epsilon \downarrow 0$, the blow-up indicates they are approximating Dirac delta distributions associated with point spectrum (eigenvalues).

\section{Examples in high dimensions}
\label{numexample_sect}

In this section, we implement our algorithms on a range of numerical examples of practical interest, including turbulent fluid flow (\cref{sec:turbulent}), molecular dynamics (\cref{sec:ifabp}), and climate forecasting (\cref{sec:antarctic,sec:seaheight}). The datasets used throughout are obtained from real data or high-fidelity numerical simulations. All the numerical examples in this section are performed on a standard laptop with $16$GB of RAM and an Intel\textsuperscript{\textregistered} Core\textsuperscript{TM} i7-9750H CPU @ 2.60\,GHz. Despite considering high-dimensional state spaces, our algorithms are fast to train and evaluate since they predominantly exploit the smaller number of observables. Moreover, the forecasting methods introduced in \cref{sec:errorcontrolpfmd} often significantly reduce the model's size while maintaining or improving accuracy. Hence, as discussed in \cref{sec:errorcontrolpfmd}, using the Perron--Frobenius operator rather than the Koopman operator for predictions is computationally cheaper when handling large dimensions $d$.

Throughout this section, we use Mat{\'e}rn kernels rather than Wendland kernels to generate RKHSs equivalent to Sobolev spaces, as Mat{\'e}rn kernels offer more stable evaluations. The Wendland kernel function $\varphi_{d,k}(r)$ must be of the form $(1-r)_+^{\lfloor{d/2}\rfloor +k+1}q(r)$ for $q$ a polynomial, due to the requirement of having $\lfloor d/2\rfloor+k$ continuous derivatives at $r=1$ \cite[Thm.~9.13]{wendland_scattered_2004}, and is unstable to evaluate numerically for large $d$. On the other hand, the key parameter that influences the power of $\sigma\|x\|$ to take and the order of the Bessel function to use is $n-d/2$ for the Mat{\'e}rn kernel, which can be kept small regardless of how large $d$ is by choosing a sufficiently large value of $n$.

\subsection{Turbulent channel flow}
\label{sec:turbulent}

\begin{figure}[th]
    \centering
    \includegraphics[width=0.48\linewidth]{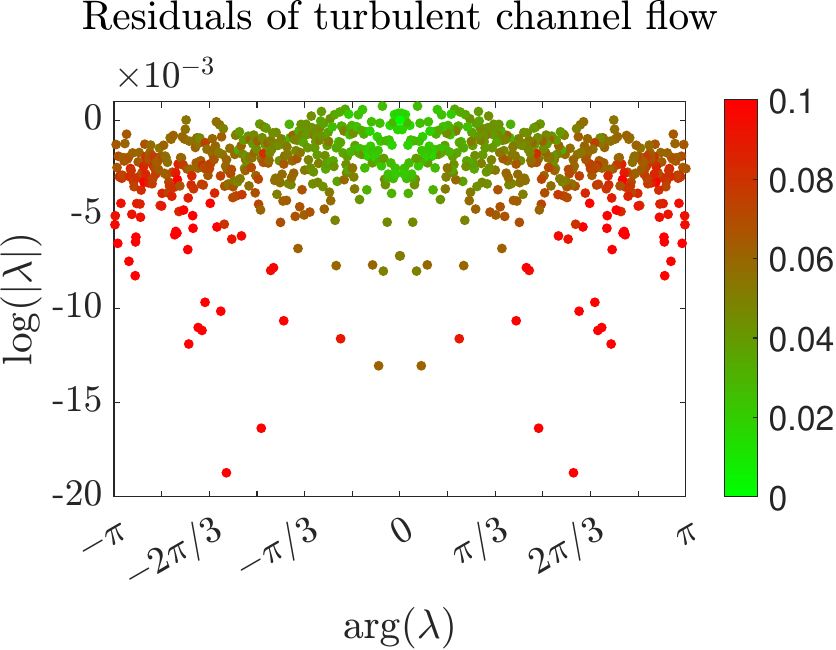}\hfill
    \includegraphics[width=0.48\linewidth]{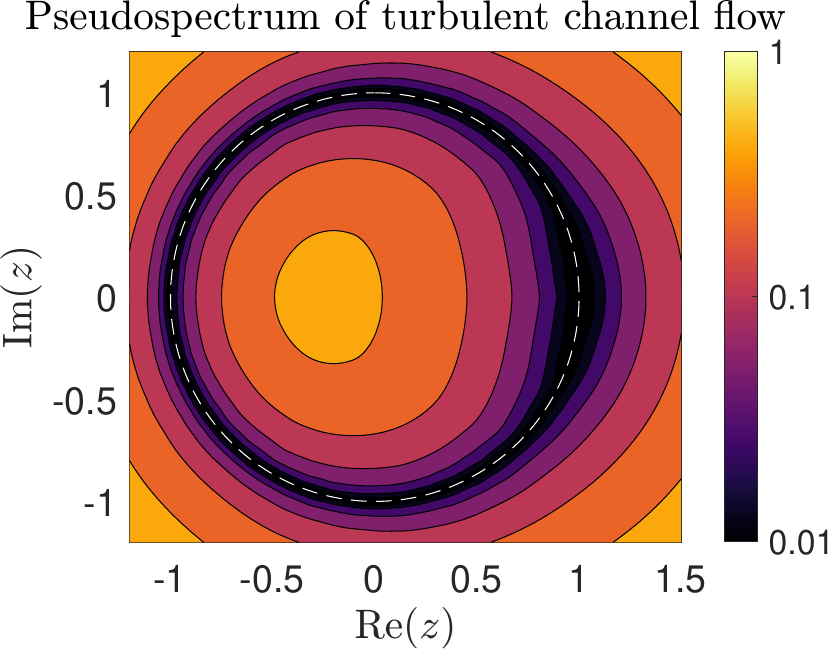}
    \caption{Turbulent channel flow. Left: The eigenvalues outputted by kEDMD using a Mat{\'e}rn kernel are plotted based on their argument and (the log of) their modulus. The colors highlight the size of their residuals, computed by SpecRKHS-Eig (\cref{evalverif_alg}). Right: Pseudospectral contours for the same system using the same kernel, computed by SpecRKHS-PseudoPF (\cref{pspecadjoint_alg}). The white dotted line displays the unit circle $\{z\in\mathbb{C}:|z|=1\}$.}
    \label{fig:channelresidualspspec}
\end{figure}

\begin{figure}[th!]
    \centering
    \begin{overpic}[height=0.2\linewidth]{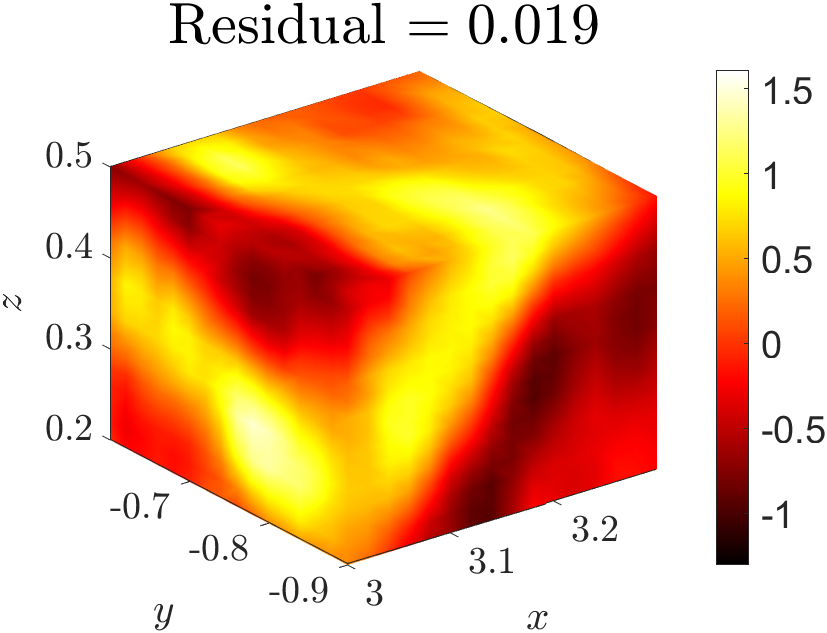}
       \put(0,71){(a)} 
    \end{overpic}
    \hspace{0.1cm}
    \includegraphics[height=0.2\linewidth]{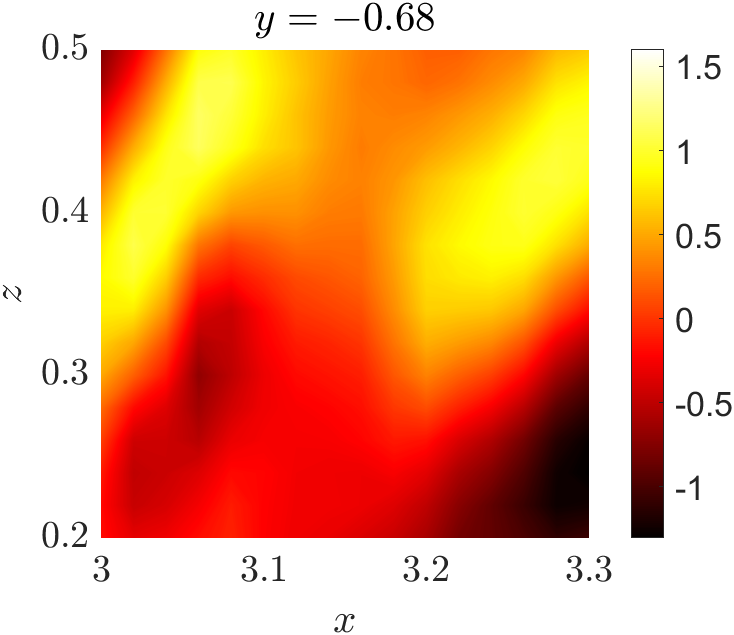}
    \includegraphics[height=0.2\linewidth]{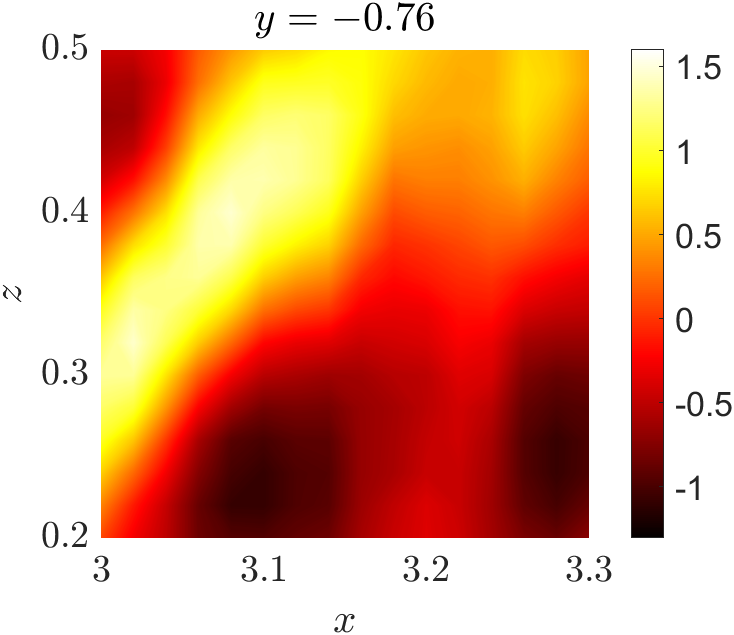}
    \includegraphics[height=0.2\linewidth]{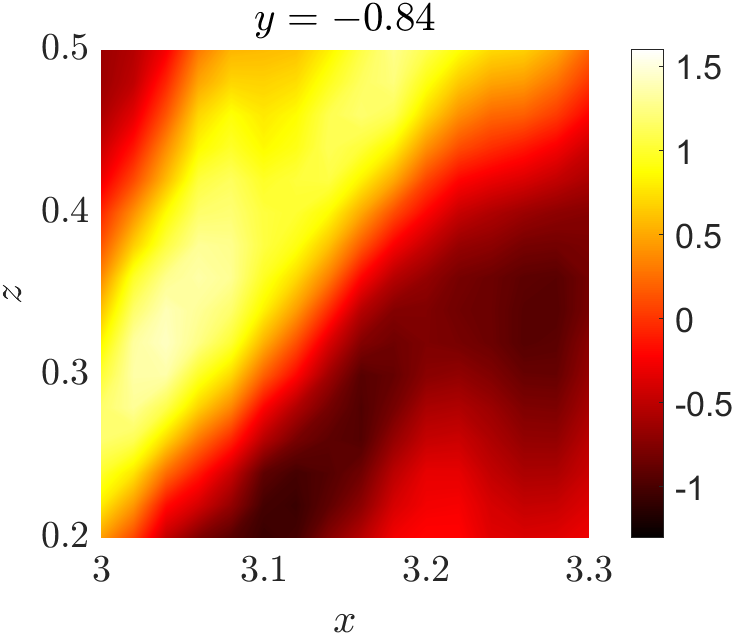}\\
   \begin{overpic}[height=0.2\linewidth]{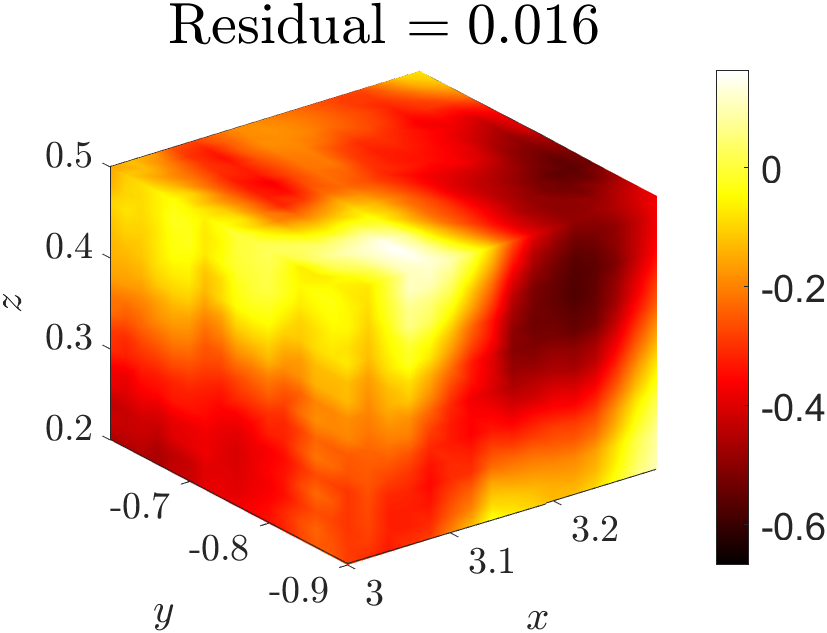}
          \put(0,71){(b)} 
   \end{overpic}
        \hspace{0.1cm}
    \includegraphics[height=0.2\linewidth]{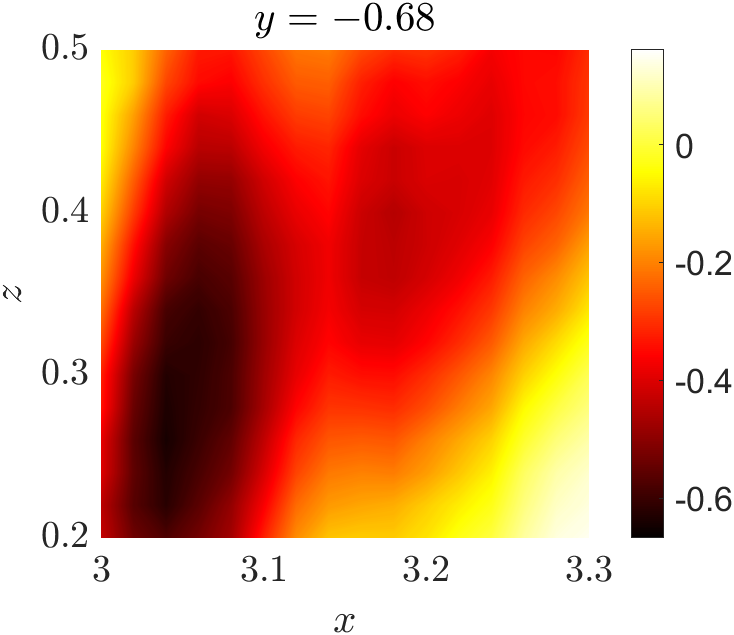}
    \includegraphics[height=0.2\linewidth]{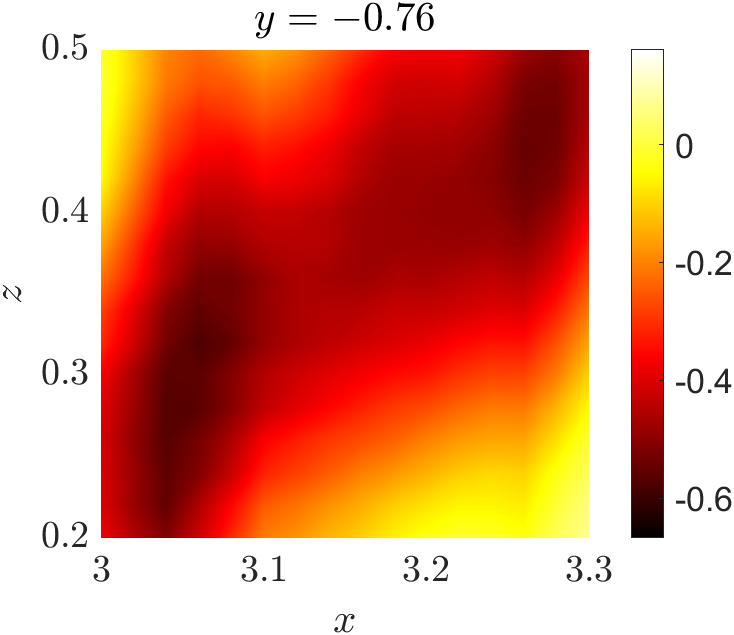}
    \includegraphics[height=0.2\linewidth]{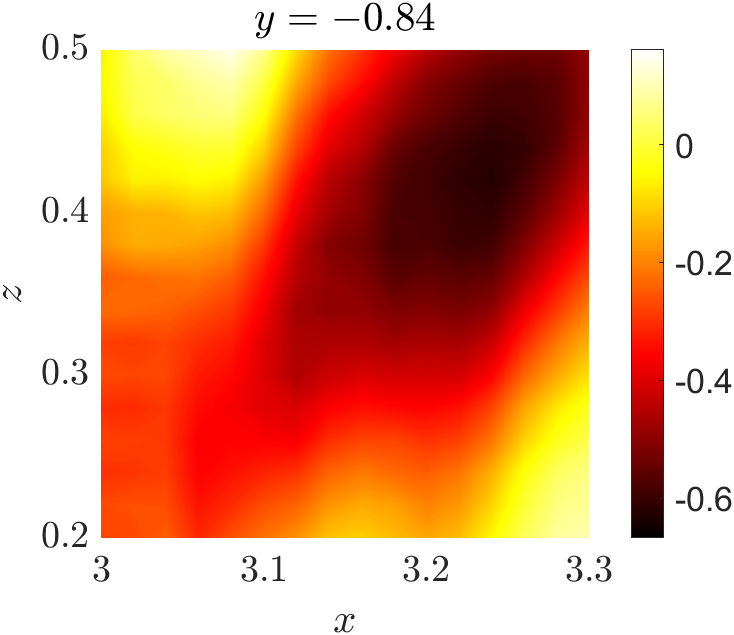}\\
    \begin{overpic}[height=0.2\linewidth]{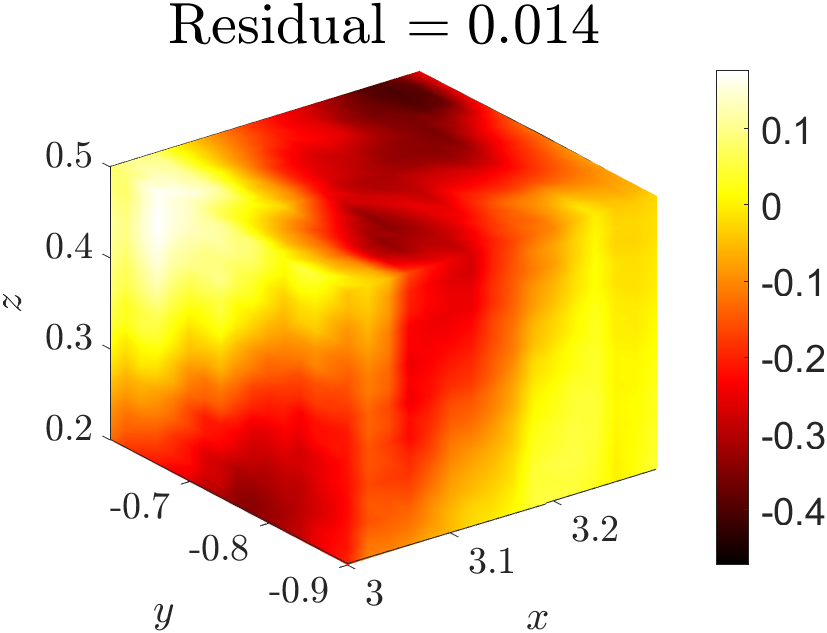}
    \put(0,71){(c)} 
   \end{overpic}
        \hspace{0.1cm}
    \includegraphics[height=0.2\linewidth]{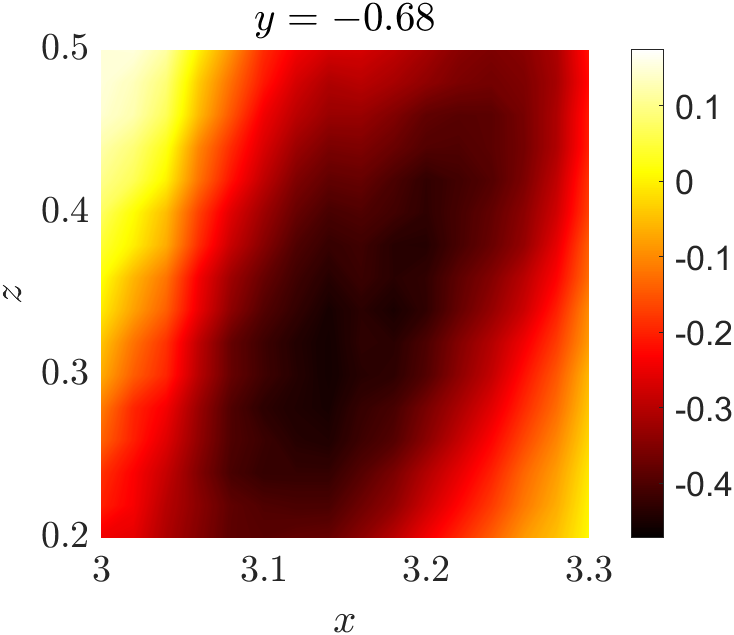}
    \includegraphics[height=0.2\linewidth]{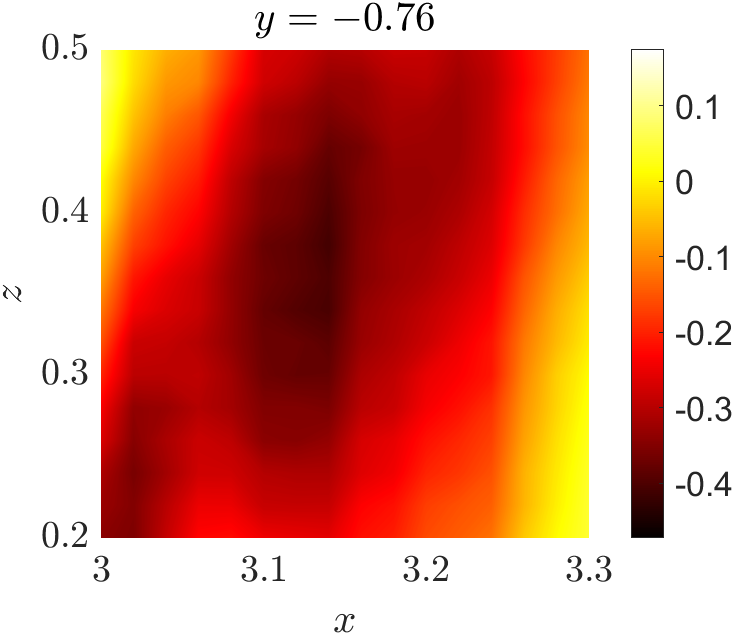}
    \includegraphics[height=0.2\linewidth]{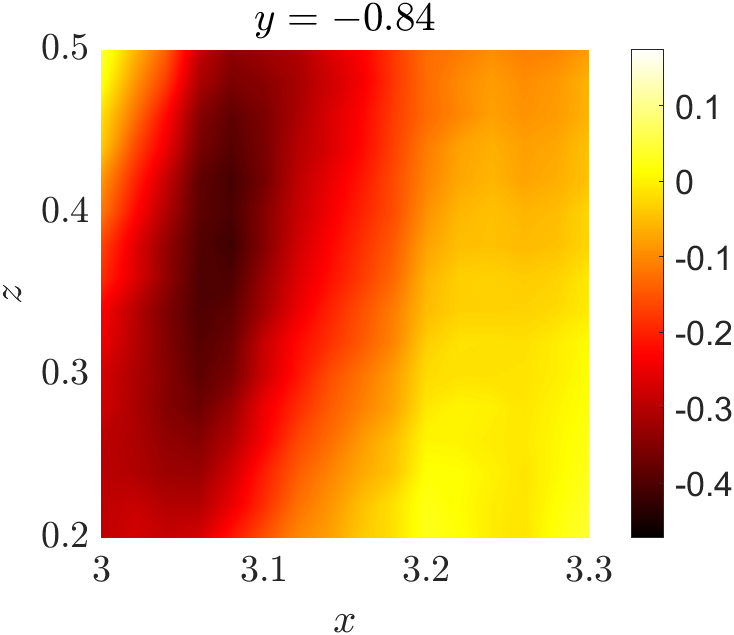}
    \caption{Turbulent channel flow. (a)-(c) Perron--Frobenius modes corresponding to three verified eigenfunctions for the turbulent channel flow. Each mode is plotted first in three dimensions with its residual computed by SpecRKHS-Eig (\cref{evalverif_alg}), along with two-dimensional slices along a fixed $y$-direction.}
    \label{fig:channelpdfmodes}
\end{figure}

As a first example to demonstrate the practicality of our algorithms in high dimensions, we consider the problem of turbulent channel flow, with data obtained from \cite{graham_web_2016}. The system is a turbulent flow in a channel of size $8\pi\times 2\times 3\pi$, with friction velocity Reynolds number of $Re\approx 1000$, friction velocity $u=0.0499$ and viscosity $5\times 10^{-5}$. The simulation solves the three-dimensional incompressible Navier-Stokes equation on $2048\times 512\times 1536$ nodes; we sample the pressure at $4096$ equally spaced nodes on the subdomain $[3.0,3.3]\times [-0.9,-0.6]\times [0.2,0.5]$. We take $800$ snapshots of a single trajectory at time intervals $0.0325$, so the trajectory lasts $26$ seconds, and use the Mat{\'e}rn kernel defined by
$$
    \mathfrak{K}(x,y)=\begin{cases}
        1,                                     & \text{if } x=y,   \\
        (\sigma\|x-y\|_2)K_1(\sigma\|x-y\|_2), & \text{otherwise},
    \end{cases}
$$
where $K_1$ is a modified Bessel function of the second kind.

Given this setup, we use SpecRKHS-Eig (\cref{evalverif_alg}) to compute residuals of eigenpairs computed by kEDMD, and then use SpecRKHS-PseudoPF (\cref{pspecadjoint_alg}) to compute the approximate point pseudospectrum of the Perron--Frobenius operator. We observe in \cref{fig:channelresidualspspec} that most verified eigenpairs lie close to $\lambda=1$, as can also be seen in the pseudospectral plot. The `bulge' in the pseudospectral plot around $\lambda=1$ shows that the Perron--Frobenius operator is highly non-normal, consistent with transient energy growth typical in turbulent flows.

We also compute the Perron--Frobenius modes corresponding to the verified eigenfunctions. In particular, by solving a least squares problem, we can expand the eigenfunction $\psi$ in terms of the observables $g_i(x)=[x]_i$ (where $[x]_i$ denotes the $ith$ component of $x$), i.e., $\psi=\sum_{i=1}^dc_ig_i$, and plot the Perron--Frobenius modes $c_i$ at the corresponding locations. \cref{fig:channelpdfmodes} shows these modes, plotted both on the full $3$D domain and in slices with fixed $y$-coordinate, along with the residuals of the corresponding eigenfunction. The three leading verified eigenfunctions correspond to slowly decaying modes in the flow. By examining the Perron–Frobenius modes (spatial structures) associated with these eigenfunctions (\cref{fig:channelpdfmodes}), we can identify regions in the channel that oscillate or decay coherently. Each mode highlights a spatial pattern of pressure variation, and its residual indicates how perfectly it follows $\lambda^n$ scaling.

\subsection{Intestinal fatty acid binding proteins}
\label{sec:ifabp}

\begin{figure}[t]
    \centering
    \includegraphics[width=0.48\linewidth]{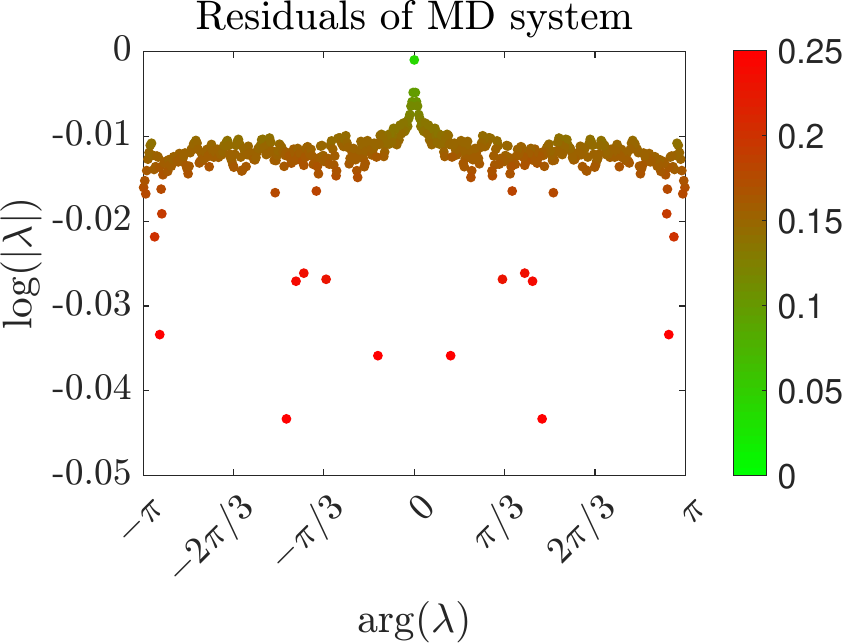}\hfill
    \includegraphics[width=0.43\linewidth]{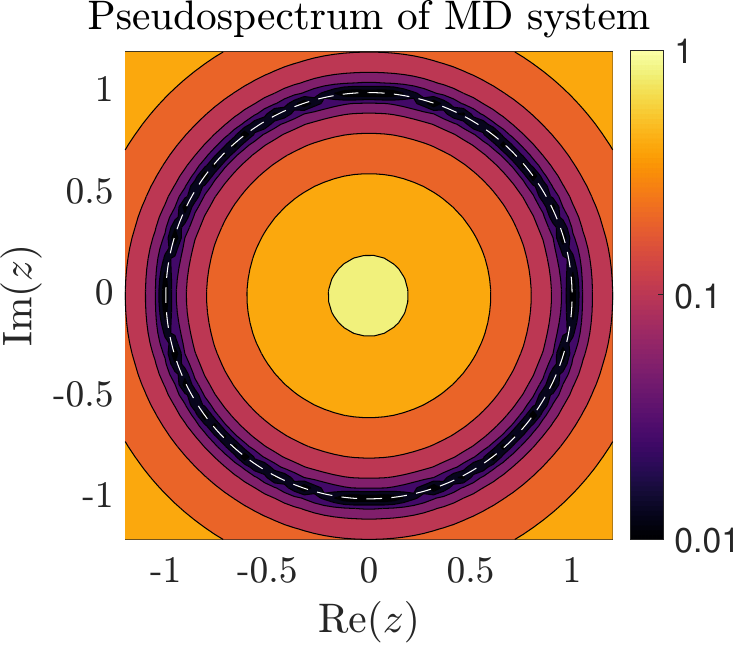}
    \caption{Molecular dynamics (MD) dataset. Left: The eigenvalues outputted by kEDMD using a Mat{\'e}rn kernel are plotted based on their argument and (the log of) their modulus. The color shows the size of their residuals, computed by SpecRKHS-Eig (\cref{evalverif_alg}). Right: Pseudospectral contours for the same system using the same kernel, computed by SpecRKHS-PseudoPF (\cref{pspecadjoint_alg}). The white dotted line displays the unit circle $\{z\in\mathbb{C}:|z|=1\}$.}
    \label{fig:mdresidualspspec}
\end{figure}

In this example, we use a molecular dynamics (MD) dataset \cite{Beckstein2018} consisting of a trajectory of 500 snapshots for an intestinal fatty acid binding protein (I-FABP), including both water and ions as solvents. The trajectory was computed using CHARMM with a $2$fs timestep, and the snapshots were saved every $1$ ps; the trajectory was then RMSD-fitted to the protein. The state space consists of the $x$, $y$, and $z$ positions of the $12445$ particles making up the protein, giving a dimension of $d=37335$.

We choose $n=18668$ such that $n-d/2=1/2>0$ and use the Mat{\'e}rn kernel:
$$\mathfrak{K}(x,y)=\begin{cases}
        1,                                     & x=y,              \\
        (\sigma\|x-y\|_2)K_1(\sigma\|x-y\|_2), & \text{otherwise},
    \end{cases}$$
whose native space corresponds to the Sobolev space $H^{18668}(\mathbb{R}^{37335})$. To ensure well-conditioned matrices, we select $\sigma=1/2000$. We run SpecRKHS-Eig (\cref{evalverif_alg}) and SpecRKHS-PseudoPF (\cref{pspecadjoint_alg}) on this example and report the results in \cref{fig:mdresidualspspec}. There are no significant damped modes (no $|\lambda|<1$ eigenvalues beyond numerical noise). The approximate point spectrum appears to be the entire unit circle. The pseudospectral contours are nearly circular (\cref{fig:mdresidualspspec}) with $\|(\koop_F^*-zI)^{-1}\|^{-1}= \mathrm{dist}(z,\mathbb{T})$, suggesting the Perron--Frobenius operator is close to normal (possibly unitary). Physically, this might indicate that the system's dynamics (on the timescale of the simulation) are largely conservative without clear exponential decay.

Motivated by this and the fact that $\|G-M\|/\|G\|\approx 3\times 10^{-3}$ is small (possibly corresponding to a quasi-conservative molecular vibration), we apply SpecRKHS-UniMeasure (\cref{unispecmeas_alg}) to compute the spectral measure of $\koop_F^*$. To remove the eigenvalue at $\lambda=1$, we select our observable $g=\sum_{i=1}^N c_i\mathfrak{K}_{x_i}$ to be orthogonal to the constant function (or, strictly speaking, any approximation of the constant function on a compact set containing all the data points), which is equivalent to $\sum_{i=1}^N c_i=0$. As in \cref{sec:mobiusmaps}, we use a rational kernel with $m=6$ and poles $a_j=2j/(m+1)-1+i$ for $j\in\{1,\ldots,m\}$. \cref{fig:mdspecmeas} shows the smoothed spectral measure for various $\epsilon$, where there are non-trivial regions of continuous spectra throughout the unit circle.

This suggests that the molecular system does not have a small set of dominant Koopman modes (aside from the trivial mode); instead, its dynamics contain a continuum of frequencies/modes – consistent with a high-dimensional chaotic motion. Our algorithms successfully detected this continuous spectral component, which would be impossible to see with standard kEDMD alone.

\begin{figure}[t]
    \centering
    \includegraphics[width=0.48\linewidth]{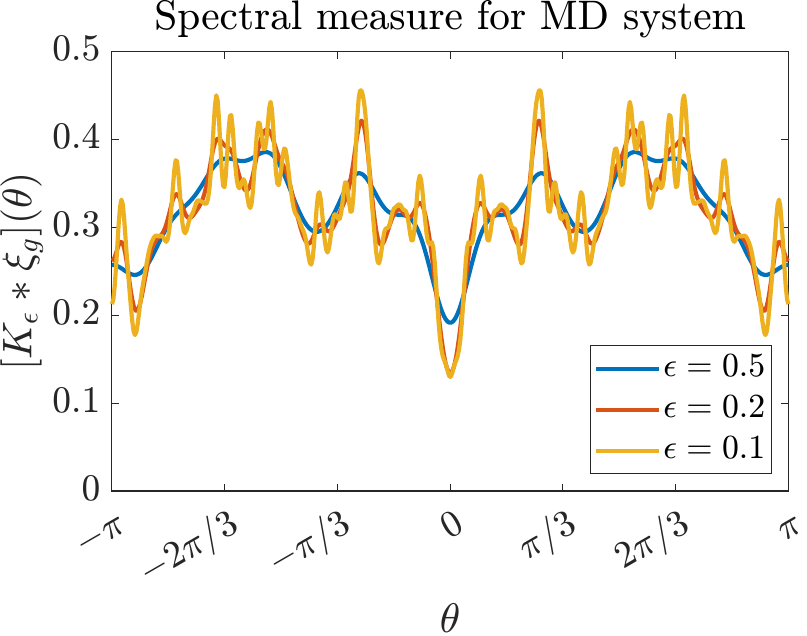}
    \caption{Molecular dynamics (MD) dataset. Approximation to the spectral measure of molecular dynamics system with respect to an observable orthogonal to the constant function using SpecRKHS-UniMeasure     (\cref{unispecmeas_alg}).}
    \label{fig:mdspecmeas}
\end{figure}

\subsection{Antarctic sea ice concentration}
\label{sec:antarctic}

\begin{figure}[t]
    \centering
    \includegraphics[width=0.48\linewidth]{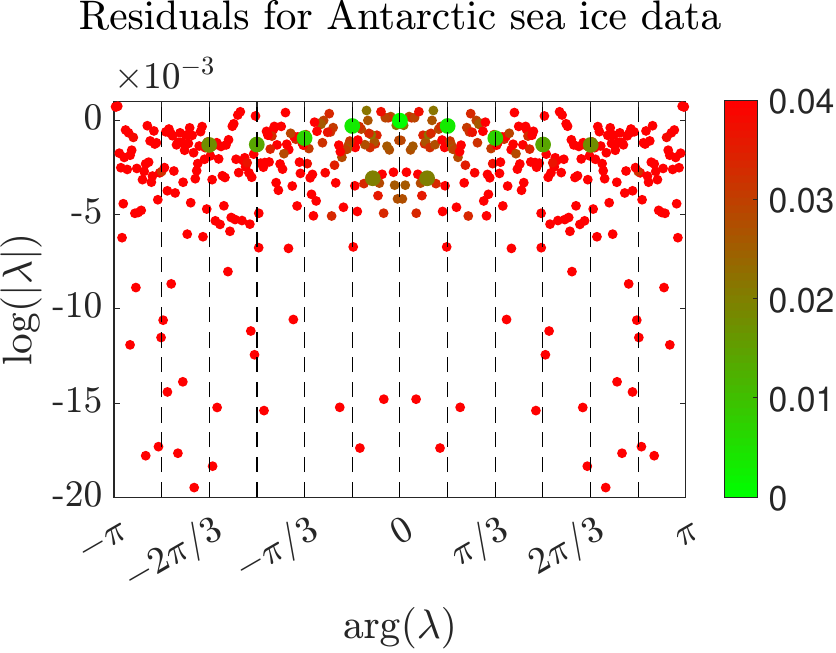}\hfill
    \includegraphics[width=0.43\linewidth]{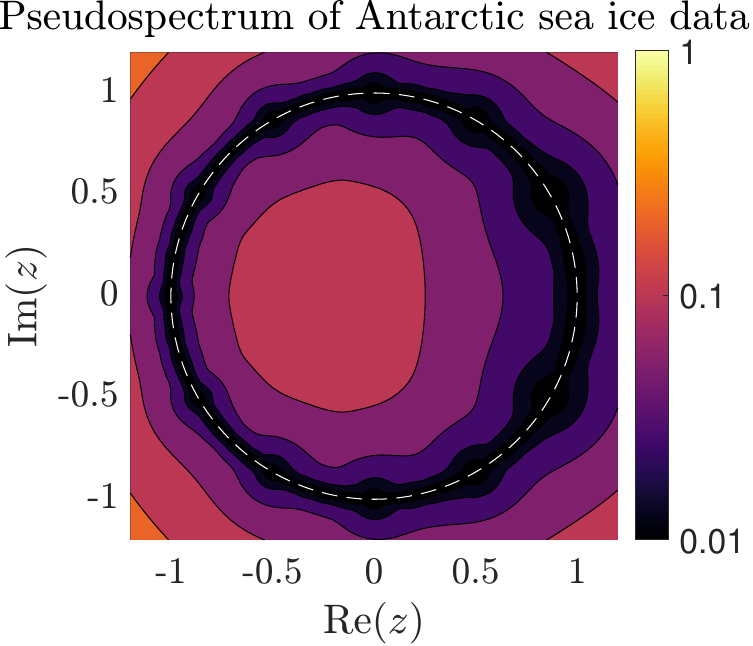}
    \caption{Antarctic sea ice data. Left: The eigenvalues outputted by kEDMD using a Mat{\'e}rn kernel are plotted based on their argument and (the log of) their modulus. The color coding shows the size of their residuals, as computed by SpecRKHS-Eig (\cref{evalverif_alg}). The dotted black lines correspond to the lines $\mathrm{arg}(\lambda)=\pi k/6$ for $k=-5,\dots,5$, and the $11$ dominant eigenvalues which we use for predictions are larger. Right: The plot shows pseudospectral contours for the same system using the same kernel, as computed by SpecRKHS-PseudoPF (\cref{pspecadjoint_alg}). The white dotted line shows the circle $\{z\in\mathbb{C}:|z|=1\}$.}
    \label{fig:antarcticresidualspspec}
\end{figure}

\begin{figure}[th!]
    \centering
    \includegraphics[width=0.32\linewidth]{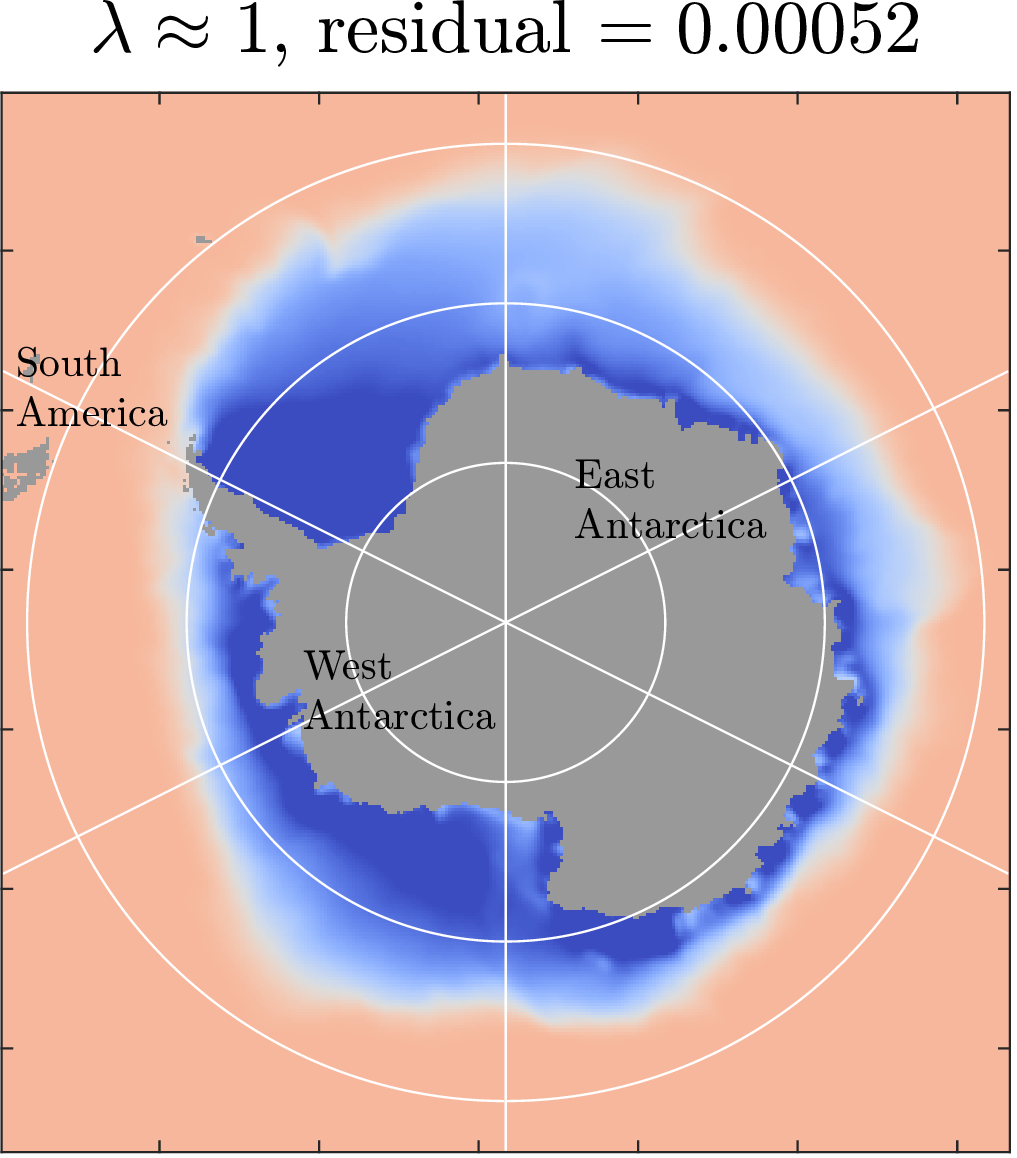}
    \includegraphics[width=0.32\linewidth]{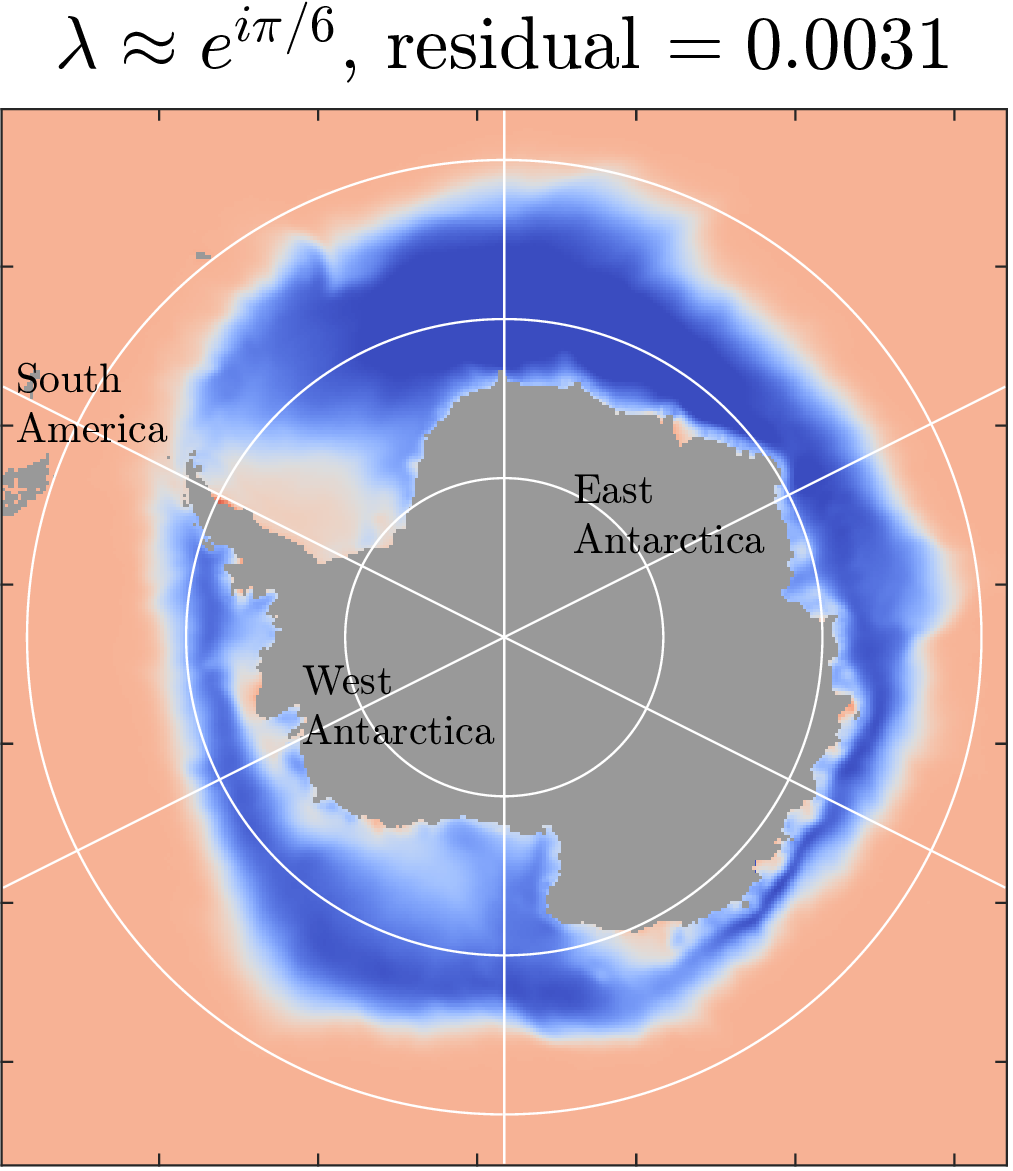}
    \includegraphics[width=0.32\linewidth]{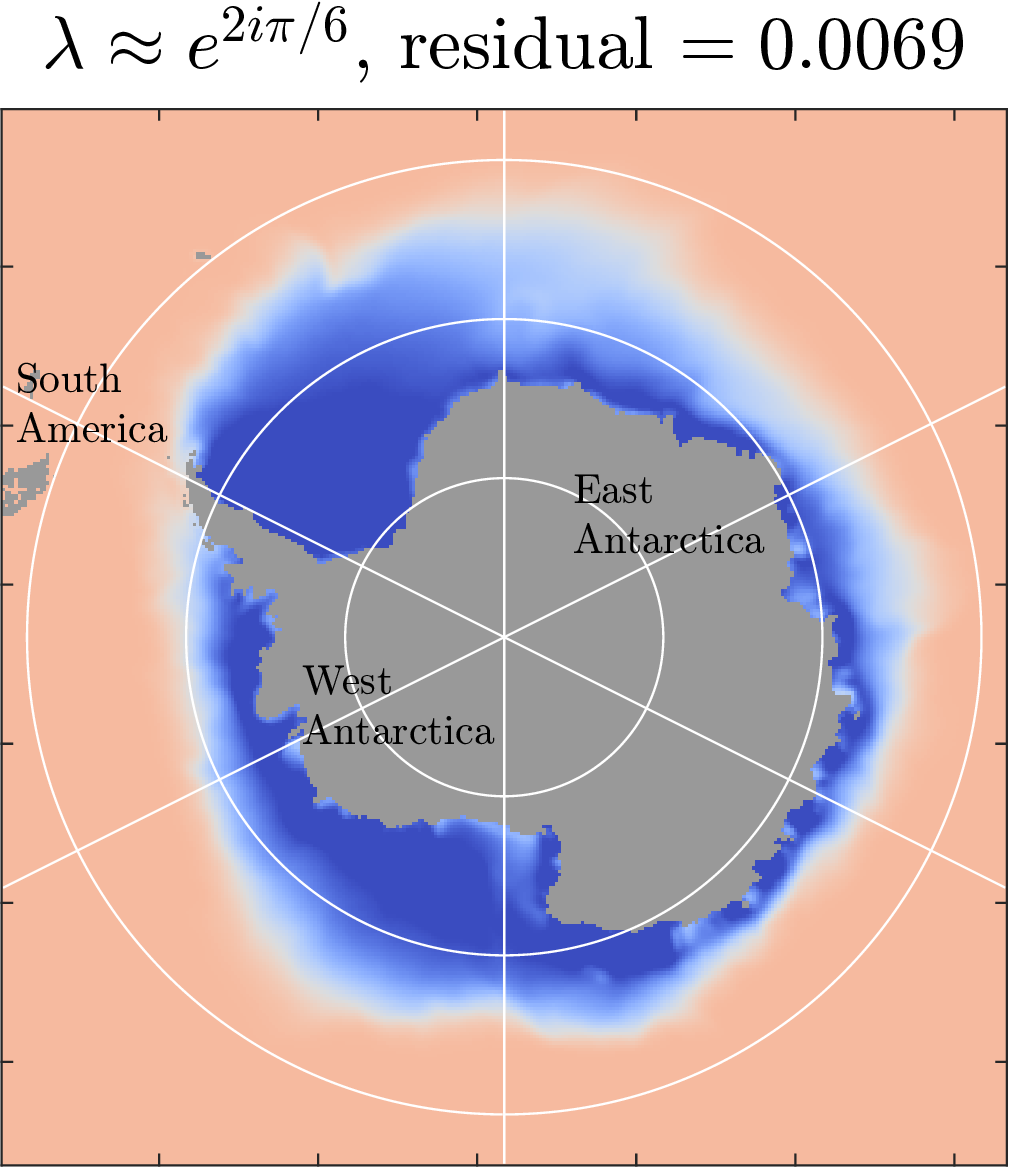}
    \caption{Antarctic sea ice data. Perron--Frobenius modes for Antarctic sea ice data, along with their residuals computed by SpecRKHS-Eig (\cref{evalverif_alg}).
    }
    \label{fig:antarcticpdfmodes}
\end{figure}

We now use monthly satellite data of sea ice concentration in the Southern Hemisphere on a grid of $25\times 25$ km regions from 1979 to 2023 inclusive \cite{hogg_exponentially_2020,sea_ice_data}. The dataset has dimension $d=82907$, and we take monthly samples across the 45 years considered, yielding $540$ snapshots. We remove the final 5 years of data to test our model's predictions and use the remaining years of data as training data to build the model. We use a Mat{\'e}rn kernel corresponding to an integer-valued Sobolev space:
\begin{equation}
    \label{eqn:maternkernelhalf}
    \mathfrak{K}(x,y)=\begin{cases}
        \sqrt{\pi/2},                                    & \text{if } x=y,   \\
        (\sigma\|x-y\|_2)^{3/2}K_{3/2}(\sigma\|x-y\|_2), & \text{otherwise}.
    \end{cases}
\end{equation}
The native space of this kernel is equivalent to the Sobolev space $H^{41454}(\mathbb{R}^{82907})$. Note that explicitly $K_{3/2}(x)=K_{-3/2}(x)=\sqrt{\pi/2}\,e^{-x}(x^{-1}+1)x^{-1/2}$, and here we take the scale factor to be $\sigma=1/10000$.

\begin{figure}[t]
    \centering
    \includegraphics[width=0.45\linewidth]{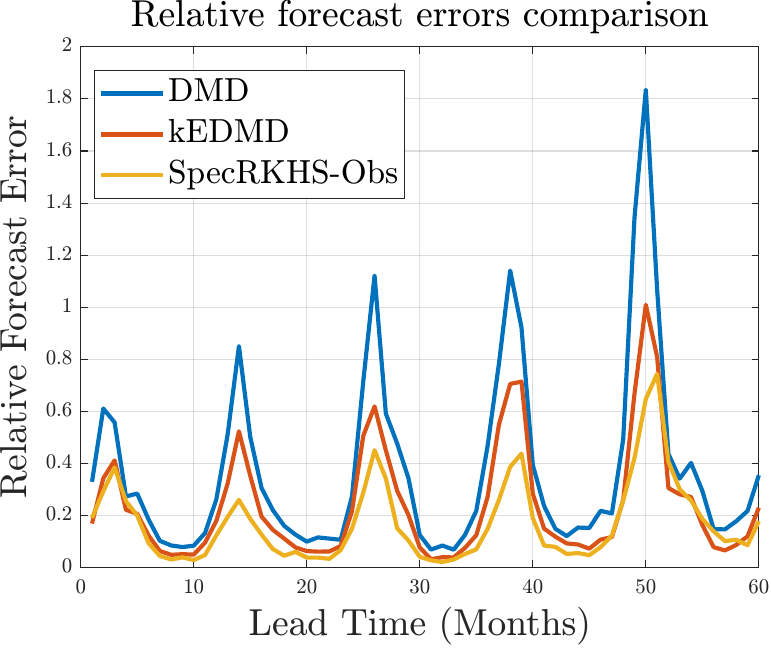}
    \caption{Antarctic sea ice data. Relative forecast errors for DMD, kEDMD, and SpecRKHS-Obs (\cref{alg_verifiedPF}), compared to the exact sea ice concentrations over a five-year period.}
    \label{fig:antarcticforecasterrors}
\end{figure}

\begin{figure}[th!]
    \centering
    \includegraphics[width=0.32\linewidth]{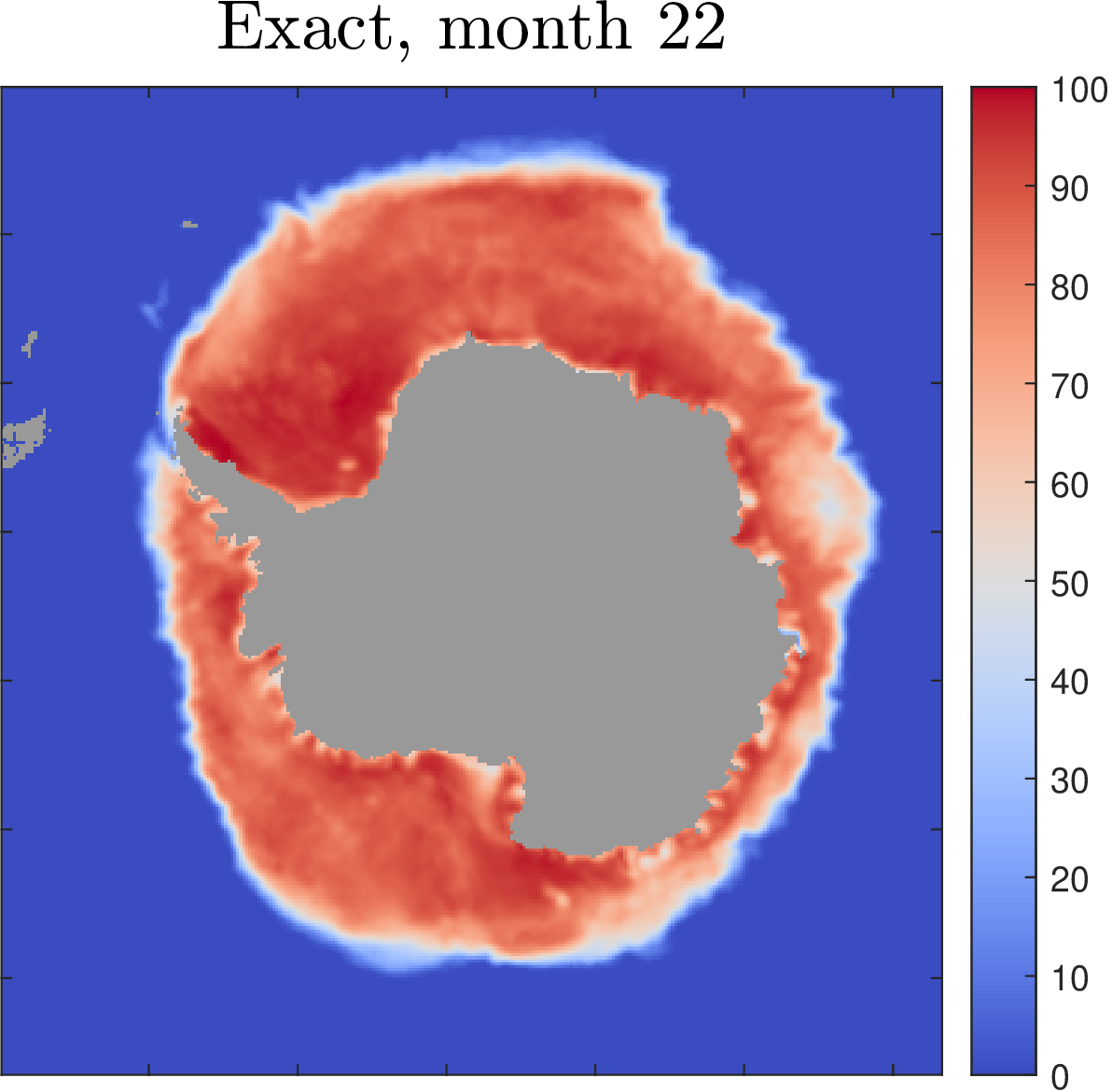}\hfill
    \includegraphics[width=0.32\linewidth]{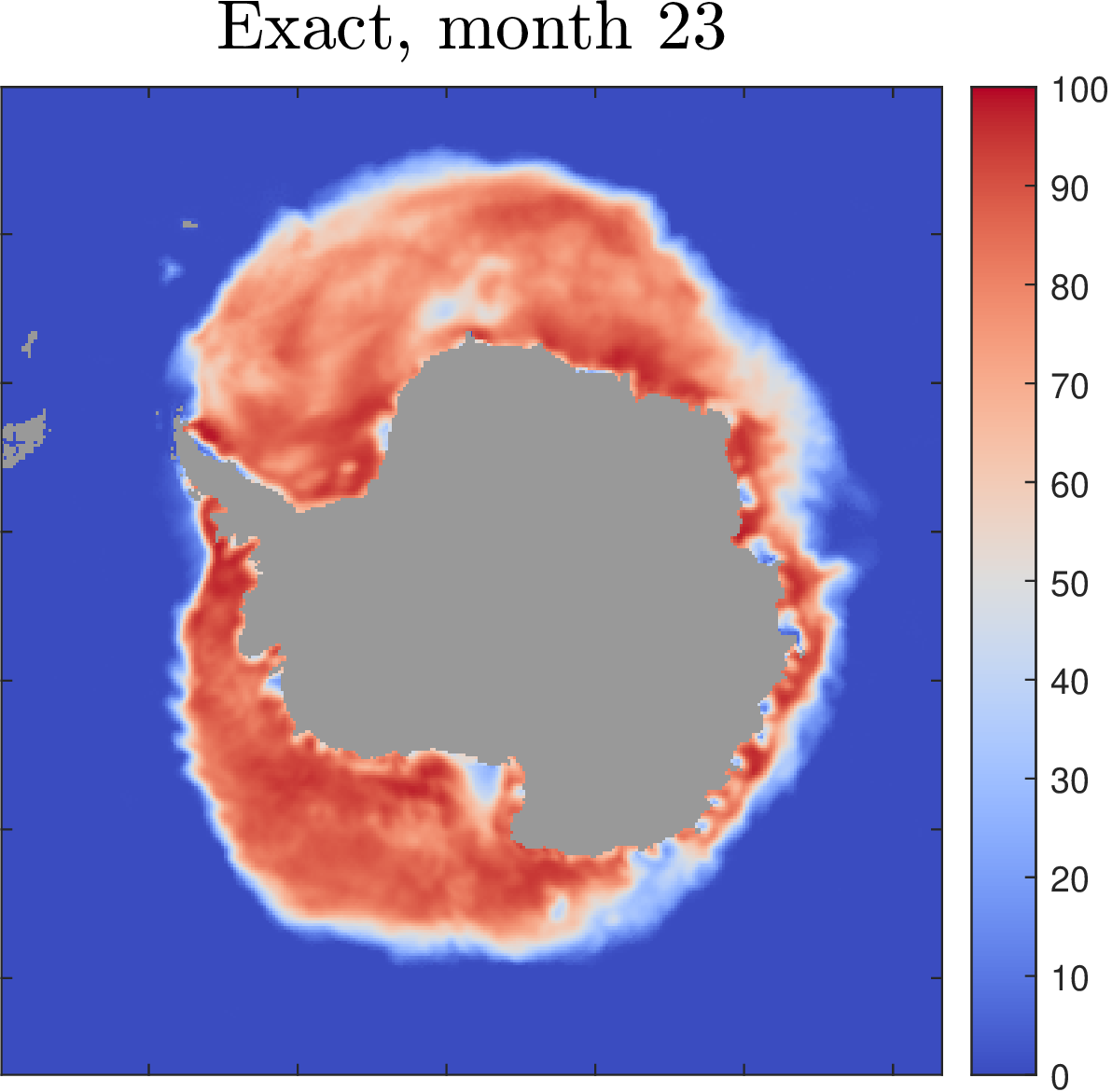}\hfill
    \includegraphics[width=0.32\linewidth]{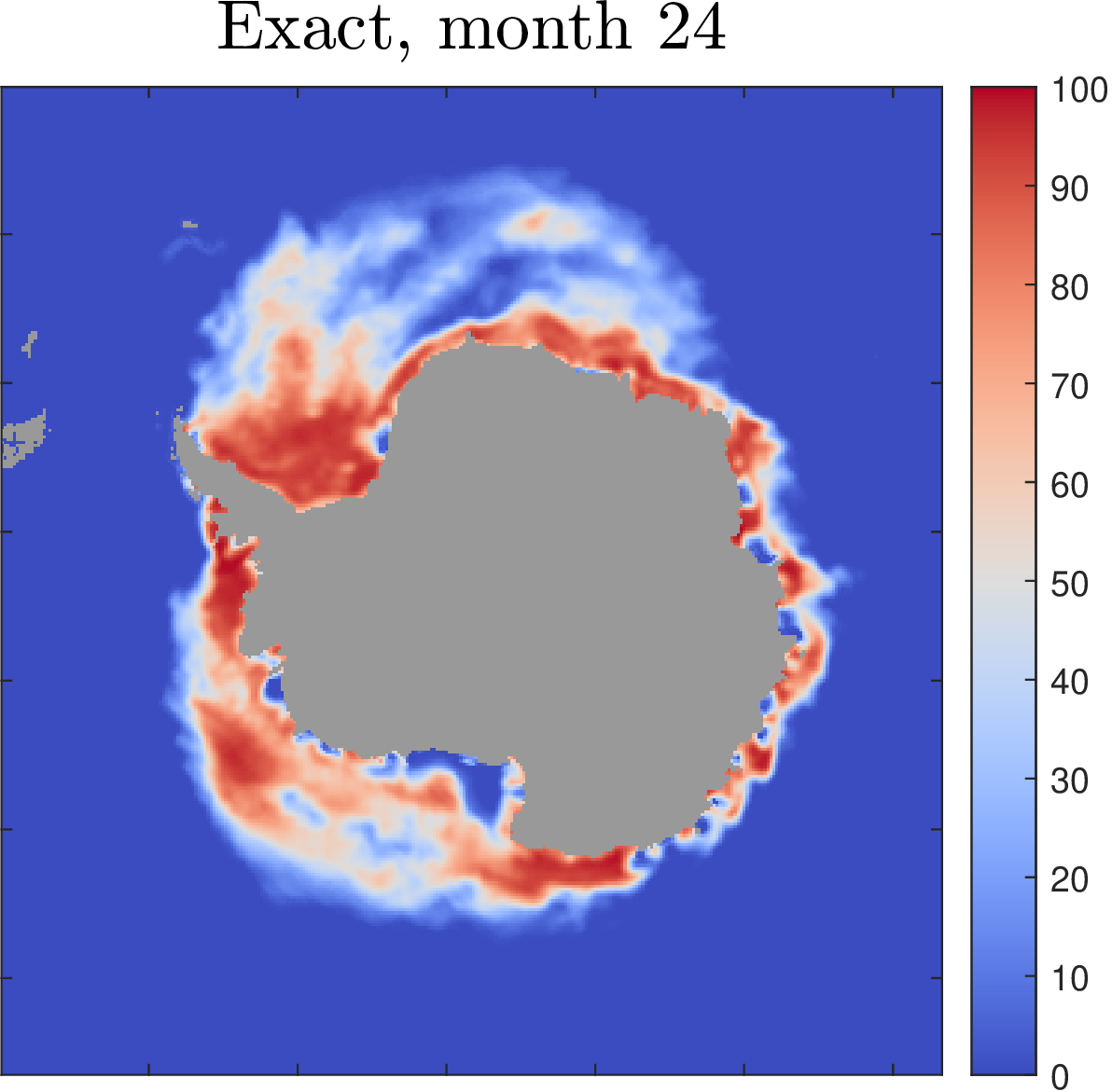}\\ \vspace{0.5cm}
    \includegraphics[width=0.32\linewidth]{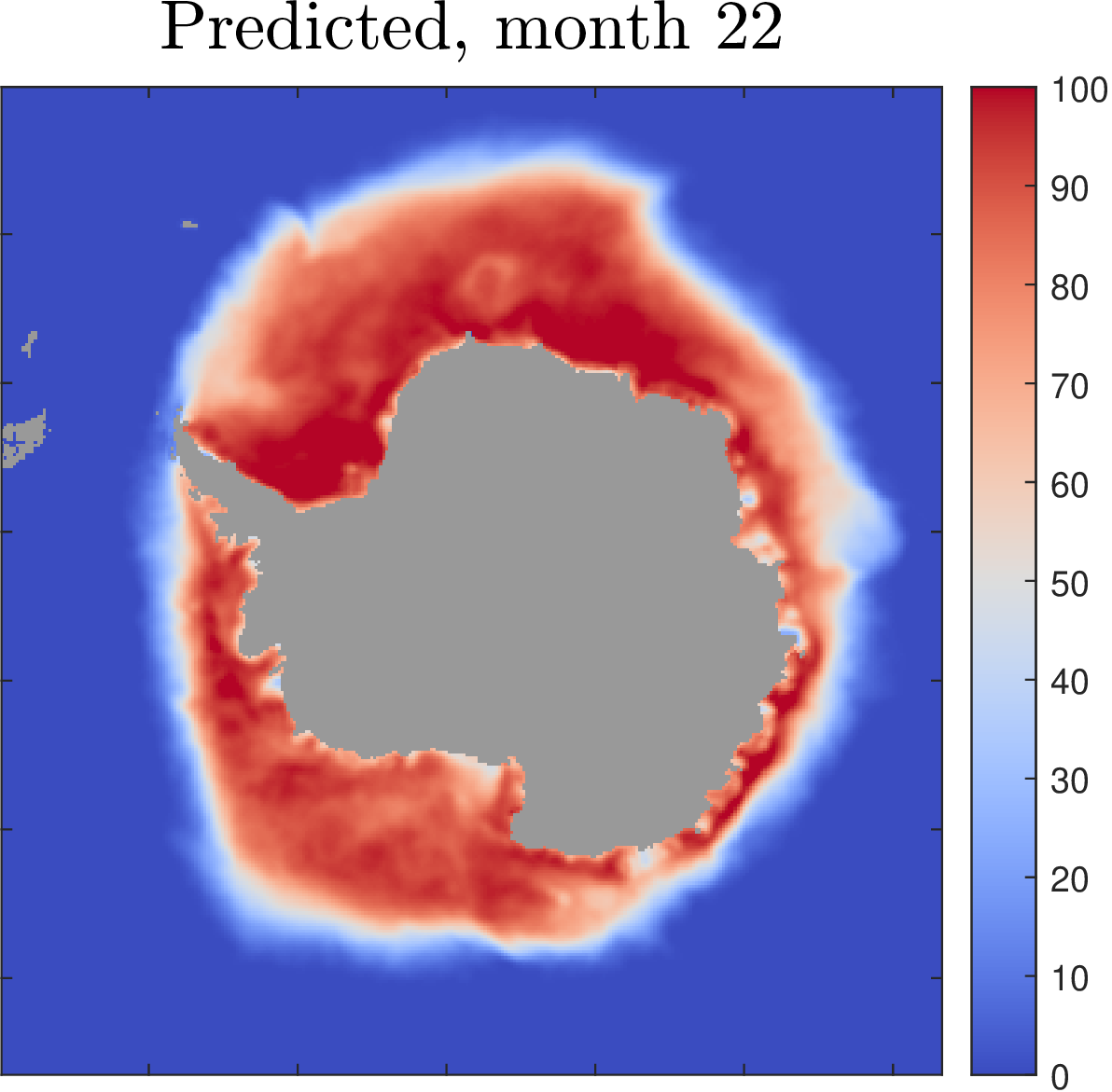}\hfill
    \includegraphics[width=0.32\linewidth]{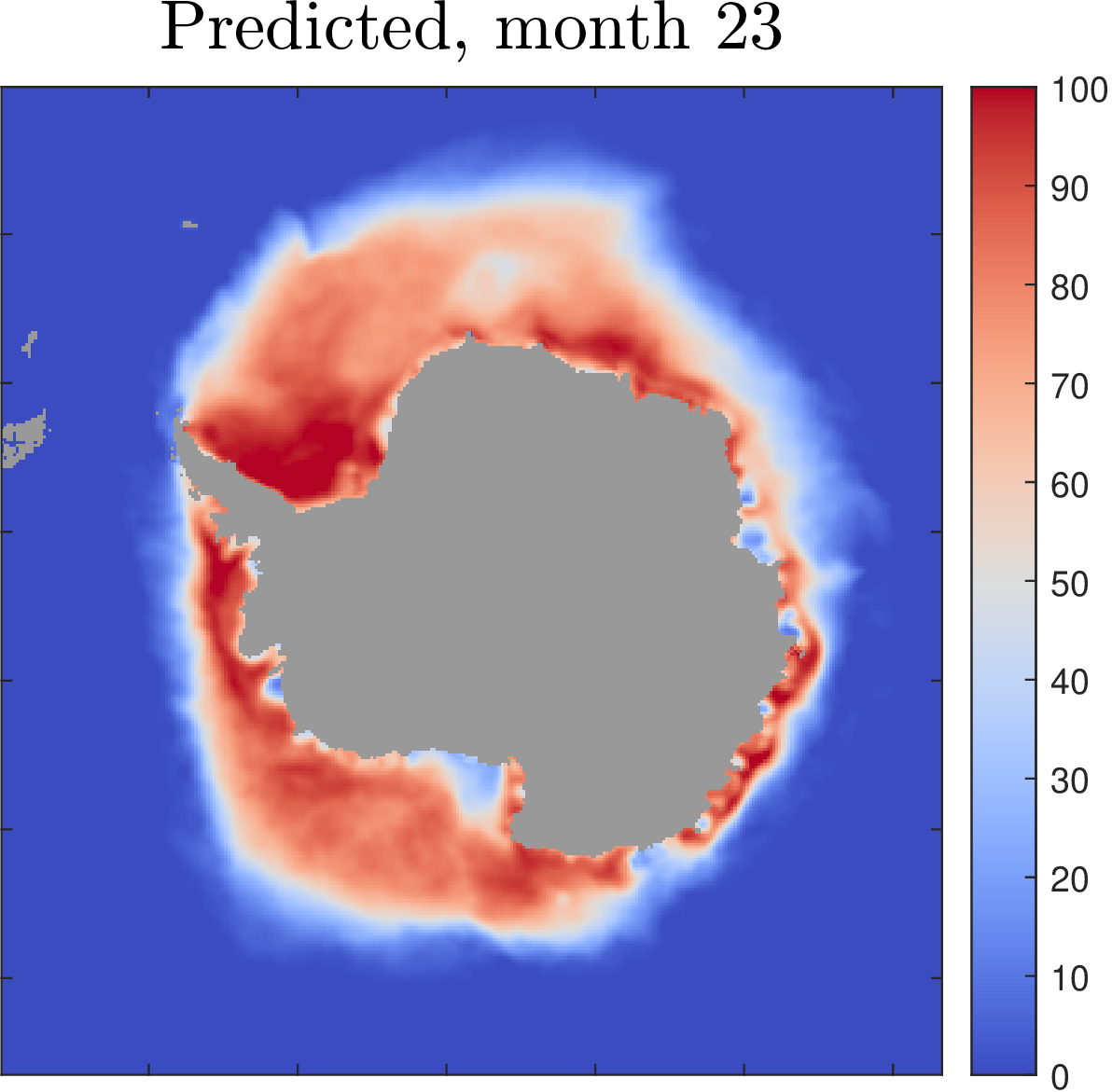}\hfill
    \includegraphics[width=0.32\linewidth]{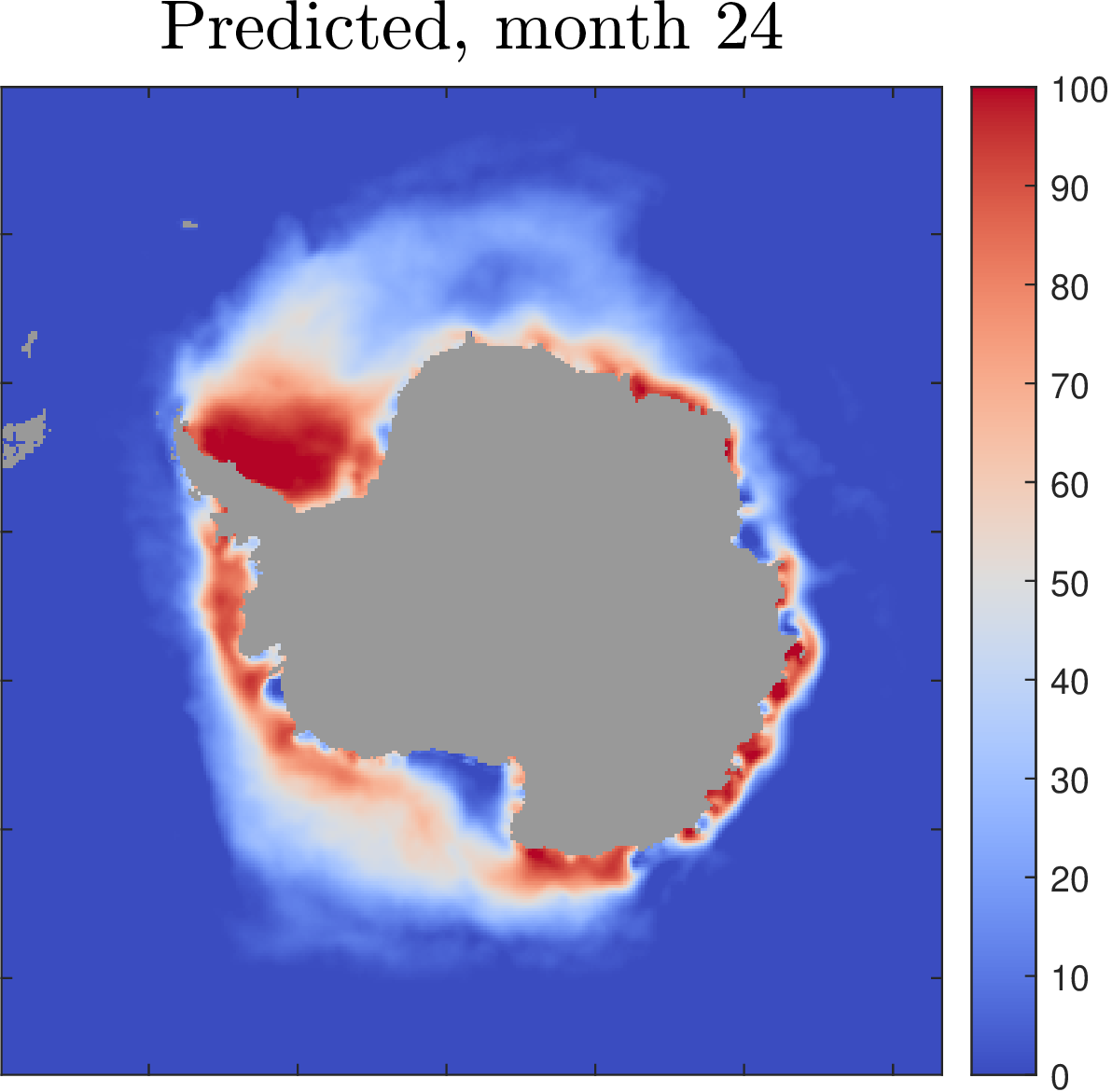}
    \caption{Antarctic sea ice data. Plots of exact and predicted Antarctic sea ice concentrations over a three-month period drawn from data the model was not trained on, using SpecRKHS-Obs (\cref{alg_verifiedPF}) applied to the state observables $g(x)=[x]_i$ for $i=1,\dots,d$.}
    \label{fig:antarcticpredictedexactplots}
\end{figure}

\cref{fig:antarcticresidualspspec} displays the outputs of SpecRKHS-Eig (\cref{evalverif_alg}) and SpecRKHS-PseudoPF (\cref{pspecadjoint_alg}) on this system. The left panel shows the kEDMD eigenvalues with residuals computed by SpecRKHS-Eig. Of the eleven eigenpairs with the smallest residuals (and hence highest coherence), nine have eigenvalues with complex arguments $\pi k/6$ for $k=-4,\dots,4$, i.e., they correspond to monthly variations of sea ice concentration, and all have magnitude less than $1$ which implies that they are decaying. The two remaining eigenpairs are more heavily decaying modes with small arguments corresponding to sea ice loss. We select these eleven eigenpairs because, beyond this point, the number of eigenpairs with residuals below any given threshold increases sharply. The eigenpairs with the lowest time variation are more dominant and have lower residuals. The pseudospectral plot on the right verifies that these are accurate approximate eigenvalues while suggesting that there may also be eigenvalues (though with larger residuals) at $\exp(i\pi k/6)$ for $k\in\{-5,5,6\}$.

In \cref{fig:antarcticpdfmodes}, we display the verified Perron--Frobenius modes outputted by SpecRKHS-Eig (\cref{evalverif_alg}). Since the modes corresponding to complex conjugate eigenvalues are the same, the nine dominant eigenpairs give rise to five unique dominant modes. The monthly modes involve both decaying and increasing sea ice concentrations, as would be expected due to seasonal variations.

Finally, we use the verified eigenvalues to perform model reduction on the KMD and predict future sea ice concentrations. We compare these predictions to the five years of test data on which we did not train the model. We consider the observables given by $g_i(x)=[x]_i$ for $i=1,\ldots,d$, and define $\m{g}(x)=(g_1(x)\,\cdots\, g_d(x))^T$. While these do not lie in the Sobolev space being considered, we can uniformly approximate them on any compact set in which we expect the predictions to lie, so there are no theoretical restrictions. Then, we can use the method of \cref{sec:errorcontrolpfmd} to predict the future evolution of these observables.

\cref{fig:antarcticforecasterrors} shows the relative forecast errors in our proposed method applied to the Antarctic sea ice dataset, compared to two standard methods: DMD and kEDMD. We significantly outperform DMD and kEDMD while also achieving significant model reduction, allowing our model to be evaluated significantly faster and identifying the dominant physical processes for deeper analysis. The kEDMD mode decomposition uses $450$ modes, but we use only the $11$ dominant modes extracted by computing residuals in our proposed method.
In \cref{fig:antarcticpredictedexactplots}, we plot the exact and predicted (using the model given by the verified Perron--Frobenius mode decomposition) distributions of Antarctic sea ice concentration over a three-month period (which is part of the test data and not the training data). Our method accurately captures the global variations in sea ice concentrations.

Our methods extracted the known seasonal cycle dynamics and provided superior predictive skills with a vastly simpler model. This demonstrates its practical value for complex geophysical time series.

\subsection{Northern Hemisphere sea surface height}
\label{sec:seaheight}

This final example examines the dynamics of the monthly mean sea surface height in the Northern Hemisphere. We use the OFES simulation of hindcast data from 1950 to the present day, which was conducted on the Earth Simulator with the support of JAMSTEC \cite{jamstec_2009}. This simulation is quasi-global, covering latitudes from $74.95$\textdegree S to $74.95$\textdegree N in intervals of $0.1$\textdegree (i.e., excluding arctic regions) and longitudes from $0.05$ to $359.95$, also in intervals $0.1$\textdegree. Since the Northern and Southern Hemispheres exhibit different behavior, we restrict latitudes from $0.05$\textdegree N to $74.95$\textdegree N. At each longitude-latitude pair, the simulation provides monthly data for the mean sea surface height, and we ignore any region that is not ocean. We use simulation data from 1995 to 2019 inclusive as training data and then data from 2020 to 2024 inclusive to test the accuracy of our Koopman-based model. This provides $300$ snapshots for training data, where each snapshot has $60330$ degrees of freedom.

We use the Mat{\'e}rn kernel
\[
    \mathfrak{K}(x,y)=\begin{cases}
        1,                                       & \text{if } x=y,   \\
        (\sigma\|x-y\|_2)^2K_2(\sigma\|x-y\|_2), & \text{otherwise},
    \end{cases}
\]
which corresponds to the Sobolev space $H^{30167}(\mathbb{R}^{60330})$, and select $\sigma=1/10000$.
\cref{fig:sealevelevals} shows residuals of kEDMD eigenpairs and pseudospectra computed using SpecRKHS-Eig (\cref{evalverif_alg}) and SpecRKHS-PseudoPF (\cref{pspecadjoint_alg}). In contrast to the previous section, the dominant eigenpairs now have eigenvalues with angles approximately $\pi k/6$ for $k\in\{-2,-1,0,1,2\}$. This implies that the Northern Hemisphere sea level height primarily has a semi-annual oscillation (6-month and 12-month cycles). In \cref{fig:sealevelpfmodes}, we plot the system's two dominant non-trivial Perron--Frobenius modes.

\begin{figure}[t]
    \centering
    \includegraphics[height=0.35\textwidth]{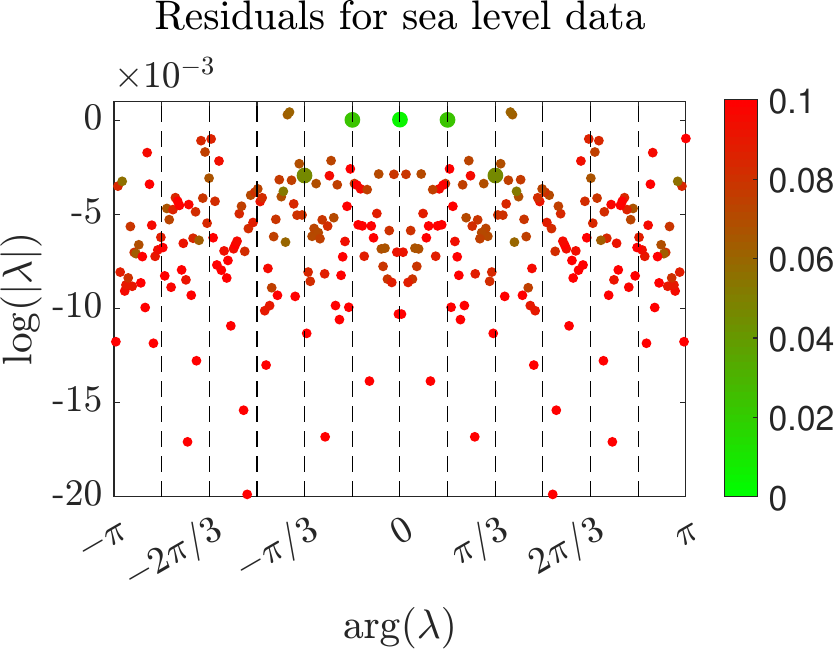}\hspace{1cm}
    \includegraphics[height=0.35\textwidth]{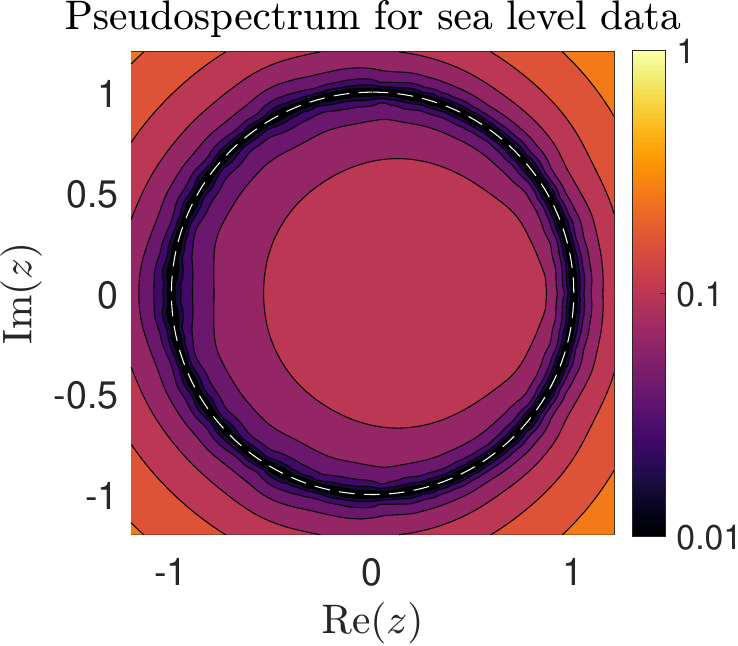}
    \caption{Northern Hemisphere sea surface height. Left: The eigenvalues outputted by kEDMD using a Mat{\'e}rn kernel are plotted based on their argument and (the log of) their modulus. The color coding shows the size of their residuals, as computed by SpecRKHS-Eig (\cref{evalverif_alg}). The dotted black lines correspond to the lines $\mathrm{arg}(\lambda)=\pi k/6$ for $k=-5,\dots,5$, and the $5$ dominant eigenvalues which we use for predictions are larger. Right: The plot shows pseudospectral contours for the same system using the same kernel, as computed by SpecRKHS-PseudoPF (\cref{pspecadjoint_alg}). The white dotted line shows the circle $\{z\in\mathbb{C}:|z|=1\}$.}
    \label{fig:sealevelevals}
\end{figure}

\begin{figure}[th!]
    \centering
    \includegraphics[height=0.2\textwidth]{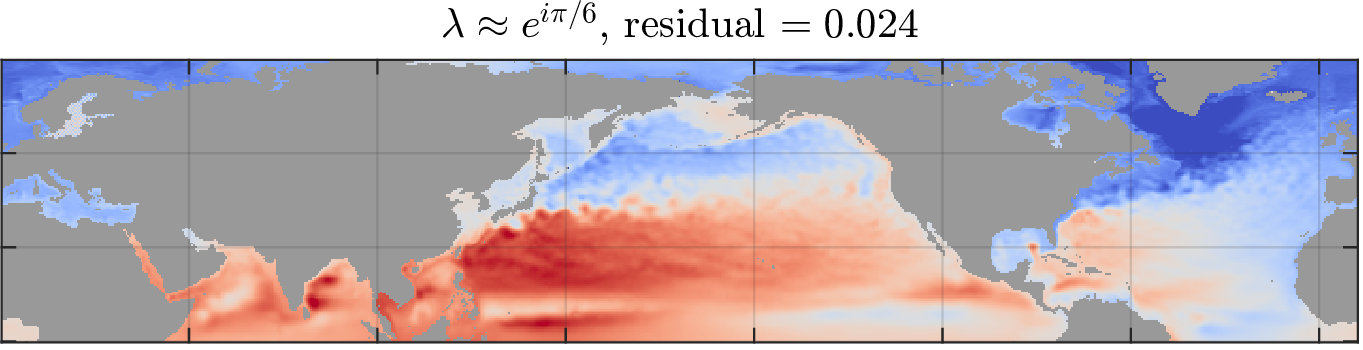}\vspace{0.5cm}
    \includegraphics[height=0.2\textwidth]{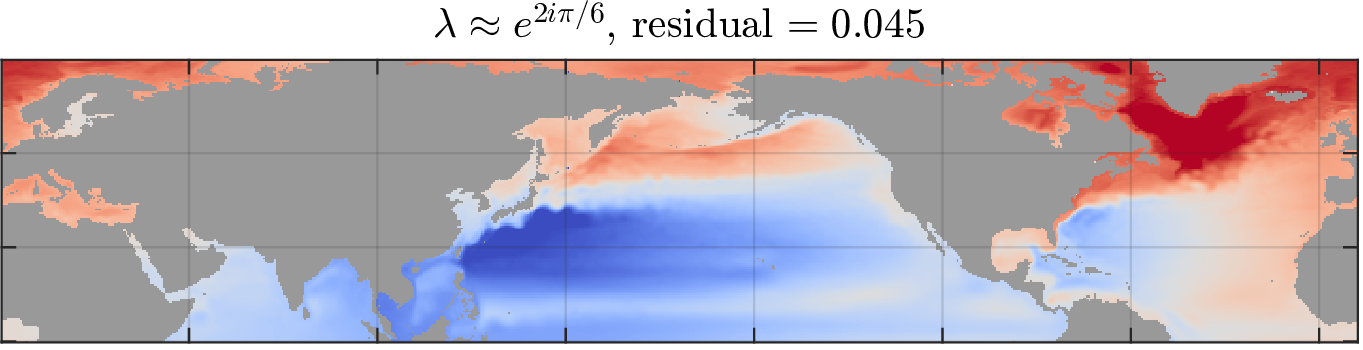}
    \caption{Northern Hemisphere sea surface height. Top: The Perron--Frobenius eigenvalue and mode corresponding to the first verified, non-trivial eigenvalue for the Northern Hemisphere sea level heights system, which has argument approximately $\pi/6$. Bottom: The Perron--Frobenius eigenvalue and mode corresponding to the first verified, non-trivial eigenvalue for the Northern Hemisphere sea level heights system, which has argument approximately $2\pi/6$. In both cases the latitude ranges from $74.95$\textdegree S to $74.95$\textdegree N the longitude from $0.05$ to $359.95$, both in intervals of $0.1$\textdegree.}
    \label{fig:sealevelpfmodes}
\end{figure}

Next, using the verified eigenpairs shown in \cref{fig:sealevelevals} and the arguments of \cref{sec:errorcontrolpfmd}, we compute predictions of future Northern Hemisphere sea level heights and compare them to the simulated values. As before, we compare the relative errors of our method to those of DMD and kEDMD, and the results are shown in \cref{fig:sealevelerror}. Our RKHS approach achieved dramatically better  forecasts than standard DMD or kEDMD. Moreover, our model uses only 5 modes versus 299 for kEDMD – a 98\% reduction in model complexity.

\begin{figure}[t]
    \centering
    \includegraphics[width=0.45\linewidth]{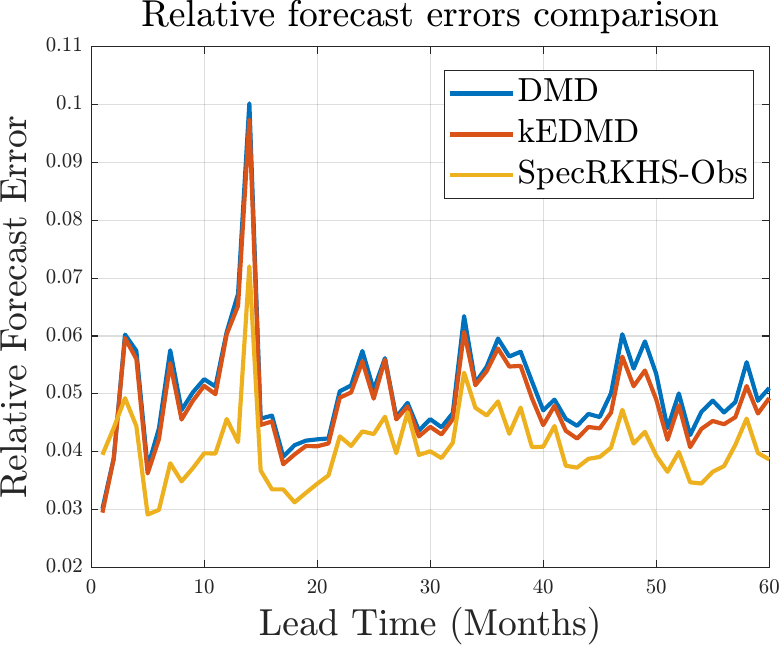}
    \caption{Relative forecast errors for DMD, kEDMD, and SpecRKHS-Obs (\cref{alg_verifiedPF}), compared to the exact Northern Hemisphere sea level heights over five years.}
    \label{fig:sealevelerror}
\end{figure}
\begin{figure}[t]
    \centering
    \includegraphics[height=0.25\linewidth]{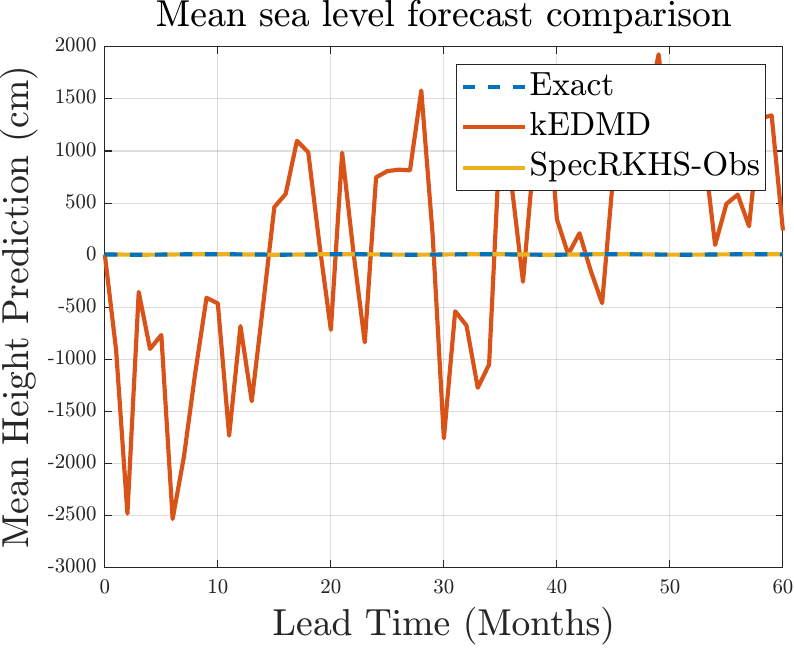}\hfill
    \includegraphics[height=0.25\linewidth]{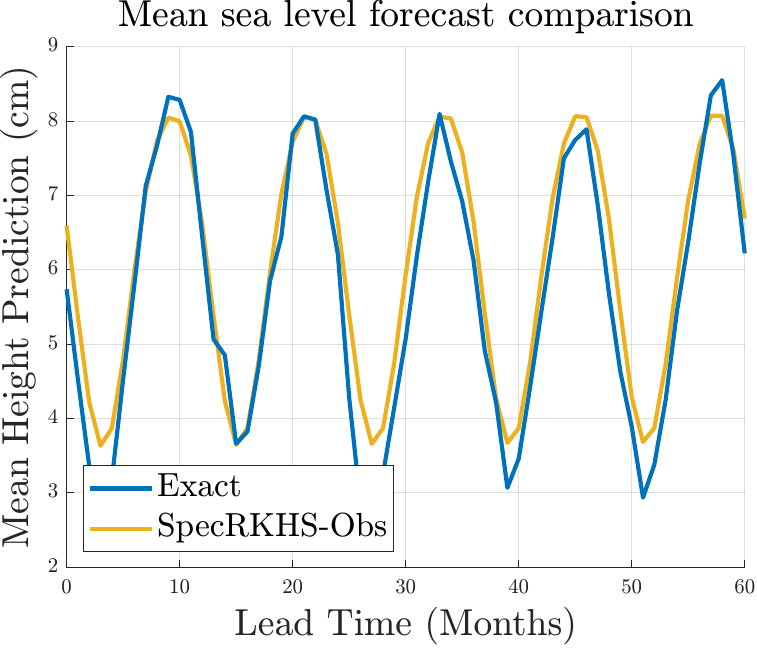}\hfill
    \includegraphics[height=0.25\linewidth]{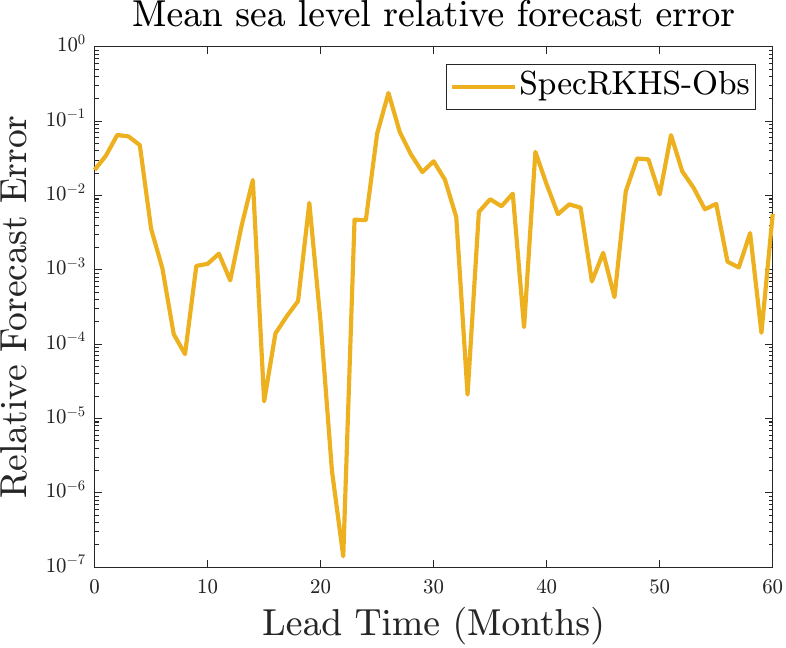}
    \caption{Left: A comparison of predictions between exact simulation data, kEDMD, and SpecRKHS-Obs (\cref{alg_verifiedPF}). Middle: The same plot as on the left but with kEDMD removed. Right: Relative forecast error between SpecRKHS-Obs and the exact mean sea level values over five years.}
    \label{fig:sealevelmeanpressurepred}
\end{figure}

\begin{figure}[th!]
    \centering
    \includegraphics[height=0.35\textwidth]{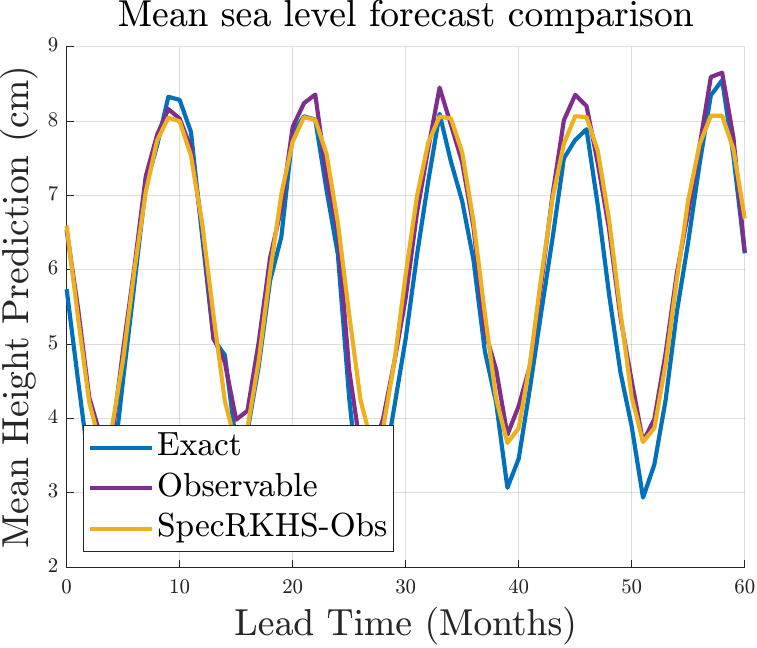}\hspace{1cm}
    \includegraphics[height=0.35\textwidth]{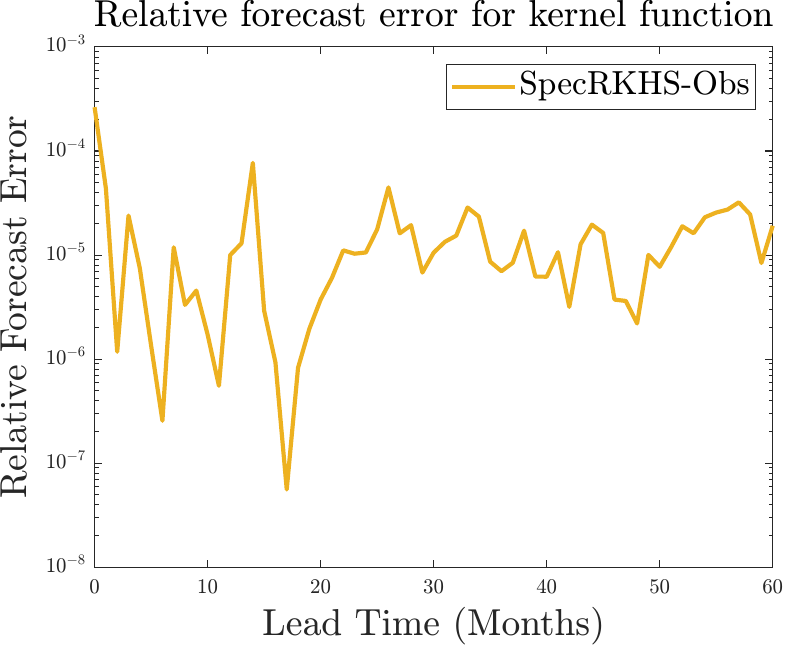}
    \caption{Left: An examination of the sources of error when predicting future trajectories using SpecRKHS-Obs (\cref{alg_verifiedPF}). Right: The relative forecast error between exact values of $\mathfrak{K}_{x_0}$ and approximations of $\mathfrak{K}_{x_0}$ by a basis of verified eigenfunctions.}
    \label{fig:sealevelobs}
\end{figure}

Finally, the Koopman approach also allows for accurate prediction of observables besides the states themselves. For example, we may be interested in the mean sea level height across the Northern Hemisphere. Again, using \cref{sec:errorcontrolpfmd}, we compute an expansion of the mean sea level height observable with respect to the approximate eigenbasis by solving a least squares problem and using this to form the KMDs. We plot the future forecasts for kEDMD and the proposed method and compare them to the exact values of the mean sea level heights on the left-hand side and middle of \cref{fig:sealevelmeanpressurepred}. While the proposed method accurately captures the six-month variation from highest to lowest mean sea level height, the kEDMD predictions are flawed with several overly high magnitude Koopman modes. A different choice of kernel can improve the kEDMD predictions, but this shows that our proposed method is much more robust to the choice of kernel due to the eigenvalue verification procedure. The right-hand side of \cref{fig:sealevelmeanpressurepred} shows the relative error in the predictions of our method.

In addition to the two components of the error mentioned in \cref{sec:errorcontrolpfmd}, a third arises because we do not know the mean sea level observable exactly. The first component arises when approximating (in a least-squares sense) the kernel function at the starting value ${x_0}$ in the eigenbasis (the term involving $\delta$ in \cref{error_pfmd}), the second from the evolution of the verified eigenfunctions themselves (the term involving $\epsilon$ in \cref{error_pfmd}) and the third from errors in expanding the mean sea level observable in terms of kernel function. The left-hand side of \cref{fig:sealevelobs} shows the exact value of the mean sea level height, the exact evolution of the eigenbasis expanded observable approximating the mean sea level height, and the evolution of this observable according to the proposed method. Most of the errors come from the third component, i.e., being unable to express the mean sea level observable exactly, rather than the first two. In particular, the right-hand side of \cref{fig:sealevelobs} shows the relative forecast error between the exact values of observable $\mathfrak{K}_{x_0}$ on the data and approximations of $\mathfrak{K}_{x_0}$ using the verified Perron--Frobenius mode decomposition as discussed in \cref{sec:errorcontrolpfmd}; we see that this error is small, reinforcing the claim of the previous sentence.

\section*{Acknowledgements}
This work was supported by the Office of Naval Research (ONR) under grant N00014-23-1-2729. The authors would like to thank James Hogg for useful discussions about the Antarctic sea ice dataset.

\addcontentsline{toc}{section}{References}
\small
\linespread{0.9}
\bibliography{RKHS}
\bibliographystyle{abbrv}
\end{document}